\documentclass[11pt, reqno]{amsart}

\usepackage{amssymb,amsmath,amsthm,amsfonts,amscd,bbm}
\usepackage{enumitem}
\usepackage{pdflscape}
\usepackage{caption}
\usepackage{bm}
\usepackage{comment}

\usepackage{ifpdf}
\ifpdf
\usepackage[pdftex]{graphicx}
\else
\usepackage[dvips]{graphicx}
\fi
\usepackage[all]{xy}
\usepackage{hyperref}
\usepackage{xcolor}
\usepackage{graphicx}
\usepackage[outdir=./]{epstopdf}
\usepackage{enumitem}
\usepackage{tensor}
\usepackage{tikz-cd}
\usepackage{multicol}
\usepackage{mathtools}
\usepackage{tocvsec2}
\usepackage{bbm}
\usepackage{longtable}
\usepackage{fullpage}

\newcommand{\doi}[1]{\href{https://dx.doi.org/#1}{{\small DOI:#1}}}

\newcommand{\googlebooks}[1]{(preview at \href{https://books.google.com/books?id=#1}{google books})}

\newcommand{\numdam}[1]{}

\usepackage{mathrsfs}
\DeclareMathAlphabet{\mathpzc}{OT1}{pzc}{m}{it}

\def\semicolon{;}
\def\applytolist#1{
    \expandafter\def\csname multi#1\endcsname##1{
        \def\multiack{##1}\ifx\multiack\semicolon
            \def\next{\relax}
        \else
            \csname #1\endcsname{##1}
            \def\next{\csname multi#1\endcsname}
        \fi
        \next}
    \csname multi#1\endcsname}

\def\calc#1{\expandafter\def\csname c#1\endcsname{{\mathcal #1}}}
\applytolist{calc}QWERTYUIOPLKJHGFDSAZXCVBNM;
\def\bbc#1{\expandafter\def\csname bb#1\endcsname{{\mathbb #1}}}
\applytolist{bbc}QWERTYUIOPLKJHGFDSAZXCVBNM;
\def\bfc#1{\expandafter\def\csname bf#1\endcsname{{\mathbf #1}}}
\applytolist{bfc}QWERTYUIOPLKJHGFDSAZXCVBNM;
\def\sfc#1{\expandafter\def\csname s#1\endcsname{{\sf #1}}}
\applytolist{sfc}QWERTYUIOPLKJHGFDSAZXCVBNM;
\def\fc#1{\expandafter\def\csname f#1\endcsname{{\mathfrak #1}}}
\applytolist{fc}QWERTYUIOPLKJHGFDSAZXCVBNM;

\usepackage{tikz}
\usepackage{tikz-cd}

\makeatletter
\def\fixtikzforbreqn#1#2{%
  \protected\edef#1{\noexpand\ifmmode\mathchar\the\mathcode`#2 \noexpand\else#2\noexpand\fi}%
}
\fixtikzforbreqn\tikz@nonactivesemicolon;
\fixtikzforbreqn\tikz@nonactivecolon:
\fixtikzforbreqn\tikz@nonactivebar|
\fixtikzforbreqn\tikz@nonactiveexlmark!
\makeatother

\usepackage{breqn}   % Fix : | ; ! in tikzcd. And must put makealetter-other between tikzcd package and breqn package

\usetikzlibrary{arrows,backgrounds,patterns.meta}
\usetikzlibrary{positioning,shadings,cd}
\usetikzlibrary{shapes}
\usetikzlibrary{backgrounds}
\usetikzlibrary{decorations,decorations.pathreplacing,decorations.markings,decorations.pathmorphing}
\usetikzlibrary{fit,calc,through}
\usetikzlibrary{external}
\usetikzlibrary{arrows}
\tikzset{vertex/.style = {shape=circle,draw,fill=black,inner sep=0pt,minimum size=5pt}}
\tikzset{edge/.style = {->,> = latex', bend right}}
\tikzset{
	super thick/.style={line width=3pt}
}
\tikzset{
    quadruple/.style args={[#1] in [#2] in [#3] in [#4]}{
        #1,preaction={preaction={preaction={draw,#4},draw,#3}, draw,#2}
    }
}
\tikzstyle{shaded}=[fill=red!10!blue!20!gray!30!white]
\tikzstyle{unshaded}=[fill=white]
\tikzstyle{empty box}=[circle, draw, thick, fill=white, opaque, inner sep=2mm]
\tikzstyle{annular}=[scale=.7, inner sep=1mm, baseline]
\tikzstyle{rectangular}=[scale=.75, inner sep=1mm, baseline=-.1cm]
\tikzstyle{mid>}=[decoration={markings, mark=at position 0.5 with {\arrow{>}}}, postaction={decorate}]
\tikzstyle{mid<}=[decoration={markings, mark=at position 0.5 with {\arrow{<}}}, postaction={decorate}]
\tikzstyle{over}=[double, draw=white, super thick, double=]
\tikzstyle{snake}=[decorate, decoration={snake, segment length=1mm, amplitude=.3mm}]
\tikzstyle{saw}=[decorate, decoration={saw, segment length=.7mm, amplitude=.25mm}]
\tikzstyle{knot}=[preaction={super thick, white, draw}]

\tikzstyle{coupon}=[draw, very thick, rectangle, rounded corners=5pt]
\tikzset{Rightarrow/.style={double equal sign distance,>={Implies},->},
triplecd/.style={-,preaction={draw,Rightarrow}},
quadruplecd/.style={preaction={draw,Rightarrow,
shorten >=0pt
},
shorten >=1pt,
-,double,double
distance=0.2pt}}
\tikzset{
    tripleline/.style args={[#1] in [#2] in [#3]}{
        #1,preaction={preaction={draw,#3},draw,#2}
    }
}
\tikzstyle{triple}=[tripleline={[line width=.15mm,black] in
      [line width=.7mm,white] in
      [line width=1mm,black]}] 
\tikzset{
    quadrupleline/.style args={[#1] in [#2] in [#3] in [#4]}{
        #1,preaction={preaction={preaction={draw,#4},draw,#3}, draw,#2}
    }
}
\tikzstyle{quadruple}=[quadrupleline={[line width=.3mm,white] in
      [line width=.6mm,black] in
      [line width=1.2mm,white] in
      [line width=1.5mm,black]}]
      
\newcommand{\roundNbox}[6]{
	\draw[rounded corners=5pt, very thick, #1] ($#2+(-#3,-#3)+(-#4,0)$) rectangle ($#2+(#3,#3)+(#5,0)$);
	\coordinate (ZZa) at ($#2+(-#4,0)$);
	\coordinate (ZZb) at ($#2+(#5,0)$);
	\node at ($1/2*(ZZa)+1/2*(ZZb)$) {#6};
}

\newcommand{\tikzmath}[2][]
     {\vcenter{\hbox{\begin{tikzpicture}[#1]#2
                     \end{tikzpicture}}}
     }

\theoremstyle{plain}
\newtheorem{thm}{Theorem}[section]
\newtheorem*{thm*}{Theorem}
\newtheorem{thmalpha}{Theorem}

\newtheorem{cor}[thm]{Corollary}
\newtheorem{coralpha}[thmalpha]{Corollary}
\newtheorem*{cor*}{Corollary}

\newtheorem*{conj*}{Conjecture}
\newtheorem{lem}[thm]{Lemma}
\newtheorem*{lem*}{Lemma}

\newtheorem{prop}[thm]{Proposition}
\newtheorem{quest}[thm]{Question}

\newtheorem*{quest*}{Question}
\newtheorem*{claim*}{Claim}

\theoremstyle{definition}
\newtheorem{defn}[thm]{Definition}

\newtheorem{facts}[thm]{Facts}
\newtheorem{construction}[thm]{Construction}

\newtheorem{assumption}[thm]{Assumption}
\newtheorem{nota}[thm]{Notation}

\newtheorem{ex}[thm]{Example}
\newtheorem{sub-ex}[thm]{Sub-Example}
\newtheorem{counter-ex}[thm]{Counter-Example}
\newtheorem{rem}[thm]{Remark}
\newtheorem*{rem*}{Remark}

\newtheorem{warn}[thm]{Warning}  
     
\usepackage{xcolor}
\definecolor{dark-red}{rgb}{0.7,0.25,0.25}
\definecolor{dark-blue}{rgb}{0.15,0.15,0.55}
\definecolor{medium-blue}{rgb}{0,0,.8}
\definecolor{DarkGreen}{RGB}{0,150,0}
\definecolor{rho}{named}{red}
\hypersetup{
   colorlinks, linkcolor={purple},
   citecolor={medium-blue}, urlcolor={medium-blue}
}

\newcommand{\NColor}{orange}
\newcommand{\XColor}{red} % module string color
\newcommand{\YColor}{blue}
\newcommand{\ZColor}{violet}
\newcommand{\QsColor}{black}
\newcommand{\zColor}{blue} % object string color
\newcommand{\wColor}{red}
\newcommand{\HColor}{gray!75} % Hilbert space color
\newcommand{\rColor}{gray!20} % M region shading
\newcommand{\QrColor}{gray!50} % Q region shading

\newcommand{\alt}{\operatorname{alt}}
\newcommand{\id}{\operatorname{id}}

\newcommand{\Inv}{\operatorname{Inv}}

\newcommand{\Irr}{\operatorname{Irr}}

\newcommand{\ONB}{\operatorname{ONB}}

\newcommand{\Forget}{\operatorname{Forget}}

\newcommand{\op}{\operatorname{op}}

\newcommand{\tr}{\operatorname{tr}}
\newcommand{\Tr}{\operatorname{Tr}}
\newcommand{\Hom}{\operatorname{Hom}}

\newcommand{\End}{\operatorname{End}}
\newcommand{\Out}{\operatorname{Out}}
\newcommand{\loc}{\operatorname{loc}}
\newcommand{\ai}{\operatorname{ai}}
\newcommand{\ct}{\operatorname{ct}}

\newcommand{\Aut}{\operatorname{Aut}}

\newcommand{\Ad}{\operatorname{Ad}}

\newcommand{\co}{\overline{\operatorname{co}}}
\newcommand{\Chi}{\tilde{\chi}}
\newcommand{\Add}{\mathsf{Add}}
\newcommand{\Proj}{\mathsf{Proj}}

\newcommand{\fgpMod}{\mathsf{Mod_{fgp}}}
\newcommand{\Bim}{\mathsf{Bim}}

\newcommand{\Vect}{{\mathsf{Vect}}}

\newcommand{\Fun}{\mathsf{Fun}}
\newcommand{\fdHilb}{\mathsf{Hilb_{fd}}}

\newcommand{\fgpBim}{\mathsf{Bim_{fgp}}}

\newcommand{\WStarRCorr}{\mathsf{W^*Alg}}

\newcommand{\fgpWStarRCorr}{\mathsf{W^*Alg_{fgp}}}

\newcommand{\locEnd}{{\End}_{\loc}}
\newcommand{\ctEnd}{{\End}_{\ct}}
\newcommand{\aiEnd}{{\End}_{\ai}}
\newcommand{\set}[2]{\left\{#1 \middle| #2\right\}}

\def\altdb{\vadjust{\vbox to 0pt{\vss\hbox{\kern \hsize
\quad{\dbend}}\kern\baselineskip\kern-10pt}}}

\newcommand{\noshow}[1]{}
\renewcommand{\MR}[1]{}

\title{A categorical Connes' $\chi(M)$}
\author{Quan Chen, Corey Jones, and David Penneys}

\begin{document}

\begin{abstract}
Popa introduced the tensor category $\tilde{\chi}(M)$ of approximately inner, centrally trivial bimodules of a $\rm{II}_{1}$ factor $M$, generalizing Connes' $\chi(M)$. 
We extend Popa's notions to define the $\rm W^*$-tensor category $\operatorname{End}_{\rm loc}(\mathcal{C})$ of local endofunctors on a $\rm W^*$-category $\mathcal{C}$.
We construct a unitary braiding on $\operatorname{End}_{\rm loc}(\mathcal{C})$, giving a new construction of a braided tensor category associated to an arbitrary $\rm W^*$-category.
For the $\rm W^*$-category of finite modules over a $\rm{II}_{1}$ factor, this yields a unitary braiding on Popa's $\tilde{\chi}(M)$, which extends Jones' $\kappa$ invariant for $\chi(M)$. 

Given a finite depth inclusion $M_{0}\subseteq M_{1}$ of non-Gamma $\rm{II}_1$ factors, we show that the braided unitary tensor category $\tilde{\chi}(M_{\infty})$ is equivalent to the Drinfeld center of the standard invariant, where $M_{\infty}$ is the inductive limit of the associated Jones tower. 
This implies that for any pair of finite depth non-Gamma subfactors $N_{0}\subseteq N_{1}$ and $M_{0}\subseteq M_{1}$, if the standard invariants are not Morita equivalent, then the inductive limit factors $N_{\infty}$ and $M_{\infty}$ are not stably isomorphic.

\end{abstract}

\maketitle
\tableofcontents

\section{Introduction}

Tensor categories have come to play an important role in noncommutative analysis, arising as categories of bimodules of $\rm C^*$ and von Neumann algebras and as representation categories of compact quantum groups. 
In subfactor theory, the standard invariant of a finite index $\rm II_1$ subfactor \cite{MR1334479,math.QA/9909027} can be described by a unitary tensor category (a.k.a.~a semisimple rigid $\rm C^*$-tensor category), together with a chosen unitary Frobenius algebra object internal to this category \cite{MR1966524,MR1027496,MR2909758,MR3948170}. 
In operator algebraic approaches to quantum field theories (AQFT) and topologically ordered spin systems, braided tensor categories arise in the DHR theory of superselection sectors of nets of von Neumann algebras \cite{MR0297259, MR1405610, MR1838752, MR3426207, arXiv:2106.15741}. 
Here, the presence of a braiding yields an incredibly rich structure theory which does not have an obvious analog in the purely `noncommutative' world of ordinary subfactors.

There are, however, less widely recognized instances of braided tensor categories arising in the theory of $\rm{II}_{1}$ factors. 
In \cite{MR377534}, Connes introduced an invariant $\chi(M)$ of a $\rm{II}_{1}$ factor $M$, which is the abelian subgroup of $\Out(M)$ consisting of the image of approximately inner and centrally trivial automorphisms.
In \cite{MR585235}, V.~Jones defines a quadratic form $\kappa$ on the group $\chi(M)$. 
Eilenberg and MacLane have shown that an abelian group together with a quadratic form defines (uniquely up to braided equivalence) a braided 2-group \cite{MR65163}, which linearizes to a braided unitary tensor category. 
Generalizations of these constructions in the relative context were studied by Kawahigashi \cite{MR1230287, MR1321700} and utilized in the study of orbit equivalence of group actions by Ioana \cite{MR2342972}.

In his groundbreaking work \cite{MR1317367, MR1339767}, Popa introduced notions of approximately inner and centrally free for finite index subfactors, 
which played a key role in his remarkable classification result for subfactors in terms of their standard invariants. 
In \cite[Def.~2.5]{MR1317367}, 
Popa also considers a definition of a centrally trivial subfactor (as an `opposite' to his notion of centrally free), and in \cite[Rem.~2.7]{MR1317367}, he discusses how this definition and his definition of approximately inner subfactor have a natural generalization to bimodules. 
He introduces the unitary tensor category $\Chi(M)$ of dualizable approximately inner and centrally trivial bimodules of a $\rm{II}_{1}$ factor, generalizing Connes' $\chi(M)$. 
Popa asks whether this $\Chi(M)$ is a `commutative' tensor category, i.e., does it admit a braiding?
In this article, we answer Popa's question positively. 

\begin{thmalpha}
\label{thm:ExistsBraiding}
Let $M$ be a $\rm{II}_{1}$ factor. 
Then $\Chi(M)$ admits a unitary braiding (see Equation \eqref{eq:uXY}). 
Furthermore, if 
$N$ is another $\rm{II}_{1}$ factor stably isomorphic to $M$, then $\Chi(M)\cong \Chi(N)$ as braided unitary tensor categories. 
\end{thmalpha}

Notably, $\Chi(M)$ recovers Connes' $\chi(M)$ as the equivalence classes of invertible bimodules in $\Chi(M)$ whose left and right von Neumann dimension are equal (to one), and our braiding on $\Chi(M)$ recovers Jones' $\kappa$ (see Example \ref{ex:ConnesChi} for more details). 
We may thus think of $\Chi(M)$ as a unitary braided categorical extension of the braided 2-group $(\chi(M), \kappa)$.

The existence of a braiding on $\Chi(M)$ is surprising from a categorical viewpoint.
Indeed, von Neumann algebras form a 2-category whose 1-morphisms are bimodules and whose 2-morphisms are intertwiners.
The unitary tensor category $\Chi(M)$
is a full subcategory of $\End(M)\cong \Bim(M)$ in this 2-category.
Braidings arise formally in the context of 3-categories by looking at endomorphisms of some identity 1-morphism.
This is the algebraic structure underlying conformal nets, which produces unitary modular tensor categories in the rational setting \cite{MR1838752,MR3439097,MR3773743}.
The presence of a braiding, together with its behavior under local extensions in Theorem \ref{thmalpha:LocalExtension} below, suggests that von Neumann algebras may be objects in a yet to be discovered 3-category.

As expected from experience with Connes' $\chi$, using \cite{MR377534} and \cite{MR1317367,MR2661553},
it is straightforward to show that
$\Chi(R)$, $\Chi(N)$ and $\Chi(R\otimes N)$ are trivial where $R$ denotes the hyperfinite $\rm{II}_{1}$ factor and $N$ is any non-Gamma $\rm{II}_{1}$ factor. 
In order to leverage these facts to compute some non-trivial examples, we prove the following theorem, which is similar in spirit to Connes' short exact sequence \cite{MR377534} (c.f.~\cite[Prop.~6.4]{MR3424476} for a parallel result in the conformal net context). 

\begin{thmalpha}
\label{thmalpha:LocalExtension}
Let $N\subseteq M$ be a finite index $\rm{II}_{1}$ subfactor such that $Q:={}_{N} L^{2}M_{N}\in\Chi(N)$ is a commutative Q-system. 
Then $\Chi(M)\cong \Chi(N)_Q^\loc$ as braided unitary tensor categories.
\end{thmalpha}

In \cite{MR2661553}, Popa studied Connes' $\chi$ in the context of inductive limits of Jones towers of finite depth finite index non-Gamma $\rm II_1$ subfactors $N\subseteq M$. 
He shows $\chi(M_{\infty})=1$ for a large class of non-Gamma inclusions $N\subseteq M$ for which $M_{\infty}$ is McDuff but not isomorphic to $R\otimes N$ for $N$ non-Gamma, resolving a question of Connes. Building off Popa's techniques, in this paper we will directly compute $\Chi(M_{\infty})$ as a unitary braided tensor category.

To state our results, recall the standard invariant of a finite depth finite index $\rm II_1$ subfactor $N\subseteq M$ consists of the indecomposable $2\times 2$ multifusion category $\cC(N\subseteq M)$ 
of $N-N$, $N-M$, $M-M$, and $M-N$ bimodules generated by $L^2M$, together with the choice of generating object ${}_NL^2M_M$. 
We define Morita equivalence of two standard invariants $\cC(N_{1}\subseteq N_2)$ and $\cC(M_1\subseteq M_{2})$ as Morita equivalence of the underlying multifusion categories \cite[\S7.12]{MR3242743}. 
The Drinfeld center $\cZ(\cC(N\subseteq M))$ is a braided unitary fusion category, and indecomposable multifusion categories are Morita equivalent if and only if their Drinfeld centers are equivalent as braided fusion categories \cite[\S8.5]{MR3242743}.  We have the following theorem.

\begin{thmalpha}
\label{thmalpha:InductiveLimit}
Let $N\subseteq M$ be a finite depth finite index inclusion of non-Gamma $\rm{II}_{1}$ factors, and let $M_{\infty}$ denote the inductive limit factor of the Jones tower. 
Then $\Chi(M_{\infty})\cong \cZ(\cC(N\subseteq M))$ as braided unitary tensor categories.
\end{thmalpha}

We get the following immediate corollary.

\begin{coralpha}
If $N_{1}\subseteq N_{2}$ and $M_{1}\subseteq M_{2}$ are finite depth inclusions of non-Gamma $\rm{II}_{1}$ factors with $\cC(N_{1}\subseteq N_{2})$ not Morita equivalent to $\cC(M_{1}\subseteq M_{2})$, then the $\rm II_1$ factors $N_{\infty}$ and $M_{\infty}$ are not stably isomorphic.
\end{coralpha}

This corollary shows that remarkably, the stable isomorphism class of the inductive limit $\rm{II}_{1}$ factor $M_\infty$ remembers the standard invariant of the initial finite depth subfactor $N\subseteq M$ up to Morita equivalence. 
In fact, our result complements the rigidity result of Popa, which states that the $\rm{II}_{1}$ factor $M_{\infty}$ remembers the inclusion $N\subseteq M$ up to weak equivalence \cite[Def.~3.1.4 and Cor.~3.6]{MR2888236}. As another consequence, our computation for $\Chi(M_{\infty})$ also computes the ordinary $\chi(M_{\infty})$ and $\kappa$ invariants as the braided subcategory of invertible objects in $\cZ(\cC(N\subseteq M))$. 
By Popa's result \cite[Thm.~4.2]{MR2661553}, this is isomorphic to the relative $\chi(M'\cap M_{\infty} \subseteq N'\cap M_{\infty} )$ studied by Kawahigashi \cite{MR1230287}. 
The latter inclusion is a finite index hyperfinite $\rm II_1$ subfactor with standard invariant equivalent to $\cC(N\subseteq M)$.

There are many examples of finite depth non-Gamma inclusions. Popa and Shlyakhtenko showed that every subfactor standard invariant can be realized as an inclusion of $\rm II_1$ factors isomorphic to $L\bbF_{\infty}$ \cite{MR2051399}. 
In \cite{MR2732052}, Guionnet-Jones-Shlyakhtenko provide an alternative realization of finite depth standard invariants as inclusions of interpolated free group factors \cite{MR2807103}
in their diagrammatic reproof of Popa's celebrated subfactor reconstruction theorem \cite{MR1278111}.
As every indecomposable unitary multifusion category is Morita equivalent to any of its unitary fusion category diagonal summands, we obtain the following corollary.

\begin{coralpha}
For $\cC$ a unitary fusion category, its Drinfeld center $\cZ(\cC)$ is realized as $\Chi(M)$ for some McDuff $\rm{II}_{1}$ factor $M$.
\end{coralpha}

While we have stated our results above for $\Chi(M)$, our analysis actually occurs in a much more general categorical setting.  
We define the notions of approximately inner and centrally trivial for endofunctors on an arbitrary $\rm W^*$-category with separable preduals (see \S\ref{ApproxInnerCentTrival}). Functors which are both are said to be \textit{local}. 
We construct a unitary braiding on this category (without any dualizability assumptions), and prove all the axioms are satisfied here. 
Thus we get a new construction of a canonical braided $\rm W^*$-tensor category $\locEnd(\cC)$ from an arbitrary $\rm W^*$-category $\cC$.

When $\cC=\fgpMod(M)$, the finitely generated projective modules of a separable finite von Neumann algebra, under the well known equivalence between $\End(\fgpMod(M))$ and $\fgpBim(M)$,
our definitions of approximately inner and centrally trivial agree with Popa's from \cite{MR1317367}. 
We use this equivalence to express the braiding as a bimodule intertwiner in Equation \eqref{eq:uXY}.
Aside from the greater generality, one of the reasons we use the functor language for our constructions and proofs is that the category of endofunctors is strict, making some commuting diagrams significantly simpler. 
In addition, in the categorically oriented functor approach, the concept of approximately commuting diagram significantly reduces the complexity of proofs and their verification.

%%%%%%%%%%%%%%%%%%%%%%%%%%%%%%%%%%%%%%%%%%%%%%%%%%%%%%%%%%%%%%%%%%%%%%%%%%%%%
\subsection*{Outline}

In \S\ref{sec:Preliminaries}, we recall the notion of $\rm W^*$-category, and we pay special attention to the $\rm W^*$ 2-category of von Neumann algebras, bimodules, and intertwiners.
In \S\ref{sec:CMC*Cat}, we discuss the canonical $\sigma$-strong* topology on the hom spaces of a $\rm W^*$-category, which is essential to our construction.

In \S\ref{sec:W*ApproxNatTrans}, we introduce the notion of an approximate natural transformation between endofunctors of a $\rm W^*$-category, which we use to define the notions of approximately inner and centrally trivial for endofunctors.
We show that the local endofunctors which are both approximately inner and centrally trivial admit a canonical unitary braiding in \S\ref{sec:RelativeBraiding}.

In \S\ref{sec:Translation}, we translate our construction into the language of dualizable bimodules over a $\rm II_1$ factor $M$, and we calculate many examples of $\Chi(M)$ in \S\ref{sec:Examples}.
In \S\ref{sec:LocalExtensionChapter}, we prove Theorem \ref{thmalpha:LocalExtension}, and in \S\ref{sec:Calculation}, we prove Theorem \ref{thmalpha:InductiveLimit}.
To prove these theorems, we make heavy use of the Q-system realization machinery developed in \cite{MR3948170,2105.12010} and the graphical calculus for unitary tensor categories.

%%%%%%%%%%%%%%%%%%%%%%%%%%%%%%%%%%%%%%%%%%%%%%%%%%%%%%%%%%%%%%%%%%%%%%%%%%%%%
\subsection*{Acknowledgements}

This project began through conversations with Vaughan Jones. 
This work would not have been possible without his generous support and encouragement.
We dedicate this article to his memory.
The authors would like to thank 
Dietmar Bisch, 
Andr\'e Henriques, 
Yasuyuki Kawahigashi,
Sorin Popa, and
Stefaan Vaes
for helpful comments and conversations.
Quan Chen and David Penneys were supported by NSF grant DMS 1654159.
Corey Jones was supported by NSF grant DMS 1901082/2100531.

%%%%%%%%%%%%%%%%%%%%%%%%%%%%%%%%%%%%%%%%%%%%%%%%%%%%%%%%%%%%%%%%%%%%%%%%%%%%
%%%%%%%%%%%%%%%%%%%%%%%%%%%%%%%%%%%%%%%%%%%%%%%%%%%%%%%%%%%%%%%%%%%%%%%%%%%%
%%%%%%%%%%%%%%%%%%%%%%%%%%%%%%%%%%%%%%%%%%%%%%%%%%%%%%%%%%%%%%%%%%%%%%%%%%%%
\section{Preliminaries}
\label{sec:Preliminaries}

We assume the reader is relatively familiar with the basics of von Neumann algebras, in particular $\rm II_1$ factors, where our main references include \cite{MR1873025,MR2188261,JonesVNA,ClaireSorinII_1}.
Most von Neumann algebras that appear in this article are assumed to be \emph{separable} (their preduals are separable), with the exception of ultraproducts in \S\ref{sec:Examples} below.

We also assume the reader is relatively familiar with the basics of tensor categories and 2-categories, where our main references include \cite{MR3242743,MR3971584,2002.06055}.
Of particular importance is the graphical string diagrammatic calculus for 2-categories and tensor categories \cite[\S1.1.1 and 8.1.2]{MR3971584}.
For a 2-category $\cC$, objects are represented by 2D shaded regions, 1-morphisms are represented by labelled 1D strands read from left-to-right, and 2-morphisms are represented by labelled 0D coupons which are read bottom-to-top.
1-composition is read left-to-right similar to the relative tensor product of bimodules, and 2-composition is read bottom-to-top.

These string diagrams are formally dual to pasting diagrams, and typically associators and unitors are suppressed whenever possible.
As a tensor category is equivalent to a 2-category with one object, the graphical calculus for tensor categories has no shadings for regions;
objects are represented by labelled 1D strings, and morphisms are represented by labelled 0D coupons read bottom-to-top.
Tensor product is read left-to-right, and composition of morphisms is read bottom-to-top.
Our 2-categories and tensor categories are $\rm C^*/W^*$ (see \S\ref{sec:W*Cat} and \ref{sec:Modules} below), and we represent the $\dag$-operation by vertical reflection of diagrams.

%%%%%%%%%%%%%%%%%%%%%%%%%%%%%%%%%%%%%%%%%%%%%%%%%%%%%%%%%%%%%%%%%%%%%%%%%%%%
\subsection{\texorpdfstring{$\rm C^*/W^*$}{C*/W*}-categories}
\label{sec:W*Cat}

We begin with the basics of $\rm C^*$ and $\rm W^*$-categories.
The latter were first introduced in \cite{MR808930}.

\begin{defn} 
A \textit{$\rm C^*$-category} is a $\bbC$-linear category $\cC$ such that:
\begin{itemize}
    \item 
    for each pair of objects $a,b\in \cC$, there is a conjugate linear involution $\dag:\cC(a\to b)\rightarrow \cC(b\to a)$, satisfying $(f\cdot g)^\dag=g^\dag\cdot f^\dag$,
    \item
    for each pair of objects $a,b\in \cC$, there is a Banach norm on $\cC(a\to b)$ satisfying $\|f\|^{2}=\|f^\dag\cdot f\|=\|f\cdot f^\dag\|$ for all $f\in \cC(a\to b)$, and
    \item
    for all $f\in \cC(a\to b)$, $f^\dag\cdot f$ is a positive element in the $\rm C^*$-algebra $\cC(a\to a)$.
    That is, there is a $g\in \cC(a\to a)$ such that $f^\dag \cdot f = g^\dag \cdot g$.
\end{itemize}
A \emph{$\rm W^*$-category} is a $\rm C^*$-category such that every hom space $\cC(a\to b)$ admits a predual Banach space.
We call a $\rm W^*$-category \emph{separable} if all such preduals are separable Banach spaces.
\end{defn}

\begin{assumption}
In this article, we assume all $\rm C^*/W^*$-categories are \emph{unitarily Cauchy complete}, meaning they admit all finite orthogonal direct sums and are orthogonal projection complete.
There is a formal construction to complete any $\rm C^*/W^*$-category which satisfies a universal property; we refer the reader to \cite[\S3.1.1]{1810.06076} for more details.
\end{assumption}

\begin{rem}
Every unitarily Cauchy complete $\rm C^*/W^*$-category admits a canonical left $\fdHilb$-module category structure.
That is, for each $c\in \cC$ and finite dimensional Hilbert space $H$, there is an object $H\rhd c\in \cC$, unique up to canonical unitary isomorphism.
Moreover, there is a canonical unitary associator $H\rhd (K\rhd a)\cong (H\otimes K)\rhd a$.
In the sequel, we will assume our $\fdHilb$-module category structure is \emph{strictly unital}, i.e., $\bbC\rhd c = c$ for all $c\in \cC$.
\end{rem}

\begin{defn}
A \emph{$\dag$-functor} between between $\rm C^*$-categories is a functor $F: \cC\to \cD$ such that $F(f^\dag) = F(f)^\dag$ for all morphisms $f$ in $\cC$.
Two $\rm C^*$-categories are \emph{(unitarily) equivalent} if there are $\dag$-functors each way whose appropriate composites are unitarily naturally isomorphic to the appropriate identity $\dag$-functors.

For $\rm W^*$-categories, one restricts to the \emph{normal} $\dag$-functors which are weak*-continuous on hom spaces.
Equivalence is defined similarly as before, but restricting to normal $\dag$-functors.
\end{defn}

\begin{ex}
\label{ex:RightFGPMod}
The most important example of a $\rm W^*$-category for our article is the finitely generated projective right modules for a $\rm II_1$ factor $M$.
There are two dagger equivalent such categories that one can work with:
\begin{itemize}
    \item 
    Hilbert spaces $H$ equipped with a normal right $M$-action such that the von Neumann dimension $\dim(H_M)$ is finite, or
    \item
    finitely generated projective right Hilbert $\rm W^*$-modules (see \S\ref{sec:Modules} below for more details).
\end{itemize}
To see the equivalence, the map from the first to the second is taking bounded vectors (the $\xi \in H$ such that $\widehat{m} \mapsto \xi m$ extends to a bounded map $L^2M \to H$), and the map from the second to the first is $-\otimes_M L^2M$ (the inner product is given by $\langle \eta\otimes \widehat{m} , \xi\otimes \widehat{n} \rangle:=\langle \langle \xi | \eta\rangle_M \widehat{m}, \widehat{n}\rangle_{L^2M}$).

We will use the second definition above for the convenience that we may state many results for all $\xi \in X_M$ rather than for all bounded vectors.
However, one can work with the first definition provided that one restricts to bounded vectors when appropriate.
\end{ex}

\begin{ex}
For a separable $\rm C^*$-algebra $A$, $\mathsf{Rep}(A)$ is the $\rm W^*$-category of (non-degenerate) $*$-representations of $A$ on separable Hilbert spaces. This category is relevant in the operator algebraic study of quantum  statistical mechanics.
\end{ex}

\begin{defn}
\label{defn:sigmaFH}
Given a $\fdHilb$-module $\rm C^*$ category $\cC$, a finite dimensional Hilbert space $H$, and any $\dag$-functor $F\in  \End(\cC)$, we have a canonical \emph{braid-like} unitary natural isomorphism 
$$
\sigma_{F,H}: F(H\rhd -)\rightarrow H\rhd F(-)
$$
defined as follows. 
For an orthonormal basis $\{e_{i}\}$, we may identify its elements as bounded operators $e_i\in B(\bbC, H)$ defined by $1\mapsto e_{i}$. 
Then $e^{\dag}_{i}\in B(H, \bbC)$ is given by $e^{\dag}_{j}e_{i}=\delta_{i,j}$.
We then define $\sigma_{F,H}$ in components by 
$$
\sigma^{a}_{F,H}:=\sum_{i} (e_{i}\rhd F(1_{a}))\cdot F(e^{\dag}_{i}\rhd 1_{a}),
$$
which does not depend on the choice of orthonormal basis of $H$.
 
Unitarity is straightforward to verify. 
To show naturality, let $f\in \cC(a\to b)$. 
Then 
\begin{align*}
(1_{H}\rhd F(f))\cdot \sigma^{a}_{F,H} 
& = \sum_{i} (e_{i}\rhd F(f))\cdot F(e^{\dag}_{i}\rhd 1_{a}) = \sum_{i} (e_{i}\rhd 1_{b})\cdot F(f)\cdot F(e^{\dag}_{i}\rhd 1_{a}) 
\\
& = \sum_{i} (e_{i}\rhd 1_{b})\cdot F(e^{\dag}_{i}\rhd f)=\sigma^{b}_{F,H}\cdot F(1_{H}\rhd f). \qedhere
\end{align*}
\end{defn}

The family $\sigma$ also satisfies the following monoidality conditions (where we have suppressed the module category associator).

\begin{prop}
\label{prop:sigmabraid}
\mbox{}
\begin{enumerate}
    \item For any $F\in  \End(\cC)$, $\sigma^{a}_{F,H\otimes K}=(1_{H}\rhd \sigma^{a}_{F,K})\cdot \sigma^{K\rhd a}_{F,H}$.
    \item For any $G\in  \End(\cC)$ we have $\sigma^{a}_{F\circ G, H}=\sigma^{G(a)}_{F,H}\cdot F(\sigma^{a}_{G,H})$.
\end{enumerate} 
\end{prop}
\begin{proof}
Let $\{f_j\}$ be an orthonormal basis for $K$, which we identify with bounded operators $f_j\in B(\bbC,K)$.
Then
\begin{align*}
(1_H\rhd\sigma^a_{F,K})\cdot \sigma^{K\rhd a}_{F,H} 
& = \sum_{i,j} (1_H\rhd f_j\rhd F(1_a))\cdot (1_H\rhd F(f_j^\dag\rhd 1_a))\cdot (e_i\rhd F(1_{K\rhd a})\cdot F(e_i^\dag\rhd 1_{K\rhd a}) \\
& = \sum_{i,j} (1_H\rhd f_j\rhd F(1_a))\cdot (e_i\rhd F(1_a))\cdot F(f_j^\dag\rhd 1_a) \cdot F(e_i^\dag\rhd 1_{K\rhd a})
= \sigma^a_{F,H\otimes K} \\
\intertext{and}
\sigma^{G(a)}_{F,H}\cdot F(\sigma^{a}_{G,H}) 
& = \sum_i (e_i\rhd FG(1_a))\cdot F(e_i^\dag\rhd G(1_a))\cdot F(e_i\rhd G(1_a))\cdot FG(e_i^\dag\rhd 1_a) \\
& = \sum_i (e_i\rhd FG(1_a))\cdot FG(e_i^\dag\rhd 1_a) 
= \sigma^{a}_{F\circ G, H}.
\qedhere
\end{align*}
\end{proof}

%%%%%%%%%%%%%%%%%%%%%%%%%%%%%%%%%%%%%%%%%%%%%%%%%%%%%%%%%%%%%%%%%%%%%%%%%%%%%%%%%%%%
\subsection{Modules and correspondences of \texorpdfstring{$\rm W^*$}{W*}-algebras}
\label{sec:Modules}

We now recall the definition of the $\rm W^*$ 2-category $\fgpWStarRCorr$ of finitely generated projective right $\rm W^*$-correspondences, after which we formally define the finitely generated projective right modules $\fgpMod(M)$.
Our exposition follows \cite[\S2.2]{2105.12010}, which was adapted from \cite[\S8]{MR2111973}.
Other references include \cite{MR355613,MR0367670}.

\begin{defn}
The $\rm W^*$ 2-category $\fgpWStarRCorr$ is given as follows.
\begin{itemize}
\item 
objects are von Neumann algebras
\item
1-morphisms are finitely generated projective right $\rm W^*$-corresponendences.
In more detail, given von Neumann algebras $A,B$, a 1-morphism ${}_AX_B$ is a Banach space equipped with a right $B$-action and a right $B$-valued inner product satisfying
\begin{itemize}
    \item $\langle \eta| \xi_1 +\xi_2 b\rangle_B = \langle \eta | \xi_1\rangle_B + \langle \eta | \xi_2\rangle_B b$,
    \item $\langle \eta_1 + \eta_2 b| \xi\rangle_B = \langle \eta_1 | \xi\rangle_B+ b^*\langle \eta_2 | \xi\rangle_B$,
    \item $\langle \eta|\xi\rangle_B^* = \langle \xi | \eta\rangle_B$, and
    \item $\langle \xi | \xi \rangle_B \geq 0$ with equality if and only if $\xi = 0$.
\end{itemize}
By the Cauchy-Schwarz inequality, $\|\langle \xi|\xi\rangle_B\|_B$ defines a norm on $X$, which is required to be complete.
Moreover, we require the left $A$-action to be by adjointable operators.

The finitely generated projective condition says that as a right $B$-module, $X_B$ is unitarily isomorphic to $pB^n$ for some (adjointable) orthogonal projection $p\in \End_{-B}(B^n)$.

The $\rm W^*$ condition amounts to requiring that:
\begin{itemize}
    \item ${}_AX_B$ has a predual,
    \item the $B$-valued inner product $\langle \,\cdot\,|\,\cdot\,\rangle_B$ is separately weak*-continuous, and
    \item the left $A$-action $A\to \End(X_B)$ is normal.
\end{itemize}
Composition of 1-morphisms is the relative tensor product.
\item
2-morphisms ${}_AX_B\Rightarrow {}_AY_B$ are the adjointable right $B$-linear operators that commute with the left $A$-action.
\end{itemize}
\end{defn}

\begin{defn}
For a von Neumann algebra $M$, we define $\fgpMod(M):=\fgpWStarRCorr(\bbC \to M)$.

Given $X_M \in \fgpMod(M)$, a finite $X_M$-basis 
is a finite subset $\{\beta\}\subseteq X$ such that $\sum_{\beta} \beta\langle \beta|\xi\rangle_M = \xi$ for all $\xi\in X$ \cite[\S1.1.3]{MR1278111},
\cite[\S3.1.1]{1111.1362}.
As we only work with finitely generated projective modules in this article, all $X_M$-bases will be finite, so we omit the word `finite' without confusion.
We call such a basis \emph{orthogonal} if $\langle \beta|\beta'\rangle_M$ is equal to $\delta_{\beta=\beta'}$ times an orthogonal projection in $M$.
Given an $X_M$-basis, one can always obtain an orthogonal $X_M$-basis using the Gram-Schmidt orthogonalization procedure \cite[Lem.~8.5.2]{ClaireSorinII_1}.
\end{defn}

\begin{rem}
\label{rem:OppositeIssue}
It is well known that the $\rm W^*$-tensor category $\End(\fgpMod(M))$ of normal $\dag$-endofunctors is dagger equivalent to $\fgpBim(M)^{\rm mp}$, the monoidal opposite of $\fgpBim(M)$.
That is, every normal $\dag$-endofunctor of $\fgpMod(M)$ is of the form $-\boxtimes_M X$ for some $X\in \fgpMod(M)$.
For example, this equivalence is explained in \cite[\S3.2]{2004.08271} for infinite von Neumann algebras using the fact that $\End(\fgpMod(M))$ is unitarily equivalent to the orthogonal projection completion of the $\rm W^*$-tensor category $\End(M)$.
\end{rem}

\begin{rem}
There is another way that the category of $\End(\fgpMod(N))$ is used in practice, particularly among the $\rm{II}_{1}$ factor community. 
This stems from the fact that $\fgpMod(N)$ is equivalent to the unitary Cauchy completion of $N$ thought of as a $\rm W^*$-category with one object. 
Objects in the completion are pairs $(n,p)$, where $n\in \bbN$ and $p\in M_{n}(N)$ is a projection. 
Morphisms $(n,p)\to (m,q)$ are elements of $qM_{m\times n}(N)p$.
By the universal property of Cauchy completion, an endofunctor is determined by where it sends $(1,1_{N})$ together with its action on $\End((1, 1_{N}))\cong N$. 
In other words, an endofunctor in $\End(\fgpMod(N))$ is completely determined up to unitary natural isomorphism by a (unital) homomorphism $N\to pM_{n}(N)p$ for some projection $p\in M_{n}(N)$, called a \emph{cofinite morphism} of $N$ in \cite{MR1049618}. 
Furthermore, a natural transformation is uniquely determined by its $(1,1_{N})$-component.
\end{rem}

%%%%%%%%%%%%%%%%%%%%%%%%%%%%%%%%%%%%%%%%%%%%%%%%%%%%%%%%%%%%%%%%%%%%%%%%%%%%%%%%%%%%%%%%%%%%%%%%%%%%%%%%%
\subsection{The \texorpdfstring{$\sigma$}{sigma}-strong* topology on a \texorpdfstring{$\rm W^*$}{W*}-category}
\label{sec:CMC*Cat}

Let $\cC$ be a separable $\rm W^*$-category, which has a canonical weak* topology on each hom space.

\begin{defn}
For each $a,b\in \cC$,
the $\sigma$-\emph{strong* topology} $\tau$ on $\cC(a\to b)$
is defined as follows:
$f_i \to 0$ $\sigma$-strong* if and only if
$f_i^\dag f_i \to 0$
and
$f_if_i^\dag \to 0$
weak* ($\sigma$-weakly).
\end{defn}

\begin{facts}
The $\sigma$-strong* topology $\tau$ on the hom spaces of $\cC$ satisfies the following properties:
\begin{enumerate}[label=($\tau$\arabic*)]
\item
\label{tau:composition}
composition is jointly $\tau$-continuous on norm bounded subsets (if $\{a_{n}\}, \{b_n\}$ are uniformly norm bounded, and $a_{n}\rightarrow a$ and $b_{n}\rightarrow b$ then $a_{n}b_{n}\rightarrow ab$).
\item
\label{tau:dagcontinuity}
$\dag$ is $\tau$-continuous on norm bounded subsets of morphism spaces.
\item
\label{tau:metrizable}
$\tau$ restricted to the unit ball of any morphism space is completely metrizable.
\end{enumerate}
\end{facts}

The following proposition basically follows from \cite[III.2.2.2]{MR2188261}.

\begin{prop}
Suppose $M,N$ are von Neumann algebras and $\Phi: M \to N$ is a unital $*$-homomorphism.
Then $\Phi$ is normal if and only if it is $\sigma$-strong* continuous on the unit ball of $M$.
\end{prop}
\begin{proof}
If $\Phi$ is $\sigma$-strong* continuous on bounded sets, then for any increasing bounded net $(x_i)$ in $M$ with $x_i \nearrow x$, $x_i \to x$ in the $\sigma$-strong* topology.
Hence $\Phi(x_i)\nearrow \Phi(x)$, and $\Phi$ is normal.

The converse argument is similar to \cite[III.2.2.2]{MR2188261}.
If $x_i \to 0$ $\sigma$-strong*, then $x_ix_i^*\to 0$ and $x_i^*x_i \to 0$ $\sigma$-weakly.
Hence $\Phi(x_i)^*\Phi(x_i) \to 0$ and $\Phi(x_i)\Phi(x_i)^* \to 0$ $\sigma$-weakly, which implies $\Phi(x_i) \to 0$ $\sigma$-strong*.
\end{proof}

As all $\rm W^*$-categories were assumed to admit finite orthogonal direct sums, we have the following immediate corollary.

\begin{cor}
Suppose $\cC,\cD$ are $\rm W^*$-categories and $F: \cC\to \cD$ is a $\dag$-functor.
Then $F$ is normal if and only if $F$ is $\sigma$-strong* continuous on norm bounded subsets of hom spaces in $\cC$.
\end{cor}

\begin{ex}
Let $(M,\tr_M)$ be a finite von Neumann algebra equipped with a faithful tracial state $\tau$.
In this case, the $\sigma$-strong* topology on the unit ball of $M$ is exactly the $\|\cdot\|_2$-topology, where $\|x\|_2^2 := \tr_M(x^*x)$ \cite[Prop.~9.1.1]{JonesVNA}.
In fact, we may also describe the entire $\sigma$-strong* topology $\tau$ on $\fgpMod(M)$ on norm bounded sets as coming from a $\|\cdot\|_2$-norm induced by canonical commutant traces. 

In more detail, for $X_M\in \fgpMod(M)$, the canonical \emph{commutant trace} 
\cite[\S1.1.3(c)]{MR1278111},
\cite[Def.~3.1.4]{1111.1362} 
on the finite von Neumann algebra $\End(X_M)$ is given by
$$
\Tr_X(f)
:=
\sum_{b} \tr_M(\langle b| xb\rangle^X_M)
$$
where $\{b\}$ is any finite $X_M$-basis.
Observe that $\Tr_X$ is independent of the choice of basis.
When $M$ is a $\rm II_1$ factor, $\Tr(\id_X)$ equals the right von Neumann dimension of $X\otimes_M L^2M$.

Observe that the maps $(\Tr_X)_{X\in \fgpMod(M)}$ endow $\fgpMod(M)$ with a \emph{unitary categorical trace} in the spirit of \cite[Def.~3.7]{MR3019263} (see also \cite[Def.~3.59]{1810.06076}).
Indeed, for $f\in \Hom(X_M \to Y_M)$, 
we choose a finite $X_M$-basis $\{b\}$ and a finite $Y_M$-basis $\{c\}$,
and we calculate
\begin{align*}
\Tr_X(f^\dag f)
&=
\sum_{b} \tr_M(\langle b| f^\dag fb\rangle^X_M)
=
\sum_{b} \tr_M(\langle fb| fb\rangle^Y_M)
=
\sum_{b,c} \tr_M(\langle fb| c \langle c|fb\rangle^Y_M \rangle^Y_M )
\\&
=
\sum_{b,c} \tr_M(\langle fb| c\rangle^Y_M  \langle c|fb\rangle^Y_M )
=
\sum_{b,c} \tr_M( \langle f^\dag c|b\rangle^X_M \langle b| f^\dag c\rangle^X_M )
=
\sum_{b,c} \tr_M(  \langle b\langle b|f^\dag c\rangle^X_M| f^\dag c\rangle^X_M )
\\&=
\sum_{c} \tr_M(  f^\dag c| f^\dag c\rangle^X_M )
=
\sum_{c} \tr_M(  c| ff^\dag c\rangle^X_M )
%\\&
=
\Tr_Y(ff^\dag).
\end{align*}

Using this categorical trace, for $f\in \Hom(X_M \to Y_M)$, we define $\|f\|_2 := \Tr_X(f^\dag\circ f)^{1/2}$.
Observe that $\|\cdot\|_2$ on $\Hom(X_M \to Y_M)$ is exactly 
the restriction of 
$\|\cdot\|_2$ on $\End(X_M \oplus Y_M)$, which is again defined using the canonical commutant trace $\Tr_{X\oplus Y}$.
Thus the $\sigma$-strong* topology $\tau$ on $\fgpMod(M)$ exactly corresponds to the $\|\cdot\|_2$-topology on norm-bounded sets.
\end{ex}

The following remark will be used later in \S\ref{sec:LocalExtension}.

\begin{rem}
Suppose $(M,\tr_M)$ is a finite von Neumann algebra equipped with a normal faithful trace $\tr_M$.
Suppose $X_M\in \fgpMod(M)$ and $N\subseteq (M,\tr_M)$ is strongly Markov inclusion \cite[Def.~2.8]{MR2812459}, i.e., there is a finite $M_N$-basis $\{c\}$ which satisfies $[M:N]:=\sum_c cc^* \in [1,\infty)$.
(This definition was based on \cite{MR945550} and \cite[\S1.1.3 and 1.1.4]{MR1278111}.)
Here, the right $N$-valued inner product on $M_N$ is given by $\langle a|b\rangle_N = E_N(a^*b)$ where $E_N: M\to N$ is the unique trace-preserving conditional expectation.

We now compare $\Tr_{X_M}$ and $\Tr_{X\boxtimes_M M_N}$.
If $\{b\}$ is a basis for $X_M$ and $\{c\}$ is a basis for $M_N$, then $\{b\boxtimes c\}$ is a basis for $X\boxtimes_M M_N$.
Thus for $f\in \End(X_M)$,
\begin{align*}
\Tr_{X\boxtimes_M M_N}(f\boxtimes \id_{M_N})
&=
\sum_{b,c} \tr_N(\langle b\boxtimes c | fb\boxtimes c\rangle^{X\boxtimes_M M_N}_N)
=
\sum_{b,c} \tr_N(E_N(c^* \langle b |fb\rangle^X_M c))
\\&=
\sum_{b,c} \tr_M(c^* \langle b |fb\rangle^X_M c)
=
\sum_{b,c} \tr_M(\langle b |fb\rangle^X_M cc^*  )
\\&=
[M:N] 
\sum_{b} \tr_M(\langle b |fb\rangle^X_M)
=
[M:N] 
\Tr_X(f).
\end{align*}
This says the (faithful) restriction functor $-\boxtimes M_N : \fgpMod(M) \to \fgpMod(N)$ is a  continuous embedding (a homeomorphism onto its image) of hom spaces with respect to the $\tau_M-\tau_N$ topologies.
This means that $f_n \to 0$ in $\tau_M$ if and only if $f_n \boxtimes \id_{M_N} \to 0$ in $\tau_N$.
\end{rem}

%%%%%%%%%%%%%%%%%%%%%%%%%%%%%%%%%%%%%%%%%%%%%%%%%%%%%%%%%%%%%%%%%%%%%%%%%%%%%%%%%%
%%%%%%%%%%%%%%%%%%%%%%%%%%%%%%%%%%%%%%%%%%%%%%%%%%%%%%%%%%%%%%%%%%%%%%%%%%%%%%%%%%
%%%%%%%%%%%%%%%%%%%%%%%%%%%%%%%%%%%%%%%%%%%%%%%%%%%%%%%%%%%%%%%%%%%%%%%%%%%%%%%%%%
\section{Approximate natural transformations and local endofunctors}
\label{sec:W*ApproxNatTrans}

In this section, given a $\rm W^*$-category, we define its canonical braided $\rm W^*$-tensor category of local endofunctors  which are both approximately inner and centrally trivial.
To define these notions, we first introduce the concept of an approximate natural transformation.

%%%%%%%%%%%%%%%%%%%%%%%%%%%%%%%%%%%%%%%%%%%%%%%%%%%%%%%%%%%%%%%%%
\subsection{Approximate natural transformations}

Suppose $\cC$ is a separable $\rm W^*$-category, and recall $\End(\cC)$ denotes the normal $\dag$-endofunctors of $\cC$.
We define $\ell^{\infty}(\bbN,\cC)$ as the $\rm W^*$-category with the same objects as $\cC$ and whose morphisms are uniformly norm-bounded sequences of morphisms in $\cC$.
The composition and $\dag$ in $\ell^{\infty}(\bbN,\cC)$ are defined pointwise. 

\begin{defn}
For each $a,b\in \cC$, we define 
\begin{align*}
\cI(a\to b)
&:=
\set{
f=(f_{n})\in \ell^{\infty}(\bbN,\cC(a\to b))
\,}{\, 
f_{n}\rightarrow_{\tau} 0
}
\\
\cC^{\infty}(a\to b)
&:=
\set{f\in \ell^{\infty}(\bbN,\cC(a\to b))
\,}{\,
\parbox{6.2cm}{
$\forall\, g\in \cI(b\to c), g\cdot f\in \cI(a\to c)$
and
$\forall\, h\in \cI(d\to a), f\cdot h\in \cI(d\to b)$
}
}.
\end{align*}
We view $\cC^\infty$ as the \emph{idealizer} of $\cI$ in $\ell^{\infty}(\bbN,\cC)$.
We call $f\in \cC^{\infty}(a\to b)$ an \emph{approximate morphism}, and we say two morphisms $f,g\in \cC^{\infty}(a\to b)$ are \emph{equivalent} or \textit{approximately equal} if $f-g\in \cI(a\to b)$.
Observe that $\cC^\infty$ is a category under pointwise composition of approximate morphisms.
By \ref{tau:composition} and \ref{tau:dagcontinuity}, $\cI(a\to b)\subseteq \cC^\infty(a\to b)$ for all $a,b$, and $\cI$ defines a $\dag$-closed ideal in $\cC^\infty$.

We define a $\dag$-category $\widetilde{\cC}$ 
with the same objects as $\cC$, and hom spaces
$\widetilde{\cC}(a\to b):=\cC^\infty(a\to b)/\cI(a\to b)$.
For $f\in \cC^\infty(a\to b)$, we write $\widetilde{f}$ for its image in $\widetilde{\cC}(a\to b)$.
Observe we can view $\cC$ as a $\dag$-subcategory of $\widetilde{\cC}$ by mapping $f\in \cC(a\to b)$ 
to the image of the constant sequence $\widetilde{(f)} \in \widetilde{\cC}(a\to b)$.
In what follows,
we identify $f\in\cC(a\to b)$ with $(f)_{n\in\bbN}\in \cC^\infty(a\to b)$ and $\widetilde{(f)}_{n\in\bbN}\in \widetilde{\cC}(a\to b)$.
\end{defn}

\begin{defn}
A pasting diagram in $\cC^\infty$ is said to \emph{approximately commute} if the corresponding pasting diagram in the quotient $\widetilde{\cC}=\cC^\infty/\cI$ actually commutes.
That is, given $a,b,c,d\in\cC$
and morphisms $f\in \cC^\infty(a\to b)$, $g\in \cC^\infty(b\to d)$, $h\in \cC^\infty(a\to c)$, and $k\in \cC^\infty(c\to d)$,
the diagram
\[
\begin{tikzcd}
a \arrow{r}{f_n} \arrow[swap]{d}{h_n} & b \arrow{d}{g_n}
\\
c \arrow{r}{k_n} & d
\end{tikzcd}
\]
approximately commutes if 
%$\widetilde{g}\cdot \widetilde{f}-\widetilde{k}\cdot \widetilde{h} =0$ in $\widetilde{\cC}$, or equivalently, 
$g_n\cdot f_n-k_n\cdot h_n\to_\tau 0$.
\end{defn}

For a normal $\dag$-endofunctor $F\in\End(\cC)$, for any $f=(f_n)\in\cC^\infty(a\to b)$,
$(F(f_n))\in \cC^\infty(F(a)\to F(b))$, 
so $F$ descends to a $\dag$-endofunctor on $\widetilde{\cC}$.

\begin{defn}
Given two functors $F,G\in \End(\cC)$ an \textit{approximate natural transformation} is a family $\{\eta^{a}\in \cC^{\infty}(F(a)\rightarrow G(a) \}_{a\in \cC}$
such that for every $f\in \cC(a\to b)$, $\eta^{b}\cdot f=f\cdot \eta^{a}$ in $\widetilde{\cC}$. 
In other words, the following diagram approximately commutes:
$$
\begin{tikzcd}
F(a) 
\arrow[d,"F(f)"']
\arrow[r,"\eta^a_n"]
&
G(a)
\arrow[d,"G(f)"]
\\
F(b) 
\arrow[r,"\eta^b_n"]
&
G(b).
\end{tikzcd}
$$
Clearly every natural transformation $\eta: F\Rightarrow G$ gives an approximate natural transformation.
\end{defn}

\begin{warn}
The collection of normal $\dag$-endofunctors of $\cC$ and approximate natural transformations between them  (up to $\cI$) clearly forms a $\dag$-category %(in fact, a $\rm C^*$-category), 
containing $\End(\cC)$ as a (non-full) subcategory. 
It is tempting to think that this should also form a tensor category, 
with tensor product being 
composition of endofunctors as usual. 
However, the horizontal composition of two approximate natural transformations is not well-defined in general. 
This is a fundamental point: the category of endofunctors and approximately natural transformations is \emph{not} a tensor category in general, as endomorphisms of the `unit' in this category (see Definition \ref{defn:CentralSequenceAsApproximateNT}) is not a commutative algebra.
\end{warn}

\begin{defn}
We call $v\in\cC^\infty(a\to b)$ an \textit{approximate isometry} if its image $\widetilde{v}\in\widetilde{\cC}(a\to b)$ is an isometry. 
For $\dag$-endofunctors $F,G\in\End(\cC)$, an approximate natural transformation $v:F\Rightarrow G$ is called an \textit{approximately isometric natural transformation} if $v^a\in \cC^\infty(F(a)\to G(a))$ is an approximate isometry for each $a\in\cC$.
\end{defn}

\begin{rem}(Arrow flipping)\label{rem:ArrowFlipping} 
Suppose
$f\in \cC^\infty(a\to c)$, $v\in \cC^\infty(a\to b)$, and $w\in \cC^\infty(b\to c)$ such that the diagram
$$
\begin{tikzcd}
a\arrow{rr}{f_n} \arrow[swap]{dr}{v_n} & & c\\
 & b \arrow[swap]{ur}{w_n} & 
\end{tikzcd}
$$ 
approximately commutes.
\begin{itemize}
\item 
If $w$ is an approximate isometry, then
$
\begin{tikzcd}
a\arrow{rr}{f_n} \arrow[swap]{dr}{v_n} & & c \arrow{dl}{w^{\dag}_n}\\
 & b  & 
\end{tikzcd}
$
approximately commutes. 
\item 
If $v$ is an approximate coisometry, then
$
\begin{tikzcd}
a\arrow{rr}{f_n} & & c \\
 & b \arrow{ul}{v_n^{\dag}} \arrow[swap]{ur}{w_n} & 
\end{tikzcd}
$ 
approximately commutes. 
\end{itemize}
Indeed, these remarks follow since $\widetilde{v}$, respectively $\widetilde{w}$, is an actual isometry, respectively coisometry, in $\widetilde{\cC}$.
\end{rem}

\begin{defn} 
\label{defn:CentralSequenceAsApproximateNT}
An approximately natural transformation from the identity functor to itself is called a \textit{central sequence}.
\end{defn}

\begin{lem}
\label{lem:CauchyCompletionPreservesCentralSequences}
Suppose $\cC$ is a separable $\rm W^*$-category which is not necessarily unitarily Cauchy complete.
Every central sequence of $\cC$ has a canonical extension to the unitary Cauchy completion of $\cC$.
Moreover, this extension gives a bijective correspondence between
the equivalence classes of central sequences of $\cC$ and 
the equivalence classes of central sequences of the unitary Cauchy completion of $\cC$.
\end{lem}
\begin{proof}
We proceed in 2 steps;
first, we show the result for the orthogonal direct sum completion $\Add(\cC)$, and second, we show the result for the orthogonal projection completion $\Proj(\cC)$.
\medskip

\item[\underline{Step 1:}]
We claim a central sequence $\eta$ of $\cC$ gives a central sequence of $\Add(\cC)$ by $\Add(\eta)_{ii}:=\pi_i^\dag \eta^{c_i} \pi_i$ for $\bigoplus_i c_i\in \Add(\cC)$, where $\pi_j : \bigoplus c_i \to c_j$ denotes the canonical projections satisfying $\sum_i \pi_i^\dag\pi_i = \id_{\bigoplus_i c_i}$ and $\pi_i \circ \pi_j^\dag = \delta_{i=j}\id_{c_i}$. 
Let $f^{n}=(f^{n}_{ji})\in \Add(\cC)^\infty(\bigoplus_{i} c_{i}\rightarrow \bigoplus_{j} d_{j})$, with each $(f^{n}_{ji})\in \cC^\infty(c_{i}\rightarrow d_{j})$ a bounded sequence of morphisms. 
Then clearly $(f^{n})\in \cI(\bigoplus_i c_{i}\rightarrow \bigoplus_j d_{j}) $ if and only if each $(f^{n}_{ji})\in \cI(c_i \rightarrow d_{j})$. 
Therefore since $\eta\in \cC^{\infty}$, we have $\Add(\eta)\in\Add(\cC)^\infty$. 

To see that $\Add(\eta)$ defines a natural transformation of the identity functor, let $f=(f_{ji})\in \Add(\cC)(\bigoplus c_{i}\rightarrow \bigoplus d_{j})$ as above. 
Then in $\widetilde{\cC}$, $\Add(\eta)f=(\eta^{d_j}_{n}f_{ji})= (f_{ji}\eta^{c_i}_{n})=f\Add(\eta)$. 
Moreover, the assignment $\eta\mapsto \Add(\eta)$ clearly preserves equivalence of central sequences. 

Conversely, given a central sequence $\eta$ of $\Add(\cC)$, we automatically get a central sequence of $\cC$ by considering the canonical embedding $\cC \hookrightarrow \Add(\cC)$.
Moreover, for $\bigoplus_i c_i \in \Add(\cC)$, the off-diagonal terms of $\eta^{\bigoplus_i c_i}$ go to zero $\sigma$-strong* as $\eta^{\bigoplus_i c_i}$ approximately commutes with the $\pi_j$, and the diagonal term corresponding to $c_i$ must be approximately equivalent to $\eta^{c_i}$.

\medskip
\item[\underline{Step 2:}]
A central sequence $\eta$ of $\cC$ gives a central sequence of $ \Proj(\cC)$ by defining
$\Proj(\eta)^{(c,p)}:= p\eta^c p$ for $(c,p)\in \Proj(\cC)$.
Given $(a,p),(b,q)\in \Proj(\cC)$,
$\cI((a,p) \to (b,q)) = q\cI(a\to b)p$, so $\Proj(\eta)\in \Proj(\cC)^\infty$. Moreover, this construction preserves equivalence of central sequences.

Conversely, given a central sequence $\eta$ of $\Proj(\cC)$, we automatically get a central sequence of $\cC$ by considering the canonical embedding $\cC \hookrightarrow \Proj(\cC)$. 
Clearly starting with a central sequence in $\cC$, extending to $\Proj(\cC)$ as above, and restricting back to $\cC$ yields the same central sequence. 
In the other direction, let $\eta$ be a central sequence in $\Proj(\cC)$.
We need to show that in $\widetilde{\Proj(\cC)}$, $p\eta^{(c,\id_{c})}p= \eta^{(c,p)}$. 
But note we can view an orthogonal projection $p\in \End_\cC(c)$ as a morphism in $\Proj(\cC)((c,p)\to (c,\id_c))$.
Then in $\widetilde{\Proj(\cC)}$, $\eta^{(c,p)}=\eta^{(c,p)}p=(\eta^{(c,p)}p)p=p\eta^{(c,\id_c)}p$ as desired.
\end{proof}

\begin{ex}
\label{ex:MCentralSequencesAgree}
For a separable von Neumann algebra $M$, central sequences in the sense of Definition \ref{defn:CentralSequenceAsApproximateNT} for the $\rm W^*$-category $\fgpMod(M)$ exactly agree with the usual notion of central sequences of $M$. 
Indeed, consider the category with one object $M_M$ whose endomorphisms is $M$ acting by left multiplication.
An approximate natural transformation of the identiy functor is exactly a sequence $(x_n)$ such that $x_n m - mx_n \to 0$ in the $\sigma$-strong* topology for all $m\in M$.
Now since $\fgpMod(M)$ is the unitary Cauchcy completion of this one object category, the claim follows by Lemma \ref{lem:CauchyCompletionPreservesCentralSequences}.
\end{ex}

The next lemma will be important in the next subsection.

\begin{lem}
\label{Lem:ApproNT&CentralSq}
Suppose $H,K$ are finite dimensional Hilbert spaces with orthonormal bases $\{e_i\},\{f_j\}$ respectively.
For a collection of maps $(\eta^a \in \cC^\infty (H\rhd a \to K\rhd a))_{a\in \cC}$,
we define $\eta^a_{i,j}:=(f_j^\dag\rhd 1_a)\eta^a(e_i\rhd 1_a)\in \cC^\infty(a\to a)$.
Then
$\eta:H\rhd- \Rightarrow K\rhd-$ defines an approximate natural transformation if and only if $\eta_{i,j}$ is a central sequence for each $i,j$.
\end{lem}
\begin{proof}
Note that $\eta^a = \sum_{i,j} (f_j\rhd 1_a)\eta^a_{i,j}(e_i^\dag\rhd 1_a)$.
Then for $g\in\cC(a\to b)$,
\begin{align*}
g\eta^a_{i,j} - \eta^b_{i,j} g
& = g (f_j^\dag\rhd 1_a) \eta^a (e_i\rhd 1_a) - (f_j^\dag\rhd 1_b) \eta^b  (e_i\rhd 1_b) g \\
& = (f_j^\dag\rhd 1_a) (1_K\rhd g)\eta^a (e_i\rhd 1_a) - (f_j^\dag\rhd 1_b) \eta^b   (1_H\rhd g)(e_i\rhd 1_b) \\
& = (f_j^\dag\rhd 1_b) \left((1_K\rhd g) \eta^a -\eta^b  (1_H\rhd g)\right)  (e_i\rhd 1_a) ,
\\
\intertext{and}
(1_K\rhd g) \eta^a - \eta^b (1_H\rhd g) 
&=
\sum_{i,j} (f_j\rhd 1_b)\left(g \eta^a_{i,j} - \eta^b_{i,j} g \right) (e_i^\dag\rhd 1_a).
\end{align*}
Thus 
$(1_K\rhd g) \eta^a - \eta^b (1_H\rhd g) =0$ in $\widetilde{\cC}$
if and only if
$g\eta^a_{i,j} - \eta^b_{i,j} g =0$ in $\widetilde{\cC}$ for all $i,j$.
\end{proof}

%%%%%%%%%%%%%%%%%%%%%%%%%%%%%%%%%%%%%%%%%%%%%%%%%%%%%%%%%%%%%%%%%
\subsection{Centrally trivial and approximately inner endofunctors}\label{ApproxInnerCentTrival}

For the rest of this section, $\cC$ is a fixed separable $\rm W^*$-category. 

\begin{defn}
\label{Defn:CentrallyTrivial}
A functor $F\in\End(\cC)$ is called \textit{centrally trivial} if for all finite dimensional Hilbert spaces and all approximate natural transformations $\eta:H\rhd- \Rightarrow K\rhd-$, the following diagram approximately commutes. 
$$
\begin{tikzcd}
F(H\rhd a) 
\arrow{r}{F(\eta^{a}_n)} 
\arrow[swap]{d}{\sigma^{a}_{F,H}} 
& F(K\rhd a) 
\arrow{d}{\sigma^{a}_{F,K}} 
\\
H\rhd F(a) 
\arrow{r}{\eta^{F(a)}_n} 
& K\rhd F(a)
\end{tikzcd}
$$
\end{defn}

\begin{prop}
\label{Prop:CentralTrivial&CentralSq}
A functor $F$ is centrally trivial if and only if for all central sequences $\eta$, 
$F(\eta^a)=\eta^{F(a)}$ in $\widetilde{\cC}$.
\end{prop}
\begin{proof}
It is clear that $F(\eta^a)=\eta^{F(a)}$ in $\widetilde{\cC}$ for all central sequences $\eta$ if $F$ is centrally trivial.
Conversely, 
for each approximate natural transformation $\eta:H\rhd -\Rightarrow K\rhd -$,
$\eta_{i,j}$ defined in Lemma \ref{Lem:ApproNT&CentralSq} is a central sequence for each $i,j$.
Then by Definition \ref{defn:sigmaFH},
\begin{align*}
\sigma^a_{F,K} F(\eta^a) (\sigma^a_{F,H})^\dag 
& = \sum_j (f_j\rhd F(1_a)) F(f_j^\dag\rhd 1_a) F(\eta^a) \sum_i F(e_i\rhd 1_a) (e_i^\dag\rhd F(1_a)) \\
& = \sum_{i,j} (f_j\rhd F(1_a)) F(\eta_{i,j}^a)  (e_i^\dag\rhd F(1_a)) \\
& = \sum_{i,j} (f_j\rhd F(1_a)) \eta_{i,j}^{F(a)}  (e_i^\dag\rhd F(1_a)) 
= \eta^{F(a)},
\end{align*}
which implies that $F$ is centrally trivial.
\end{proof}

\begin{defn}
We denote the full subcategory of $\End(\cC)$ of centrally trivial endofunctors by $\ctEnd(\cC)$.
\end{defn}

\begin{prop}
\label{prop:ctEnd(C)FullReplete}
$\ctEnd(\cC)$ is a replete unitarily Cauchy complete $\rm W^*$-tensor subcategory of $\End(\cC)$. 
\end{prop}
\begin{proof}
Suppose $G,G'$ are centrally trivial.
For each central sequence $\eta$,
by Proposition \ref{Prop:CentralTrivial&CentralSq}, we have
$$G(G'(\eta^a))=G(\eta^{G'(a)})=\eta^{G(G'(a))},$$
which implies $G\circ G'$ is centrally trivial.

Now suppose $G$ is centrally trivial and $v:F\Rightarrow G$ is an isometric natural transformation for some other endofunctor $F\in\End(\cC)$.
That $F(\eta^a)=\eta^{F(a)}$ for all central sequences $\eta$ follows from the following approximately commuting diagram
\begin{equation}
\label{eq:IsometryCT}
\begin{tikzcd}
F(a)
\arrow[rrr, bend left=40, "F(\eta_n^{a})"]
\arrow{r}{v^a}
\arrow[rrr, bend right=40, "\eta_n^{F(a)}"] 
& 
G(a)\arrow[r, bend left= 30, "G(\eta_n^{a})"] 
\arrow[r, bend right=30, "\eta_n^{G(a)}"]
& 
G(a)
\arrow{r}{(v^{a})^\dag}
&
F(a)
\end{tikzcd},
\end{equation}
where we have used Remark \ref{rem:ArrowFlipping} to flip the arrow on the right.
Considering the case when $v$ is unitary shows that $\ctEnd(\cC)$ is replete.

As $\cC$ admits orthogonal direct sums, so does $\End(\cC)$.
Suppose $G,G'$ are centrally trivial, and let $\eta$ be a central sequence. 
Then
$$
(G\oplus G')(\eta^a) = G(\eta^a)\oplus G'(\eta^a) = \eta^{G(a)}\oplus \eta^{G'(a)}.
$$
By the proof of Step 1 of Lemma \ref{lem:CauchyCompletionPreservesCentralSequences}, $\eta^{G(a)}\oplus \eta^{G'(a)}=\eta^{(G\oplus G')(a)}$ in $\widetilde{\cC}$, and thus $G\oplus G'$ is centrally trivial.

As $\cC$ is orthogonal projection complete, so is $\End(\cC)$.
Suppose $G$ is centrally trivial and $\pi : G\Rightarrow G$ is an orthogonal projection natural transformation.
Then $\pi$ orthogonally splits in $\End(\cC)$, so there is an $F\in \End(\cC)$ and an isometry natural transformation $v: F\Rightarrow G$.
But then $F\in \ctEnd(\cC)$ by \eqref{eq:IsometryCT} above.
\end{proof}

\begin{ex}
For a finite dimensional Hilbert space $H$, the functor $H\rhd -$ is centrally trivial.
Note that the identity functor $\id_\cC$ is centrally trivial, and thus so is $\bigoplus_i \id_\cC$. 
Since the functor $H\rhd -$ is equivalent to $\bigoplus_i \id_\cC$,
by Proposition \ref{prop:ctEnd(C)FullReplete}, $H\rhd -$ is centrally trivial.
\end{ex}

\begin{defn}
\label{Defn:ApproximatelyInner}
A functor $F\in\End(\cC)$ is called \textit{approximately inner} if there exists a finite dimensional Hilbert space $H$ and an approximate isometry natural transformation $v: F\Rightarrow H\rhd -$ in $\End(\widetilde{\cC})$, i.e.,
$$
\begin{tikzcd}
F(a) \arrow{r}{v^{a}_n} \arrow[swap]{d}{F(f)} & H\rhd a \arrow{d}{1_{H}\rhd f} \\
F(b) \arrow{r}{v^{b}_n} & H\rhd b
\end{tikzcd}
$$
approximately commutes.
The pair $(v,H)$ is called an \textit{approximating sequence} for $F$.
\end{defn}

\begin{defn}
We denote the full subcategory of $\End(\cC)$ of approximately inner endofunctors by $\aiEnd(\cC)$.
\end{defn}

\begin{prop}
\label{prop:aiRepleteComplete}
$\aiEnd(\cC)$ is a replete unitarily Cauchy complete $\rm W^*$-tensor subcategory of $\End(\cC)$.
\end{prop}
\begin{proof}
Suppose $F$ and $F'$ are approximately inner, with approximating sequences $(v,H)$ and $(w,K)$ respectively. 
We claim $(\alpha^{-1}_{H,K,-}\cdot v^{K\rhd -}\cdot F(w),H\otimes K)$ is an approximating sequence for $F\circ F'$,
where the unitary natural transformation $\alpha_{H,K,-}:(H\otimes K)\rhd -\Rightarrow H\rhd (K\rhd -)$ is the left module associator.
For $f\in \cC(a\to b)$, consider the diagram
$$
\begin{tikzcd}[column sep=3em]
{F(F'(a))}
\arrow["F(w^a_n)", rightarrow]{r} 
\arrow["F(F'(f))", rightarrow]{d} 
&
{F(K\rhd a)} 
\arrow["v_n^{K\rhd a}", rightarrow]{r} 
\arrow["F(1_K\rhd f)", rightarrow]{d} 
&
{H\rhd (K\rhd a)}
\arrow["\alpha^\dag_{H,K,a}", rightarrow]{r} 
\arrow["1_H\rhd(1_K\rhd f)", rightarrow]{d} 
&
{(H\otimes K)\rhd a}
\arrow["1_{H\otimes K}\rhd f", rightarrow]{d} 
\\
{F(F'(b))} 
\arrow["F(w_n^b)"', rightarrow]{r} 
&
{F(K\rhd b)}
\arrow["v_n^{K\rhd b}"', rightarrow]{r} 
&
{H\rhd (K\rhd b)}
\arrow["\alpha^\dag_{H,K,b}"', rightarrow]{r} 
&
{(H\otimes K)\rhd b}
\end{tikzcd}.
$$
The left square commutes because $F'$ is approximately inner and $F$ is $\tau$-continuous on bounded subsets.
The middle square commutes because $F$ is approximately inner.
The right square commutes because $\alpha^\dag_{H,K,-}$ is natural.

Now suppose $F$ is approximately inner with approximating sequence $(v,H)$
and $u:G\Rightarrow F$ is a isometric natural transformation for some other endofunctor $G\in \End(\cC)$.
It is easy to see $(v\cdot u,H)$ is an approximating sequence for $G$ from the fact that $v\cdot u$ is again approximately isometric.
The case when $u$ is unitary shows that $\aiEnd(\cC)$ is replete.
The case when $u$ is an isometry shows that since $\cC$ is orthogonal projection complete, then so is $\aiEnd(\cC)$.

Finally, as $\cC$ is orthogonal direct sum complete, the orthogonal direct sum of functors is defined. 
If $F$ and $F'$ are approximately inner with approximating sequence $(v,H)$ and $(w,K)$ respectively, then $(v\oplus w, H\oplus K)$ is an approximating sequence for $F\oplus F'$, so $F\oplus F' \in \aiEnd(\cC)$.
\end{proof}

%%%%%%%%%%%%%%%%%%%%%%%%%%%%%%%%%%%%%%%%%%%%%%%%%%%%%%%%%%%%%%%%%%%%%%%
\subsection{Relative braiding between centrally trivial and approximately inner endofunctors}
\label{sec:RelativeBraiding}
The goal of this section is to show that the subcategories $\ctEnd(\cC)$ and $\aiEnd(\cC)$ of $\End(\cC)$ `commute' with each other in the sense of Definition \ref{Defn:centralizing} below. 

\begin{prop}
\label{prop:vwdagCommuteG}
Suppose $F\in\aiEnd(\cC)$ with approximating sequence $(v,H)$ and $G\in \ctEnd(\cC)$.
Then $(v^{G(a)})^\dag\cdot\sigma^a_{G,H}\cdot G(v^a)\in\widetilde{\cC}(G(F(a))\to F(G(a)))$ is independent of the choice of approximating sequence $(v,H)$ for $F$.
\end{prop}
\begin{proof}
Suppose $(v,H)$ and $(w,K)$ are approximating sequences for the endofunctor $F\in\aiEnd(\cC)$.
Observe that $vw^\dag:K\rhd - \Rightarrow H\rhd -$ is approximately natural. 
Now since $G$ is centrally trivial,
by Definition \ref{Defn:CentrallyTrivial},
$$
\begin{tikzcd}
& G(F(a))  
\arrow{dr}{G(w_n^{a})} &    
\\
G(H\rhd a)\arrow[swap]{d}{\sigma^{a}_{G,H}} 
\arrow{ur}{G((v_n^{a})^{\dag})} 
& & G(K\rhd a)
\arrow{d}{\sigma^{a}_{G,K}}
\\
H\rhd G(a)
\arrow{r}{(v_n^{G(a)})^{\dag}} 
& F(G(a)) 
\arrow{r}{w_n^{G(a)}}    
& K\rhd G(a)                                              
\end{tikzcd}
$$
approximately commutes for each object $a\in\cC$.
Then by Remark \ref{rem:ArrowFlipping},
$$
\begin{tikzcd}
& G(F(a)) 
\arrow[swap]{dl}{G(v_n^{a})} 
\arrow{dr}{G(w_n^{a})} &
\\
G(H\rhd a)
\arrow[swap]{d}{\sigma^{a}_{G,H}} &   & 
G(K\rhd a)
\arrow{d}{\sigma^{a}_{G,K}}
\\
H\rhd G(a)
\arrow{r}{(v_n^{G(a)})^{\dag}} 
& F(G(a)) 
& K\rhd G(a)
\arrow[swap]{l}{(w_n^{G(a)})^{\dag}}                                              
\end{tikzcd}
$$
approximately commutes.
Therefore, $(v^{G(a)})^\dag\cdot\sigma^a_{G,H}\cdot G(v^a)\in\widetilde{\cC}(G(F(a))\to F(G(a)))$ is independent of the choice of approximating sequences $(v,H)$ for $F$.
\end{proof}

\begin{lem}\label{lem:CauchySequenceBySubsequence}
Suppose $(x_n)$ is a sequence in a metric space $(X,d)$
which satisfies the following property:
\begin{itemize}
\item 
For all functions $k: \bbN \to \bbN$ such that $n< k_n\leq k_{n+1}$,
$d(x_n,x_{k_n})\to 0$.
(Note here that $(x_{k_n})$ is not quite a subsequence as terms can repeat.)
\end{itemize}
Then $(x_n)$ is Cauchy.
\end{lem}
\begin{proof}
Suppose for contradiction that $(x_n)$ is not Cauchy.
Then there is an $\varepsilon>0$ such that for all $N\in \bbN$, there are $m,n>N$ such that $d(x_m,x_n)\geq \varepsilon$.
Pick a subsequence $(x_{\ell_n})$ with $\ell_1>1$ and $d(x_{\ell_{2n-1}},x_{\ell_{2n}})\geq \varepsilon$.
We now define a function $k:\bbN \to \bbN$ by
$$
k_n:= 
\begin{cases}
\ell_1
&\text{if }1\le n < \ell_1
\\
\ell_{i+1}
&\text{if }\ell_i \leq n < \ell_{i+1}.
\end{cases}
$$
Since $k: \bbN \to \bbN$ satisfies $n< k_n \leq k_{n+1}$, we have $d(x_n , x_{k_n})\to 0$.
But observe that when $i$ is odd and $n=\ell_i$, we have
$$
d(x_{\ell_i} , x_{k_{\ell_i}})
=
d(x_{\ell_i} , x_{\ell_{i+1}})
\geq \varepsilon,
$$
a contradiction.
\end{proof}

\begin{thm}
\label{ThmDef:uGFa}
Suppose $F\in \aiEnd(\cC)$ and $G\in \ctEnd(\cC)$.
For each $a\in\cC$,
there exists a unique morphism in $\cC$, $u^a_{G,F}\in \cC(G(F(a))\to F(G(a)))$, such that for all approximate sequences $(v,H)$, 
the following diagram approximately commutes. 
$$
\begin{tikzcd}[column sep=3.4em]
{G(F(a))}
\arrow[" u^a_{G,F}", dashrightarrow]{r} 
\arrow["G(v_n^a)"', rightarrow]{d} 
&
{F(G(a))} 
\\
{G(H\rhd a)} 
\arrow["\sigma^a_{G,H}"', rightarrow]{r} 
&
{H\rhd G(a)}
\arrow["(v_n^{G(a)})^\dag"', rightarrow]{u} 
\end{tikzcd}
$$
\end{thm}
\begin{proof}
Suppose $(v,H)$ is an approximating sequence for $F$.
Then for any function $k:\bbN\to\bbN$ such that $n<k_n\le k_{n+1}$, $(v':=(v_{k_n})_n,H)$ is also an approximating sequence for $F$. 
By Proposition \ref{prop:vwdagCommuteG}, 
$$(v^{G(a)})^\dag\cdot\sigma^a_{G,H}\cdot G(v^a) = ({v'}^{G(a)})^\dag\cdot\sigma^a_{G,H}\cdot G({v'}^a)$$
in $\widetilde{\cC}(G(F(a))\to F(G(a)))$.
By \ref{tau:metrizable}, 
the $\tau$ topology on 
any bounded subspace of $\cC(G(F(a))\to F(G(a)))$ is completely metrizable.
Then by Lemma \ref{lem:CauchySequenceBySubsequence} and the definition of $\widetilde{\cC}$, 
the bounded sequence
$$(v^{G(a)})^\dag\cdot\sigma^a_{G,H}\cdot G(v^a) = \left((v_n^{G(a)})^\dag\cdot\sigma^a_{G,H}\cdot G(v_n^a)\right)_n$$
is a Cauchy sequence,
and we denote $u^a_{G,F}$ to be its unique limit.
Again, by Proposition \ref{prop:vwdagCommuteG}, $u^a_{G,F}$ does not depend on the choice of the approximating sequence for $F$.
\end{proof}

\begin{prop}
\label{prop:uGFUnitary}
For $F\in \aiEnd(\cC)$ and $G\in \ctEnd(\cC)$, 
$u^a_{G,F}$ is unitary.
\end{prop}
\begin{proof}
By Proposition \ref{prop:vwdagCommuteG}, 
for an approximating sequence $(v,H)$ for $F$, 
the following diagram approximately commutes.
$$
\begin{tikzcd}[column sep=4.5em]
G(H\rhd a) 
\arrow{r}{G(v_n^{a}\cdot (v_n^a)^\dag)} 
\arrow[swap]{d}{\sigma^{a}_{G,H}} 
& G(H\rhd a) 
\arrow{d}{\sigma^{a}_{G,H}} 
\\
H\rhd G(a) 
\arrow{r}{v_n^{G(a)}\cdot (v_n^{G(a)})^\dag} 
& H\rhd G(a)
\end{tikzcd}
$$
Since $\sigma^a_{G,H}$ is unitary, then in $\widetilde{\cC}$, we have
\begin{align*}
u^a_{G,F} \cdot (u^a_{G,F})^\dag 
& = \left((v^{G(a)})^\dag\cdot \sigma^a_{G,H} \cdot G(v^a)\right)\cdot \left(G(v^a)^\dag \cdot (\sigma^a_{G,H})^\dag \cdot v^{G(a)}\right) \\
& = (v^{G(a)})^\dag\cdot \left(\sigma^a_{G,H} \cdot G(v^a) \cdot G(v^a)^\dag \cdot (\sigma^a_{G,H})^\dag \right) \cdot v^{G(a)}\\
& = (v^{G(a)})^\dag \cdot\left(v^{G(a)}\cdot(v^{G(a)})^\dag \right)\cdot v^{G(a)}\\
& = \id_{F(G(a))},\\
(u^a_{G,F})^\dag \cdot u^a_{G,F} 
& = \left(G(v^a)^\dag \cdot (\sigma^a_{G,H})^\dag \cdot v^{G(a)}\right) \cdot \left((v^{G(a)})^\dag\cdot \sigma^a_{G,H} \cdot G(v^a)\right) \\
& = G(v^a)^\dag \cdot \left((\sigma^a_{G,H})^\dag \cdot v^{G(a)} \cdot (v^{G(a)})^\dag\cdot \sigma^a_{G,H}\right) \cdot G(v^a) \\
& = G(v^a)^\dag \cdot \left(G(v^a) \cdot G(v^a)^\dag\right) \cdot G(v^a) \\
& = \id_{G(F(a))}.
\end{align*}
Since the inclusion $\cC\hookrightarrow \widetilde{\cC}$ is faithful, we are finished.
\end{proof}

\begin{prop}
\label{prop:uGF-Natural}
For $F\in \aiEnd(\cC)$ and $G\in \ctEnd(\cC)$, 
the family $u_{G,F}:\{u^{a}_{G,F}\}_{a\in \cC}$ is a natural transformation $G\circ F \Rightarrow F\circ G$.
\end{prop}
\begin{proof}
For $f\in\cC(a\to b)$, it suffices to prove the following diagram approximately commutes.
\[
\begin{tikzcd}
{G(F(a))}
\arrow["u^a_{G,F}", rightarrow]{rrr} 
\arrow["G(F(f))", rightarrow, swap]{ddd} 
\arrow["G(v_n^a)", rightarrow]{rd} 
&
{}
&
{}
&
{F(G(a))}
\arrow["F(G(f))", rightarrow]{ddd} 
\\
{}
&
{G(H\rhd a)}
\arrow["\sigma^a_{G,H}", rightarrow]{r} 
\arrow["G(1_H\rhd f)", rightarrow, swap]{d} 
&
{H\rhd G(a)}
\arrow["1_H\rhd G(f)", rightarrow]{d} 
\arrow["(v_n^{G(a)})^\dag", rightarrow]{ru} 
&
{}
\\
{}
&
{G(H\rhd b)}
\arrow["\sigma^b_{G,H}", rightarrow, swap]{r} 
&
{H\rhd G(b)}
\arrow["(v_n^{G(b)})^\dag", rightarrow, swap]{rd} 
&
{}
\\
{G(F(b))}
\arrow["u^b_{G,F}", rightarrow, swap]{rrr} 
\arrow["G(v_n^b)", rightarrow, swap]{ru} 
&
{}
&
{}
&
{F(G(b))}
\end{tikzcd}
\]
The top and bottom squares commute by definition of $u^a_{G,F}$ by Theorem \ref{ThmDef:uGFa}. 
The left square commute because $F$ is approximately inner and $G$ is $\tau$-continuous on bounded sets.
The right square commutes by approximate naturality.
The middle square commutes by naturality of $\sigma$.
\end{proof}

In the next definition, we use strict monoidal categories simply because the case we care about is strict, but one obtains the general definition by inserting coheretors where appropriate.

\begin{defn}\label{Defn:centralizing} Let $\cC$ be a (strict) $\rm C^*$-tensor category, and $\cA, \cB$ full, replete $\dag$ tensor subcategories of $\cC$. A \textit{centralizing structure} on the pair $(\cA, \cB)$ is a family of unitary isomorphisms $u_{a,b}:a\otimes b\to b\otimes a$ for $a\in \cA$ and $b\in \cB$ 
satisfying the following conditions:
\begin{enumerate}
%\item
%(Unital) For $a\in \cA$, $b\in \cB$, $u_{a, 1}=1_{a}$, $u_{1, b}=1_{b}$.
%\This condition is automatic cf.~\cite[Ex.~8.1.6]{MR3242743}.
\item
(Natural) For $a,a'\in \cA$ and $f\in \cC(a\to a')$ and $b,b\in \cB$, $g\in \cC(b\to b')$, $(g\otimes f)\circ u_{a,b}=u_{a',b'}\circ (f\otimes g)$.
\item
(Braid relation 1) For $a\in \cA$, $b,b'\in \cB$, $(1_{a}\otimes u_{a,b'})\circ (u_{a,b}\otimes 1_{b'})=u_{a, b\otimes b'}$.
\item
(Braid relation 2) For $a,a'\in \cA$ and $b\in \cB$, $(u_{a,b}\otimes 1_{a'})\circ (1_{a}\otimes u_{a',b})=u_{a\otimes a', b}$.
\end{enumerate}
\end{defn}

The goal of this section is to construct a \textit{centralizing structure} for the pair of full, replete subcategories $( \ctEnd(\cC),  \aiEnd(\cC))$ inside $ \End(\cC)$.

\begin{prop}
\label{prop:uGF-Centralizing-Natural}
For $F\in \aiEnd(\cC)$ and $G\in \ctEnd(\cC)$, 
$u_{G,F}$ satisfies condition (1) in Definition \ref{Defn:centralizing}.
\end{prop}
\begin{proof}
Suppose $G,G'\in\ctEnd(\cC)$ and $F,F'\in\aiEnd(\cC)$ with approximating sequences $(v,H)$ and $(w,K)$ respectively. 
Let $\eta\in\End(\cC)( F\Rightarrow F' )$ be a natural transformation.
We shall show $u^a_{G,F'}\cdot G(\eta^a) =\eta^{G(a)}\cdot u^a_{G,F}$.
It suffices to prove the following diagram approximately commutes.
\[
\begin{tikzcd}
{G(F(a))}
\arrow["u^a_{G,F}", rightarrow]{rrr} 
\arrow["G(\eta^a)", rightarrow, swap]{ddd} 
\arrow["G(v_n^a)", rightarrow]{rd} 
&
{}
&
{}
&
{F(G(a))}
\arrow["\eta^{G(a)}", rightarrow]{ddd} 
\\
{}
&
{G(H\rhd a)}
\arrow["\sigma^a_{G,H}", rightarrow]{r} 
\arrow["G((w_n\cdot\eta\cdot v_n^\dag)^a)", rightarrow, swap]{d} 
&
{H\rhd G(a)}
\arrow["(w_n\cdot\eta\cdot v_n^\dag)^{G(a)}", rightarrow]{d} 
\arrow["(v_n^{G(a)})^\dag", rightarrow]{ru} 
&
{}
\\
{}
&
{G(K\rhd a)}
\arrow["\sigma^a_{G,K}", rightarrow, swap]{r} 
&
{K\rhd G(a)}
\arrow["(w_n^{G(a)})^\dag", rightarrow, swap]{rd} 
&
{}
\\
{G(F'(a))}
\arrow["u^a_{G,F'}", rightarrow, swap]{rrr} 
\arrow["G(w_n^a)", rightarrow, swap]{ru} 
&
{}
&
{}
&
{F'(G(a))}
\end{tikzcd}
\]
The left/right squares commute since $v,w$ are approximately isometric.
The top/bottom squares commute by the definition of $u_{G,F}$.
The middle square commutes because $G$ is centrally trivial.

Let $\psi\in\End(\cC)(G\Rightarrow G')$.
We shall show $u^a_{G',F}\cdot \psi^{F(a)} = F(\psi^a)\cdot u^a_{G,F}$.
It suffices to prove the following diagram approximately commutes.
\[
\begin{tikzcd}
{G(F(a))}
\arrow["u^a_{G,F}", rightarrow]{rrr} 
\arrow["\psi^{F(a)}", rightarrow, swap]{ddd} 
\arrow["G(v_n^a)", rightarrow]{rd} 
&
{}
&
{}
&
{F(G(a))}
\arrow["\psi^{G(a)}", rightarrow]{ddd} 
\\
{}
&
{G(H\rhd a)}
\arrow["\sigma^a_{G,H}", rightarrow]{r} 
\arrow["\psi^{H\rhd a}", rightarrow, swap]{d} 
&
{H\rhd G(a)}
\arrow["1_H\rhd \psi^a", rightarrow]{d} 
\arrow["(v_n^{G(a)})^\dag", rightarrow]{ru} 
&
{}
\\
{}
&
{G'(H\rhd a)}
\arrow["\sigma^a_{G',H}", rightarrow, swap]{r} 
&
{H\rhd G'(a)}
\arrow["(v_n^{G'(a)})^\dag", rightarrow, swap]{rd} 
&
{}
\\
{G'(F(a))}
\arrow["u^a_{G',F}", rightarrow, swap]{rrr} 
\arrow["G'(v_n^a)", rightarrow, swap]{ru} 
&
{}
&
{}
&
{F(G'(a))}
\end{tikzcd}
\]
The left square commutes by the naturality of $\psi$.
The right square commutes because $F$ is approximately inner.
The top/bottom squares commute by the definition of $u_{G,F}$.
The middle square commutes by the naturality of $\sigma$.
\end{proof}

\begin{prop}
\label{braid4}
For $F\in \aiEnd(\cC)$ and $G,G'\in \ctEnd(\cC)$, 
$u_{G,F}$ satisfies condition (2) of Definition \ref{Defn:centralizing}.
\end{prop}
\begin{proof}
It suffices to show that for an approximating sequence $(v,H)$ for $F$, the following diagram approximately commutes.
\[
\begin{tikzcd}[column sep=5em]
GG'F(a)
\arrow[dddddd, swap, bend right=55, "G(u^a_{G',F})"']
\arrow[rrrddd, bend left=25, "u^a_{G\circ G',F}"]
\arrow[dd, "G(G'(v_n^a))"]
\arrow[rdd, "G(G'(v_n^a))"]
& & & \\ 
& & & & \\ 
GG'(H\rhd a)
\arrow[r, swap, "G(G'(v_n^a(v_n^a)^\dag))"]
\arrow[dd, "G(\sigma^a_{G',H})"]
& GG'(H\rhd a)
\arrow[rd, "\sigma^a_{G\circ G',H}"]
\arrow[dd, swap, "G(\sigma^a_{G',H})"]
& & \\
& & 
H\rhd GG'(a)
\arrow[r, "(v_n^{G(G'(a))})^\dag"]
& FGG'(a)
\\
G(H\rhd G'(a))
\arrow[dd, "G(v_n^{G'(a)})^\dag"]
\arrow[r, "G(v_n^{G'(a)}(v_n^{G'(a)})^\dag)"]
& G(H\rhd G'(a))
\arrow[ru, swap, "\sigma^{G'(a)}_{G,H}"]
& & \\
& & & & \\ 
GFG'(a)
\arrow[ruu, swap, "G(v_n^{G'(a)})"]
\arrow[rrruuu, bend right=25, swap, "u^{G'(a)}_{G,F}"]
& & &
\end{tikzcd}
\]
For any choice of approximating sequence, the outer three cells approximately commute by the definition of $u$ from Theorem \ref{ThmDef:uGFa}.
The upper triangle approximately commutes as $v$ is approximately isometric.
The middle triangle approximately commutes by Proposition \ref{prop:sigmabraid}(2). 
The lower triangle is trivial. 
Finally, the remaining square approximately commutes because $vv^\dag:H\rhd -\Rightarrow H\rhd -$ is approximately natural and $G'$ is centrally trivial (see also Proposition \ref{prop:uGFUnitary}). 
\end{proof}

\begin{prop}
\label{braid3}
For $F,F'\in \aiEnd(\cC)$ and $G\in \ctEnd(\cC)$, 
$u_{G,F}$ satisfies condition (3) of Definition \ref{Defn:centralizing}.
\end{prop}
\begin{proof}
We must prove for each fixed $a\in \cC$, $F(u^a_{G,F'})u_{G,F}^{F'(a)}=u^a_{G,F\circ F'}$.
To do so, we carefully choose approximating sequences $(v,H)$ and $(w,K)$ for $F$ and $F'$ respectively such that the following diagram approximately commutes.
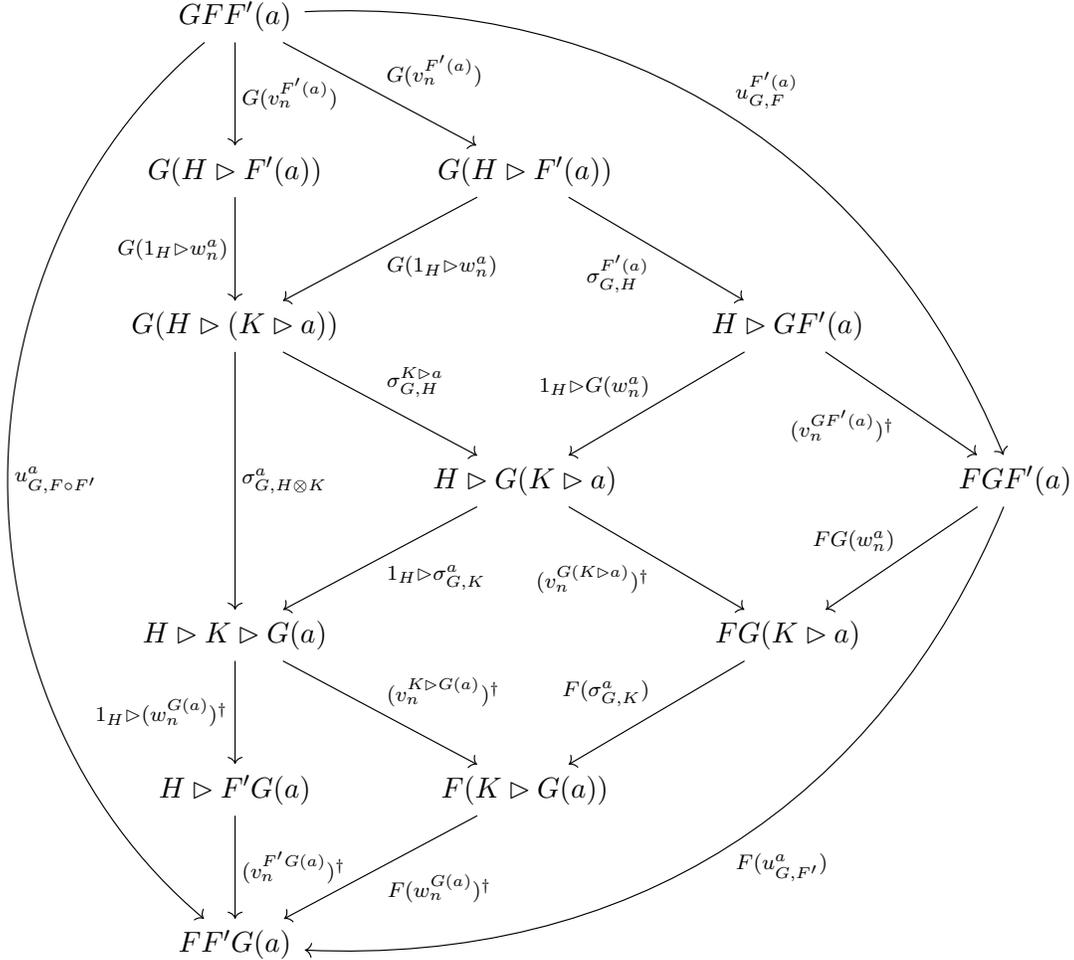
\[
\begin{tikzcd}
GFF'(a)
\arrow{dd}{G(v_n^{F'(a)})}
\arrow[ddr, "G(v_n^{F'(a)})"] 
\arrow[dddddddddddd, swap, bend right=50, "u^{a}_{G, F\circ F'}"'] 
\arrow[ddddddrrr, bend left=35, "u^{F'(a)}_{G,F}"]
& & & \\
& & & \\
G(H\rhd F'(a))
\arrow[dd, swap, "G(1_{H}\rhd w_n^{a})"]
& G(H\rhd F'(a))
\arrow{ddl}{G(1_{H}\rhd w_n^{a})} 
\arrow[ddr, swap, "\sigma^{F'(a)}_{G,H}"]
& & \\
& & & \\
G(H\rhd(K\rhd a))
\arrow{dddd}{\sigma^{a}_{G, H\otimes K}} 
\arrow{ddr}{\sigma^{K\rhd a}_{G,H}}
& & H\rhd GF'(a) 
\arrow{ddl}[swap]{1_{H}\rhd G(w_n^{a})} 
\arrow[swap]{ddr}{(v_n^{GF'(a)})^{\dag}} 
& \\
& & & \\
& H\rhd G(K\rhd a)
\arrow{ddl}{1_{H}\rhd \sigma^{a}_{G,K}} 
\arrow[swap]{ddr}{(v_n^{G(K\rhd a)})^{\dag}} & & FGF'(a) 
\arrow[swap]{ddl}{FG(w_n^{a})}  
\arrow[ddddddlll, bend left=35, "F(u^{a}_{G,F'})"]\\
& & & \\
H\rhd K\rhd G(a) 
\arrow[swap]{dd}{1_{H}\rhd (w_n^{G(a)})^{\dag}} 
\arrow{ddr}{(v_n^{K\rhd G(a)})^{\dag}} 
& & FG(K\rhd a) 
\arrow[swap]{ddl}{F(\sigma^{a}_{G,K})} & \\
& & & \\
H\rhd F'G(a) 
\arrow{dd}{(v_n^{F'G(a)})^{\dag} } 
& F(K\rhd G(a)) 
\arrow[ddl, "F(w_n^{G(a)})^{\dag}"] 
& & \\
& & & \\
FF'G(a) 
& & & \\
\end{tikzcd}
\]
For any choices of approximating sequences, the outer three cells approximately commute by the definition of $u$ from Theorem \ref{ThmDef:uGFa}. 
The upper left square is trivial. 
The adjacent square to its lower right approximately commutes by the naturality of $\sigma_{G,H}$.
The left middle triangle approximately commutes by Proposition \ref{prop:sigmabraid}(1),
and the square to the lower right of this triangle approximately commutes by approximate naturality of $v$.
This leaves us to consider the two remaining squares
$$
\begin{tikzcd}[column sep=3.4em]
{H\rhd K\rhd G(a)}
\arrow["(v_n^{K\rhd G(a)})^\dag", rightarrow]{r} 
\arrow["1\rhd(w_n^{G(a)})^\dag"', rightarrow]{d} 
&
{F(K\rhd G(a))} 
\arrow["F(w_n^{G(a)})^\dag", rightarrow]{d} 
\\
{H\rhd F'G(a)} 
\arrow["(v_n^{F'G(a)})^\dag"', rightarrow]{r} 
&
{FF'G(a)}
\end{tikzcd}
\quad\text{and}\quad
\begin{tikzcd}[column sep=3.4em]
{H\rhd GF'(a)}
\arrow["(v_n^{GF'(a)})^\dag", rightarrow]{r} 
\arrow["1_H\rhd G(w_n^a)"', rightarrow]{d} 
&
{FGF'(a)} 
\arrow["FG(w_n^a)", rightarrow]{d} 
\\
{H\rhd G(K\rhd a)} 
\arrow["(v_n^{G(K\rhd a)})^\dag"', rightarrow]{r} 
&
{FG(K\rhd a)}
\end{tikzcd}
\,.
$$
These squares may not approximately commute for an arbitrary choice of approximating sequences, but we can get around this by replacing $v$ with a subsequence, as any approximating sequence for $F$ can be used.
Indeed, for $b,c\in \cC$, let $d_{b\to c}$ denote a metric inducing the $\tau$-topology on bounded subsets of $\cC(b\to c)$.
Since $(w_n^{G(a)})^\dag$ and $G(w_n^a)$ are morphisms in $\cC$ for each fixed $n$, using approximate naturality of $v^\dag$,
we may inductively choose $1\leq k_{n-1} < k_n$ so that both
\begin{align*}
d_{H\rhd K\rhd G(a)\to FF'G(a)}\left(F(w_n^{G(a)})^\dag\cdot (v_{k_n}^{K\rhd G(a)^\dag}),\  (v_{k_n}^{F'G(a)^\dag})\cdot (1_H\rhd (w_n^{G(a)})^\dag) \right)
&< \frac{1}{n}
\qquad\qquad\text{and}
\\
d_{H\rhd GF'(a) \to FG(K\rhd a)}\left(FG(w_n^a)\cdot (v_{k_n}^{GF'(a)^\dag}),\ (v_{k_n}^{G(K\rhd a)^\dag})\cdot (1_H\rhd G(w_n^a) \right)
&< \frac{1}{n}
\end{align*}
simultaneously.
Thus replacing $(v_n^c)$ with $(v_{k_n}^c)$ for all $c\in\cC$, 
the previous arguments still hold, as $((v_{k_n}^c)_c,H)$ is still an approximating sequence for $F$, and
these two squares approximately commute for our fixed $a\in \cC$.
Since we only need to verify condition (2) of Definition \ref{Defn:centralizing} for one $a\in\cC$ at a time, the result follows.
\end{proof}

\begin{defn}
We define the category $\locEnd(\cC)$ of \textit{local endofunctors}, to be 
the full
$\rm W^*$-monoidal subcategory of $\End(\cC)$ whose objects are normal $\dag$-endofunctors which are both approximately inner and centrally trivial. 
By Propositions \ref{prop:ctEnd(C)FullReplete} and \ref{prop:aiRepleteComplete}, $\locEnd(\cC)$ is replete and unitarily Cauchy complete.
The family of unitary natural transformations $u_{G,F}: G\circ F\Rightarrow F\circ G$ equips $\locEnd(\cC)$ with the structure of a \textit{braided $\rm W^*$-tensor category}.
\end{defn}

%%%%%%%%%%%%%%%%%%%%%%%%%%%%%%%%%%%%%%%%%%%%%%%%%%%%%%%%%%%%%%%
%%%%%%%%%%%%%%%%%%%%%%%%%%%%%%%%%%%%%%%%%%%%%%%%%%%%%%%%%%%%%%%
%%%%%%%%%%%%%%%%%%%%%%%%%%%%%%%%%%%%%%%%%%%%%%%%%%%%%%%%%%%%%%%
\section{\texorpdfstring{$\Chi(M)$}{chi(M)} for finite von Neumann algebras via bimodules}
\label{sec:Translation}

In this section, we give the main application of the construction of the last section to give a definition of $\Chi(M)$ for a $\rm II_1$-factor $M$ in terms of bimodules. 

Given a W*-tensor category $\cC$, the \textit{dualizable part}, denoted $\cC_{\rm dualizable}$, is the full tensor subcategory whose objects have two sided duals.
Observe that if $\End_\cC(1_\cC)$ is finite dimensional, then $\cC_{\rm dualizable}$ is a rigid $\rm C^*/W^*$ tensor category.
If moreover $\cC_{\rm dualizable}$ is semisimple (equivalently orthogonal projection complete), it is called a \emph{unitary multitensor category} \cite{MR4133163}; it is called a \emph{unitary tensor category} if $\End_\cC(1_\cC)$ is one dimensional.

\begin{defn}
Given a von Neumann algebra $M$, we denote the braided unitary tensor category $\Chi(M):=\locEnd(\fgpMod(M))_{\rm dualizable}$.
\end{defn}

Identifying $\End(\fgpMod(M))\simeq\fgpBim(M)^{\rm mp}$ as in Remark \ref{rem:OppositeIssue}, we call a bimodule $X\in \fgpBim(M)$ approximately inner (respectively centrally trivial) if the functor $-\boxtimes_M X$ is approximately inner (respectively centrally trivial). 
We see the underlying unitary tensor category of $\Chi(M)$ agrees with the definition of $\Chi(M)$ from \cite[Rem.~2.7]{MR1317367}.
We may thus think of $\Chi(M)$ as the dualizable approximately inner and centrally trivial bimodules of $M$ whose conjugate bimodule is also approximately inner and centrally trivial. 

Since $\locEnd(\fgpMod(M))$ is braided, we get a monoidal equivalence from $\Chi(M)$ to its monoidal opposite $\Chi(M)^{\rm mp}$, which allows us to bypass this opposite issue. 
We address this in detail in Remark \ref{rem:OppositeWithBraiding} below and the discussion thereafter.

Note the dualizable objects in $\fgpBim(M)$ are precisely the bifinite Hilbert bimodules of $M$, and the dual object is the conjugate bimodule. 
It is easy to see that $H$ is centrally trivial if and only if $\overline{H}$ is centrally trivial, but for approximately inner, any such relationship is not obvious. 
Thus it may be possible for a centrally trivial bifinite bimodule to be approximately inner, but its conjugate bimodule may not be approximately inner. 
We do not have an example, but we cannot rule this out at this time.

%%%%%%%%%%%%%%%%%%%%%%%%%%%%%%%%%%%%%%%%%%%%%%%%%%%%%%%%%%%%%%%%%%%%%%%%%%%%%%%%%
\subsection{Module and bimodule realization}
\label{sec:ModuleRealization}

In order to translate the results of \S\ref{sec:W*ApproxNatTrans}, we use the graphical calculus for $\fgpMod(M)$ as a right $\fgpBim(M)$-module $\rm W^*$-category.
One way to do this is to use the realization graphical calculus from \cite{2105.12010} based on \cite[\S4.1]{MR3221289}.
We only introduce the part of the graphical calculus that we need in this section, and we introduce the rest of the graphical calculus for Q-system realization in \S\ref{sect:Q-systemRealization} below.

In the 2D graphical calculus for $\WStarRCorr$, von Neumann algebras are denoted by shaded regions, bimodules are denoted by 1D strands, and intertwiners are denoted by 0D coupons.
For the rest of this section, let $M$ be a $\rm II_1$ factor.
Of particular importance is the right $M$-module $M_M$.
The missing label on the left hand side is inferred to be $\bbC$, which is always represented by the empty shading.
That is, we identify $\fgpMod(M)=\WStarRCorr(\bbC \to M)$ in the 2D graphical calculus.
We denote ${}_\bbC M_M$ by a dashed line which is shaded by $M$ on the right hand side.
$$
\tikzmath{\filldraw[\rColor, rounded corners=5, very thin, baseline=1cm] (0,0) rectangle (.6,.6);}=M
\qquad
\qquad
\tikzmath{
\filldraw[white, rounded corners=5, very thin, baseline=1cm] (0,0) rectangle (.6,.6);
\draw[dotted, rounded corners=5pt] (0,0) rectangle (.6,.6);
}=\bbC
\qquad
\qquad
\tikzmath{
\begin{scope}
\clip[rounded corners = 5] (0,0) rectangle (.6,.6);
\filldraw[\rColor] (.3,0) rectangle (.6,.6);
\end{scope}
\draw[dashed] (.3,0) -- (.3,.6);
\draw[dotted, rounded corners=5pt] (.3,0) -- (0,0) -- (0,.6) -- (.3,.6);
}={}_\bbC M_M.
$$
A bounded, adjointable intertwiner $f: Y_M \to Z_M$ is denoted graphically by
$$
\tikzmath{
\begin{scope}
\clip[rounded corners = 5] (-.3,-.7) rectangle (.6,.7);
\filldraw[\rColor] (0,-.7) -- (0,.7) -- (.7,.7) -- (.7,-.7);
\end{scope}
\draw[\YColor, thick] (0,-.7) -- (0,0);
\draw[\ZColor, thick] (0,0) -- (0,.7);
\roundNbox{unshaded}{(0,0)}{.3}{0}{0}{\scriptsize{$f$}};
}
:Y_M \to Z_M;
\qquad\qquad\qquad
\tikzmath{
\begin{scope}
\clip[rounded corners = 5] (0,0) rectangle (.6,.6);
\filldraw[\rColor] (.3,0) rectangle (.6,.6);
\end{scope}
\draw[\YColor, thick] (.3,0) -- (.3,.6);
\draw[dotted, rounded corners=5pt] (.3,0) -- (0,0) -- (0,.6) -- (.3,.6);
}=Y_M
\qquad\text{and}\qquad
\tikzmath{
\begin{scope}
\clip[rounded corners = 5] (0,0) rectangle (.6,.6);
\filldraw[\rColor] (.3,0) rectangle (.6,.6);
\end{scope}
\draw[\ZColor, thick] (.3,0) -- (.3,.6);
\draw[dotted, rounded corners=5pt] (.3,0) -- (0,0) -- (0,.6) -- (.3,.6);
}=Z_M.
$$

\begin{construction}[{\cite[Const.~4.1]{2105.12010}}]
Given $Y_M\in \fgpMod(M)$, the map $x\mapsto L_y$ where $L_y(m) := ym$ gives a canonical isomorphism $Y_M \cong |Y|_M:= \Hom(M_M \to Y_M)$ 
such that $\langle x|y\rangle_M = L_x^\dag L_y$.
In \S\ref{sec:CTandAIbimods} and \S\ref{sec:LocalExtension} below, we make heavy use of this identification.
\[
\tikzmath{
\begin{scope}
\clip[rounded corners = 5] (-.3,-.7) rectangle (.6,.7);
\filldraw[\rColor] (0,-.7) -- (0,.7) -- (.7,.7) -- (.7,-.7);
\end{scope}
\draw[dashed] (0,-.7) -- (0,0);
\draw[\YColor,thick] (0,0) -- (0,.7);
\roundNbox{unshaded}{(0,0)}{.3}{0}{0}{\scriptsize{$y$}};
}
\in 
|Y|:= \Hom(M_M \to Y_M).
\]
The right $M$-action is given by identifying $M=\End(M_M)$ and stacking coupons:
\[
\tikzmath{
\begin{scope}
\clip[rounded corners = 5] (-.3,-1.7) rectangle (.6,.7);
\filldraw[\rColor] (0,-1.7) -- (0,.7) -- (.7,.7) -- (.7,-1.7);
\end{scope}
\draw[dashed] (0,-1.7) -- (0,0);
\draw[\YColor,thick] (0,0) -- (0,.7);
\roundNbox{unshaded}{(0,0)}{.3}{0}{0}{\scriptsize{$y$}};
\roundNbox{unshaded}{(0,-1)}{.3}{0}{0}{\scriptsize{$a$}};
}\,
=
y\lhd a.
\]
The condition that $\{c_j\}$ is a $Y_M$-basis can be written graphically as
\[
\sum_j
\tikzmath{
\begin{scope}
\clip[rounded corners = 5] (-.15,-1.2) rectangle (.6,1.2);
\filldraw[\rColor] (0,-1.2) rectangle (.6,1.2);
\end{scope}
\draw[dashed] (0,-.5) -- (0,.5);
\draw[\YColor,thick] (0,.5) -- (0,1.2);
\draw[\YColor,thick] (0,-.5) -- (0,-1.2);
\roundNbox{unshaded}{(0,.5)}{.3}{0}{0}{\scriptsize{$c_j$}};
\roundNbox{unshaded}{(0,-.5)}{.3}{0}{0}{\scriptsize{$c_j^\dag$}};
}
=
\tikzmath{
\begin{scope}
\clip[rounded corners = 5] (-.15,-.5) rectangle (.45,.5);
\filldraw[\rColor] (0,-.5) rectangle (.45,.5);
\end{scope}
\draw[\YColor,thick] (0,-.5) -- (0,.5);
}\,.
\]

If $X\in \fgpBim(M)$, then we define the realization $|X|$ slightly differently:
$$
\tikzmath{
\begin{scope}
\clip[rounded corners = 5] (-.3,-.7) rectangle (.6,.7);
\filldraw[\rColor] (0,-.7) -- (0,0) -- (-.2,0) -- (-.2,.7) -- (.7,.7) -- (.7,-.7);
\end{scope}
\draw[dashed] (0,-.7) -- (0,0);
\draw[dashed] (-.2,0) -- (-.2,.7);
\draw[\XColor,thick] (.2,0) -- (.2,.7);
\roundNbox{unshaded}{(0,0)}{.3}{.05}{.05}{\scriptsize{$x$}};
}
\in
|X|:= \Hom(M_M \to M\boxtimes_M X_M).
$$
While this definition is canonically isomorphic to the previous definition via the canonical unitor $M\boxtimes_M X \cong X$, this definition offers the advantage of depicting both the left and the right $M$-actions graphically by
\[
\tikzmath{
\begin{scope}
\clip[rounded corners = 5] (-.4,-.7) rectangle (.75,1.7);
\filldraw[\rColor] (0,-.7) -- (0,0) -- (-.25,0) -- (-.25,1.7) -- (.75,1.7) -- (.75,-.7);
\end{scope}
\draw[dashed] (-.25,0) -- (-.25,1.7);
\draw[\XColor,thick] (.25,0) -- (.25,1.7);
\draw[dashed] (0,-.7) -- (0,0);
\roundNbox{unshaded}{(-.25,1)}{.3}{0}{0}{\scriptsize{$a$}};
\roundNbox{unshaded}{(0,0)}{.3}{.15}{.15}{\scriptsize{$x$}};
}
\,=
a\rhd x
\qquad\text{and}\qquad
\tikzmath{
\begin{scope}
\clip[rounded corners = 5] (-.4,-1.7) rectangle (.7,.7);
\filldraw[\rColor] (0,-1.7) -- (0,0) -- (-.2,0) -- (-.2,.7) -- (.75,.7) -- (.75,-1.7);
\end{scope}
\draw[dashed] (-.2,0) -- (-.2,.7);
\draw[\XColor,thick] (.2,0) -- (.2,.7);
\draw[dashed] (0,-1.7) -- (0,0);
\roundNbox{unshaded}{(0,0)}{.3}{.05}{.05}{\scriptsize{$x$}};
\roundNbox{unshaded}{(0,-1)}{.3}{0}{0}{\scriptsize{$b$}};
}
\,=
x\lhd b.
\]
\end{construction}

%%%%%%%%%%%%%%%%%%%%%%%%%%%%%%%%%%%%%%%%%%%%%%%%%%%%%%%%%%%%%%%%%%%%%%%%%%%%%%%%%%%%
\subsection{Centrally trivial and approximately inner bimodules}
\label{sec:CTandAIbimods}
In this section we clarify the equivalence between our definitions of approximately inner and centrally trivial given in terms of endofunctors, and Popa's original definitions as translated to bimodules, which are much more natural from the point of view of a single von Neumann algebra \cite[Rem.~2.7]{MR1317367}.
For this section, let $M$ be a finite separable von Neumann algebra with faithful normal trace $\tr_M$.

\begin{nota}
For norm bounded sequences $(f_n)_n,(g_n)_n\subseteq \Hom(X_M\to Y_M)$,
we write $f_n\approx g_n$ if $\lim_n\|f_n-g_n\|_2 = 0$. 
For $f\in \Hom(X_M\to Y_M)$, we write $f\approx f_n$ if $\lim_n \|f-f_n\|_2 = 0$.
As a consequence, $f\approx g$ if and only if $f=g$.

We remark that since composition is jointly $\tau$-continuous on norm-bounded subsets of hom spaces, we may use $f_n\approx g_n$ as a \emph{local relation} amongst morphisms in $\fgpMod(M)$.
If we precompose with an appropriate $h$, we still have $f_n \circ h \approx g_n\circ h$, and similarly for composing on the other side, or both sides simultaneously.
\end{nota}

\begin{prop}
\label{prop:BimodualCT&CS}
$X$ is centrally trivial over $M$ if and only if for all central sequences $(a_n)_n\subseteq M$ and for all $x\in X$, $\|a_nx-xa_n\|_2\to 0$, i.e.,
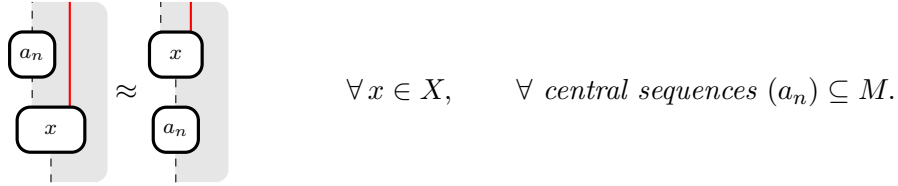
\[
\tikzmath{
\begin{scope}
\clip[rounded corners = 5] (-.4,-.7) rectangle (.75,1.7);
\filldraw[\rColor] (0,-.7) -- (0,0) -- (-.25,0) -- (-.25,1.7) -- (.75,1.7) -- (.75,-.7);
\end{scope}
\draw[dashed] (-.25,0) -- (-.25,1.7);
\draw[\XColor,thick] (.25,0) -- (.25,1.7);
\draw[dashed] (0,-.7) -- (0,0);
\roundNbox{unshaded}{(-.25,1)}{.3}{0}{0}{\scriptsize{$a_n$}};
\roundNbox{unshaded}{(0,0)}{.3}{.15}{.15}{\scriptsize{$x$}};
}
\approx
\tikzmath{
\begin{scope}
\clip[rounded corners = 5] (-.4,-1.7) rectangle (.7,.7);
\filldraw[\rColor] (0,-1.7) -- (0,0) -- (-.2,0) -- (-.2,.7) -- (.75,.7) -- (.75,-1.7);
\end{scope}
\draw[dashed] (-.2,0) -- (-.2,.7);
\draw[\XColor,thick] (.2,0) -- (.2,.7);
\draw[dashed] (0,-1.7) -- (0,0);
\roundNbox{unshaded}{(0,0)}{.3}{.05}{.05}{\scriptsize{$x$}};
\roundNbox{unshaded}{(0,-1)}{.3}{0}{0}{\scriptsize{$a_n$}};
}
\qquad\qquad
\forall\, x\in X,
\qquad
\forall\text{ central sequences }(a_n)\subseteq M.
\]
\end{prop}
\begin{proof}
Recall a central sequence of $\fgpMod(M)$ is a natural transformation of the identity functor.
By Example \ref{ex:MCentralSequencesAgree}, equivalence classes of these central sequences agree with the usual equivalence classes of central sequences of $M$.
The result now follows directly from Proposition \ref{Prop:CentralTrivial&CentralSq}, which translates into the displayed condition in the statement of the proposition. 
\end{proof}

Now by definition, the functor $-\boxtimes_M X$ is approximately inner if there exists an approximately natural isometry $v_n^Y:Y\boxtimes_M X\to Y\otimes H$, i.e., we have
\[
v_n^Y=
\tikzmath{
\begin{scope}
\clip[rounded corners = 5] (-.15,-.7) rectangle (.9,.7);
\filldraw[\rColor] (0,-.7) -- (0,.7) -- (1,.7) -- (1,-.7);
\end{scope}
\draw[\YColor,thick] (0,-.7) node[below]{$\scriptstyle Y$} -- (0,.7);
\draw[\HColor,thick] (.4,0) -- (.4,.7) node[above,black]{$\scriptstyle H$};
\draw[\XColor,thick] (.4,0) -- (.4,-.7) node[below]{$\scriptstyle X$};
\roundNbox{unshaded}{(.2,0)}{.3}{.1}{.1}{\scriptsize{$v_n^{Y}$}};
}
\qquad\text{such that}\qquad
\tikzmath{
\begin{scope}
\clip[rounded corners = 5] (-.15,-1.7) rectangle (1,.7);
\filldraw[\rColor] (0,-1.7) -- (0,.7) -- (1,.7) -- (1,-1.7);
\end{scope}
\draw[\YColor,thick] (0,-.7) -- (0,.7);
\draw[\XColor,thick] (.5,-1.7) -- (.5,0);
\draw[\ZColor,thick] (0,-1.7) -- (0,-.7);
\draw[\HColor,thick] (.5,0) -- (.5,.7);
\roundNbox{unshaded}{(.2,0)}{.3}{.2}{.2}{\scriptsize{$v_n^{Y}$}};
\roundNbox{unshaded}{(0,-1)}{.3}{0}{0}{\scriptsize{$f$}};
}
\approx
\tikzmath{
\begin{scope}
\clip[rounded corners = 5] (-.15,-.7) rectangle (1,1.7);
\filldraw[\rColor] (0,-.7) -- (0,1.7) -- (1,1.7) -- (1,-.7);
\end{scope}
\draw[\YColor,thick] (0,1) -- (0,1.7);
\draw[\XColor,thick] (.5,-.7) -- (.5,0);
\draw[\ZColor,thick] (0,-.7) -- (0,1);
\draw[\HColor,thick] (.5,0) -- (.5,1.7);
\roundNbox{unshaded}{(.2,0)}{.3}{.2}{.2}{\scriptsize{$v_n^{Z}$}};
\roundNbox{unshaded}{(0,1)}{.3}{0}{0}{\scriptsize{$f$}};
}
\qquad\text{and}\qquad
\tikzmath{
\begin{scope}
\clip[rounded corners = 5] (-.15,-1.7) rectangle (.9,.7);
\filldraw[\rColor] (0,-1.7) -- (0,.7) -- (1,.7) -- (1,-1.7);
\end{scope}
\draw[\YColor,thick] (0,-1.7) -- (0,.7);
\draw[\XColor,thick] (.4,-1.7) -- (.4,-1);
\draw[\HColor,thick] (.4,-1) -- (.4,0);
\draw[\XColor,thick] (.4,0) -- (.4,.7);
\roundNbox{unshaded}{(.2,0)}{.3}{.1}{.1}{\scriptsize{$(v_n^{Y})^\dag$}};
\roundNbox{unshaded}{(.2,-1)}{.3}{.1}{.1}{\scriptsize{$v_n^{Y}$}};
}
\approx
\tikzmath{
\begin{scope}
\clip[rounded corners = 5] (-.15,-1.7) rectangle (.9,.7);
\filldraw[\rColor] (0,-1.7) -- (0,.7) -- (1,.7) -- (1,-1.7);
\end{scope}
\draw[\YColor,thick] (0,-1.7) -- (0,.7);
\draw[\XColor,thick] (.4,-1.7) -- (.4,.7);
}
\]
for all intertwiners $f\in \Hom(Z_M\to Y_M)$.

\begin{defn}
\label{defn:ApproximatePPBasis}
For $Y\in\fgpMod(M)$, an \emph{approximate $Y_M$-basis} is a sequence $\{b_i^{(n)}\}_{i=1}^{m(n)}\subseteq Y$ such that 
$\sup_n m(n) <\infty$,
$\sup_{i,n} \|\langle b_i^{(n)}| b_i^{(n)}\rangle_M^Y\| <\infty$,
and
$$
\lim_{n\to\infty} \left\|x-\sum_{i=1}^m b_i^{(n)}\langle b_i^{(n)}|x\rangle_M^Y \right\|_2=0
\qquad\qquad
\forall\,x\in Y.
$$
For $X\in \fgpBim(X)$, an \emph{approximately inner} $X_M$-basis is an approximate $X_M$-basis such that
\begin{equation}
\label{eq:AICommutativity}
\left\|ab_i^{(n)}-b_i^{(n)}a\right\|_2\to 0
\qquad\qquad
\forall\,a\in M.
\end{equation}
\end{defn}

\begin{prop}
\label{prop:AI&APPB}
A bimodule $X$ is approximately inner over $M$ if and only if there exists an approximately inner $X_M$-basis.
\end{prop}
\begin{proof}
Suppose $X$ is approximately inner.
By Definition \ref{Defn:ApproximatelyInner}, there exists a finite dimensional Hilbert space $H$ and an approximate natural isometry $v=(v_n):-\boxtimes_M X\to H\rhd-$. 
Let $\{e_i\}$ be an orthonormal basis of $H$, and define $b_i^{(n)}$ as follows.
\[
\tikzmath{
\begin{scope}
\clip[rounded corners = 5] (-.3,-.7) rectangle (.6,.7);
\filldraw[\rColor] (0,-.7) -- (0,0) -- (-.2,0) -- (-.2,.7) -- (.7,.7) -- (.7,-.7);
\end{scope}
\draw[dashed] (0,-.7) -- (0,0);
\draw[dashed] (-.2,0) -- (-.2,.7);
\draw[\XColor,thick] (.2,0) -- (.2,.7);
\roundNbox{unshaded}{(0,0)}{.3}{.05}{.05}{\scriptsize{$b_i^{(n)}$}};
}
: = 
\tikzmath{
\begin{scope}
\clip[rounded corners = 5] (-.5,-.7) rectangle (.8,.7);
\filldraw[\rColor] (-.25,-.7) -- (-.25,.7) -- (1,.7) -- (1,-.7);
\end{scope}
\draw[dashed] (-.25,-.7) -- (-.25,.7);
\draw[\XColor,thick] (.25,0) -- (.25,.7);
\draw[\HColor,thick] (.25,0) -- (.25,-.55);
\filldraw[\HColor] (.25,-.55) circle (.05cm);
\roundNbox{unshaded}{(0,0)}{.3}{.15}{.15}{\scriptsize{$(v_n^{M})^\dag$}};
\node at (.5,-.55) {\scriptsize{$e_i$}};
}
\,.
\]
Observe that for all $x\in X$, we have
\[
\sum_{i} b_i^{(n)}\langle b_i^{(n)}|x\rangle_M^X
=
\sum_i
\tikzmath{
\begin{scope}
\clip[rounded corners = 5] (-.5,-2.2) rectangle (.8,1.2);
\filldraw[\rColor] (0,-2.2) -- (0,-1.5) -- (-.2,-1.5) -- (-.2,-.5) -- (0,-.5) -- (0,.5) -- (-.2,.5) -- (-.2,1.2) -- (.8,1.2) -- (.8,-2.2);
\end{scope}
\draw[dashed] (-.2,-1.5) -- (-.2,-.5);
\draw[dashed] (-.2,.5) -- (-.2,1.2);
\draw[dashed] (0,-.5) -- (0,.5);
\draw[dashed] (0,-1.5) -- (0,-2.2);
\draw[\XColor,thick] (.2,-1.5) -- (.2,-.5);
\draw[\XColor,thick] (.2,.5) -- (.2,1.2);
\roundNbox{unshaded}{(0,.5)}{.3}{.2}{.2}{\scriptsize{$b_i^{(n)}$}};
\roundNbox{unshaded}{(0,-.5)}{.3}{.2}{.2}{\scriptsize{$(b_i^{(n)})^\dag$}};
\roundNbox{unshaded}{(0,-1.5)}{.3}{.2}{.2}{\scriptsize{$x$}};
}
=
\sum_i 
\tikzmath{
\begin{scope}
\clip[rounded corners = 5] (-.5,-2.4) rectangle (1,1.4);
\filldraw[\rColor] (0,-2.4) -- (0,-1.5) -- (-.3,-1.5) -- (-.3,1.4) -- (1,1.4) -- (1,-2.4);
\end{scope}
\draw[dashed] (-.3,-1.7) -- (-.3,1.4);
\draw[dashed] (0,-2.4) -- (0,-1.7);
\draw[\XColor,thick] (.3,-1.7) -- (.3,-.7);
\draw[\XColor,thick] (.3,.7) -- (.3,1.4);
\draw[\HColor,thick] (.3,.15) -- (.3,.4);
\draw[\HColor,thick] (.3,-.15) -- (.3,-.4);
\filldraw[\HColor] (.3,.15) circle (.05cm);
\filldraw[\HColor] (.3,-.15) circle (.05cm);
\roundNbox{unshaded}{(0,.7)}{.3}{.2}{.2}{\scriptsize{$(v_n^{M})^\dag$}};
\roundNbox{unshaded}{(0,-.7)}{.3}{.2}{.2}{\scriptsize{$v_n^{M}$}};
\roundNbox{unshaded}{(0,-1.7)}{.3}{.2}{.2}{\scriptsize{$x$}};
\node at (.55,.15) {\scriptsize{$e_i$}};
\node at (.55,-.15) {\scriptsize{$e_i^*$}};
}
=
\tikzmath{
\begin{scope}
\clip[rounded corners = 5] (-.5,-2.2) rectangle (1,1.2);
\filldraw[\rColor] (0,-2.2) -- (0,-1.5) -- (-.3,-1.5) -- (-.3,1.2) -- (1,1.2) -- (1,-2.2);
\end{scope}
\draw[dashed] (-.3,-1.5) -- (-.3,1.2);
\draw[dashed] (0,-2.2) -- (0,-1.5);
\draw[\XColor,thick] (.3,-1.5) -- (.3,-.5);
\draw[\XColor,thick] (.3,.5) -- (.3,1.2);
\draw[\HColor,thick] (.3,-.5) -- (.3,.5);
\roundNbox{unshaded}{(0,.5)}{.3}{.2}{.2}{\scriptsize{$(v_n^{M})^\dag$}};
\roundNbox{unshaded}{(0,-.5)}{.3}{.2}{.2}{\scriptsize{$v_n^{M}$}};
\roundNbox{unshaded}{(0,-1.5)}{.3}{.2}{.2}{\scriptsize{$x$}};
}
=
(v_n^{M})^\dag(v_n^{M}(x)).
\]
The condition \eqref{eq:AICommutativity} follows immediately by approximate naturality of $v$ on the $M$-component.

Conversely, starting with an approximately inner $X_M$-basis, we define
\[
\tikzmath{
\begin{scope}
\clip[rounded corners = 5] (-.5,-.7) rectangle (.8,.7);
\filldraw[\rColor] (-.25,-.7) -- (-.25,.7) -- (1,.7) -- (1,-.7);
\end{scope}
\draw[dashed] (-.25,-.7) -- (-.25,.7);
\draw[\HColor,thick] (.25,0) -- (.25,.7);
\draw[\XColor,thick] (.25,0) -- (.25,-.7);
\roundNbox{unshaded}{(0,0)}{.3}{.15}{.15}{\scriptsize{$v_n^{M}$}};
\node at (.25,.9) {\scriptsize{$H$}};
}
:=
\sum_i
\tikzmath{
\begin{scope}
\clip[rounded corners = 5] (-.7,-.7) rectangle (.9,1.1);
\filldraw[\rColor] (-.25,-.7) -- (-.25,1.1) -- (1,1.1) -- (1,-.7);
\end{scope}
\draw[dashed] (-.25,-.7) -- (-.25,1.1);
\draw[\XColor,thick] (.25,0) -- (.25,-.7);
\draw[\HColor,thick] (.25,0) -- (.25,.55);
\draw[\HColor,thick] (.25,.85) -- (.25,1.1);
\filldraw[\HColor] (.25,.55) circle (.05cm);
\filldraw[\HColor] (.25,.85) circle (.05cm);
\roundNbox{unshaded}{(0,0)}{.3}{.15}{.15}{\scriptsize{$v_n^{M}$}};
\node at (.5,.55) {\scriptsize{$e_i^*$}};
\node at (.5,.85) {\scriptsize{$e_i$}};
}
=
\sum_i
\tikzmath{
\begin{scope}
\clip[rounded corners = 5] (-.3,-.7) rectangle (.8,.8);
\filldraw[\rColor] (-.2,-.7) -- (-.2,.8) -- (.8,.8) -- (.8,-.7);
\end{scope}
\draw[dashed] (-.2,-.7) -- (-.2,.8);
\draw[\XColor,thick] (.2,0) -- (.2,-.7);
\draw[\HColor] (.2,.55) -- (.2,.8);
\filldraw[\HColor] (.2,.55) circle (.05cm);
\roundNbox{unshaded}{(0,0)}{.3}{.2}{.2}{\scriptsize{$(b_i^{(n)})^\dag$}};
\node at (.45,.55) {\scriptsize{$e_i$}};
}\,,
\]
and we define each $v_n^Y$ for $Y\in\fgpMod(M)$ in terms of $v_n^{M}$ and a $Y_M$-basis $\{c_j\}$:
\[
\tikzmath{
\begin{scope}
\clip[rounded corners = 5] (-.15,-.7) rectangle (.9,.7);
\filldraw[\rColor] (0,-.7) -- (0,.7) -- (1,.7) -- (1,-.7);
\end{scope}
\draw[\YColor,thick] (0,-.7) node[below]{$\scriptstyle Y$} -- (0,.7);
\draw[\HColor,thick] (.4,0) -- (.4,.7);
\draw[\XColor,thick] (.4,0) -- (.4,-.7) node[below]{$\scriptstyle X$};
\roundNbox{unshaded}{(.2,0)}{.3}{.1}{.1}{\scriptsize{$v_n^{Y}$}};
}
:=
\sum_j
\tikzmath{
\begin{scope}
\clip[rounded corners = 5] (-.45,-1.7) rectangle (.9,1.7);
\filldraw[\rColor] (-.3,-1.7) rectangle (.9,1.7);
\end{scope}
\draw[dashed] (-.3,-1) -- (-.3,1);
\draw[\HColor,thick] (.3,0) -- (.3,1.7);
\draw[\XColor,thick] (.3,0) -- (.3,-1.7);
\draw[\YColor,thick] (-.3,-1.7) -- (-.3,-1);
\draw[\YColor,thick] (-.3,1) -- (-.3,1.7);
\roundNbox{unshaded}{(0,0)}{.3}{.3}{.3}{\scriptsize{$v_n^{M}$}};
\roundNbox{unshaded}{(-.3,1)}{.3}{0}{0}{\scriptsize{$c_j$}};
\roundNbox{unshaded}{(-.3,-1)}{.3}{0}{0}{\scriptsize{$c_j^\dag$}};
}\,.
\]
We have $v_n^Y$ is norm-bounded as $\sum_j \|c_j\|_2^2<\infty$.
To see the approximate naturality of $\{v_n^Y\}_{Y\in\fgpMod(M)}$, for $f\in \Hom(Z_M\to Y_M)$,
\[
\tikzmath{
\begin{scope}
\clip[rounded corners = 5] (-.15,-1.7) rectangle (1,.7);
\filldraw[\rColor] (0,-1.7) -- (0,.7) -- (1,.7) -- (1,-1.7);
\end{scope}
\draw[\YColor,thick] (0,-.7) -- (0,.7) node[above]{$\scriptstyle Y$};
\draw[\XColor,thick] (.5,-1.7) node[below]{$\scriptstyle X$} -- (.5,0);
\draw[\ZColor,thick] (0,-1.7) node[below]{$\scriptstyle Z$} -- (0,-.7);
\draw[\HColor,thick] (.5,0) -- (.5,.7);
\roundNbox{unshaded}{(.2,0)}{.3}{.2}{.2}{\scriptsize{$v_n^{Y}$}};
\roundNbox{unshaded}{(0,-1)}{.3}{0}{0}{\scriptsize{$f$}};
}
=
\sum_j
\tikzmath{
\begin{scope}
\clip[rounded corners = 5] (-.45,-2.7) rectangle (.9,1.7);
\filldraw[\rColor] (-.3,-2.7) rectangle (.9,1.7);
\end{scope}
\draw[dashed] (-.3,-1) -- (-.3,1);
\draw[\HColor,thick] (.3,0) -- (.3,1.7);
\draw[\XColor,thick] (.3,0) -- (.3,-2.7);
\draw[\YColor,thick] (-.3,-2) -- (-.3,-1);
\draw[\YColor,thick] (-.3,1) -- (-.3,1.7);
\draw[\ZColor,thick] (-.3,-2.7) -- (-.3,-2);
\roundNbox{unshaded}{(0,0)}{.3}{.3}{.3}{\scriptsize{$v_n^{M}$}};
\roundNbox{unshaded}{(-.3,1)}{.3}{0}{0}{\scriptsize{$c_j$}};
\roundNbox{unshaded}{(-.3,-1)}{.3}{0}{0}{\scriptsize{$c_j^\dag$}};
\roundNbox{unshaded}{(-.3,-2)}{.3}{0}{0}{\scriptsize{$f$}};
}
=
\sum_{j,k}
\tikzmath{
\begin{scope}
\clip[rounded corners = 5] (-.45,-4.7) rectangle (.9,1.7);
\filldraw[\rColor] (-.3,-4.7) rectangle (.9,1.7);
\end{scope}
\draw[dashed] (-.3,-1) -- (-.3,1);
\draw[dashed] (-.3,-4) -- (-.3,-3);
\draw[\HColor,thick] (.3,0) -- (.3,1.7);
\draw[\XColor,thick] (.3,0) -- (.3,-4.7);
\draw[\YColor,thick] (-.3,-2) -- (-.3,-1);
\draw[\YColor,thick] (-.3,1) -- (-.3,1.7);
\draw[\ZColor,thick] (-.3,-3) -- (-.3,-2);
\draw[\ZColor,thick] (-.3,-4.7) -- (-.3,-4);
\roundNbox{unshaded}{(0,0)}{.3}{.3}{.3}{\scriptsize{$v_n^{M}$}};
\roundNbox{unshaded}{(-.3,1)}{.3}{0}{0}{\scriptsize{$c_j$}};
\roundNbox{unshaded}{(-.3,-1)}{.3}{0}{0}{\scriptsize{$c_j^\dag$}};
\roundNbox{unshaded}{(-.3,-2)}{.3}{0}{0}{\scriptsize{$f$}};
\roundNbox{unshaded}{(-.3,-3)}{.3}{0}{0}{\scriptsize{$d_k$}};
\roundNbox{unshaded}{(-.3,-4)}{.3}{0}{0}{\scriptsize{$d_k^\dag$}};
}
\approx
\sum_{j,k}
\tikzmath{
\begin{scope}
\clip[rounded corners = 5] (-.45,-1.7) rectangle (.9,4.7);
\filldraw[\rColor] (-.3,-1.7) rectangle (.9,4.7);
\end{scope}
\draw[dashed] (-.3,-1) -- (-.3,1);
\draw[dashed] (-.3,3) -- (-.3,4);
\draw[\HColor,thick] (.3,0) -- (.3,4.7);
\draw[\XColor,thick] (.3,0) -- (.3,-1.7);
\draw[\ZColor,thick] (-.3,-1.7) -- (-.3,-1);
\draw[\ZColor,thick] (-.3,1) -- (-.3,2);
\draw[\YColor,thick] (-.3,2) -- (-.3,3);
\draw[\YColor,thick] (-.3,4) -- (-.3,4.7);
\roundNbox{unshaded}{(0,0)}{.3}{.3}{.3}{\scriptsize{$v_n^{M}$}};
\roundNbox{unshaded}{(-.3,1)}{.3}{0}{0}{\scriptsize{$d_k$}};
\roundNbox{unshaded}{(-.3,-1)}{.3}{0}{0}{\scriptsize{$d_k^\dag$}};
\roundNbox{unshaded}{(-.3,2)}{.3}{0}{0}{\scriptsize{$f$}};
\roundNbox{unshaded}{(-.3,3)}{.3}{0}{0}{\scriptsize{$c_j^\dag$}};
\roundNbox{unshaded}{(-.3,4)}{.3}{0}{0}{\scriptsize{$c_j$}};
}
=
\sum_k
\tikzmath{
\begin{scope}
\clip[rounded corners = 5] (-.45,-1.7) rectangle (.9,2.7);
\filldraw[\rColor] (-.3,-1.7) rectangle (.9,2.7);
\end{scope}
\draw[dashed] (-.3,-1) -- (-.3,1);
\draw[\HColor,thick] (.3,0) -- (.3,2.7);
\draw[\XColor,thick] (.3,0) -- (.3,-1.7);
\draw[\ZColor,thick] (-.3,-1.7) -- (-.3,-1);
\draw[\ZColor,thick] (-.3,1) -- (-.3,2);
\draw[\YColor,thick] (-.3,2) -- (-.3,2.7);
\roundNbox{unshaded}{(0,0)}{.3}{.3}{.3}{\scriptsize{$v_n^{M}$}};
\roundNbox{unshaded}{(-.3,1)}{.3}{0}{0}{\scriptsize{$d_k$}};
\roundNbox{unshaded}{(-.3,-1)}{.3}{0}{0}{\scriptsize{$d_k^\dag$}};
\roundNbox{unshaded}{(-.3,2)}{.3}{0}{0}{\scriptsize{$f$}};
}
=
\tikzmath{
\begin{scope}
\clip[rounded corners = 5] (-.15,-.7) rectangle (1,1.7);
\filldraw[\rColor] (0,-.7) -- (0,1.7) -- (1,1.7) -- (1,-.7);
\end{scope}
\draw[\YColor,thick] (0,1) -- (0,1.7);
\draw[\XColor,thick] (.5,-.7) -- (.5,0);
\draw[\ZColor,thick] (0,-.7) -- (0,1);
\draw[\HColor,thick] (.5,0) -- (.5,1.7);
\roundNbox{unshaded}{(.2,0)}{.3}{.2}{.2}{\scriptsize{$v_n^{Z}$}};
\roundNbox{unshaded}{(0,1)}{.3}{0}{0}{\scriptsize{$f$}};
}\,.
\]
We see $v_n^Y$ is independent of the choice of $Y_M$-basis by taking $f=\id_Y$ above.
\end{proof}

In \cite[Def.~1.1]{MR1317367}, Popa gave a definition of approximate innerness for a finite index $\rm II_1$ subfactor $N\subseteq M$.
We can also view ${}_NM_N \in \fgpBim(N)$ with $\langle x|y\rangle_N := E_N(x^*y)$.
We will show in Proposition \ref{PopaEq} below that
$N\subseteq M$ is approximately inner in the sense of \cite[Def.~1.1]{MR1317367} if and only if ${}_NM_N$ is approximately inner.

We quickly recall the notion of ultraproduct for $\rm II_1$ (sub)factors following \cite[\S5.4]{ClaireSorinII_1}.
Let $\omega$ be a non-principal unltrafilter on $\bbN$. 
For a $\rm{II}_{1}$ factor $N$, define $N^{\omega}=\ell^{\infty}(\bbN, M)/\cI$, where $\cI$ is the ideal of sequences which converge to $0$ in $\|\cdot \|_{2}$ along $\omega$. 
Then $N^{\omega}$ is a $\rm II_1$ factor, with trace given by taking the limit along $\omega$.

Now consider a finite index $\rm{II}_{1}$ subfactor $N\subseteq M$. 
Then $N^{\omega}\subseteq M^{\omega}$ is another $\rm II_1$ subfactor with the same Jones index and trace preserving condition expectation extending $E:M\rightarrow N$. 
Using this expectation, we can view $M^{\omega}$ as an $N^{\omega}$ bimodule. 
We can also consider the inclusion $N_{\omega}:=N^{\prime}\cap N^{\omega}\subseteq N^{\prime}\cap M^{\omega}=M_{\omega}$, and $E$ restricts to the trace preserving  conditional expectation $E: M_{\omega}\rightarrow N_{\omega}$.
We recall the following definition due to Popa.

\begin{defn}[{\cite[Def.~1.1 and Prop.~1.2]{MR1317367}}]
A finite index $\rm II_1$ subfactor $N\subseteq M$ is called \textit{approximately inner} if the inclusion $N_{\omega}\subseteq_{E} M_{\omega}$ is $[M:N]^{-1}$-Markov, i.e., there is a (finite) Pimsner-Popa basis $\{b\}$ for $M_\omega$ over $N_\omega$ which satisfies $\sum_b bb^* \in [1,\infty)$.
\end{defn}

\begin{prop}
\label{PopaEq}
A finite index $\rm II_1$ subfactor $N\subseteq M$ is approximately inner if and only if ${}_{N} M_{N} \in \fgpBim(N)$ is approximately inner.
\end{prop}
\begin{proof}
It is easy to see that for a finite index subfactor $N\subseteq M$, if ${}_NM_N\in \fgpBim(N)$ is approximately inner, then an approximate $M_N$-basis 
gives an honest Pimsner-Popa basis for $M^\omega$ over $N^\omega$, and the extra commutativity condition \eqref{eq:AICommutativity} means this Pimsner-Popa basis lies in $M_\omega=N'\cap M^\omega$. 
Thus we have a Pimsner-Popa basis for $M_\omega$ over $N_\omega$ which lifts to a Pimsner-Popa basis for $M^\omega$ over $N^\omega$, which implies $N_\omega \subseteq_E M_\omega$ is $[M:N]^{-1}$-Markov.
%(Note that although our Pimsner-Popa basis is not orthogonal, one can use the Gram-Schmidt orthogonalization procedure \cite[Lem.~8.5.2]{ClaireSorinII_1} to orthogonalize this basis.

Conversely, suppose $N\subseteq M$ is approximately inner.
Let $\{b^{(n)}_{i}\}^{k}_{i=1}$ be a representative of the Pimsner-Popa basis for $N_{\omega}\subseteq_{E} M_{\omega}$ which lifts to a Pimsner-Popa basis for $N^{\omega}\subseteq_{E} M^{\omega}$, which exists by \cite[Proposition 2.2 (3)]{MR1339767}.
Note that for every $x\in M$ and $y\in N$,
$$
\left\|x-\sum b^{(n)}_{i}E((b^{(n)}_{i})^{*}x)\right\|_{2}\longrightarrow_{\omega}0
\qquad\text{and}\qquad
\left\|yb^{(n)}_{i}-b^{(n)}_{i}y\right\|_{2}\longrightarrow_{\omega} 0.
$$
Now let $F\subseteq N$ and $G\subseteq M$ be countable $\sigma$-strong* dense subsets, and write $F=\bigcup_{l} F_{l}$ and $G=\bigcup_{l} G_{l}$ where the $F_{l}$ and $G_{l}$ increasing sequences of finite subsets. 
For each $l$, we have
$$
\sum_{x\in G_{l}}\left\|x-\sum b^{(n)}_{i}E((b^{(n)}_{i})^{*}x)\right\|_{2}+\sum_{y\in F_{l}} \left\|yb^{(n)}_{i}-b^{(n)}_{i}y\right\|_{2}\longrightarrow_{\omega} 0.
$$
Therefore there exists a subsequence of the $b^{(n_{j})}_{i}$ such that
$$
\sum_{x\in G_{l}}\left\|x-\sum b^{(n_{j})}_{i}E((b^{(n_{j})}_{i})^{*}x)\right\|_{2}+\sum_{y\in F_{l}} \left\|yb^{(n_{j})}_{i}-b^{(n_{j})}_{i}y\right\|_{2}\longrightarrow 0.
$$
In particular, we can choose $c^{(l)}_{i}:=b^{(n_{j})}_{i}$ where $n_{j}$ is sufficiently large so that the above sum is less than $2^{-l}$.

We claim is the sequence $\{c^{(l)}_{i}\}$ is an approximately inner $L^2M_N$-basis.
Clearly the collection $\{c^{(l)}_{i}\}$ satisfies the conditions from Definition \ref{defn:ApproximatePPBasis} for $x\in G$ while it satisfies \eqref{eq:AICommutativity} for $y\in F$. 
The result follows since $G$ and $F$ are $\sigma$-strong* dense in $M$ and $N$ respectively.
\end{proof}

\begin{rem}[{cf.~\cite[Ex.~8.1.9]{MR3242743}}]
\label{rem:OppositeWithBraiding}
For a monoidal category $\cC$, the \emph{monoidal opposite} $\cC^{\rm mp}$ is the same category with the opposite monoidal product given by
$a \otimes_{\rm mp} b := b \otimes_\cC a$.
It is equipped with the inverse associator: $$
\alpha_{a,b,c}^{\rm mp} 
:
(a \otimes_{\rm mp} b) \otimes_{\rm mp} c 
:= 
c \otimes (b \otimes a)
\xrightarrow{\alpha_{c,b,a}^{-1}} 
(c \otimes b) \otimes a 
=: 
a \otimes_{\rm mp} (b \otimes_{\rm mp} c). 
$$
Now suppose $\cC$ is a braided monoidal category.
Observe that the braiding endows the linear equivalence $\cC \to \cC^{\rm mp}$
with a monoidal structure
$$
\mu_{a,b}
:
\id(a\otimes b) 
=
a\otimes b
\xrightarrow{\beta_{a,b}}
b\otimes a
=
\id(a)\otimes_{\rm mp} \id(b)
$$
giving an equivalence of monoidal categories $\cC\simeq \cC^{\rm mp}$.
Now transporting $\beta$ to $\cC^{\rm mp}$ along this monoidal equivalence endows $\cC^{\rm mp}$ with the braiding
$$
\beta^{\rm mp}_{a,b}
:=
\mu_{b,a} \circ \id(\beta_{a,b})\circ \mu^{-1}_{a,b}
:
a \otimes_{\rm mp} b 
=
b \otimes a
\xrightarrow{
\beta_{b,a}\circ \beta_{a,b}\circ \beta_{a,b}^{-1}
=
\beta_{b,a}
} 
a \otimes b
=
b \otimes_{\rm mp} a
$$
such that $(\id,\mu): (\cC, \alpha, \beta) \to (\cC^{\rm mp}, \alpha^{\rm mp}, \beta^{\rm mp})$ is a braided monoidal equivalence.
(Observe that if we chose our monoidal structure for $\id$ to be $\mu_{a,b}:=\beta_{b,a}^{-1}$, we would still obtain $\beta^{\rm mp}_{a,b} = \beta_{b,a}$ as the transported braiding on $\cC^{\rm mp}$.)
\end{rem}

Using the above remark, we now translate the definition of the unitary braiding into the language of bimodules.
For $X,Y \in \fgpBim(M)$,
we write $u_{X,Y}=u^{L^2M}_{-\boxtimes_M Y, -\boxtimes_M X}$.
Choosing approximately inner $X_M,Y_M$-bases $\{b_i^{(n)}\}_i,\{c_j^{(n)}\}_j$
and ordinary $X_M,Y_M$-bases $\{b_i\}_i,\{c_j\}_j$ respectively,
unpacking Theorem \ref{ThmDef:uGFa} and Definition \ref{defn:sigmaFH} gives the following formulas.
\begin{align}
u_{X,Y} &=
\tikzmath{
\begin{scope}
\clip[rounded corners = 5] (-.55,-.5) rectangle (.7,.5);
\filldraw[\rColor] (-.4,-.5) -- (-.4,.5) -- (.7,.5) -- (.7,-.5);
\end{scope}
\draw[dashed] (-.4,-.5) -- (-.4,.5);
\draw[\YColor,thick] (.4,-.5) node[below]{$\scriptstyle Y$} .. controls ++(90:.45cm) and ++(270:.45cm) .. (0,.5);
\filldraw[\rColor] (.2,0) circle (.05cm);
\draw[\XColor,thick] (0,-.5) node[below]{$\scriptstyle X$} .. controls ++(90:.45cm) and ++(270:.45cm) .. (.4,.5);
}
\approx
\tikzmath{
\begin{scope}
\clip[rounded corners = 5] (-.7,-1.5) rectangle (.9,1.5);
\filldraw[\rColor] (-.4,-1.5) -- (-.4,1.5) -- (1,1.5) -- (1,-1.5);
\end{scope}
\draw[dashed] (-.4,-1.5) -- (-.4,1.5);
\draw[\XColor,thick] (0,-1.5) -- (0,-.8);
\draw[\YColor,thick] (.4,-1.5) -- (.4,-.5) .. controls ++(90:.45cm) and ++(270:.45cm) .. (0,.5) -- (0,1.5);
\filldraw[\rColor] (.2,0) circle (.05cm);
\draw[\XColor,thick] (.4,.8) -- (.4,1.5);
\draw[\HColor,thick] (0,-.5) .. controls ++(90:.45cm) and ++(270:.45cm) .. (.4,.5);
\roundNbox{unshaded}{(0,.8)}{.3}{.35}{.35}{\scriptsize{$(v_n^Y)^\dag$}};
\roundNbox{unshaded}{(-.2,-.8)}{.3}{.15}{.15}{\scriptsize{$v_n^{M}$}};
}
=
\sum_i
\tikzmath{
\begin{scope}
\clip[rounded corners = 5] (-.7,-1.5) rectangle (.9,1.5);
\filldraw[\rColor] (-.4,-1.5) -- (-.4,1.5) -- (1,1.5) -- (1,-1.5);
\end{scope}
\draw[dashed] (-.4,-1.5) -- (-.4,1.5);
\draw[\XColor,thick] (0,-1.5) -- (0,-.8);
\draw[\HColor,thick] (0,-.5) -- (0,-.25);
\draw[\HColor,thick] (.4,.5) -- (.4,.25);
\filldraw[\HColor] (0,-.25) circle (.05cm);
\filldraw[\HColor] (.4,.25) circle (.05cm);
\draw[\YColor,thick] (.4,-1.5) -- (.4,-.5) .. controls ++(90:.45cm) and ++(270:.45cm) .. (0,.5) -- (0,1.5);
\draw[\XColor,thick] (.4,.8) -- (.4,1.5);
\roundNbox{unshaded}{(0,.8)}{.3}{.35}{.35}{\scriptsize{$(v_n^Y)^\dag$}};
\roundNbox{unshaded}{(-.2,-.8)}{.3}{.15}{.15}{\scriptsize{$v_n^{M}$}};
\node at (.65,.25) {\scriptsize{$e_i$}};
}
=
\sum_{i,j}
\tikzmath{
\begin{scope}
\clip[rounded corners = 5] (-.7,-1.5) rectangle (.9,3.5);
\filldraw[\rColor] (-.4,-1.5) -- (-.4,.5) -- (-.2,.5) -- (-.2,2.5) -- (-.4,2.5) -- (-.4,3.5) -- (1,3.5) -- (1,-1.5);
\end{scope}
\draw[dashed] (-.4,-1.5) -- (-.4,.5);
\draw[\XColor,thick] (0,-1.5) -- (0,-.8);
\draw[\HColor,thick] (0,-.5) -- (0,-.25);
\draw[\HColor,thick] (.4,1.5) -- (.4,1.25);
\filldraw[\HColor] (0,-.25) circle (.05cm);
\filldraw[\HColor] (.4,1.25) circle (.05cm);
\draw[\YColor,thick] (.4,-1.5) -- (.4,-.5) .. controls ++(90:.45cm) and ++(270:.45cm) .. (0,.5);
\draw[dashed] (-.2,.5) -- (-.2,2.5);
\draw[\XColor,thick] (.4,1.8) -- (.4,3.5);
\draw[dashed] (-.4,2.5) -- (-.4,3.5);
\draw[\YColor,thick] (0,2.5) -- (0,3.5);
\roundNbox{unshaded}{(-.2,.8)}{.3}{.15}{.15}{\scriptsize{$c_j^\dag$}};
\roundNbox{unshaded}{(-.2,2.8)}{.3}{.15}{.15}{\scriptsize{$c_j$}};
\roundNbox{unshaded}{(0,1.8)}{.3}{.35}{.35}{\scriptsize{$(v_n^{M})^\dag$}};
\roundNbox{unshaded}{(-.2,-.8)}{.3}{.15}{.15}{\scriptsize{$v_n^{M}$}};
\node at (.65,1.25) {\scriptsize{$e_i$}};
}
=
\sum_{i,j}
\tikzmath{
\begin{scope}
\clip[rounded corners = 5] (-.35,-2.2) rectangle (1,2.2);
\filldraw[\rColor] (-.2,-2.2) -- (-.2,-1.5) -- (0,-1.5) -- (0,-.5) -- (.3,-.5) -- (.3,.5) -- (0,.5) -- (0,1.5) -- (-.2,1.5) -- (-.2,2.2) -- (1.05,2.2) -- (1.05,-2.2);
\end{scope}
\draw[dashed] (-.2,-2.2) -- (-.2,-1.5);
\draw[\XColor,thick] (.2,-2.2) -- (.2,-1.5);
\draw[dashed] (0,-1.5) -- (0,-.5);
\draw[\YColor,thick] (.6,-2.2) -- (.6,-.5);
\draw[dashed] (.3,-.5) -- (.3,.5);
\draw[dashed] (0,.5) -- (0,1.5);
\draw[\XColor,thick] (.6,.5) -- (.6,2.2);
\draw[dashed] (-.2,1.5) -- (-.2,2.2);
\draw[\YColor,thick] (.2,1.5) -- (.2,2.2);
\roundNbox{unshaded}{(0,1.5)}{.3}{.1}{.1}{\scriptsize{$c_j$}};
\roundNbox{unshaded}{(.3,.5)}{.3}{.15}{.15}{\scriptsize{$b_i^{(n)}$}};
\roundNbox{unshaded}{(.3,-.5)}{.3}{.15}{.15}{\scriptsize{$c_j^\dag$}};
\roundNbox{unshaded}{(0,-1.5)}{.3}{.1}{.1}{\scriptsize{$(b_i^{(\!n\!)}\!)^\dag$}};
}
\label{eq:OverBraiding}
\\
u_{Y,X}^\dag&=
\tikzmath{
\begin{scope}
\clip[rounded corners = 5] (-.55,-.5) rectangle (.7,.5);
\filldraw[\rColor] (-.4,-.5) -- (-.4,.5) -- (.7,.5) -- (.7,-.5);
\end{scope}
\draw[dashed] (-.4,-.5) -- (-.4,.5);
\draw[\XColor,thick] (0,-.5) .. controls ++(90:.45cm) and ++(270:.45cm) .. (.4,.5);
\filldraw[\rColor] (.2,0) circle (.05cm);
\draw[\YColor,thick] (.4,-.5) .. controls ++(90:.45cm) and ++(270:.45cm) .. (0,.5);
}
\approx
\tikzmath{
\begin{scope}
\clip[rounded corners = 5] (-.55,-1.5) rectangle (.9,1.5);
\filldraw[\rColor] (-.4,-1.5) -- (-.4,1.5) -- (1,1.5) -- (1,-1.5);
\end{scope}
\draw[dashed] (-.4,-1.5) -- (-.4,1.5);
\draw[\XColor,thick] (0,-1.5) -- (0,-.5) .. controls ++(90:.45cm) and ++(270:.45cm) .. (.4,.5) -- (.4,1.5);
\filldraw[\rColor] (.2,0) circle (.05cm);
\draw[\HColor,thick] (.4,-.5) .. controls ++(90:.45cm) and ++(270:.45cm) .. (0,.5);
\draw[\YColor,thick] (.4,-1.5) -- (.4,-.5);
\draw[\YColor,thick] (0,.5) -- (0,1.5);
\roundNbox{unshaded}{(-.2,.8)}{.3}{.15}{.15}{\scriptsize{$(w_n^{M})^\dag$}};
\roundNbox{unshaded}{(0,-.8)}{.3}{.35}{.35}{\scriptsize{$w_n^X$}};
}
=
\sum_j
\tikzmath{
\begin{scope}
\clip[rounded corners = 5] (-.55,-1.5) rectangle (.9,1.5);
\filldraw[\rColor] (-.4,-1.5) -- (-.4,1.5) -- (1,1.5) -- (1,-1.5);
\end{scope}
\draw[dashed] (-.4,-1.5) -- (-.4,1.5);
\draw[\HColor,thick] (0,.5) -- (0,.25);
\draw[\HColor,thick] (.4,-.25) -- (.4,-.5);
\filldraw[\HColor] (0,.25) circle (.05cm);
\filldraw[\HColor] (.4,-.25) circle (.05cm);
\draw[\XColor,thick] (0,-1.5) -- (0,-.5) .. controls ++(90:.45cm) and ++(270:.45cm) .. (.4,.5) -- (.4,1.5);
\draw[\YColor,thick] (.4,-1.5) -- (.4,-.5);
\draw[\YColor,thick] (0,.5) -- (0,1.5);
\roundNbox{unshaded}{(-.2,.8)}{.3}{.15}{.15}{\scriptsize{$(w_n^{M})^\dag$}};
\roundNbox{unshaded}{(0,-.8)}{.3}{.35}{.35}{\scriptsize{$w_n^{X}$}};
\node at (-.2,.25) {\scriptsize{$f_j$}};
}
=
\sum_{i.j}
\tikzmath{
\begin{scope}
\clip[rounded corners = 5] (-.55,-3.5) rectangle (.9,1.5);
\filldraw[\rColor] (-.4,-3.5) -- (-.4,-2.5) -- (-.2,-2.5) -- (-.2,-.5) -- (-.4,-.5) -- (-.4,1.5) -- (1,1.5) -- (1,-3.5);
\end{scope}
\draw[dashed] (-.4,-.5) -- (-.4,1.5);
\draw[\HColor,thick] (0,.5) -- (0,.25);
\draw[\HColor,thick] (.4,-1.25) -- (.4,-1.5);
\filldraw[\HColor] (0,.25) circle (.05cm);
\filldraw[\HColor] (.4,-1.25) circle (.05cm);
\draw[\XColor,thick] (0,-.5) .. controls ++(90:.45cm) and ++(270:.45cm) .. (.4,.5) -- (.4,1.5);
\draw[dashed] (-.2,-2.5) -- (-.2,-.5);
\draw[\YColor,thick] (.4,-3.5) -- (.4,-1.5);
\draw[\YColor,thick] (0,.5) -- (0,1.5);
\draw[dashed] (-.4,-3.5) -- (-.4,-2.5);
\draw[\XColor,thick] (0,-3.5) -- (0,-2.5);
\roundNbox{unshaded}{(-.2,.8)}{.3}{.15}{.15}{\scriptsize{$(w_n^{M})^\dag$}};
\roundNbox{unshaded}{(0,-1.8)}{.3}{.35}{.35}{\scriptsize{$w_n^{M}$}};
\roundNbox{unshaded}{(-.2,-.8)}{.3}{.15}{.15}{\scriptsize{$b_i$}};
\roundNbox{unshaded}{(-.2,-2.8)}{.3}{.15}{.15}{\scriptsize{$b_i^\dag$}};
\node at (-.2,.25) {\scriptsize{$f_j$}};
}
=
\sum_{i,j}
\tikzmath{
\begin{scope}
\clip[rounded corners = 5] (-.35,-2.2) rectangle (1,2.2);
\filldraw[\rColor] (-.2,-2.2) -- (-.2,-1.5) -- (0,-1.5) -- (0,-.5) -- (.3,-.5) -- (.3,.5) -- (0,.5) -- (0,1.5) -- (-.2,1.5) -- (-.2,2.2) -- (1.05,2.2) -- (1.05,-2.2);
\end{scope}
\draw[dashed] (-.2,-2.2) -- (-.2,-1.5);
\draw[\XColor,thick] (.2,-2.2) -- (.2,-1.5);
\draw[dashed] (0,-1.5) -- (0,-.5);
\draw[\YColor,thick] (.6,-2.2) -- (.6,-.5);
\draw[dashed] (.3,-.5) -- (.3,.5);
\draw[dashed] (0,.5) -- (0,1.5);
\draw[\XColor,thick] (.6,.5) -- (.6,2.2);
\draw[dashed] (-.2,1.5) -- (-.2,2.2);
\draw[\YColor,thick] (.2,1.5) -- (.2,2.2);
\roundNbox{unshaded}{(0,1.5)}{.3}{.1}{.1}{\scriptsize{$c_j^{(n)}$}};
\roundNbox{unshaded}{(.3,.5)}{.3}{.15}{.15}{\scriptsize{$b_i$}};
\roundNbox{unshaded}{(.3,-.5)}{.3}{.15}{.15}{\scriptsize{$(c_j^{(\!n\!)}\!)^\dag$}};
\roundNbox{unshaded}{(0,-1.5)}{.3}{.1}{.1}{\scriptsize{$b_i^\dag$}};
}
\label{eq:UnderBraiding}
\end{align}
That is, the braidings $u_{X,Y}, u_{Y,X}^\dag$ can be expressed as the following $\|\cdot\|_2$-limits of $M$-finite rank operators:
\begin{equation}
\label{eq:uXY}
u_{X,Y}
=
\lim_\tau\sum_{i,j} | c_j \boxtimes b_i^{(n)} \rangle\langle b_i^{(n)} \boxtimes c_j |
\qquad\text{and}\qquad
u_{Y,X}^\dag
=
\lim_\tau\sum_{i,j} | c_j^{(n)} \boxtimes b_i \rangle\langle b_i \boxtimes c_j^{(n)} |.
\end{equation}
The content of \S\ref{sec:W*ApproxNatTrans} then translates into the facts that:
\begin{itemize}
\item 
The limit $u_{X,Y}$ exists and is independent of the choice of (approximately inner) $X_M,Y_M$-bases (Theorem \ref{ThmDef:uGFa} and Proposition \ref{prop:uGF-Natural}),
\item
$u_{X,Y}$ is unitary (Proposition \ref{prop:uGFUnitary}),
\item
$u_{X,Y}$ is natural in $X$ and $Y$ (Proposition \ref{prop:uGF-Centralizing-Natural}), and
\item
$u_{X,Y}$ satisfies the braid relations (Propositions \ref{braid3} and \ref{braid4}).
\end{itemize}

%%%%%%%%%%%%%%%%%%%%%%%%%%%%%%%%%%%%%%%%%%%%%%%%%%%%%%%%%%%%%%%%%%%%%%%%%%%%%%%%%%%%%
\subsection{Examples}
\label{sec:Examples}

We now compute many examples of $\Chi(M)$ for various von Neumann algebras $M$.
In this section, in order to easily make contact with other results in the literature, we work with the Hilbert space version of $\fgpMod(M)$, i.e., Hilbert spaces with a normal right $M$-action such that $\dim(H_M)<\infty$ (see Example \ref{ex:RightFGPMod} for the equivalence with $\fgpWStarRCorr(\bbC \to M)$).

Now we recall two different notions of central sequence in the von Neumann algebra literature which agree for finite von Neumann algebras but are subtly distinct in general. The standard notion to use outside the finite setting is \textit{centralizing sequence}, introduced by Connes in \cite{MR0358374}. 
One way to define this is a say a norm bounded sequence $\{x_{i}\}\subseteq M$ is centralizing if for all $\eta\in L^{2}M$, $\|x_{i}\eta-\eta x_{i}\|_2\rightarrow 0$. 
This is in contrast to \textit{central sequences}, which satisfy $x_{i}m-mx_{i}\rightarrow 0$, in the strong*-topology (the latter notion is compatible with our use of the term central sequence). Note every centralizing sequence is central, but in general the converse is not true. For finite von Neumann algebras, however, these two notions agree. 
In both settings, we say a sequence is \textit{trivial} if there exists a scalar $\lambda$ such that $x_{i}-\lambda 1_{M}\rightarrow 0$ in the strong*-topology.

For our formalism, the notion of central sequence is the correct one; however, we will occasionally need to make use of results that are stated in terms of centralizing sequences. 

\begin{ex}
[$L^{\infty}(X, \mu)$] 
For the abelian von Neumann algebra $L^{\infty}(X,\mu)$ over a finite measure space $(X, \mu)$, constant sequences are centrally trivial. 
Thus any centrally trivial right-finite correspondence is simply a finitely generated projective module, made into a bimodule by defining the right action to be the left action. 
All of these correspondences are inner, hence approximately inner. 
This braided category is equivalent to the symmetric monoidal category of finitely generated projective modules, which is equivalent to the category of finite dimensional measurable Hilbert bundles over $(X, \mu)$ \cite{MR641217}.
\end{ex}

\begin{ex}
[Connes' $\chi(M)$ \cite{MR377534} and Jones' $\kappa$ \cite{MR585235}]
\label{ex:ConnesChi}

Let $M$ be a separable finite von Neumann algebra with faithful normal trace $\tr_M$.
We show that Connes' $\chi(M)$ embeds as a multiplicative subgroup of the monoid of equivalence classes of invertible objects in $\Chi(M)$.

For $\alpha\in\Aut(M)$, we define the corresponding $M-M$ bimodule as ${}_\alpha L^2M$.
Denoting the image of $1_M$ in ${}_\alpha L^2M$ by $\Omega$,
the left action is given by $a\rhd x\Omega:=\alpha(a)x\Omega$, 
the right action is given by $x\Omega\lhd b=xb\Omega$, and 
the $M$-valued inner product is given by $\langle x\Omega|x'\Omega\rangle^{L^2M}_M: = x^*x'$.
Moreover, the map $\alpha \mapsto {}_\alpha L^2M$ descends to a group isomorphism from $\Out(M)$ onto the group of unitary equivalence classes of 
invertible $M-M$ bimodules $H$ such that $\dim({}_MH) = \dim(H_M) = 1$.

Recall that an automorphism $\alpha$ is approximately inner if there exists a sequence of unitaries $u_{n}\in U(M)$ such that $\|\alpha(x)u_n - u_{n} x\|_2 \to 0$ for all $x\in M$.
According to Proposition \ref{prop:AI&APPB}, to show ${}_\alpha L^2M$ is  approximately inner, 
it suffices to show $\{u_n\Omega\}$ is an approximate Pimsner-Popa basis centralized by $M$.
It is clear that $u_n\rhd \langle u_n\Omega|x\Omega\rangle^{{}_\alpha L^2M}_M = x$ for all $x\in M$,
and for all $a\in M$,
$$
\|a\rhd u_n\Omega - u_n\Omega\lhd a\|_2 
= 
\|\alpha(a)u_n\Omega - u_na\Omega\|_2
\longrightarrow 
0.
$$
Therefore, ${}_\alpha L^2M$ is approximately inner.
By Proposition \ref{prop:BimodualCT&CS},
the bimodule ${}_\alpha L^2M$ is centrally trivial if $\alpha$ is centrally trivial. 

Now assume $M$ is a $\rm{II}_{1}$ factor. 
The subcategory of $\Chi(M)$ spanned by the image of $\chi(M)$ is a pointed braided unitary tensor category. 
The entire braided tensor structure here is uniquely determined by the quadratic form $\kappa$ on $\chi(M)$ determined by $u:=u_{{}_\alpha L^2M,{}_\alpha L^2M}=\kappa(\alpha) 1_{{}_\alpha L^2M\boxtimes_{M} {}_\alpha L^2M}$.
Since
$$
u=
\|\cdot\|_2-\lim_n |\Omega\boxtimes u_n\Omega\rangle \langle u_n\Omega\boxtimes \Omega|,
$$
for an arbitrary bounded vector $\Omega\boxtimes m\Omega$, 
we see 
$$
\kappa(\alpha)(\Omega \boxtimes m\Omega)
=
u(\Omega\boxtimes m\Omega)
= 
\Omega \boxtimes \|\cdot\|_2-\lim_n u^{*}_{n}\alpha(u_{n})m\Omega.
$$
Thus this $\kappa$ is precisely Jones' quadratic form $\kappa$ on $\chi(M)$ \cite[Def.~2.4]{MR585235}.

Further references for Connes' $\chi(M)$ and Jones' $\kappa$ include \cite{VaughanNotesOnChi,JianduanChenThesis,MR1642584}.
\end{ex}

\begin{ex}
[$\Chi(R)$ is trivial]

The following proposition specialized to the case $S=\bbC$ shows that the only centrally trivial bimodules of $R$ are inner, and thus $\Chi(R)$ is trivial.
This extends Connes' result that $\chi(R)$ is trivial \cite{MR377534}.

\begin{prop}
\label{CTRHYPER}
Let $M=R\otimes S$ where $R$ is the hyperfinite $\rm{II}_1$ factor and $S$ is any factor. Let $H$ be a separable $M-M$ Hilbert bimodule such that for all $\xi\in H$ and all central sequences $x=(x_{n})_{n\in \mathbb{N}}\subseteq R$, $\|(x_{n}\otimes 1_{S})\xi -\xi(x_{n}\otimes 1_{S})\|_{2}\rightarrow 0$. 
Then $H\cong L^{2}R\otimes K$ for some $S-S$ bimodule $K$.
\end{prop}
\begin{proof}
Let $\xi\in H$ with $\|\xi\|=1$. Represent $R=\otimes^{\infty}_{i=1} M_{2}(\bbC)$, and denote $R_{n}=\otimes^{n}_{i=1} M_{2}(\bbC)\subseteq R$. We claim there is some $n_0$ so that for all unitaries $u\in R^{\prime}_{n_0}\cap R$, $\|(u\otimes 1_{S})\xi-\xi (u\otimes 1_{S})\|<\frac{1}{2}$. Otherwise, we could find a sequence of unitaries $u_{n}\in R^{\prime}_{n}\cap R$ with $\|(u_{n}\otimes 1_{S})\xi-\xi (u_{n}\otimes 1_{S})\|\ge \frac{1}{2}$. However by construction the sequence $(u_{n})_{n\in \bbN}\subseteq R$ is central, contradicting the hypothesis.

Choose such an $n_{0}$, and consider the weakly compact convex subset 
$$
\co^{w}\set{(u\otimes 1_{S})\xi(u^{*}\otimes 1_{S})}{ u\in R^{\prime}_{n_0}\cap R}\subseteq H.
$$ 
This has a unique element of minimal norm, $\xi_{0}$. By hypothesis 

$$
2\|\xi\|^{2}-2\operatorname{Re}\langle (u\otimes 1_{S})\xi(u^{*}\otimes 1_{S}), \xi\rangle\ 
=
\|(u\otimes 1_{S})\xi(u^{*}\otimes 1_{S})-\xi \|^{2}
<
\frac{1}{4}
$$
and thus
$$
\frac{7}{8} <\ \operatorname{Re}\langle (u\otimes 1_{S})\xi(u^{*}\otimes 1_{S}), \xi\rangle.
$$
Therefore $\xi_{0}\ne 0$.  
Since $\|(u\otimes 1_{S})\xi_{0}(u^{*}\otimes 1_{S})\|=\|\xi_{0}\|$,
uniqueness of $\xi_0$ implies
$(u\otimes 1_{S})\xi_0=\xi_0 (u\otimes 1_{S})$ for all unitaries (hence all elements) $u\in R^{\prime}_{n}\cap R$. 
Note that $R_{n}\xi_{0} R_{n}$ is a cyclic bimodule over the finite dimensional full matrix algebra $R_{n}$, and thus contains a non-zero $R_{n}$ central vector $\xi_{1}$ (which is evidently $R^{\prime}_{n}\cap R$-central). 
Thus since $R=\langle R_{n}, R^{\prime}_{n}\cap R\rangle$, we see that $\xi_{1}$ is R-central and thus $R$-bounded by \cite[Lem.~3.20]{MR3040370}.
This means the $R-R$ bimodule $H_{1}:=\overline{(R\otimes 1_{S})\xi_{1} (R\otimes 1_{S})}^{\|\cdot\|}$ is canonically isomorphic to $L^{2}R$ as an $R$-$R$ bimodule via the map defined by $1_{R}\rightarrow \xi_{1}$. 
Choosing a vector $\xi^{\prime}\in H^{\perp}_{1}$, we can repeat this procedure, obtaining a decomposition as $R$-bimodules $H\cong L^{2}R\otimes K $ where $K$ is a separable multiplicity space. But note $K\cong \Omega_{R}\otimes K$ is the space of $R\otimes 1_{S}$ central vectors hence is closed under the left and right actions of $1_{R}\otimes S$.
\end{proof}

\end{ex}

The following proposition is well known to experts.
We record it here for completeness and convenience of the reader.

\begin{prop}\label{quasitrivialsubfactor}
Let $N\subseteq M$ be a finite index $\rm{II}_{1}$ subfactor. 
Then $M=N\vee (N^{\prime}\cap M)$ if and only if $ L^{2}M\cong \bigoplus_{i}L^{2}N$ as $N-N$ bimodules.
\end{prop}
\begin{proof}
Suppose $M=N\vee (N^{\prime}\cap M)$. 
Since $N^{\prime}\cap M$ is finite dimensional \cite{MR0696688}, pick an orthonormal basis $ \{e_{i}\}\subseteq N^{\prime}\cap M $ with respect to $\tr_{M}=E_N|_{N'\cap M}$.
Then each $e_iL^2N \cong L^2N$,
and
$L^{2}M\cong \bigoplus_{i} e_{i}L^{2}N$ 
as 
$$
\langle e_in_i, e_j n_j \rangle_{L^2M} 
=
\tr_M(n_j^*e_j^*e_i n_i)
=
\tr_N(n_j^* E_N(e_j^*e_i)n_i)
=
\delta_{i=j}\langle n_i, n_j\rangle_{L^2N}.
$$
The converse is obvious.
\end{proof}

\begin{ex}
[$\Chi(N)$ is trivial for $N$ non-Gamma]

Let $N$ be a non-Gamma $\rm{II}_1$ factor, and $H\in \Chi(N)$ irreducible. 
Setting $X=L^{2}N\oplus H$ and $M:= |X\boxtimes_{N} \overline{X}|$, we get a finite index $\rm II_1$ subfactor $N\subseteq M$ by Example \ref{ex:RealizeSeparatedQSystem} which is approximately inner by Proposition \ref{PopaEq}.
By \cite[Proposition 2.6 (iv)]{MR1339767}, $M=N\vee (N^{\prime}\cap M)$,  
so by Proposition \ref{quasitrivialsubfactor}, ${}_{N} L^{2}M_{N}\cong \bigoplus_{i} {}_{N} L^{2}N_{N} $. 
On the other hand, $L^{2}M\cong X\boxtimes_{N} \overline{X}\cong L^{2}N\oplus H\oplus \overline{H}\oplus (H\boxtimes_{N} \overline{H})$ contains $H$ as an irreducible summand, and thus we have $H\cong L^{2}N$ as an $N$-$N$ bimodule.
\end{ex}

\begin{ex}
[$\Chi(R\otimes N)$ is trivial for $N$ non-Gamma] \label{ex:Chi_fus(RboxtimesN)}

Let $H\in \Chi(R\otimes N)$ be irreducible, and consider $X=L^{2}(R\otimes N)\oplus H$.
Then since $H$ is centrally trivial, so is $X$.
By Lemma \ref{CTRHYPER}, $X\cong L^{2}(R)\otimes K$ where $K$ is an  $N$-$N$ bifinite bimodule.
In particular, setting $M:=|K\boxtimes_{N} \overline{K}|$, $N\subseteq M$ is a finite index $\rm II_1$ subfactor by Example \ref{ex:RealizeSeparatedQSystem}. 
Furthermore, since $N$ is non-Gamma, $M$ is non-Gamma by \cite[Prop.~1.11]{MR860811}.

Now since $H$ is approximately inner, the $R\otimes N-R\otimes N$ bimodule $X\boxtimes_{R\otimes N} \overline{X}\cong L^{2}(R\otimes M)$ is approximately inner, hence $R\otimes N\subseteq R\otimes M$ is a finite index approximately inner subfactor.
But note this subfactor is simply $R\otimes (N\subseteq M)$. 
Thus the inclusion $(R\otimes N)_{\omega}\subseteq (R\otimes M)_{\omega}$
is $[M:N]^{-1}$-Markov, and has a finite Pimsner-Popa basis lifting to a Pimsner-Popa basis of $(R\otimes N)^{\omega}\subseteq (R\otimes M)^{\omega}$ (where $\omega$ is a non-principal ultrafilter).

Furthermore, since 
$R \otimes N= (R\otimes 1)\ \vee (1\otimes N)$,
we have the equality
\begin{equation}
\label{eq:RNcommutant}
(R\otimes N)^{\prime}\cap (R\otimes M)^{\omega}
=
\left((1\otimes N)^{\prime}\cap (R\otimes M)^{\omega}\right) \cap \left( (R\otimes 1)^{\prime} \cap (R\otimes M)^{\omega}\right).
\end{equation}
By \cite[Prop.~3.2(1)]{MR2661553}, the inclusion $1\otimes N\subseteq R\otimes M$ has spectral gap, and thus 
\begin{equation}
\label{eq:1Ncommutant}
(1\otimes N)^{\prime}\cap (R\otimes M)^{\omega}=\left((1\otimes N)^{\prime}\cap (R\otimes M)\right)^{\omega}=(R\otimes (N^{\prime}\cap M))^{\omega}.
\end{equation}
Since $(R\otimes 1)^{\omega}\subseteq (R\otimes M)^{\omega}$, 
combining \eqref{eq:RNcommutant} and \eqref{eq:1Ncommutant},
we have
$$
(R\otimes N)^{\prime}\cap (R\otimes M)^{\omega}= (R\otimes (N^{\prime} \cap M))^{\omega} \cap \left( (R\otimes 1)^{\prime} \cap (R\otimes M)^{\omega}\right)=(R\otimes 1)^{\prime} \cap (R \otimes (N^{\prime}\cap M))^{\omega}.
$$
Thus there is a finite Pimsner-Popa basis $\{\tilde{m}_{i}\}$ for $(R\otimes N)^{\omega}\subseteq (R\otimes M)^{\omega}$ with $\tilde{m}_{i}=\{b^{(k)}_{i}\}_{k\in \bbN}$ where each $b^{(k)}_{i}\in R\otimes (N^{\prime}\cap M)$.

In particular, for any $m\in M$ and any $\varepsilon>0$, there exists finitely many elements $r_{i}\in R$, $k_{i}\in N^{\prime}\cap M$, and $n_{i}\in N$ such that $\|(1_{R}\otimes m)-\sum_{i} r_{i}\otimes k_{i}n_{i}\|_{2}<\varepsilon.$
But then applying the trace preserving conditional expectation $E=\tr_R\otimes \id_M:R\otimes M\rightarrow M$, we have 
$$
\|m-\sum_{i} \tr_{R}(r_{i})k_{i}n_{i}\|_{2}=\|E((1_{R}\otimes m)-\sum_{i} r_{i}\otimes k_{i}n_{i})\|_{2}\le \|(1_{R}\otimes m)-\sum_{i} r_{i}\otimes k_{i}n_{i}\|_{2}<\varepsilon.
$$
Therefore $N\vee (N^{\prime}\cap M)=M$. By Proposition \ref{quasitrivialsubfactor}, this implies that ${}_{N}L^{2}M_{N}\cong \bigoplus_{i} {}_{N}L^{2}N_{N}$. 
But since $ |X\boxtimes_{R\otimes N} \overline{X}|$ is isomorphic to $L^{2}R\otimes |K\boxtimes_{R\otimes N} \overline{K}|$ as $N-N$ bimodules, we have that $X\boxtimes_{R\otimes N} \overline{X} \cong \bigoplus_{i} L^{2}(R\otimes N) $ as $R\otimes N-R\otimes N$ bimodules. 
But recall $X\boxtimes_{R\otimes N} \overline{X}\cong L^{2}(R\otimes N)\oplus H\oplus \overline{H}\oplus (H\boxtimes_{R\otimes N} \overline{H})$, which contains $H$ as an irreducible summand. Thus $H$ is trivial.
\end{ex}

%%%%%%%%%%%%%%%%%%%%%%%%%%%%%%%%%%%%%%%%%%%%%%%%%%%%%%%%%%%%%%%%%%%%%%%%%%%%%%%%%%
%%%%%%%%%%%%%%%%%%%%%%%%%%%%%%%%%%%%%%%%%%%%%%%%%%%%%%%%%%%%%%%%%%%%%%%%%%%%%%%%%%
%%%%%%%%%%%%%%%%%%%%%%%%%%%%%%%%%%%%%%%%%%%%%%%%%%%%%%%%%%%%%%%%%%%%%%%%%%%%%%%%%%
\section{Local extension}
\label{sec:LocalExtensionChapter}

In this section,  we prove Theorem \ref{thmalpha:LocalExtension}, i.e., $\Chi(|Q|) = \Chi(M)_Q^{\loc}$ for a commutative Q-system $Q\in \Chi(M)$, where $|Q|$ is the realization of $Q$ defined in \S\ref{sect:Q-systemRealization} below.
This result appears as Theorem \ref{thm:LocalExtension}.

%%%%%%%%%%%%%%%%%%%%%%%%%%%%%%%%%%%%%%%%%%%%%%%%%%%%%%%%%%%%%%%%%%%%%%%%%%%%%%%%%
\subsection{Q-system realization}
\label{sect:Q-systemRealization}

Q-systems are unitary versions of Frobenius algebra objects which were originally introduced by Longo in \cite{MR1257245} to describe the canonical endomorphism for type $\rm III$ subfactors \cite{MR1027496}.
In this section, we define the \emph{realization} procedure \cite{MR3948170,2105.12010} 
(based on \cite{MR2097363} and \cite[\S4.1]{MR3221289})
which 
given a Q-system $Q$ over a $\rm II_1$ factor $M$, recovers a von Neumann algebra $|Q|$ containing $M$ and a conditional expectation $E_M: |Q| \to M$ with finite Pimsner-Popa index.
This story works in much broader generality, but we restrict to $\rm II_1$ factors here for ease of exposition.
As in \S\ref{sec:ModuleRealization}, for this section, $M$ is a $\rm II_1$ factor, and we denote $\bbC,M, {}_\bbC M_M$ as before:
$$
\tikzmath{\filldraw[\rColor, rounded corners=5, very thin, baseline=1cm] (0,0) rectangle (.6,.6);}=M
\qquad
\qquad
\tikzmath{
\filldraw[white, rounded corners=5, very thin, baseline=1cm] (0,0) rectangle (.6,.6);
\draw[dotted, rounded corners=5pt] (0,0) rectangle (.6,.6);
}=\bbC
\qquad
\qquad
\tikzmath{
\begin{scope}
\clip[rounded corners = 5] (0,0) rectangle (.6,.6);
\filldraw[\rColor] (.3,0) rectangle (.6,.6);
\end{scope}
\draw[dashed] (.3,0) -- (.3,.6);
\draw[dotted, rounded corners=5pt] (.3,0) -- (0,0) -- (0,.6) -- (.3,.6);
}={}_\bbC M_M.
$$

\begin{defn}
Given a $\rm II_1$ factor $M$, a \emph{Q-system} in $\fgpBim(M)$ is a triple $(Q,m,i)$ 
where $Q\in \Bim(M)$ is bimodule, and $m: Q\boxtimes_M Q \to Q$ and $i: M \to Q$ are bounded maps that satisfy certain relations best described graphically.
Representing $M$ by a shaded region and $M$ by a strand, $m$ is a trivalent vertex, and $i$ is a univalent vertex; adjoints are represented by vertical reflection.
$$
\tikzmath{\filldraw[\rColor, rounded corners=5, very thin, baseline=1cm] (0,0) rectangle (.6,.6);}=M
\qquad
\tikzmath{
\fill[\rColor, rounded corners=5pt ] (0,0) rectangle (.6,.6);
\draw[\QsColor,thick] (.3,0) -- (.3,.6);
}={}_MQ_M.
\qquad
\tikzmath{
\fill[\rColor, rounded corners=5pt] (-.3,0) rectangle (.9,.6);
\draw[\QsColor,thick] (0,0) arc (180:0:.3cm);
\draw[\QsColor,thick] (.3,.3) -- (.3,.6);
\filldraw[\QsColor] (.3,.3) circle (.05cm);
}=m
\qquad
\tikzmath{
\fill[\rColor, rounded corners=5pt] (-.3,0) rectangle (.9,-.6);
\draw[\QsColor,thick] (0,0) arc (-180:0:.3cm);
\draw[\QsColor,thick] (.3,-.3) -- (.3,-.6);
\filldraw[\QsColor] (.3,-.3) circle (.05cm);
}=m^\dag
\qquad
\tikzmath{
\fill[\rColor, rounded corners=5pt] (0,0) rectangle (.6,.6);
\draw[\QsColor,thick] (.3,.3) -- (.3,.6);
\filldraw[\QsColor] (.3,.3) circle (.05cm);
}=i
\qquad
\tikzmath{
\fill[\rColor, rounded corners=5pt] (0,0) rectangle (.6,-.6);
\draw[\QsColor,thick] (.3,-.3) -- (.3,-.6);
\filldraw[\QsColor] (.3,-.3) circle (.05cm);
}=i^\dag.
$$
The axioms that $m,i$ must satisfy are associativity, unitality, the Frobenius relations, and unitary separability.
We refer the reader to \cite[\S3.1]{2105.12010} for a full discussion with many helpful diagrams. 
We call a Q-system $Q\in \fgpBim(M)$ \emph{connected} if $\Hom_{M-M}(M\to Q) =\bbC i$.
\end{defn}

\begin{defn}[{\cite[\S4.1]{2105.12010}}]
For a Q-system $(Q,m,i)\in \Bim(M)$, 
its \emph{realization}
$|Q|$ is the unital $*$-algebra
with underlying vector space is $\Hom_{\bbC - M}({}_{\bbC}M_M \to {}_{\bbC}M\boxtimes_M Q_M)$,
whose elements are denoted by
\[
\tikzmath{
\begin{scope}
\clip[rounded corners = 5] (-.3,-.7) rectangle (.6,.7);
\filldraw[\rColor] (0,-.7) -- (0,0) -- (-.15,0) -- (-.15,.7) -- (.7,.7) -- (.7,-.7);
\end{scope}
\draw[dashed] (0,-.7) -- (0,0);
\draw[dashed] (-.15,0) -- (-.15,.7);
\draw[\QsColor,thick] (.15,0) -- (.15,.7);
\roundNbox{unshaded}{(0,0)}{.3}{0}{0}{\scriptsize{$q$}};
}
\in |Q| 
:=
\Hom_{\bbC - M}({}_{\bbC}M_M \to {}_{\bbC}M\boxtimes_M Q_M).
\]
The multiplication, unit, and adjoint, respectively, of $|Q|$ are given by
\[
q_1\cdot q_2:=
\tikzmath{
\begin{scope}
\clip[rounded corners = 5] (-1,-.7) rectangle (.5,2);
\filldraw[\rColor] (-.2,-.7) -- (-.2,0) -- (-.4,.3) .. controls ++(90:.2cm) and ++(270:.2cm) .. (-.6,.7) -- (-.8,1) -- (-.8,2.8) -- (.9,2.8) -- (.9,-.7);
\end{scope}
\draw[dashed] (-.2,-.7) -- (-.2,0);
\draw[dashed] (-.4,.3) .. controls ++(90:.2cm) and ++(270:.2cm) .. (-.6,.7);
\draw[dashed] (-.8,1) -- (-.8,2);
\filldraw[\rColor] (-.2,1.8) circle (.05cm);
\draw[\QsColor,thick] (-.4,1.3) -- (-.4,1.5) arc (180:0:.2cm) -- (0,0);
\draw[\QsColor,thick] (-.2,1.7) -- (-.2,2); 
\filldraw[\QsColor] (-.2,1.7) circle (.05cm);
\roundNbox{unshaded}{(-.6,1)}{.3}{.1}{.1}{\scriptsize{$q_1$}};
\roundNbox{unshaded}{(-.2,0)}{.3}{.1}{.1}{\scriptsize{$q_2$}};
}
\qquad\qquad
1_{|Q|}
:=
\tikzmath{
\begin{scope}
\clip[rounded corners = 5] (-.5,-.4) rectangle (.3,.4);
\filldraw[\rColor] (-.3,-.5) rectangle (.3,.5);
\end{scope}
\draw[dashed] (-.3,-.4) -- (-.3,.4);
\draw[\QsColor,thick] (0,0) -- (0,.4);
\filldraw[\QsColor] (0,0) circle (.05cm);
}
\qquad\qquad
q^*:=
\tikzmath{
\begin{scope}
\clip[rounded corners = 5] (-.5,-1.2) rectangle (.85,.7);
\filldraw[\rColor] (-.15,-1.2) -- (-.15,0) -- (0,0) -- (0,.7) -- (.85,.7) -- (.85,-1.2);
\end{scope}
\draw[dashed] (0,0) -- (0,.7);
\draw[dashed] (-.15,-1.2) -- (-.15,0);
\draw[\QsColor,thick] (.15,0) -- (.15,-.5) arc (-180:0:.2cm) -- (.55,.7);
\draw[\QsColor,thick] (.35,-1) -- (.35,-.7);
\filldraw[\QsColor] (.35,-.7) circle (.05cm);
\filldraw[\QsColor] (.35,-1) circle (.05cm);
\roundNbox{unshaded}{(0,0)}{.3}{0}{0}{\scriptsize{$q^\dag$}};
}
\]
Identifying $M=\End(M_M)$, the inclusion $M\hookrightarrow |Q|$ is given by
$$
\End({}_\bbC M_M)
\ni
\tikzmath{
\begin{scope}
\clip[rounded corners = 5] (-.3,-.7) rectangle (.6,.7);
\filldraw[\rColor] (0,-.7) rectangle (.6,.7);
\end{scope}
\draw[dashed] (0,-.7) -- (0,.7);
\roundNbox{unshaded}{(0,0)}{.3}{0}{0}{\scriptsize{$m$}};
}
\mapsto 
\tikzmath{
\begin{scope}
\clip[rounded corners = 5] (-.3,-.7) rectangle (.6,.7);
\filldraw[\rColor] (0,-.7) rectangle (.6,.7);
\end{scope}
\draw[dashed] (0,-.7) -- (0,.7);
\draw[\QsColor, thick] (.2,.5) -- (.2,.7);
\filldraw[\QsColor] (.2,.5) circle (.05cm);
\roundNbox{unshaded}{(0,0)}{.3}{0}{0}{\scriptsize{$m$}};
}
\in |Q|
$$
By \cite[Rem.~4.4 and Prop.~4.6]{2105.12010}, $|Q|$ is a finite (Pimsner-Popa) index von Neumann algebra over $M$.
Moreover, $|Q|$ is a $\rm II_1$ factor if and only if $\End_{Q-Q}(Q)=\bbC\id_Q$.
In this case, the unique trace-preserving conditional expectation is given by
$$
E_M(q) =
\frac{1}{[|Q|:M]}
\cdot\,
\tikzmath{
\begin{scope}
\clip[rounded corners = 5] (-.3,-.7) rectangle (.6,.7);
\filldraw[\rColor] (0,-.7) -- (0,0) -- (-.15,0) -- (-.15,.7) -- (.7,.7) -- (.7,-.7);
\end{scope}
\draw[dashed] (0,-.7) -- (0,0);
\draw[dashed] (-.15,0) -- (-.15,.7);
\draw[\QsColor,thick] (.15,0) -- (.15,.5);
\filldraw[\QsColor] (.15,.5) circle (.05cm);
\roundNbox{unshaded}{(0,0)}{.3}{0}{0}{\scriptsize{$q$}};
}
$$
\end{defn}

\begin{ex}[Canonical Q-systems]
\label{ex:RealizeSeparatedQSystem}
Suppose ${}_NH_P \in \fgpWStarRCorr(N\to P)$.
Then $H\boxtimes_P \overline{H}$ is a Q-system in $\fgpBim(N)$.
By \cite[Prop.~3.1]{MR703809}, there is a canonical isomorphism $H\boxtimes_P \overline{H} \cong L^2(P^{\op})'\cap B(H)$, and thus $|H\boxtimes_P \overline{H}|\cong (P^{\op})'=:M$.
Observe that the relative commutant of $N \subseteq M$ is exactly given by
$N'\cap M = N'\cap (P^{\op})' = \End_{N-P}(H)$.
\end{ex}

\begin{rem}
\label{rem:AllQSystems}
Let $N$ be a $\rm II_1$ factor.
If $Q$ is a connected Q-system in $\fgpBim(N)$, then $N\subseteq |Q|$ is a finite index irreducible $\rm II_1$ subfactor.
By the $\rm W^*$ version of \cite[Prop.~4.16]{2105.12010}, $|Q|$ is also a connected Q-system in $\fgpBim(N)$, and $Q\cong |Q|$ as Q-systems.
Hence the connected Q-systems in $\fgpBim(N)$ are exactly $\rm II_1$ factors $M$ containing $N$ such that $N\subseteq M$ is finite index and irreducible.
\end{rem}

\begin{defn}[{\cite[\S4.2]{2105.12010}}]
\label{defn:RealizationOfBimodules}
Suppose $P,Q\in\fgpBim(M)$ are two Q-systems and $X$ is a $P-Q$ bimodule.
The \emph{realization} $|X|:=\Hom_{\bbC-M}({}_\bbC M_M \to {}_\bbC M \boxtimes_M X_M)$ of $X$ is a $|P|-|Q|$ bimodule whose elements are denoted by
$$
\tikzmath{
\begin{scope}
\clip[rounded corners = 5] (-.3,-.7) rectangle (.6,.7);
\filldraw[\rColor] (0,-.7) -- (0,0) -- (-.15,0) -- (-.15,.7) -- (.7,.7) -- (.7,-.7);
\end{scope}
\draw[dashed] (0,-.7) -- (0,0);
\draw[dashed] (-.15,0) -- (-.15,.7);
\draw[\XColor,thick] (.15,0) -- (.15,.7);
\roundNbox{unshaded}{(0,0)}{.3}{0}{0}{\scriptsize{$x$}};
}
\in |X| 
:=
\Hom_{\bbC - M}({}_{\bbC}M_M \to {}_{\bbC}M\boxtimes_M X_M).
$$
with left $|P|$ and right $|Q|$-actions given respectively by
\[
p\rhd x
:=
\tikzmath{
\begin{scope}
\clip[rounded corners = 5] (-1,-.7) rectangle (.5,2);
\filldraw[\rColor] (-.2,-.7) -- (-.2,0) -- (-.4,.3) .. controls ++(90:.2cm) and ++(270:.2cm) .. (-.6,.7) -- (-.8,1) -- (-.8,2.8) -- (.9,2.8) -- (.9,-.7);
\end{scope}
\draw[dashed] (-.2,-.7) -- (-.2,0);
\draw[dashed] (-.4,.3) .. controls ++(90:.2cm) and ++(270:.2cm) .. (-.6,.7);
\draw[dashed] (-.8,1) -- (-.8,2);
\filldraw[\rColor] (-.2,1.8) circle (.05cm);
\draw[\QsColor,thick] (-.4,1.3) arc (180:90:.4cm);
\draw[\XColor,thick] (0,0) -- (0,2); 
\filldraw[\XColor] (0,1.7) circle (.05cm);
\roundNbox{unshaded}{(-.6,1)}{.3}{.1}{.1}{\scriptsize{$p$}};
\roundNbox{unshaded}{(-.2,0)}{.3}{.1}{.1}{\scriptsize{$x$}};
}
\qquad\qquad
x\lhd q:=
\tikzmath{
\begin{scope}
\clip[rounded corners = 5] (-1,-.7) rectangle (.5,2);
\filldraw[\rColor] (-.2,-.7) -- (-.2,0) -- (-.4,.3) .. controls ++(90:.2cm) and ++(270:.2cm) .. (-.6,.7) -- (-.8,1) -- (-.8,2.8) -- (.9,2.8) -- (.9,-.7);
\end{scope}
\draw[dashed] (-.2,-.7) -- (-.2,0);
\draw[dashed] (-.4,.3) .. controls ++(90:.2cm) and ++(270:.2cm) .. (-.6,.7);
\draw[dashed] (-.8,1) -- (-.8,2);
\filldraw[\rColor] (-.2,1.8) circle (.05cm);
\draw[\XColor,thick] (-.4,1.3) -- (-.4,2);
\draw[\QsColor,thick] (0,0) -- (0,1.3) arc (0:90:.4cm); 
\filldraw[\XColor] (-.4,1.7) circle (.05cm);
\roundNbox{unshaded}{(-.6,1)}{.3}{.1}{.1}{\scriptsize{$x$}};
\roundNbox{unshaded}{(-.2,0)}{.3}{.1}{.1}{\scriptsize{$q$}};
}
=x\lhd q
\qquad\qquad
\langle x_1|x_2\rangle_{|Q|}^X
:=
\tikzmath{
\begin{scope}
\clip[rounded corners = 5] (-.4,-1.4) rectangle (.9,1.4);
\filldraw[\rColor] (0,-1.4) -- (0,-.5) -- (-.2,-.5) -- (-.2,.5) -- (0,.5) -- (0,1.4) -- (1.1,1.4) -- (1.1,-1.4);
\end{scope}
\draw[dashed] (0,.5) -- (0,1.4);
\draw[dashed] (0,-1.4) -- (0,-.5);
\draw[dashed] (-.2,-.5) -- (-.2,.5);
\draw[\XColor,thick] (.2,-.5) -- (.2,.5);
\draw[\QsColor,thick] (.2,0) arc (-90:0:.4cm) -- (.6,1.4);
\filldraw[\XColor] (.2,0) circle (.05cm);
\roundNbox{unshaded}{(0,.7)}{.3}{.05}{.05}{\scriptsize{$x_1^\dag$}};
\roundNbox{unshaded}{(0,-.7)}{.3}{.05}{.05}{\scriptsize{$x_2$}};
}
\]
%This should be a remark in realization, and we should just cite this remark.
%TODO once we get a referee report.
Clearly $|X|$ has a predual as it is a hom space in a $\rm W^*$-category.
By \cite[Lem.~2.3 and Prop.~2.4]{2105.12010}, the right $|Q|$-valued inner product is separately weak*-continuous, and the left action of $|P|$ on $X$ is normal.
Hence $|X|\in \fgpWStarRCorr(|P|\to |Q|)$ is a $\rm W^*$-correspondence. 
\end{defn}

\begin{nota}
Given another Q-system $R\in \fgpBim(M)$ and a $Q-R$ bimodule $Y$, 
there is a notion of the relative tensor product of $X$ with $Y$ over $Q$, denoted $X\otimes_Q Y$.
We refer the reader to \cite[\S3.2]{2105.12010} for the detailed definition.
We have two canonical projectors which we denote graphically as follows:
\begin{equation}
\label{eq:CanonicalProjectors}
\tikzmath{
\draw[thin, dotted, rounded corners = 5] (-.2,0) -- (-.5,0) -- (-.5,.8) -- (-.2,.8);
\draw[thin, dotted, rounded corners = 5] (.2,0) -- (.5,0) -- (.5,.8) -- (.2,.8);
\filldraw[\rColor] (-.2,0) rectangle (.2,.4);
\filldraw[\rColor] (-.2,.4) rectangle (.2,.8);
\draw[\XColor,thick] (-.2,0) -- (-.2,.8);
\draw[\YColor,thick] (.2,0) -- (.2,.8);
\draw[\QsColor, thick] (-.2,.4) -- (.2,.4);
}
:X\boxtimes_M Y \to X\otimes_Q Y
\qquad\text{and}\qquad
\tikzmath{
\draw[thin, dotted, rounded corners = 5] (-.2,0) -- (-.5,0) -- (-.5,.8) -- (-.2,.8);
\draw[thin, dotted, rounded corners = 5] (.2,0) -- (.5,0) -- (.5,.8) -- (.2,.8);
\filldraw[\rColor] (-.2,0) rectangle (.2,.4);
\filldraw[\QrColor] (-.2,.4) rectangle (.2,.8);
\draw[\XColor,thick] (-.2,0) -- (-.2,.8);
\draw[\YColor,thick] (.2,0) -- (.2,.8);
\draw[\QsColor, thick] (-.2,.4) -- (.2,.4);
}
:|X\boxtimes_M Y| \to |X|\boxtimes_{|Q|} |Y|;
\qquad
\tikzmath{
\fill[\QrColor, rounded corners=5] (0,0) rectangle (.6,.6);
}
=
|Q|.
\end{equation}
For the second diagram, we omit the external shadings, which may denote either a left/right $M$-action, or a left $|P|$ and right $|R|$-action depending on context.
\end{nota}

\begin{rem}
\label{rem:QBimodules}
By the $\rm W^*$ version of \cite[Thm.~A]{2105.12010},
realization $|\,\cdot\,|$ gives a dagger 2-equivalence from the Q-system completion of $\WStarRCorr$ to $\WStarRCorr$.
Thus $|\,\cdot\,|$ gives an equivalence from the unitary tensor category of 
$Q-Q$ bimodules in $\fgpBim(M)$
with the tensor product $\otimes_Q$
to $\fgpBim(|Q|)$ with the Connes fusion relative tensor product $\boxtimes_Q$.
Moreover, the canonical tensorator $\mu_{X,Y}:|X|\boxtimes_{|Q|} |Y| \to |X\otimes_Q Y|$ fits into a commuting diagram with the canonical projectors \eqref{eq:CanonicalProjectors}:
\begin{equation}
\label{eq:CanonicalProjectorTriangle}
\tikzmath{
\node (a) at (0,0) {$|X\boxtimes_M Y|$};
\node (b) at (4,1) {$|X|\boxtimes_{|Q|}|Y|$};
\node (c) at (4,-1) {$|X\otimes_QY|$};
\draw[->] (a) -- (b);
\draw[->] (a) -- (c);
\draw[->] (b) -- node[right]{$\scriptstyle \mu_{X,Y}$} (c);
\coordinate (a) at (1.2,-1.2);
\draw ($ (a) + (-.05,0) $) -- ($ (a) + (-.05,.6) $) ;
\draw ($ (a) + (.85,0) $) -- ($ (a) + (.85,.6) $) ;
\draw[dotted, thin, rounded corners=5] ($ (a) + (.2,0) $) -- (a) -- ($ (a) + (0,.6) $) -- ($ (a) + (.2,.6) $);
\draw[dotted, thin, rounded corners=5] ($ (a) + (.6,0) $) -- ($ (a) + (.8,0) $)  -- ($ (a) + (.8,.6) $) -- ($ (a) + (.6,.6) $);
\filldraw[\rColor] ($ (a) + (.2,0) $) rectangle ($ (a) + (.6,.6)$);
\draw[\XColor,thick] ($ (a) + (.2,0) $) -- ($ (a) + (.2,.6) $);
\draw[\YColor,thick] ($ (a) + (.6,0) $) -- ($ (a) + (.6,.6) $);
\draw[\QsColor, thick] ($ (a) + (.2,.3) $) -- ($ (a) + (.6,.3) $);
\coordinate (b) at (1.2,.6);
\draw[dotted, thin, rounded corners=5] ($ (b) + (.2,0) $) -- (b) -- ($ (b) + (0,.6) $) -- ($ (b) + (.2,.6) $);
\draw[dotted, thin, rounded corners=5] ($ (b) + (.6,0) $) -- ($ (b) + (.8,0) $)  -- ($ (b) + (.8,.6) $) -- ($ (b) + (.6,.6) $);
\filldraw[\rColor] ($ (b) + (.2,.3) $) rectangle ($ (b) + (.6,0)$);
\filldraw[\QrColor] ($ (b) + (.2,.3) $) rectangle ($ (b) + (.6,.6)$);
\draw[\XColor,thick] ($ (b) + (.2,0) $) -- ($ (b) + (.2,.6) $);
\draw[\YColor,thick] ($ (b) + (.6,0) $) -- ($ (b) + (.6,.6) $);
\draw[\QsColor, thick] ($ (b) + (.2,.3) $) -- ($ (b) + (.6,.3) $);
}
\end{equation}
\end{rem}

%%%%%%%%%%%%%%%%%%%%%%%%%%%%%%%%%%%%%%%%%%%%%%%%%%%%%%%%%%%%%%%%%%%%%%%%%%%%%%%%%
\subsection{Local extension}
\label{sec:LocalExtension}

We now turn to the proof that $\Chi(|Q|) = \Chi(M)_Q^{\loc}$.

\begin{defn}
A Q-system $Q$ in a unitary braided tensor category $\cC$ is called \textit{commutative} if 
$$
\tikzmath{
\begin{scope}
\clip[rounded corners = 5] (-.3,-.5) rectangle (.7,1);
\filldraw[\rColor] (-.4,-.5) -- (-.4,1) -- (.7,1) -- (.7,-.5);
\end{scope}
\draw[\QsColor,thick] (0,-.5) .. controls ++(90:.45cm) and ++(270:.45cm) .. (.4,.5);
\draw[\QsColor,thick] (0,.5) arc (180:0:.2cm);
\draw[\QsColor,thick] (.2,.7) -- (.2,1);
\filldraw[\rColor] (.2,0) circle (.05cm);
\filldraw[\QsColor] (.2,.7) circle (.05cm);
\draw[\QsColor,thick] (.4,-.5) .. controls ++(90:.45cm) and ++(270:.45cm) .. (0,.5);
}
=
\tikzmath{
\fill[\rColor, rounded corners=5pt] (-.3,0) rectangle (.9,.6);
\draw[\QsColor,thick] (0,0) arc (180:0:.3cm);
\draw[\QsColor,thick] (.3,.3) -- (.3,.6);
\filldraw[\QsColor] (.3,.3) circle (.05cm);
}
=
\tikzmath{
\begin{scope}
\clip[rounded corners = 5] (-.3,-.5) rectangle (.7,1);
\filldraw[\rColor] (-.4,-.5) -- (-.4,1) -- (.7,1) -- (.7,-.5);
\end{scope}
\draw[\QsColor,thick] (.4,-.5) .. controls ++(90:.45cm) and ++(270:.45cm) .. (0,.5);
\draw[\QsColor,thick] (0,.5) arc (180:0:.2cm);
\draw[\QsColor,thick] (.2,.7) -- (.2,1);
\filldraw[\rColor] (.2,0) circle (.05cm);
\filldraw[\QsColor] (.2,.7) circle (.05cm);
\draw[\QsColor,thick] (0,-.5) .. controls ++(90:.45cm) and ++(270:.45cm) .. (.4,.5);
}\,.
$$
Suppose $Q\in\cC$ is a commutative connected Q-system and let $X\in\cC$ be a right $Q$-module. We say $X$ is  \textit{local} if 
\[
\tikzmath{
\begin{scope}
\clip[rounded corners = 5] (-.3,-.5) rectangle (.7,1.3);
\filldraw[\rColor] (-.4,-.5) -- (-.4,1.3) -- (.7,1.3) -- (.7,-.5);
\end{scope}
\draw[\QsColor,thick] (0,-.5) .. controls ++(90:.45cm) and ++(270:.45cm) .. (.4,.5) arc (0:90:.4cm);
\filldraw[\rColor] (.2,0) circle (.05cm);
\draw[\XColor,thick] (.4,-.5) .. controls ++(90:.45cm) and ++(270:.45cm) .. (0,.5) -- (0,1.3);
\filldraw[\XColor] (0,.9) circle (.05cm);
}
=
\tikzmath{
\begin{scope}
\clip[rounded corners = 5] (-.3,-.5) rectangle (.7,1.3);
\filldraw[\rColor] (-.4,-.5) -- (-.4,1.3) -- (.7,1.3) -- (.7,-.5);
\end{scope}
\draw[\XColor,thick] (.4,-.5) .. controls ++(90:.45cm) and ++(270:.45cm) .. (0,.5) -- (0,1.3);
\filldraw[\rColor] (.2,0) circle (.05cm);
\draw[\QsColor,thick] (0,-.5) .. controls ++(90:.45cm) and ++(270:.45cm) .. (.4,.5) arc (0:90:.4cm);
\filldraw[\XColor] (0,.9) circle (.05cm);
}
\]
We can turn a local $Q$-module $X_Q$ into a $Q-Q$ bimodule by defining the left action map by
\[
\tikzmath{
\begin{scope}
\clip[rounded corners = 5] (-.7,0) rectangle (.3,1);
\filldraw[\rColor] (-.7,-.5) -- (-.7,1.3) -- (.7,1.3) -- (.7,-.5);
\end{scope}
\draw[\XColor,thick] (0,0) -- (0,1);
\draw[\QsColor,thick] (-.4,0) -- (-.4,.2) arc (180:90:.4cm);
\filldraw[\XColor] (0,.6) circle (.05cm);
}
: = 
\tikzmath{
\begin{scope}
\clip[rounded corners = 5] (-.3,-.5) rectangle (.7,1.3);
\filldraw[\rColor] (-.4,-.5) -- (-.4,1.3) -- (.7,1.3) -- (.7,-.5);
\end{scope}
\draw[\QsColor,thick] (0,-.5) .. controls ++(90:.45cm) and ++(270:.45cm) .. (.4,.5) arc (0:90:.4cm);
\filldraw[\rColor] (.2,0) circle (.05cm);
\draw[\XColor,thick] (.4,-.5) .. controls ++(90:.45cm) and ++(270:.45cm) .. (0,.5) -- (0,1.3);
\filldraw[\XColor] (0,.9) circle (.05cm);
}
=
\tikzmath{
\begin{scope}
\clip[rounded corners = 5] (-.3,-.5) rectangle (.7,1.3);
\filldraw[\rColor] (-.4,-.5) -- (-.4,1.3) -- (.7,1.3) -- (.7,-.5);
\end{scope}
\draw[\XColor,thick] (.4,-.5) .. controls ++(90:.45cm) and ++(270:.45cm) .. (0,.5) -- (0,1.3);
\filldraw[\rColor] (.2,0) circle (.05cm);
\draw[\QsColor,thick] (0,-.5) .. controls ++(90:.45cm) and ++(270:.45cm) .. (.4,.5) arc (0:90:.4cm);
\filldraw[\XColor] (0,.9) circle (.05cm);
}\,.
\]
It is well known that under the bimodule relative tensor product $\otimes_Q$, 
the collection of local $Q$-modules $\cC^{\loc}_{Q}$ is a unitary braided tensor category, where the braiding is inherited from $\cC$. 
\end{defn}

Our goal now is to prove
$\Chi(|Q|)\cong\Chi(M)_Q^{\loc}$ as braided unitary tensor categories.

\begin{lem}
\label{lem:CSinMCSin|Q|}
If $(a_n)_n\subseteq M$ is a central sequence, then $(a_n)_n$ is also a central sequence in $|Q|$.
\end{lem}
\begin{proof}
Since $Q$ is centrally trivial, for all $q\in |Q|$, $\|a_n q-qa_n\|_2\to 0$ by Proposition \ref{prop:BimodualCT&CS}.
\end{proof}

\begin{prop}
\label{prop:APPBofQisCSinQ}
A Q-system $Q\in \Chi(M)$ is commutative if and only if 
for every approximate inner $Q_M$-basis $\{q_i^{(n)}\}$,
$(q_i^{(n)})_n$ is a central sequence in $|Q|$ for 
each fixed $i$.
\end{prop}
\begin{proof}
Suppose $\{q_j\}_j$ is a $Q_M$-basis and $\{q_i^{(n)}\}_i$ is an approximate $Q_M$-basis.
By \eqref{eq:OverBraiding} and \eqref{eq:UnderBraiding},
\begin{equation}
\label{eq:ExpandBraidingsQ}
\tikzmath{
\begin{scope}
\clip[rounded corners = 5] (-.55,-.5) rectangle (.7,.5);
\filldraw[\rColor] (-.4,-.5) -- (-.4,.5) -- (.7,.5) -- (.7,-.5);
\end{scope}
\draw[dashed] (-.4,-.5) -- (-.4,.5);
\draw[\QsColor,thick] (.4,-.5) .. controls ++(90:.45cm) and ++(270:.45cm) .. (0,.5);
\filldraw[\rColor] (.2,0) circle (.05cm);
\draw[\QsColor,thick] (0,-.5) .. controls ++(90:.45cm) and ++(270:.45cm) .. (.4,.5);
}
\approx
\sum_{i,j}
\tikzmath{
\begin{scope}
\clip[rounded corners = 5] (-.35,-2.2) rectangle (1,2.2);
\filldraw[\rColor] (-.2,-2.2) -- (-.2,-1.5) -- (0,-1.5) -- (0,-.5) -- (.3,-.5) -- (.3,.5) -- (0,.5) -- (0,1.5) -- (-.2,1.5) -- (-.2,2.2) -- (1.05,2.2) -- (1.05,-2.2);
\end{scope}
\draw[dashed] (-.2,-2.2) -- (-.2,-1.5);
\draw[\QsColor,thick] (.2,-2.2) -- (.2,-1.5);
\draw[dashed] (0,-1.5) -- (0,-.5);
\draw[\QsColor,thick] (.6,-2.2) -- (.6,-.5);
\draw[dashed] (.3,-.5) -- (.3,.5);
\draw[dashed] (0,.5) -- (0,1.5);
\draw[\QsColor,thick] (.6,.5) -- (.6,2.2);
\draw[dashed] (-.2,1.5) -- (-.2,2.2);
\draw[\QsColor,thick] (.2,1.5) -- (.2,2.2);
\roundNbox{unshaded}{(0,1.5)}{.3}{.1}{.1}{\scriptsize{$q_j$}};
\roundNbox{unshaded}{(.3,.5)}{.3}{.15}{.15}{\scriptsize{$q_i^{(n)}$}};
\roundNbox{unshaded}{(.3,-.5)}{.3}{.15}{.15}{\scriptsize{$q_j^\dag$}};
\roundNbox{unshaded}{(0,-1.5)}{.3}{.1}{.1}{\scriptsize{$(q_i^{(\!n\!)}\!)^\dag$}};
}
\qquad\text{and}\qquad
\tikzmath{
\begin{scope}
\clip[rounded corners = 5] (-.55,-.5) rectangle (.7,.5);
\filldraw[\rColor] (-.4,-.5) -- (-.4,.5) -- (.7,.5) -- (.7,-.5);
\end{scope}
\draw[dashed] (-.4,-.5) -- (-.4,.5);
\draw[\QsColor,thick] (0,-.5) .. controls ++(90:.45cm) and ++(270:.45cm) .. (.4,.5);
\filldraw[\rColor] (.2,0) circle (.05cm);
\draw[\QsColor,thick] (.4,-.5) .. controls ++(90:.45cm) and ++(270:.45cm) .. (0,.5);
}
\approx
\sum_{i,j}
\tikzmath{
\begin{scope}
\clip[rounded corners = 5] (-.35,-2.2) rectangle (1,2.2);
\filldraw[\rColor] (-.2,-2.2) -- (-.2,-1.5) -- (0,-1.5) -- (0,-.5) -- (.3,-.5) -- (.3,.5) -- (0,.5) -- (0,1.5) -- (-.2,1.5) -- (-.2,2.2) -- (1.05,2.2) -- (1.05,-2.2);
\end{scope}
\draw[dashed] (-.2,-2.2) -- (-.2,-1.5);
\draw[\QsColor,thick] (.2,-2.2) -- (.2,-1.5);
\draw[dashed] (0,-1.5) -- (0,-.5);
\draw[\QsColor,thick] (.6,-2.2) -- (.6,-.5);
\draw[dashed] (.3,-.5) -- (.3,.5);
\draw[dashed] (0,.5) -- (0,1.5);
\draw[\QsColor,thick] (.6,.5) -- (.6,2.2);
\draw[dashed] (-.2,1.5) -- (-.2,2.2);
\draw[\QsColor,thick] (.2,1.5) -- (.2,2.2);
\roundNbox{unshaded}{(0,1.5)}{.3}{.1}{.1}{\scriptsize{$q_j^{(n)}$}};
\roundNbox{unshaded}{(.3,.5)}{.3}{.15}{.15}{\scriptsize{$q_i$}};
\roundNbox{unshaded}{(.3,-.5)}{.3}{.15}{.15}{\scriptsize{$(q_j^{(\!n\!)}\!)^\dag$}};
\roundNbox{unshaded}{(0,-1.5)}{.3}{.1}{.1}{\scriptsize{$q_i^\dag$}};
}\,.
\end{equation}

First, if each $q_i^{(n)}$ is a central sequence, it is straightforward to see that $Q$ is commutative using \eqref{eq:ExpandBraidingsQ}.

Conversely, if $Q$ is commutative, to show each $q_i^{(n)}$ is a central sequence,
since $\|q_i^{(n)}a-aq_i^{(n)}\|_2\to 0$ for each $a\in M$, it suffices to prove $\|q_i^{(n)}q_j-q_jq_i^{(n)}\|_2\to 0$ for each $j$.
For each $i,j$ and $n$, define
\[
x_n^{i,j}:=q_i^{(n)}q_j-q_jq_i^{(n)}=
\tikzmath{
\begin{scope}
\clip[rounded corners = 5] (-.45,-1.2) rectangle (1.05,1.6);
\filldraw[\rColor] (.3,-1.2) -- (.3,-.5) -- (0,-.5) -- (0,.5) -- (-.2,.5) -- (-.2,2) -- (1.05,2) -- (1.05,-1.2);
\end{scope}
\draw[dashed] (-.2,.5) -- (-.2,1.6);
\draw[dashed] (0,-.5) -- (0,.5);
\draw[dashed] (.3,-1.2) -- (.3,-.5);
\draw[\QsColor,thick] (.2,.8) -- (.2,1) arc (180:0:.2cm) -- (.6,-.5);
\draw[\QsColor,thick] (.4,1.2) -- (.4,1.6);
\filldraw[\QsColor] (.4,1.2) circle (.05cm);
\roundNbox{unshaded}{(0,.5)}{.3}{.1}{.1}{\scriptsize{$q_i^{(n)}$}};
\roundNbox{unshaded}{(.3,-.5)}{.3}{.15}{.15}{\scriptsize{$q_j$}};
}
-
\tikzmath{
\begin{scope}
\clip[rounded corners = 5] (-.45,-1.2) rectangle (1.05,1.6);
\filldraw[\rColor] (.3,-1.2) -- (.3,-.5) -- (0,-.5) -- (0,.5) -- (-.2,.5) -- (-.2,2) -- (1.05,2) -- (1.05,-1.2);
\end{scope}
\draw[dashed] (-.2,.5) -- (-.2,1.6);
\draw[dashed] (0,-.5) -- (0,.5);
\draw[dashed] (.3,-1.2) -- (.3,-.5);
\draw[\QsColor,thick] (.2,.8) -- (.2,1) arc (180:0:.2cm) -- (.6,-.5);
\draw[\QsColor,thick] (.4,1.2) -- (.4,1.6);
\filldraw[\QsColor] (.4,1.2) circle (.05cm);
\roundNbox{unshaded}{(0,.5)}{.3}{.1}{.1}{\scriptsize{$q_j$}};
\roundNbox{unshaded}{(.3,-.5)}{.3}{.15}{.15}{\scriptsize{$q_i^{(n)}$}};
}
\in\,
\Hom(M_M \to Q_M).
\]
Observe that in $\End(Q_M)$, then again by \eqref{eq:ExpandBraidingsQ} we have
\begin{align*}
\sum_{i,j} x_n^{i,j} (x_n^{i,j})^\dag 
&=
\sum_{i,j}
\left(
\tikzmath{
\begin{scope}
\clip[rounded corners = 5] (-.35,-2.5) rectangle (1,2.5);
\filldraw[\rColor] (-.2,-2.5) -- (-.2,-1.5) -- (0,-1.5) -- (0,-.5) -- (.3,-.5) -- (.3,.5) -- (0,.5) -- (0,1.5) -- (-.2,1.5) -- (-.2,2.5) -- (1.05,2.5) -- (1.05,-2.5);
\end{scope}
\draw[dashed] (-.2,-2.5) -- (-.2,-1.5);
\draw[\QsColor,thick] (.2,-1.5) -- (.2,-2) arc (-180:0:.2cm) -- (.6,-.5);
\draw[dashed] (0,-1.5) -- (0,-.5);
\draw[\QsColor,thick] (.4,-2.5) -- (.4,-2.2);
\draw[dashed] (.3,-.5) -- (.3,.5);
\draw[dashed] (0,.5) -- (0,1.5);
\draw[\QsColor,thick] (.4,2.2) -- (.4,2.5);
\draw[dashed] (-.2,1.5) -- (-.2,2.5);
\draw[\QsColor,thick] (.2,1.5) -- (.2,2) arc (180:0:.2cm) -- (.6,.5);
\filldraw[\QsColor] (.4,-2.2) circle (.05cm);
\filldraw[\QsColor] (.4,2.2) circle (.05cm);
\roundNbox{unshaded}{(0,1.5)}{.3}{.1}{.1}{\scriptsize{$q_i^{(n)}$}};
\roundNbox{unshaded}{(.3,.5)}{.3}{.15}{.15}{\scriptsize{$q_j$}};
\roundNbox{unshaded}{(.3,-.5)}{.3}{.15}{.15}{\scriptsize{$q_j^\dag$}};
\roundNbox{unshaded}{(0,-1.5)}{.3}{.1}{.1}{\scriptsize{$(q_i^{(\!n\!)}\!)^\dag$}};
}
-
\tikzmath{
\begin{scope}
\clip[rounded corners = 5] (-.35,-2.5) rectangle (1,2.5);
\filldraw[\rColor] (-.2,-2.5) -- (-.2,-1.5) -- (0,-1.5) -- (0,-.5) -- (.3,-.5) -- (.3,.5) -- (0,.5) -- (0,1.5) -- (-.2,1.5) -- (-.2,2.5) -- (1.05,2.5) -- (1.05,-2.5);
\end{scope}
\draw[dashed] (-.2,-2.5) -- (-.2,-1.5);
\draw[\QsColor,thick] (.2,-1.5) -- (.2,-2) arc (-180:0:.2cm) -- (.6,-.5);
\draw[dashed] (0,-1.5) -- (0,-.5);
\draw[\QsColor,thick] (.4,-2.5) -- (.4,-2.2);
\draw[dashed] (.3,-.5) -- (.3,.5);
\draw[dashed] (0,.5) -- (0,1.5);
\draw[\QsColor,thick] (.4,2.2) -- (.4,2.5);
\draw[dashed] (-.2,1.5) -- (-.2,2.5);
\draw[\QsColor,thick] (.2,1.5) -- (.2,2) arc (180:0:.2cm) -- (.6,.5);
\filldraw[\QsColor] (.4,-2.2) circle (.05cm);
\filldraw[\QsColor] (.4,2.2) circle (.05cm);
\roundNbox{unshaded}{(0,1.5)}{.3}{.1}{.1}{\scriptsize{$q_i^{(n)}$}};
\roundNbox{unshaded}{(.3,.5)}{.3}{.15}{.15}{\scriptsize{$q_j$}};
\roundNbox{unshaded}{(.3,-.5)}{.3}{.15}{.15}{\scriptsize{$(q_i^{(\!n\!)})^\dag$}};
\roundNbox{unshaded}{(0,-1.5)}{.3}{.1}{.1}{\scriptsize{$q_j^\dag$}};
}
-
\tikzmath{
\begin{scope}
\clip[rounded corners = 5] (-.35,-2.5) rectangle (1,2.5);
\filldraw[\rColor] (-.2,-2.5) -- (-.2,-1.5) -- (0,-1.5) -- (0,-.5) -- (.3,-.5) -- (.3,.5) -- (0,.5) -- (0,1.5) -- (-.2,1.5) -- (-.2,2.5) -- (1.05,2.5) -- (1.05,-2.5);
\end{scope}
\draw[dashed] (-.2,-2.5) -- (-.2,-1.5);
\draw[\QsColor,thick] (.2,-1.5) -- (.2,-2) arc (-180:0:.2cm) -- (.6,-.5);
\draw[dashed] (0,-1.5) -- (0,-.5);
\draw[\QsColor,thick] (.4,-2.5) -- (.4,-2.2);
\draw[dashed] (.3,-.5) -- (.3,.5);
\draw[dashed] (0,.5) -- (0,1.5);
\draw[\QsColor,thick] (.4,2.2) -- (.4,2.5);
\draw[dashed] (-.2,1.5) -- (-.2,2.5);
\draw[\QsColor,thick] (.2,1.5) -- (.2,2) arc (180:0:.2cm) -- (.6,.5);
\filldraw[\QsColor] (.4,-2.2) circle (.05cm);
\filldraw[\QsColor] (.4,2.2) circle (.05cm);
\roundNbox{unshaded}{(0,1.5)}{.3}{.1}{.1}{\scriptsize{$q_j$}};
\roundNbox{unshaded}{(.3,.5)}{.3}{.15}{.15}{\scriptsize{$q_i^{(n)}$}};
\roundNbox{unshaded}{(.3,-.5)}{.3}{.15}{.15}{\scriptsize{$q_j^\dag$}};
\roundNbox{unshaded}{(0,-1.5)}{.3}{.1}{.1}{\scriptsize{$(q_i^{(\!n\!)}\!)^\dag$}};
}
+
\tikzmath{
\begin{scope}
\clip[rounded corners = 5] (-.35,-2.5) rectangle (1,2.5);
\filldraw[\rColor] (-.2,-2.5) -- (-.2,-1.5) -- (0,-1.5) -- (0,-.5) -- (.3,-.5) -- (.3,.5) -- (0,.5) -- (0,1.5) -- (-.2,1.5) -- (-.2,2.5) -- (1.05,2.5) -- (1.05,-2.5);
\end{scope}
\draw[dashed] (-.2,-2.5) -- (-.2,-1.5);
\draw[\QsColor,thick] (.2,-1.5) -- (.2,-2) arc (-180:0:.2cm) -- (.6,-.5);
\draw[dashed] (0,-1.5) -- (0,-.5);
\draw[\QsColor,thick] (.4,-2.5) -- (.4,-2.2);
\draw[dashed] (.3,-.5) -- (.3,.5);
\draw[dashed] (0,.5) -- (0,1.5);
\draw[\QsColor,thick] (.4,2.2) -- (.4,2.5);
\draw[dashed] (-.2,1.5) -- (-.2,2.5);
\draw[\QsColor,thick] (.2,1.5) -- (.2,2) arc (180:0:.2cm) -- (.6,.5);
\filldraw[\QsColor] (.4,-2.2) circle (.05cm);
\filldraw[\QsColor] (.4,2.2) circle (.05cm);
\roundNbox{unshaded}{(0,1.5)}{.3}{.1}{.1}{\scriptsize{$q_j$}};
\roundNbox{unshaded}{(.3,.5)}{.3}{.15}{.15}{\scriptsize{$q_i^{(n)}$}};
\roundNbox{unshaded}{(.3,-.5)}{.3}{.15}{.15}{\scriptsize{$(q_i^{(\!n\!)})^\dag$}};
\roundNbox{unshaded}{(0,-1.5)}{.3}{.1}{.1}{\scriptsize{$q_j^\dag$}};;
}
\right)
\displaybreak[1]\\
&\approx
\tikzmath{
\begin{scope}
\clip[rounded corners = 5] (-.8,-.6) rectangle (.6,.6);
\filldraw[\rColor] (-.6,-.6) rectangle (.6,.6);
\end{scope}
\draw[dashed] (-.6,-.6) -- (-.6,.6); 
\draw[\QsColor,thick] (0,-.6) -- (0,-.3) arc (-90:270:.3cm);
\draw[\QsColor,thick] (0,.3) -- (0,.6);
\filldraw[\QsColor] (0,-.3) circle (.05cm);
\filldraw[\QsColor] (0,.3) circle (.05cm);
}
-
\tikzmath{
\begin{scope}
\clip[rounded corners = 5] (-.5,-1) rectangle (.7,1);
\filldraw[\rColor] (-.3,-1) rectangle (.7,1);
\end{scope}
\draw[dashed] (-.3,-1) -- (-.3,1); 
\draw[\QsColor,thick] (0,-.5) .. controls ++(90:.45cm) and ++(270:.45cm) .. (.4,.5);
\draw[\QsColor,thick] (0,.5) arc (180:0:.2cm);
\draw[\QsColor,thick] (.2,.7) -- (.2,1);
\draw[\QsColor,thick] (0,-.5) arc (-180:0:.2cm);
\draw[\QsColor,thick] (.2,-.7) -- (.2,-1);
\filldraw[\rColor] (.2,0) circle (.05cm);
\draw[\QsColor,thick] (.4,-.5) .. controls ++(90:.45cm) and ++(270:.45cm) .. (0,.5);
\filldraw[\QsColor] (.2,.7) circle (.05cm);
\filldraw[\QsColor] (.2,-.7) circle (.05cm);
}
-
\tikzmath{
\begin{scope}
\clip[rounded corners = 5] (-.5,-1) rectangle (.7,1);
\filldraw[\rColor] (-.3,-1) rectangle (.7,1);
\end{scope}
\draw[dashed] (-.3,-1) -- (-.3,1); 
\draw[\QsColor,thick] (.4,-.5) .. controls ++(90:.45cm) and ++(270:.45cm) .. (0,.5);
\draw[\QsColor,thick] (0,.5) arc (180:0:.2cm);
\draw[\QsColor,thick] (.2,.7) -- (.2,1);
\draw[\QsColor,thick] (0,-.5) arc (-180:0:.2cm);
\draw[\QsColor,thick] (.2,-.7) -- (.2,-1);
\filldraw[\rColor] (.2,0) circle (.05cm);
\draw[\QsColor,thick] (0,-.5) .. controls ++(90:.45cm) and ++(270:.45cm) .. (.4,.5);
\filldraw[\QsColor] (.2,.7) circle (.05cm);
\filldraw[\QsColor] (.2,-.7) circle (.05cm);
}
+
\tikzmath{
\begin{scope}
\clip[rounded corners = 5] (-.8,-.6) rectangle (.6,.6);
\filldraw[\rColor] (-.6,-.6) rectangle (.6,.6);
\end{scope}
\draw[dashed] (-.6,-.6) -- (-.6,.6); 
\draw[\QsColor,thick] (0,-.6) -- (0,-.3) arc (-90:270:.3cm);
\draw[\QsColor,thick] (0,.3) -- (0,.6);
\filldraw[\QsColor] (0,-.3) circle (.05cm);
\filldraw[\QsColor] (0,.3) circle (.05cm);
}
= 0. 
\end{align*}
But since $\sum_{i,j} x_n^{i,j} (x_n^{i,j})^\dag$ is positive, we must have 
$\|\cdot\|_2-\lim_n x_n^{i,j} (x_n^{i,j})^\dag=0$ for each $i,j$. 
Since $\WStarRCorr$ is $\rm W^*$, we must have $\lim_n x_n^{i,j} = 0$, so each $q_i^{(n)}$ is a central sequence.
\end{proof}

\begin{lem}
\label{lem:APPAfromMto|Q|}
Suppose $X\in \fgpBim(M)$ is a right $Q$-module.
If $\{b_i^{(n)}\}$ is an approximate $X_M$-basis, then it is also an approximate $|X|_{|Q|}$-basis.
Similarly, if $\{b_i\}$ is an $X_M$-basis, then it is also an $|X|_{|Q|}$-basis.
\end{lem}
\begin{proof}
We prove the approximate version, and the ordinary version is similar, but easier.
Identifying $X_M\cong |X|_M$ as right $M$-modules, observe
\[
\sum_i b_i^{(n)}\langle b_i^{(n)}|x\rangle^X_M 
=
\sum_i
\tikzmath{
\begin{scope}
\clip[rounded corners = 5] (-.35,-2.2) rectangle (.7,1.2);
\filldraw[\rColor] (0,-2.2) -- (0,-1.5) -- (-.2,-1.55) -- (-.2,-.5) -- (0,-.5) -- (0,.5) -- (-.2,.5) -- (-.2,1.7) -- (1.3,1.7) -- (1.3,-2.55);
\end{scope}
\draw[dashed] (0,-2.2) -- (0,-1.5);
\draw[dashed] (-.2,-1.55) -- (-.2,-.5);
\draw[dashed] (0,-.5) -- (0,.5);
\draw[dashed] (-.2,.5) -- (-.2,1.2);
\draw[\XColor,thick] (.2,-1.55) -- (.2,-.5);
\draw[\XColor,thick] (.2,.5) -- (.2,1.2);
\roundNbox{unshaded}{(0,.5)}{.3}{.1}{.1}{\scriptsize{$b_i^{(n)}$}};
\roundNbox{unshaded}{(0,-.5)}{.3}{.1}{.1}{\scriptsize{$(b_i^{(\!n\!)}\!)^\dag$}};
\roundNbox{unshaded}{(0,-1.5)}{.3}{.1}{.1}{\scriptsize{$x$}};
}
\approx x
\qquad\Longrightarrow\qquad
\sum_i b_i^{(n)}\langle b_i^{(n)}|x\rangle^{|X|}_{|Q|} =
\sum_i
\tikzmath{
\begin{scope}
\clip[rounded corners = 5] (-.35,-2.6) rectangle (1,1.6);
\filldraw[\rColor] (0,-2.6) -- (0,-1.9) -- (-.2,-1.9) -- (-.2,-.5) -- (0,-.5) -- (0,.5) -- (-.2,.5) -- (-.2,1.7) -- (1.3,1.7) -- (1.3,-2.6);
\end{scope}
\draw[dashed] (0,-2.6) -- (0,-1.9);
\draw[dashed] (-.2,-1.9) -- (-.2,-.5);
\draw[dashed] (0,-.5) -- (0,.5);
\draw[dashed] (-.2,.5) -- (-.2,1.6);
\draw[\XColor,thick] (.2,-1.9) -- (.2,-.5);
\draw[\XColor,thick] (.2,.5) -- (.2,1.6);
\draw[\QsColor,thick] (.2,-1.2) arc (-90:0:.4cm) -- (.6,.8) arc (0:90:.4cm);
\filldraw[\XColor] (.2,-1.2) circle (.05cm);
\filldraw[\XColor] (.2,1.2) circle (.05cm);
\roundNbox{unshaded}{(0,.5)}{.3}{.1}{.1}{\scriptsize{$b_i^{(n)}$}};
\roundNbox{unshaded}{(0,-.5)}{.3}{.1}{.1}{\scriptsize{$(b_i^{(\!n\!)}\!)^\dag$}};
\roundNbox{unshaded}{(0,-1.9)}{.3}{.1}{.1}{\scriptsize{$x$}};
}
\approx x.
\qedhere
\]
\end{proof}

\begin{prop}
\label{Chi(M)locQsubsetChi(Q)}
Realization $|\,\cdot\,|$ takes every bimodule in $\Chi(M)_Q^\loc$ into $\Chi(|Q|)$.
\end{prop}
\begin{proof}
Suppose $X\in \Chi(M)_Q^{\loc}$.
We show ${}_{|Q|}|X|_{|Q|}$ is approximately inner and centrally trivial.

Since $X$ is approximately inner over $M$,
by Proposition \ref{prop:AI&APPB},
there is an approximately inner $X_M$-basis
$\{b_i^{(n)}\}\subseteq X$.
By Lemma \ref{lem:APPAfromMto|Q|}, $\{b_i^{(n)}\}$ is an approximate $|X|_{|Q|}$-basis. 
Now we show \eqref{eq:AICommutativity}, i.e., $\|qb_i^{(n)}-b_i^{(n)}q\|_2\to 0$ for all $q\in |Q|$.
\[
q b_i^{(n)}
=
\tikzmath{
\begin{scope}
\clip[rounded corners = 5] (-.45,-1.2) rectangle (1.05,1.6);
\filldraw[\rColor] (.3,-1.2) -- (.3,-.5) -- (0,-.5) -- (0,.5) -- (-.2,.5) -- (-.2,2) -- (1.05,2) -- (1.05,-1.2);
\end{scope}
\draw[dashed] (-.2,.5) -- (-.2,1.6);
\draw[dashed] (0,-.5) -- (0,.5);
\draw[dashed] (.3,-1.2) -- (.3,-.5);
\draw[\QsColor,thick] (.2,.8) arc (180:90:.4cm);
\draw[\XColor,thick] (.6,-.5) -- (.6,1.6);
\filldraw[\XColor] (.6,1.2) circle (.05cm);
\roundNbox{unshaded}{(0,.5)}{.3}{.1}{.1}{\scriptsize{$q$}};
\roundNbox{unshaded}{(.3,-.5)}{.3}{.15}{.15}{\scriptsize{$b_i^{(n)}$}};
}
=
\tikzmath{
\begin{scope}
\clip[rounded corners = 5] (-1,-.7) rectangle (.5,3.1);
\filldraw[\rColor] (-.2,-.7) -- (-.2,0) -- (-.4,.3) .. controls ++(90:.2cm) and ++(270:.2cm) .. (-.6,.7) -- (-.8,1) -- (-.8,3.1) -- (.9,3.1) -- (.9,-.7);
\end{scope}
\draw[dashed] (-.2,-.7) -- (-.2,0);
\draw[dashed] (-.4,.3) .. controls ++(90:.2cm) and ++(270:.2cm) .. (-.6,.7);
\draw[dashed] (-.8,1) -- (-.8,3.1);
\draw[\QsColor,thick] (-.4,1.3) .. controls ++(90:.45cm) and ++(270:.45cm) .. (0,2.3) arc (0:90:.4cm);
\filldraw[\rColor] (-.2,1.8) circle (.05cm);
\filldraw[\XColor] (-.4,2.7) circle (.05cm);
\draw[\XColor,thick] (0,0) -- (0,1.3) .. controls ++(90:.45cm) and ++(270:.45cm) .. (-.4,2.3) -- (-.4,3.1);
\roundNbox{unshaded}{(-.6,1)}{.3}{.1}{.1}{\scriptsize{$q$}};
\roundNbox{unshaded}{(-.2,0)}{.3}{.1}{.1}{\scriptsize{$b_i^{(n)}$}};
}
\approx
\sum_{j,k}
\tikzmath{
\begin{scope}
\clip[rounded corners = 5] (-.45,-4.2) rectangle (1,2.6);
\filldraw[\rColor] (.3,-4.2) -- (.3,-3.5) -- (0,-3.5) -- (0,-2.5) -- (-.2,-2.5) -- (-.2,-1.5) -- (0,-1.5) -- (0,-.5) -- (.3,-.5) -- (.3,.5) -- (0,.5) -- (0,1.5) -- (-.2,1.5) -- (-.2,2.6) -- (1,2.6) -- (1,-4.2);
\end{scope}
\draw[dashed] (.3,-4.2) -- (.3,-3.5);
\draw[dashed] (0,-3.5) -- (0,-2.5);
\draw[dashed] (-.2,-2.5) -- (-.2,-1.5);
\draw[\QsColor,thick] (.2,-2.5) -- (.2,-1.5);
\draw[dashed] (0,-1.5) -- (0,-.5);
\draw[\XColor,thick] (.6,-3.5) -- (.6,-.5);
\draw[dashed] (.3,-.5) -- (.3,.5);
\draw[dashed] (0,.5) -- (0,1.5);
\draw[\QsColor,thick] (.6,.5) -- (.6,1.8) arc (0:90:.4cm);
\draw[dashed] (-.2,1.5) -- (-.2,2.2);
\draw[\XColor,thick] (.2,1.5) -- (.2,2.6);
\filldraw[\XColor] (.2,2.2) circle (.05cm);
\roundNbox{unshaded}{(0,1.5)}{.3}{.1}{.1}{\scriptsize{$b_j^{(n)}$}};
\roundNbox{unshaded}{(.3,.5)}{.3}{.15}{.15}{\scriptsize{$q_k$}};
\roundNbox{unshaded}{(.3,-.5)}{.3}{.15}{.15}{\scriptsize{$(b_j^{(n)})^\dag$}};
\roundNbox{unshaded}{(0,-1.5)}{.3}{.1}{.1}{\scriptsize{$q_k^\dag$}};
\roundNbox{unshaded}{(0,-2.5)}{.3}{.1}{.1}{\scriptsize{$q$}};
\roundNbox{unshaded}{(.3,-3.5)}{.3}{.15}{.15}{\scriptsize{$b_i^{(n)}$}};
}
\approx
\sum_{j,k}
\tikzmath{
\begin{scope}
\clip[rounded corners = 5] (-.45,-4.2) rectangle (1,2.6);
\filldraw[\rColor] (.3,-4.2) -- (.3,-3.5) -- (0,-3.5) -- (0,-2.5) -- (.3,-2.5) -- (.3,-1.5) -- (0,-1.5) -- (0,-.5) -- (.3,-.5) -- (.3,.5) -- (0,.5) -- (0,1.5) -- (-.2,1.5) -- (-.2,2.6) -- (1,2.6) -- (1,-4.2);
\end{scope}
\draw[dashed] (.3,-4.2) -- (.3,-3.5);
\draw[dashed] (0,-3.5) -- (0,-2.5);
\draw[dashed] (.3,-2.5) -- (.3,-1.5);
\draw[\QsColor,thick] (.6,-3.5) -- (.6,-2.5);
\draw[dashed] (0,-1.5) -- (0,-.5);
\draw[\XColor,thick] (.6,-1.5) -- (.6,-.5);
\draw[dashed] (.3,-.5) -- (.3,.5);
\draw[dashed] (0,.5) -- (0,1.5);
\draw[\QsColor,thick] (.6,.5) -- (.6,1.8) arc (0:90:.4cm);
\draw[dashed] (-.2,1.5) -- (-.2,2.6);
\draw[\XColor,thick] (.2,1.5) -- (.2,2.6);
\filldraw[\XColor] (.2,2.2) circle (.05cm);
\roundNbox{unshaded}{(0,1.5)}{.3}{.1}{.1}{\scriptsize{$b_j^{(m)}$}};
\roundNbox{unshaded}{(.3,.5)}{.3}{.15}{.15}{\scriptsize{$q_k$}};
\roundNbox{unshaded}{(.3,-.5)}{.3}{.15}{.15}{\scriptsize{$(b_j^{(n)})^\dag$}};
\roundNbox{unshaded}{(.3,-1.5)}{.3}{.15}{.15}{\scriptsize{$b_i^{(n)}$}};
\roundNbox{unshaded}{(.3,-2.5)}{.3}{.15}{.15}{\scriptsize{$q_k^\dag$}};
\roundNbox{unshaded}{(.3,-3.5)}{.3}{.15}{.15}{\scriptsize{$q$}};
}
\approx
\sum_{j,k}
\tikzmath{
\begin{scope}
\clip[rounded corners = 5] (-.45,-4.2) rectangle (1,2.6);
\filldraw[\rColor] (.3,-4.2) -- (.3,-3.5) -- (0,-3.5) -- (0,-2.5) -- (.3,-2.5) -- (.3,-1.5) -- (0,-1.5) -- (0,-.5) -- (-.2,-.5) -- (-.2,.5) -- (0,.5) -- (0,1.5) -- (-.2,1.5) -- (-.2,2.6) -- (1,2.6) -- (1,-4.2);
\end{scope}
\draw[dashed] (.3,-4.2) -- (.3,-3.5);
\draw[dashed] (0,-3.5) -- (0,-2.5);
\draw[dashed] (.3,-2.5) -- (.3,-1.5);
\draw[\QsColor,thick] (.6,-3.5) -- (.6,-2.5);
\draw[dashed] (0,-1.5) -- (0,-.5);
\draw[\XColor,thick] (.2,-.5) -- (.2,.5);
\draw[dashed] (-.2,-.5) -- (-.2,.5);
\draw[dashed] (0,.5) -- (0,1.5);
\draw[\QsColor,thick] (.6,-1.5) -- (.6,1.8) arc (0:90:.4cm);
\draw[dashed] (-.2,1.5) -- (-.2,2.6);
\draw[\XColor,thick] (.2,1.5) -- (.2,2.6);
\filldraw[\XColor] (.2,2.2) circle (.05cm);
\roundNbox{unshaded}{(0,1.5)}{.3}{.1}{.1}{\scriptsize{$b_j^{(n)}$}};
\roundNbox{unshaded}{(0,.5)}{.3}{.1}{.1}{\scriptsize{$(b_j^{(\!n\!)})^\dag$}};
\roundNbox{unshaded}{(0,-.5)}{.3}{.1}{.1}{\scriptsize{$b_i^{(n)}$}};
\roundNbox{unshaded}{(.3,-1.5)}{.3}{.15}{.15}{\scriptsize{$q_k$}};
\roundNbox{unshaded}{(.3,-2.5)}{.3}{.15}{.15}{\scriptsize{$q_k^\dag$}};
\roundNbox{unshaded}{(.3,-3.5)}{.3}{.15}{.15}{\scriptsize{$q$}};
}
\approx
\tikzmath{
\begin{scope}
\clip[rounded corners = 5] (-.45,-1.2) rectangle (1.05,1.6);
\filldraw[\rColor] (.3,-1.2) -- (.3,-.5) -- (0,-.5) -- (0,.5) -- (-.2,.5) -- (-.2,1.6) -- (1.05,1.6) -- (1.05,-1.2);
\end{scope}
\draw[dashed] (-.2,.5) -- (-.2,1.6);
\draw[dashed] (0,-.5) -- (0,.5);
\draw[dashed] (.3,-1.2) -- (.3,-.5);
\draw[\XColor,thick] (.2,.5) -- (.2,1.6);
\draw[\QsColor,thick] (.6,-.5) -- (.6,.8) arc (0:90:.4cm);
\filldraw[\XColor] (.2,1.2) circle (.05cm);
\roundNbox{unshaded}{(0,.5)}{.3}{.1}{.1}{\scriptsize{$b_i^{(n)}$}};
\roundNbox{unshaded}{(.3,-.5)}{.3}{.15}{.15}{\scriptsize{$q$}};
}
= 
b_i^{(n)} q
\]
The second equality uses the hypothesis that $X$ is a local bimodule.
The third $\approx$ uses \eqref{eq:UnderBraiding}. 
The fourth $\approx$ holds because $\langle q_k|q\rangle^Q_M \in M$ and $\|[x,b_i^{(n)}]\|_2\to 0$ for $x\in M$.
The fifth $\approx$ holds because $\langle b_j^{(n)}|b_i^{(n)}\rangle^X_M$ is a central sequence in $M$ for each $i,j$ 
and $Q$ is centrally trivial over $M$.
The sixth $\approx$ holds from the definition of (approximate) basis for $X_M$ and $Q_M$.
We conclude that $|X|$ is approximately inner over $|Q|$.

Since $X$ is centrally trivial over $M$, by Proposition \ref{prop:BimodualCT&CS}, for all central sequences $(a_n)_n\subseteq M$ and all $x\in X$, $\|a_n x-xa_n\|_2 \to 0$.
If $(q_n)_n\subseteq |Q|$ is a central sequence, then 
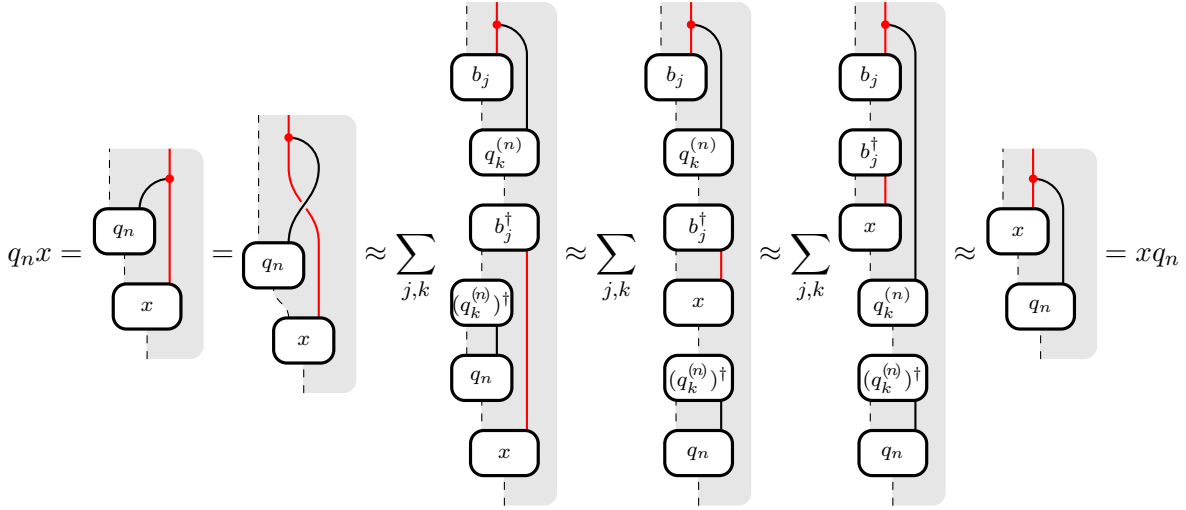
\[
q_n x
=
\tikzmath{
\begin{scope}
\clip[rounded corners = 5] (-.45,-1.2) rectangle (1.05,1.6);
\filldraw[\rColor] (.3,-1.2) -- (.3,-.5) -- (0,-.5) -- (0,.5) -- (-.2,.5) -- (-.2,2) -- (1.05,2) -- (1.05,-1.2);
\end{scope}
\draw[dashed] (-.2,.5) -- (-.2,1.6);
\draw[dashed] (0,-.5) -- (0,.5);
\draw[dashed] (.3,-1.2) -- (.3,-.5);
\draw[\QsColor,thick] (.2,.8) arc (180:90:.4cm);
\draw[\XColor,thick] (.6,-.5) -- (.6,1.6);
\filldraw[\XColor] (.6,1.2) circle (.05cm);
\roundNbox{unshaded}{(0,.5)}{.3}{.1}{.1}{\scriptsize{$q_n$}};
\roundNbox{unshaded}{(.3,-.5)}{.3}{.15}{.15}{\scriptsize{$x$}};
}
=
\tikzmath{
\begin{scope}
\clip[rounded corners = 5] (-1,-.7) rectangle (.5,3);
\filldraw[\rColor] (-.2,-.7) -- (-.2,0) -- (-.4,.3) .. controls ++(90:.2cm) and ++(270:.2cm) .. (-.6,.7) -- (-.8,1) -- (-.8,3) -- (.9,3) -- (.9,-.7);
\end{scope}
\draw[dashed] (-.2,-.7) -- (-.2,0);
\draw[dashed] (-.4,.3) .. controls ++(90:.2cm) and ++(270:.2cm) .. (-.6,.7);
\draw[dashed] (-.8,1) -- (-.8,3);
\draw[\XColor,thick] (0,0) -- (0,1.3) .. controls ++(90:.45cm) and ++(270:.45cm) .. (-.4,2.3) -- (-.4,3);
\filldraw[\rColor] (-.2,1.8) circle (.05cm);
\draw[\QsColor,thick] (-.4,1.3) .. controls ++(90:.45cm) and ++(270:.45cm) .. (0,2.3) arc (0:90:.4cm);
\filldraw[\XColor] (-.4,2.7) circle (.05cm);
\roundNbox{unshaded}{(-.6,1)}{.3}{.1}{.1}{\scriptsize{$q_n$}};
\roundNbox{unshaded}{(-.2,0)}{.3}{.1}{.1}{\scriptsize{$x$}};
}
\approx
\sum_{j,k}
\tikzmath{
\begin{scope}
\clip[rounded corners = 5] (-.45,-4.2) rectangle (1,2.5);
\filldraw[\rColor] (.3,-4.2) -- (.3,-3.5) -- (0,-3.5) -- (0,-2.5) -- (-.2,-2.5) -- (-.2,-1.5) -- (0,-1.5) -- (0,-.5) -- (.3,-.5) -- (.3,.5) -- (0,.5) -- (0,1.5) -- (-.2,1.5) -- (-.2,2.5) -- (1,2.5) -- (1,-4.2);
\end{scope}
\draw[dashed] (.3,-4.2) -- (.3,-3.5);
\draw[dashed] (0,-3.5) -- (0,-2.5);
\draw[dashed] (-.2,-2.2) -- (-.2,-1.5);
\draw[\QsColor,thick] (.2,-2.2) -- (.2,-1.5);
\draw[dashed] (0,-1.5) -- (0,-.5);
\draw[\XColor,thick] (.6,-3.5) -- (.6,-.5);
\draw[dashed] (.3,-.5) -- (.3,.5);
\draw[dashed] (0,.5) -- (0,1.5);
\draw[\QsColor,thick] (.6,.5) -- (.6,1.8) arc (0:90:.4cm);
\draw[dashed] (-.2,1.5) -- (-.2,2.5);
\draw[\XColor,thick] (.2,1.5) -- (.2,2.5);
\filldraw[\XColor] (.2,2.2) circle (.05cm);
\roundNbox{unshaded}{(0,1.5)}{.3}{.1}{.1}{\scriptsize{$b_j$}};
\roundNbox{unshaded}{(.3,.5)}{.3}{.15}{.15}{\scriptsize{$q_k^{(n)}$}};
\roundNbox{unshaded}{(.3,-.5)}{.3}{.15}{.15}{\scriptsize{$b_j^\dag$}};
\roundNbox{unshaded}{(0,-1.5)}{.3}{.1}{.1}{\scriptsize{$(q_k^{(\!n\!)})^\dag$}};
\roundNbox{unshaded}{(0,-2.5)}{.3}{.1}{.1}{\scriptsize{$q_n$}};
\roundNbox{unshaded}{(.3,-3.5)}{.3}{.15}{.15}{\scriptsize{$x$}};
}
\approx
\sum_{j,k}
\tikzmath{
\begin{scope}
\clip[rounded corners = 5] (-.45,-4.2) rectangle (1,2.5);
\filldraw[\rColor] (.3,-4.2) -- (.3,-3.5) -- (0,-3.5) -- (0,-2.5) -- (.3,-2.5) -- (.3,-1.5) -- (0,-1.5) -- (0,-.5) -- (.3,-.5) -- (.3,.5) -- (0,.5) -- (0,1.5) -- (-.2,1.5) -- (-.2,2.5) -- (1,2.5) -- (1,-4.2);
\end{scope}
\draw[dashed] (.3,-4.2) -- (.3,-3.5);
\draw[dashed] (0,-3.5) -- (0,-2.5);
\draw[dashed] (.3,-2.5) -- (.3,-1.5);
\draw[\QsColor,thick] (.6,-3.5) -- (.6,-2.5);
\draw[dashed] (0,-1.5) -- (0,-.5);
\draw[\XColor,thick] (.6,-1.5) -- (.6,-.5);
\draw[dashed] (.3,-.5) -- (.3,.5);
\draw[dashed] (0,.5) -- (0,1.5);
\draw[\QsColor,thick] (.6,.5) -- (.6,1.8) arc (0:90:.4cm);
\draw[dashed] (-.2,1.5) -- (-.2,2.5);
\draw[\XColor,thick] (.2,1.5) -- (.2,2.5);
\filldraw[\XColor] (.2,2.2) circle (.05cm);
\roundNbox{unshaded}{(0,1.5)}{.3}{.1}{.1}{\scriptsize{$b_j$}};
\roundNbox{unshaded}{(.3,.5)}{.3}{.15}{.15}{\scriptsize{$q_k^{(n)}$}};
\roundNbox{unshaded}{(.3,-.5)}{.3}{.15}{.15}{\scriptsize{$b_j^\dag$}};
\roundNbox{unshaded}{(.3,-1.5)}{.3}{.15}{.15}{\scriptsize{$x$}};
\roundNbox{unshaded}{(.3,-2.5)}{.3}{.15}{.15}{\scriptsize{$(q_k^{(\!n\!)})^\dag$}};
\roundNbox{unshaded}{(.3,-3.5)}{.3}{.15}{.15}{\scriptsize{$q_n$}};
}
\approx
\sum_{j,k}
\tikzmath{
\begin{scope}
\clip[rounded corners = 5] (-.45,-4.2) rectangle (1,2.5);
\filldraw[\rColor] (.3,-4.2) -- (.3,-3.5) -- (0,-3.5) -- (0,-2.5) -- (.3,-2.5) -- (.3,-1.5) -- (0,-1.5) -- (0,-.5) -- (-.2,-.5) -- (-.2,.5) -- (0,.5) -- (0,1.5) -- (-.2,1.5) -- (-.2,2.5) -- (1,2.5) -- (1,-4.2);
\end{scope}
\draw[dashed] (.3,-4.2) -- (.3,-3.5);
\draw[dashed] (0,-3.5) -- (0,-2.5);
\draw[dashed] (.3,-2.5) -- (.3,-1.5);
\draw[\QsColor,thick] (.6,-3.5) -- (.6,-2.5);
\draw[dashed] (0,-1.5) -- (0,-.5);
\draw[\XColor,thick] (.2,-.5) -- (.2,.5);
\draw[dashed] (-.2,-.5) -- (-.2,.5);
\draw[dashed] (0,.5) -- (0,1.5);
\draw[\QsColor,thick] (.6,-1.5) -- (.6,1.8) arc (0:90:.4cm);
\draw[dashed] (-.2,1.5) -- (-.2,2.5);
\draw[\XColor,thick] (.2,1.5) -- (.2,2.5);
\filldraw[\XColor] (.2,2.2) circle (.05cm);
\roundNbox{unshaded}{(0,1.5)}{.3}{.1}{.1}{\scriptsize{$b_j$}};
\roundNbox{unshaded}{(0,.5)}{.3}{.1}{.1}{\scriptsize{$b_j^\dag$}};
\roundNbox{unshaded}{(0,-.5)}{.3}{.1}{.1}{\scriptsize{$x$}};
\roundNbox{unshaded}{(.3,-1.5)}{.3}{.15}{.15}{\scriptsize{$q_k^{(n)}$}};
\roundNbox{unshaded}{(.3,-2.5)}{.3}{.15}{.15}{\scriptsize{$(q_k^{(\!n\!)})^\dag$}};
\roundNbox{unshaded}{(.3,-3.5)}{.3}{.15}{.15}{\scriptsize{$q_n$}};
}
\approx
\tikzmath{
\begin{scope}
\clip[rounded corners = 5] (-.45,-1.2) rectangle (1.05,1.6);
\filldraw[\rColor] (.3,-1.2) -- (.3,-.5) -- (0,-.5) -- (0,.5) -- (-.2,.5) -- (-.2,1.6) -- (1.05,1.6) -- (1.05,-1.2);
\end{scope}
\draw[dashed] (-.2,.5) -- (-.2,1.6);
\draw[dashed] (0,-.5) -- (0,.5);
\draw[dashed] (.3,-1.2) -- (.3,-.5);
\draw[\XColor,thick] (.2,.5) -- (.2,1.6);
\draw[\QsColor,thick] (.6,-.5) -- (.6,.8) arc (0:90:.4cm);
\filldraw[\XColor] (.2,1.2) circle (.05cm);
\roundNbox{unshaded}{(0,.5)}{.3}{.1}{.1}{\scriptsize{$x$}};
\roundNbox{unshaded}{(.3,-.5)}{.3}{.15}{.15}{\scriptsize{$q_n$}};
}
=
x q_n
\]
The second equality uses the hypothesis that $X$ is local.
The third $\approx$ uses \eqref{eq:OverBraiding}.
The fourth $\approx$ holds because $\langle q_k^{(n)}| q_n\rangle^Q_M$ is a central sequence in $M$, and $X$ is centrally trivial.
The fifth $\approx$ holds because $\langle b_j|x\rangle^X_M\in M$ and $\|[a,q_k^{(n)}]\|_2 \to 0$ for $a\in M$.
The sixth $\approx$ holds from the definition of (approximate) basis for $X_M$ and $Q_M$.
We conclude that $|X|$ is centrally trivial over $|Q|$.
\end{proof}

The following proposition is straightforward; we omit its proof.

\begin{prop}
\label{prop:(A)PPBfromQtoM}
\mbox{}
\begin{enumerate}
\item 
If $\{b_i\}_i$ is an $|X|_{|Q|}$-basis and $\{q_j\}_j$ is a $Q_M$-basis,
then 
$$\left\{\tikzmath
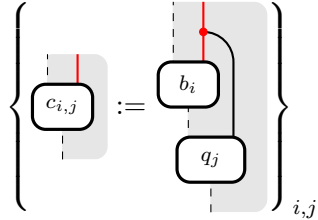
{
\begin{scope}
\clip[rounded corners = 5] (-.3,-.7) rectangle (.6,.7);
\filldraw[\rColor] (0,-.7) -- (0,0) -- (-.2,0) -- (-.2,.7) -- (.7,.7) -- (.7,-.7);
\end{scope}
\draw[dashed] (0,-.7) -- (0,0);
\draw[dashed] (-.2,0) -- (-.2,.7);
\draw[\XColor,thick] (.2,0) -- (.2,.7);
\roundNbox{unshaded}{(0,0)}{.3}{.1}{.1}{\scriptsize{$c_{i,j}$}};
}
:=
\tikzmath{
\begin{scope}
\clip[rounded corners = 5] (-.45,-1.2) rectangle (1.05,1.6);
\filldraw[\rColor] (.3,-1.2) -- (.3,-.5) -- (0,-.5) -- (0,.5) -- (-.2,.5) -- (-.2,1.6) -- (1.05,1.6) -- (1.05,-1.2);
\end{scope}
\draw[dashed] (-.2,.5) -- (-.2,1.6);
\draw[dashed] (0,-.5) -- (0,.5);
\draw[dashed] (.3,-1.2) -- (.3,-.5);
\draw[\XColor,thick] (.2,.5) -- (.2,1.6);
\draw[\QsColor,thick] (.6,-.5) -- (.6,.8) arc (0:90:.4cm);
\filldraw[\XColor] (.2,1.2) circle (.05cm);
\roundNbox{unshaded}{(0,.5)}{.3}{.1}{.1}{\scriptsize{$b_i$}};
\roundNbox{unshaded}{(.3,-.5)}{.3}{.15}{.15}{\scriptsize{$q_j$}};
}\right\}_{i,j}$$
is an $|X|_M$-basis.
\item
If $\{b_i^{(n)}\}$ is an approximate $X_{|Q|}$-basis and $\{q_j^{(n)}\}$ is an approximately inner $Q_M$-basis,
then 
$$
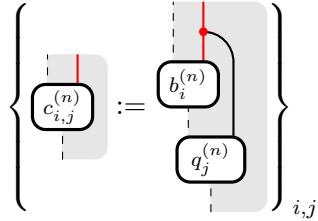
\left\{\tikzmath{
\begin{scope}
\clip[rounded corners = 5] (-.3,-.7) rectangle (.6,.7);
\filldraw[\rColor] (0,-.7) -- (0,0) -- (-.2,0) -- (-.2,.7) -- (.7,.7) -- (.7,-.7);
\end{scope}
\draw[dashed] (0,-.7) -- (0,0);
\draw[dashed] (-.2,0) -- (-.2,.7);
\draw[\XColor,thick] (.2,0) -- (.2,.7);
\roundNbox{unshaded}{(0,0)}{.3}{.1}{.1}{\scriptsize{$c_{i,j}^{(n)}$}};
}
:=
\tikzmath{
\begin{scope}
\clip[rounded corners = 5] (-.45,-1.2) rectangle (1.05,1.6);
\filldraw[\rColor] (.3,-1.2) -- (.3,-.5) -- (0,-.5) -- (0,.5) -- (-.2,.5) -- (-.2,1.6) -- (1.05,1.6) -- (1.05,-1.2);
\end{scope}
\draw[dashed] (-.2,.5) -- (-.2,1.6);
\draw[dashed] (0,-.5) -- (0,.5);
\draw[dashed] (.3,-1.2) -- (.3,-.5);
\draw[\XColor,thick] (.2,.5) -- (.2,1.6);
\draw[\QsColor,thick] (.6,-.5) -- (.6,.8) arc (0:90:.4cm);
\filldraw[\XColor] (.2,1.2) circle (.05cm);
\roundNbox{unshaded}{(0,.5)}{.3}{.1}{.1}{\scriptsize{$b_i^{(n)}$}};
\roundNbox{unshaded}{(.3,-.5)}{.3}{.15}{.15}{\scriptsize{$q_j^{(n)}$}};
}\right\}_{i,j}
$$ is an approximate $|X|_M$-basis.
Moreover, if $\{b_i^{(n)}\}_i$ is approximately inner, so is $\{c_{i,j}^{(n)}\}$.
\end{enumerate}
Since $X_M$ is canonically isomorphic to $|X|_M$, we may view (1) as an $X_M$-basis and (2) as an approximate(ly inner) $X_M$-basis.
\end{prop}

\begin{prop}
\label{Chi(Q)subsetChi(M)locQ}
Every bimodule in $\Chi(|Q|)$ is unitarily isomorphic to a realization of a bimodule in $\Chi(M)_Q^\loc$.
\end{prop}
\begin{proof}
By Remark \ref{rem:QBimodules}, it suffices to consider a $Q-Q$ bimodule  ${}_QX_Q$ in $\fgpBim(M)$ such that $|X| \in \Chi(|Q|)$.
In order to show $X\in\Chi(M)_Q^\loc$, must prove $X$ is centrally trivial and approximately inner over $M$, and $X$ is a local $Q-Q$ bimodule.

By Lemma \ref{lem:CSinMCSin|Q|}, any central sequence $(a_n)\subseteq M$ is also a central sequence in $|Q|$. 
By Proposition \ref{prop:BimodualCT&CS}, $X$ 
is centrally trivial over $|Q|$, so $X$ is centrally trivial over $M$.
By Proposition \ref{prop:(A)PPBfromQtoM} and Proposition \ref{prop:AI&APPB},
we have $X$ 
is approximately inner over $M$.
Therefore,as an $M-M$ bimodule, $X\in \Chi(M)$.

Since $X$ is already a $Q-Q$ bimodule, it remains to show $X$ is local.
Let $\{b_i\},\{q_j\}$ be $X_{|Q|},Q_M$-bases
and let
$\{b_i^{(n)}\},\{q_j^{(n)}\}$ be approximately inner $X_{|Q|},Q_M$-bases respectively.
Defining 
$\{c_{i,j}\}$ and $\{c_{i,j}^{(n)}\}$ as in
Proposition \ref{prop:(A)PPBfromQtoM}
gives an $X_M$-basis and an approximately inner $X_M$-basis respectively.

For the over braiding,
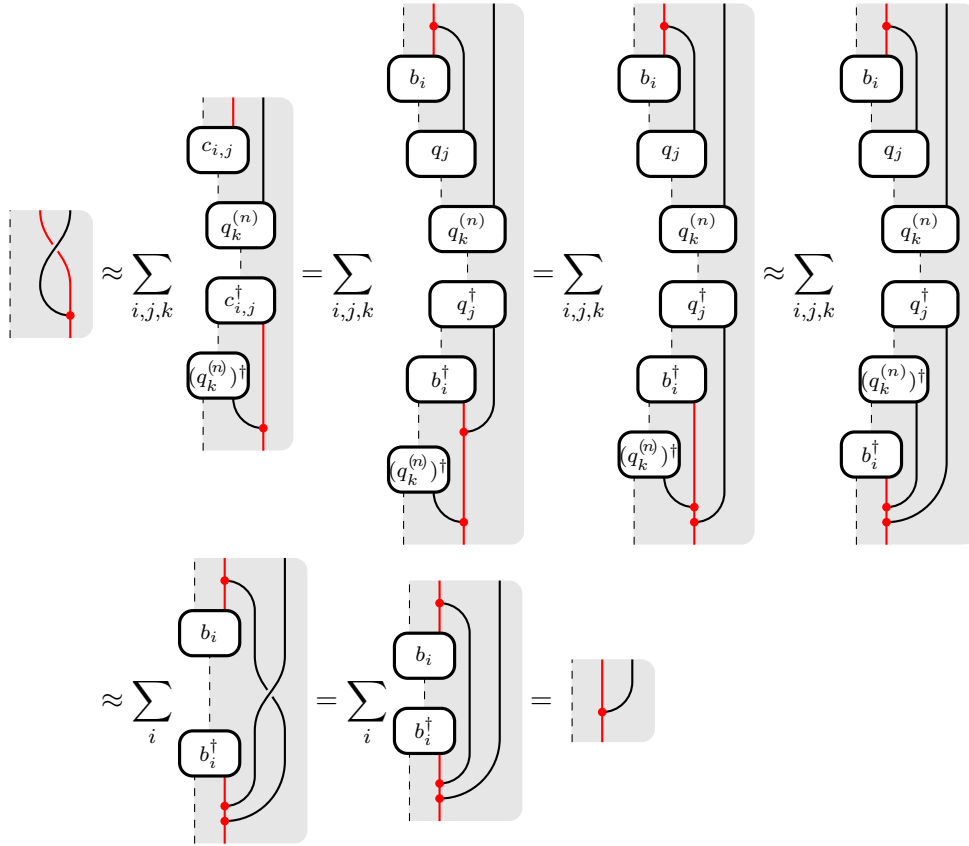
\begin{align*}
\tikzmath{
\begin{scope}
\clip[rounded corners = 5] (-.55,-1.2) rectangle (.7,.5);
\filldraw[\rColor] (-.4,-1.2) rectangle (.7,.5);
\end{scope}
\draw[dashed] (-.4,-1.2) -- (-.4,.5);
\draw[\XColor,thick] (.4,-1.2) -- (.4,-.5) .. controls ++(90:.45cm) and ++(270:.45cm) .. (0,.5);
\filldraw[\rColor] (.2,0) circle (.05cm);
\draw[\QsColor,thick] (.4,-.9) arc (270:180:.4cm) -- (0,-.5) .. controls ++(90:.45cm) and ++(270:.45cm) .. (.4,.5);
\filldraw[\XColor] (.4,-.9) circle (.05cm);
}
&\approx
\sum_{i,j,k}
\tikzmath{
\begin{scope}
\clip[rounded corners = 5] (-.35,-2.5) rectangle (1,2.2);
\filldraw[\rColor] (-.2,-2.5) -- (-.2,-1.5) -- (0,-1.5) -- (0,-.5) -- (.3,-.5) -- (.3,.5) -- (0,.5) -- (0,1.5) -- (-.2,1.5) -- (-.2,2.2) -- (1.05,2.2) -- (1.05,-2.5);
\end{scope}
\draw[dashed] (-.2,-2.5) -- (-.2,-1.5);
\draw[\QsColor,thick] (.6,-2.2) arc (270:180:.4cm);
\draw[dashed] (0,-1.5) -- (0,-.5);
\draw[\XColor,thick] (.6,-2.5) -- (.6,-.5);
\draw[dashed] (.3,-.5) -- (.3,.5);
\draw[dashed] (0,.5) -- (0,1.5);
\draw[\QsColor,thick] (.6,.5) -- (.6,2.2);
\draw[dashed] (-.2,1.5) -- (-.2,2.2);
\draw[\XColor,thick] (.2,1.5) -- (.2,2.2);
\filldraw[\XColor] (.6,-2.2) circle (.05cm);
\roundNbox{unshaded}{(0,1.5)}{.3}{.1}{.1}{\scriptsize{$c_{i,j}$}};
\roundNbox{unshaded}{(.3,.5)}{.3}{.15}{.15}{\scriptsize{$q_k^{(n)}$}};
\roundNbox{unshaded}{(.3,-.5)}{.3}{.15}{.15}{\scriptsize{$c_{i,j}^\dag$}};
\roundNbox{unshaded}{(0,-1.5)}{.3}{.1}{.1}{\scriptsize{$(q_k^{(\!n\!)})^\dag$}};
}
=
\sum_{i,j,k}
\tikzmath{
\begin{scope}
\clip[rounded corners = 5] (-.75,-3.7) rectangle (1,3.5);
\filldraw[\rColor] (-.6,-3.7) -- (-.6,-2.5) -- (-.4,-2.5) -- (-.4,-1.5) -- (-.1,-1.5) -- (-.1,-.5) -- (.25,-.5) -- (.25,.5) -- (-.1,.5) -- (-.1,1.5) -- (-.4,1.5) -- (-.4,2.5) -- (-.6,2.5) -- (-.6,3.5) -- (1,3.5) -- (1,-3.7);
\end{scope}
\draw[dashed] (-.6,-3.7) -- (-.6,-2.5);
\draw[dashed] (-.4,-2.5) -- (-.4,-1.5);
\draw[\QsColor,thick] (.2,-3.4) arc (270:180:.4cm);
\draw[dashed] (-.1,-1.5) -- (-.1,-.5);
\draw[\QsColor,thick] (.6,-.5) -- (.6,-1.8) arc (0:-90:.4cm);
\draw[\XColor,thick] (.2,-3.7) -- (.2,-1.5);
\draw[dashed] (.25,-.5) -- (.25,.5);
\draw[dashed] (-.1,.5) -- (-.1,1.5);
\draw[\QsColor,thick] (.6,.5) -- (.6,3.5);
\draw[dashed] (-.4,1.5) -- (-.4,2.5);
\draw[\QsColor,thick] (.2,1.5) -- (.2,2.8) arc (0:90:.4cm);
\draw[dashed] (-.6,2.5) -- (-.6,3.5);
\draw[\XColor,thick] (-.2,2.5) -- (-.2,3.5);
\filldraw[\XColor] (.2,-3.4) circle (.05cm);
\filldraw[\XColor] (.2,-2.2) circle (.05cm);
\filldraw[\XColor] (-.2,3.2) circle (.05cm);
\roundNbox{unshaded}{(-.4,2.5)}{.3}{.1}{.1}{\scriptsize{$b_i$}};
\roundNbox{unshaded}{(-.1,1.5)}{.3}{.15}{.15}{\scriptsize{$q_j$}};
\roundNbox{unshaded}{(.25,.5)}{.3}{.2}{.2}{\scriptsize{$q_k^{(n)}$}};
\roundNbox{unshaded}{(.25,-.5)}{.3}{.2}{.2}{\scriptsize{$q_j^\dag$}};
\roundNbox{unshaded}{(-.1,-1.5)}{.3}{.15}{.15}{\scriptsize{$b_i^\dag$}};
\roundNbox{unshaded}{(-.4,-2.7)}{.3}{.1}{.1}{\scriptsize{$(q_k^{(\!n\!)})^\dag$}};
}
=
\sum_{i,j,k}
\tikzmath{
\begin{scope}
\clip[rounded corners = 5] (-.75,-3.7) rectangle (1,3.5);
\filldraw[\rColor] (-.6,-3.7) -- (-.6,-2.5) -- (-.4,-2.5) -- (-.4,-1.5) -- (-.1,-1.5) -- (-.1,-.5) -- (.25,-.5) -- (.25,.5) -- (-.1,.5) -- (-.1,1.5) -- (-.4,1.5) -- (-.4,2.5) -- (-.6,2.5) -- (-.6,3.5) -- (1,3.5) -- (1,-3.7);
\end{scope}
\draw[dashed] (-.6,-3.7) -- (-.6,-2.5);
\draw[dashed] (-.4,-2.5) -- (-.4,-1.5);
\draw[\QsColor,thick] (.2,-3.2) arc (270:180:.4cm);
\draw[dashed] (-.1,-1.5) -- (-.1,-.5);
\draw[\QsColor,thick] (.6,-.5) -- (.6,-3) arc (0:-90:.4cm);
\draw[\XColor,thick] (.2,-3.7) -- (.2,-1.5);
\draw[dashed] (.25,-.5) -- (.25,.5);
\draw[dashed] (-.1,.5) -- (-.1,1.5);
\draw[\QsColor,thick] (.6,.5) -- (.6,3.5);
\draw[dashed] (-.4,1.5) -- (-.4,2.5);
\draw[\QsColor,thick] (.2,1.5) -- (.2,2.8) arc (0:90:.4cm);
\draw[dashed] (-.6,2.5) -- (-.6,3.5);
\draw[\XColor,thick] (-.2,2.5) -- (-.2,3.5);
\filldraw[\XColor] (.2,-3.4) circle (.05cm);
\filldraw[\XColor] (.2,-3.2) circle (.05cm);
\filldraw[\XColor] (-.2,3.2) circle (.05cm);
\roundNbox{unshaded}{(-.4,2.5)}{.3}{.1}{.1}{\scriptsize{$b_i$}};
\roundNbox{unshaded}{(-.1,1.5)}{.3}{.15}{.15}{\scriptsize{$q_j$}};
\roundNbox{unshaded}{(.25,.5)}{.3}{.2}{.2}{\scriptsize{$q_k^{(n)}$}};
\roundNbox{unshaded}{(.25,-.5)}{.3}{.2}{.2}{\scriptsize{$q_j^\dag$}};
\roundNbox{unshaded}{(-.1,-1.5)}{.3}{.15}{.15}{\scriptsize{$b_i^\dag$}};
\roundNbox{unshaded}{(-.4,-2.5)}{.3}{.1}{.1}{\scriptsize{$(q_k^{(\!n\!)})^\dag$}};
}
\approx
\sum_{i,j,k}
\tikzmath{
\begin{scope}
\clip[rounded corners = 5] (-.75,-3.7) rectangle (1,3.5);
\filldraw[\rColor] (-.6,-3.7) -- (-.6,-2.5) -- (-.4,-2.5) -- (-.4,-1.5) -- (-.1,-1.5) -- (-.1,-.5) -- (.25,-.5) -- (.25,.5) -- (-.1,.5) -- (-.1,1.5) -- (-.4,1.5) -- (-.4,2.5) -- (-.6,2.5) -- (-.6,3.5) -- (1,3.5) -- (1,-3.7);
\end{scope}
\draw[dashed] (-.6,-3.7) -- (-.6,-2.5);
\draw[dashed] (-.4,-2.5) -- (-.4,-1.5);
\draw[\QsColor,thick] (-.2,-3.2) arc (-90:0:.4cm) -- (.2,-1.5);
\draw[dashed] (-.1,-1.5) -- (-.1,-.5);
\draw[\QsColor,thick] (.6,-.5) -- (.6,-2.6) arc (0:-90:.8cm);
\draw[\XColor,thick] (-.2,-3.7) -- (-.2,-2.5);
\draw[dashed] (.25,-.5) -- (.25,.5);
\draw[dashed] (-.1,.5) -- (-.1,1.5);
\draw[\QsColor,thick] (.6,.5) -- (.6,3.5);
\draw[dashed] (-.4,1.5) -- (-.4,2.5);
\draw[\QsColor,thick] (.2,1.5) -- (.2,2.8) arc (0:90:.4cm);
\draw[dashed] (-.6,2.5) -- (-.6,3.5);
\draw[\XColor,thick] (-.2,2.5) -- (-.2,3.5);
\filldraw[\XColor] (-.2,-3.4) circle (.05cm);
\filldraw[\XColor] (-.2,-3.2) circle (.05cm);
\filldraw[\XColor] (-.2,3.2) circle (.05cm);
\roundNbox{unshaded}{(-.4,2.5)}{.3}{.1}{.1}{\scriptsize{$b_i$}};
\roundNbox{unshaded}{(-.1,1.5)}{.3}{.15}{.15}{\scriptsize{$q_j$}};
\roundNbox{unshaded}{(.25,.5)}{.3}{.2}{.2}{\scriptsize{$q_k^{(n)}$}};
\roundNbox{unshaded}{(.25,-.5)}{.3}{.2}{.2}{\scriptsize{$q_j^\dag$}};
\roundNbox{unshaded}{(-.1,-1.5)}{.3}{.15}{.15}{\scriptsize{$(q_k^{(n)})^\dag$}};
\roundNbox{unshaded}{(-.4,-2.5)}{.3}{.1}{.1}{\scriptsize{$b_i^\dag$}};
}
\displaybreak[1]\\
&\approx
\sum_i
\tikzmath{
\begin{scope}
\clip[rounded corners = 5] (-.95,-2) rectangle (.7,1.8);
\filldraw[\rColor] (-.8,-2) -- (-.8,-.8) -- (-.6,-.8) -- (-.6,.8) -- (-.8,.8) -- (-.8,1.8) -- (1.3,1.8) -- (1.3,-2);
\end{scope}
\draw[dashed] (-.8,-2) -- (-.8,-.8);
\draw[\XColor,thick] (-.4,-2) -- (-.4,-.8);
\draw[dashed] (-.6,-.8) -- (-.6,.8);
\draw[\QsColor,thick] (-.4,-1.7) arc (-90:0:.8cm) -- (.4,-.5) .. controls ++(90:.45cm) and ++(270:.45cm) .. (0,.5) -- (0,1.1) arc (0:90:.4cm);
\filldraw[\rColor] (.2,0) circle (.05cm);
\draw[\QsColor,thick] (-.4,-1.5) arc (-90:0:.4cm) -- (0,-.5) .. controls ++(90:.45cm) and ++(270:.45cm) .. (.4,.5) -- (.4,1.8);
\draw[dashed] (-.8,.8) -- (-.8,1.8);
\draw[\XColor,thick] (-.4,.8) -- (-.4,1.8);
\filldraw[\XColor] (-.4,-1.7) circle (.05cm);
\filldraw[\XColor] (-.4,-1.5) circle (.05cm);
\filldraw[\XColor] (-.4,1.5) circle (.05cm);
\roundNbox{unshaded}{(-.6,-.8)}{.3}{.1}{.1}{\scriptsize{$b_i^\dag$}};
\roundNbox{unshaded}{(-.6,.8)}{.3}{.1}{.1}{\scriptsize{$b_i$}};
}
=
\sum_i
\tikzmath{
\begin{scope}
\clip[rounded corners = 5] (-.35,-1.7) rectangle (1.3,1.5);
\filldraw[\rColor] (-.2,-1.7) -- (-.2,-.5) -- (0,-.5) -- (0,.5) -- (-.2,.5) -- (-.2,1.5) -- (1.3,1.5) -- (1.3,-1.7);
\end{scope}
\draw[dashed] (-.2,-1.7) -- (-.2,-.5);
\draw[dashed] (0,-.5) -- (0,.5);
\draw[dashed] (-.2,.5) -- (-.2,1.5);
\draw[\QsColor,thick] (.2,1.2) arc (90:0:.4cm) -- (.6,-.8) arc (0:-90:.4cm);
\draw[\QsColor,thick] (.2,-1.4) arc (-90:0:.8cm) -- (1,1.5);
\draw[\XColor,thick] (.2,-1.7) -- (.2,-.5);
\draw[\XColor,thick] (.2,.5) -- (.2,1.5);
\filldraw[\XColor] (.2,-1.4) circle (.05cm);
\filldraw[\XColor] (.2,-1.2) circle (.05cm);
\filldraw[\XColor] (.2,1.2) circle (.05cm);
\roundNbox{unshaded}{(0,.5)}{.3}{.1}{.1}{\scriptsize{$b_i$}};
\roundNbox{unshaded}{(0,-.5)}{.3}{.1}{.1}{\scriptsize{$b_i^\dag$}};
}
=
\tikzmath{
\begin{scope}
\clip[rounded corners = 5] (-.55,-.4) rectangle (.7,.7);
\filldraw[\rColor] (-.4,-.4) rectangle (.7,.7);
\end{scope}
\draw[\XColor,thick] (0,-.4) -- (0,.7);
\draw[\QsColor,thick] (0,0) arc (-90:0:.4cm) -- (.4,.7);
\draw[dashed] (-.4,-.4) -- (-.4,.7);
\filldraw[\XColor] (0,0) circle (.05cm);
}
\end{align*}
The first $\approx$ uses \eqref{eq:OverBraiding}, the second equality is the defintion of $c_{i,j}$, the third equality uses associativty of the $Q-Q$ bimodule actions, the fourth $\approx$ uses that $X$ is centrally trivial over $|Q|$, the fifth $\approx$ uses \eqref{eq:ExpandBraidingsQ}, and the sixth equation
follows from
\begin{equation}
\label{eq:UseCommutativityForRichtAction}
\tikzmath{
\begin{scope}
\clip[rounded corners = 5] (-.7,-1.6) rectangle (.7,.5);
\filldraw[\rColor] (-.7,-1.6) rectangle (.7,.5);
\end{scope}
\draw[\XColor,thick] (-.4,-1.6) -- (-.4,.5);
\draw[\QsColor,thick] (-.4,-1.3) arc (-90:0:.8cm) -- (.4,-.5) .. controls ++(90:.45cm) and ++(270:.45cm) .. (0,.5);
\filldraw[\rColor] (.2,0) circle (.05cm);
\draw[\QsColor,thick] (-.4,-1.1) arc (-90:0:.4cm) -- (0,-.5) .. controls ++(90:.45cm) and ++(270:.45cm) .. (.4,.5);
\filldraw[\XColor] (-.4,-1.3) circle (.05cm);
\filldraw[\XColor] (-.4,-1.1) circle (.05cm);
}
=
\tikzmath{
\begin{scope}
\clip[rounded corners = 5] (-.7,-1.6) rectangle (.7,.5);
\filldraw[\rColor] (-.7,-1.6) rectangle (.7,.5);
\end{scope}
\draw[\XColor,thick] (-.4,-1.6) -- (-.4,.5);
\draw[\QsColor,thick] (0,-.5) arc (-180:0:.2cm) -- (.4,-.5) .. controls ++(90:.45cm) and ++(270:.45cm) .. (0,.5);
\filldraw[\rColor] (.2,0) circle (.05cm);
\draw[\QsColor,thick] (0,-.5) .. controls ++(90:.45cm) and ++(270:.45cm) .. (.4,.5);
\draw[\QsColor,thick] (.2,-.7) arc (0:-90:.6cm);
\filldraw[\QsColor] (.2,-.7) circle (.05cm);
\filldraw[\XColor] (-.4,-1.3) circle (.05cm);
}
=
\tikzmath{
\begin{scope}
\clip[rounded corners = 5] (-.7,-.9) rectangle (.7,.5);
\filldraw[\rColor] (-.7,-1.6) rectangle (.7,.5);
\end{scope}
\draw[\XColor,thick] (-.4,-.9) -- (-.4,.5);
\draw[\QsColor,thick] (0,.5) -- (0,.2) arc (-180:0:.2cm) -- (.4,.5);
\draw[\QsColor,thick] (.2,0) arc (0:-90:.6cm);
\filldraw[\QsColor] (.2,0) circle (.05cm);
\filldraw[\XColor] (-.4,-.6) circle (.05cm);
}
=
\tikzmath{
\begin{scope}
\clip[rounded corners = 5] (-.7,-.9) rectangle (.7,.5);
\filldraw[\rColor] (-.7,-1.6) rectangle (.7,.5);
\end{scope}
\draw[\XColor,thick] (-.4,-.9) -- (-.4,.5);
\draw[\QsColor,thick] (0,.5) -- (0,0) arc (0:-90:.4cm);
\draw[\QsColor,thick] (.4,.5) -- (.4,.2) arc (0:-90:.8cm);
\filldraw[\XColor] (-.4,-.4) circle (.05cm);
\filldraw[\XColor] (-.4,-.6) circle (.05cm);
}
=
\tikzmath{
\begin{scope}
\clip[rounded corners = 5] (-.7,-1.6) rectangle (.7,.5);
\filldraw[\rColor] (-.7,-1.6) rectangle (.7,.5);
\end{scope}
\draw[\XColor,thick] (-.4,-1.6) -- (-.4,.5);
\draw[\QsColor,thick] (-.4,-1.1) arc (-90:0:.4cm) -- (0,-.5) .. controls ++(90:.45cm) and ++(270:.45cm) .. (.4,.5);
\filldraw[\rColor] (.2,0) circle (.05cm);
\draw[\QsColor,thick] (-.4,-1.3) arc (-90:0:.8cm) -- (.4,-.5) .. controls ++(90:.45cm) and ++(270:.45cm) .. (0,.5);
\filldraw[\XColor] (-.4,-1.3) circle (.05cm);
\filldraw[\XColor] (-.4,-1.1) circle (.05cm);
}\,.
\end{equation}
For the under braiding,
\begin{align*}
\tikzmath{
\begin{scope}
\clip[rounded corners = 5] (-.55,-1.2) rectangle (.7,.5);
\filldraw[\rColor] (-.4,-1.2) rectangle (.7,.5);
\end{scope}
\draw[dashed] (-.4,-1.2) -- (-.4,.5);
\draw[\QsColor,thick] (.4,-.9) arc (270:180:.4cm) -- (0,-.5) .. controls ++(90:.45cm) and ++(270:.45cm) .. (.4,.5);
\filldraw[\rColor] (.2,0) circle (.05cm);
\draw[\XColor,thick] (.4,-1.2) -- (.4,-.5) .. controls ++(90:.45cm) and ++(270:.45cm) .. (0,.5);
\filldraw[\XColor] (.4,-.9) circle (.05cm);
}
&\approx
\sum_{i,j,k}
\tikzmath{
\begin{scope}
\clip[rounded corners = 5] (-.35,-2.5) rectangle (1,2.2);
\filldraw[\rColor] (-.2,-2.5) -- (-.2,-1.5) -- (0,-1.5) -- (0,-.5) -- (.3,-.5) -- (.3,.5) -- (0,.5) -- (0,1.5) -- (-.2,1.5) -- (-.2,2.2) -- (1.05,2.2) -- (1.05,-2.5);
\end{scope}
\draw[dashed] (-.2,-2.5) -- (-.2,-1.5);
\draw[\QsColor,thick] (.6,-2.2) arc (270:180:.4cm);
\draw[dashed] (0,-1.5) -- (0,-.5);
\draw[\XColor,thick] (.6,-2.5) -- (.6,-.5);
\draw[dashed] (.3,-.5) -- (.3,.5);
\draw[dashed] (0,.5) -- (0,1.5);
\draw[\QsColor,thick] (.6,.5) -- (.6,2.2);
\draw[dashed] (-.2,1.5) -- (-.2,2.2);
\draw[\XColor,thick] (.2,1.5) -- (.2,2.2);
\filldraw[\XColor] (.6,-2.2) circle (.05cm);
\roundNbox{unshaded}{(0,1.5)}{.3}{.1}{.1}{\scriptsize{$c_{i,j}^{(n)}$}};
\roundNbox{unshaded}{(.3,.5)}{.3}{.15}{.15}{\scriptsize{$q_k$}};
\roundNbox{unshaded}{(.3,-.5)}{.3}{.15}{.15}{\scriptsize{$(c_{i,j}^{(n)})^\dag$}};
\roundNbox{unshaded}{(0,-1.5)}{.3}{.1}{.1}{\scriptsize{$q_k^\dag$}};
}
=
\sum_{i,j,k}
\tikzmath{
\begin{scope}
\clip[rounded corners = 5] (-.75,-3.7) rectangle (1,3.5);
\filldraw[\rColor] (-.6,-3.7) -- (-.6,-2.5) -- (-.4,-2.5) -- (-.4,-1.5) -- (-.1,-1.5) -- (-.1,-.5) -- (.25,-.5) -- (.25,.5) -- (-.1,.5) -- (-.1,1.5) -- (-.4,1.5) -- (-.4,2.5) -- (-.6,2.5) -- (-.6,3.5) -- (1,3.5) -- (1,-3.7);
\end{scope}
\draw[dashed] (-.6,-3.7) -- (-.6,-2.5);
\draw[dashed] (-.4,-2.5) -- (-.4,-1.5);
\draw[\QsColor,thick] (.2,-3.4) arc (270:180:.4cm);
\draw[dashed] (-.1,-1.5) -- (-.1,-.5);
\draw[\QsColor,thick] (.6,-.5) -- (.6,-1.8) arc (0:-90:.4cm);
\draw[\XColor,thick] (.2,-3.7) -- (.2,-1.5);
\draw[dashed] (.25,-.5) -- (.25,.5);
\draw[dashed] (-.1,.5) -- (-.1,1.5);
\draw[\QsColor,thick] (.6,.5) -- (.6,3.5);
\draw[dashed] (-.4,1.5) -- (-.4,2.5);
\draw[\QsColor,thick] (.2,1.5) -- (.2,2.8) arc (0:90:.4cm);
\draw[dashed] (-.6,2.5) -- (-.6,3.5);
\draw[\XColor,thick] (-.2,2.5) -- (-.2,3.5);
\filldraw[\XColor] (.2,-3.4) circle (.05cm);
\filldraw[\XColor] (.2,-2.2) circle (.05cm);
\filldraw[\XColor] (-.2,3.2) circle (.05cm);
\roundNbox{unshaded}{(-.4,2.5)}{.3}{.1}{.1}{\scriptsize{$b_i^{(n)}$}};
\roundNbox{unshaded}{(-.1,1.5)}{.3}{.15}{.15}{\scriptsize{$q_j^{(n)}$}};
\roundNbox{unshaded}{(.25,.5)}{.3}{.2}{.2}{\scriptsize{$q_k$}};
\roundNbox{unshaded}{(.25,-.5)}{.3}{.2}{.2}{\scriptsize{$(q_j^{(n)})^\dag$}};
\roundNbox{unshaded}{(-.1,-1.5)}{.3}{.15}{.15}{\scriptsize{$(b_i^{(n)})^\dag$}};
\roundNbox{unshaded}{(-.4,-2.7)}{.3}{.1}{.1}{\scriptsize{$q_k^\dag$}};
}
=
\sum_{i,j,k}
\tikzmath{
\begin{scope}
\clip[rounded corners = 5] (-.75,-3.7) rectangle (1,3.5);
\filldraw[\rColor] (-.6,-3.7) -- (-.6,-2.5) -- (-.4,-2.5) -- (-.4,-1.5) -- (-.1,-1.5) -- (-.1,-.5) -- (.25,-.5) -- (.25,.5) -- (-.1,.5) -- (-.1,1.5) -- (-.4,1.5) -- (-.4,2.5) -- (-.6,2.5) -- (-.6,3.5) -- (1,3.5) -- (1,-3.7);
\end{scope}
\draw[dashed] (-.6,-3.7) -- (-.6,-2.5);
\draw[dashed] (-.4,-2.5) -- (-.4,-1.5);
\draw[\QsColor,thick] (.2,-3.2) arc (270:180:.4cm);
\draw[dashed] (-.1,-1.5) -- (-.1,-.5);
\draw[\QsColor,thick] (.6,-.5) -- (.6,-3) arc (0:-90:.4cm);
\draw[\XColor,thick] (.2,-3.7) -- (.2,-1.5);
\draw[dashed] (.25,-.5) -- (.25,.5);
\draw[dashed] (-.1,.5) -- (-.1,1.5);
\draw[\QsColor,thick] (.6,.5) -- (.6,3.5);
\draw[dashed] (-.4,1.5) -- (-.4,2.5);
\draw[\QsColor,thick] (.2,1.5) -- (.2,2.8) arc (0:90:.4cm);
\draw[dashed] (-.6,2.5) -- (-.6,3.5);
\draw[\XColor,thick] (-.2,2.5) -- (-.2,3.5);
\filldraw[\XColor] (.2,-3.4) circle (.05cm);
\filldraw[\XColor] (.2,-3.2) circle (.05cm);
\filldraw[\XColor] (-.2,3.2) circle (.05cm);
\roundNbox{unshaded}{(-.4,2.5)}{.3}{.1}{.1}{\scriptsize{$b_i^{(n)}$}};
\roundNbox{unshaded}{(-.1,1.5)}{.3}{.15}{.15}{\scriptsize{$q_j^{(n)}$}};
\roundNbox{unshaded}{(.25,.5)}{.3}{.2}{.2}{\scriptsize{$q_k$}};
\roundNbox{unshaded}{(.25,-.5)}{.3}{.2}{.2}{\scriptsize{$(q_j^{(n)})^\dag$}};
\roundNbox{unshaded}{(-.1,-1.5)}{.3}{.15}{.15}{\scriptsize{$(b_i^{(n)})^\dag$}};
\roundNbox{unshaded}{(-.4,-2.5)}{.3}{.1}{.1}{\scriptsize{$q_k^\dag$}};
}
\approx
\sum_{i,j,k}
\tikzmath{
\begin{scope}
\clip[rounded corners = 5] (-.75,-3.7) rectangle (1,3.5);
\filldraw[\rColor] (-.6,-3.7) -- (-.6,-2.5) -- (-.4,-2.5) -- (-.4,-1.5) -- (-.1,-1.5) -- (-.1,-.5) -- (.25,-.5) -- (.25,.5) -- (-.1,.5) -- (-.1,1.5) -- (-.4,1.5) -- (-.4,2.5) -- (-.6,2.5) -- (-.6,3.5) -- (1,3.5) -- (1,-3.7);
\end{scope}
\draw[dashed] (-.6,-3.7) -- (-.6,-2.5);
\draw[dashed] (-.4,-2.5) -- (-.4,-1.5);
\draw[\QsColor,thick] (-.2,-3.2) arc (-90:0:.4cm) -- (.2,-1.5);
\draw[dashed] (-.1,-1.5) -- (-.1,-.5);
\draw[\QsColor,thick] (.6,-.5) -- (.6,-2.6) arc (0:-90:.8cm);
\draw[\XColor,thick] (-.2,-3.7) -- (-.2,-2.5);
\draw[dashed] (.25,-.5) -- (.25,.5);
\draw[dashed] (-.1,.5) -- (-.1,1.5);
\draw[\QsColor,thick] (.6,.5) -- (.6,3.5);
\draw[dashed] (-.4,1.5) -- (-.4,2.5);
\draw[\QsColor,thick] (.2,1.5) -- (.2,2.8) arc (0:90:.4cm);
\draw[dashed] (-.6,2.5) -- (-.6,3.5);
\draw[\XColor,thick] (-.2,2.5) -- (-.2,3.5);
\filldraw[\XColor] (-.2,-3.4) circle (.05cm);
\filldraw[\XColor] (-.2,-3.2) circle (.05cm);
\filldraw[\XColor] (-.2,3.2) circle (.05cm);
\roundNbox{unshaded}{(-.4,2.5)}{.3}{.1}{.1}{\scriptsize{$b_i^{(n)}$}};
\roundNbox{unshaded}{(-.1,1.5)}{.3}{.15}{.15}{\scriptsize{$q_j^{(n)}$}};
\roundNbox{unshaded}{(.25,.5)}{.3}{.2}{.2}{\scriptsize{$q_k$}};
\roundNbox{unshaded}{(.25,-.5)}{.3}{.2}{.2}{\scriptsize{$(q_j^{(n)})^\dag$}};
\roundNbox{unshaded}{(-.1,-1.5)}{.3}{.15}{.15}{\scriptsize{$q_k^\dag$}};
\roundNbox{unshaded}{(-.4,-2.5)}{.3}{.1}{.1}{\scriptsize{$(b_i^{(\!n\!)})^\dag$}};
}
\\
&\approx
\sum_i
\tikzmath{
\begin{scope}
\clip[rounded corners = 5] (-.95,-2) rectangle (.7,1.8);
\filldraw[\rColor] (-.8,-2) -- (-.8,-.8) -- (-.6,-.8) -- (-.6,.8) -- (-.8,.8) -- (-.8,1.8) -- (1.3,1.8) -- (1.3,-2);
\end{scope}
\draw[dashed] (-.8,-2) -- (-.8,-.8);
\draw[\XColor,thick] (-.4,-2) -- (-.4,-.8);
\draw[dashed] (-.6,-.8) -- (-.6,.8);
\draw[\QsColor,thick] (-.4,-1.5) arc (-90:0:.4cm) -- (0,-.5) .. controls ++(90:.45cm) and ++(270:.45cm) .. (.4,.5) -- (.4,1.8);
\filldraw[\rColor] (.2,0) circle (.05cm);
\draw[\QsColor,thick] (-.4,-1.7) arc (-90:0:.8cm) -- (.4,-.5) .. controls ++(90:.45cm) and ++(270:.45cm) .. (0,.5) -- (0,1.1) arc (0:90:.4cm);
\draw[dashed] (-.8,.8) -- (-.8,1.8);
\draw[\XColor,thick] (-.4,.8) -- (-.4,1.8);
\filldraw[\XColor] (-.4,-1.7) circle (.05cm);
\filldraw[\XColor] (-.4,-1.5) circle (.05cm);
\filldraw[\XColor] (-.4,1.5) circle (.05cm);
\roundNbox{unshaded}{(-.6,-.8)}{.3}{.1}{.1}{\scriptsize{$(b_i^{(\!n\!)})^\dag$}};
\roundNbox{unshaded}{(-.6,.8)}{.3}{.1}{.1}{\scriptsize{$b_i^{(n)}$}};
}
=
\sum_i
\tikzmath{
\begin{scope}
\clip[rounded corners = 5] (-.35,-1.7) rectangle (1.3,1.5);
\filldraw[\rColor] (-.2,-1.7) -- (-.2,-.5) -- (0,-.5) -- (0,.5) -- (-.2,.5) -- (-.2,1.5) -- (1.3,1.5) -- (1.3,-1.7);
\end{scope}
\draw[dashed] (-.2,-1.7) -- (-.2,-.5);
\draw[dashed] (0,-.5) -- (0,.5);
\draw[dashed] (-.2,.5) -- (-.2,1.5);
\draw[\QsColor,thick] (.2,1.2) arc (90:0:.4cm) -- (.6,-.8) arc (0:-90:.4cm);
\draw[\QsColor,thick] (.2,-1.4) arc (-90:0:.8cm) -- (1,1.5);
\draw[\XColor,thick] (.2,-1.7) -- (.2,-.5);
\draw[\XColor,thick] (.2,.5) -- (.2,1.5);
\filldraw[\XColor] (.2,-1.4) circle (.05cm);
\filldraw[\XColor] (.2,-1.2) circle (.05cm);
\filldraw[\XColor] (.2,1.2) circle (.05cm);
\roundNbox{unshaded}{(0,.5)}{.3}{.1}{.1}{\scriptsize{$b_i^{(n)}$}};
\roundNbox{unshaded}{(0,-.5)}{.3}{.1}{.1}{\scriptsize{$(b_i^{(\!n\!)})^\dag$}};
}
\approx
\tikzmath{
\begin{scope}
\clip[rounded corners = 5] (-.55,-.4) rectangle (.7,.7);
\filldraw[\rColor] (-.4,-.4) rectangle (.7,.7);
\end{scope}
\draw[\XColor,thick] (0,-.4) -- (0,.7);
\draw[\QsColor,thick] (0,0) arc (-90:0:.4cm) -- (.4,.7);
\draw[dashed] (-.4,-.4) -- (-.4,.7);
\filldraw[\XColor] (0,0) circle (.05cm);
}\,.
\end{align*}
The first $\approx$ uses \eqref{eq:UnderBraiding}, the second equality is the definition of $c_{i,j}^{(n)}$, the third equality uses associativity of the $Q-Q$ bimodule actions, the fourth $\approx$ uses $X$ is centrally trivial over $|Q|$ so $\|[b_i^{(n)},q_k]\|_2\to 0$,
the fifth $\approx$ uses \eqref{eq:ExpandBraidingsQ}, and the sixth equality uses \eqref{eq:UseCommutativityForRichtAction} again.
\end{proof}

\begin{thm}
\label{thm:LocalExtension}
Realization gives a braided unitary equivalence
$\Chi(M)_Q^{\loc}\to \Chi(|Q|)$.
\end{thm}
\begin{proof}
By Remark 
\ref{rem:QBimodules},
realization $|\,\cdot\,|$ gives a unitary tensor equivalence from $Q-Q$ bimodules in $\fgpBim(M)$ to $\fgpBim(|Q|)$.
By Proposition \ref{Chi(M)locQsubsetChi(Q)}, for $X\in \Chi(M)_Q^{\loc}$, $|X| \in \Chi(|Q|)$, and by Proposition \ref{Chi(Q)subsetChi(M)locQ}, every bimodule in $\Chi(|Q|)$ arises in this way.
Since 
$\Chi(M)_Q^{\loc}$ is a full subcategory of the $Q-Q$ bimodules in $\fgpBim(M)$,
$\Chi(|Q|)$ is a full subcategory of $\fgpBim(|Q|)$, and realization $|\,\cdot\,|$ is fully faithful, 
it restricts to a unitary tensor equivalence $\Chi(M)_Q^{\loc} \to \Chi(|Q|)$.

It remains to verify that  $|\,\cdot\,|:\Chi(M)_Q^{\loc} \to \Chi(|Q|)$ is braided, i.e., the following diagram commutes.
$$
\begin{tikzcd}
{|X| \boxtimes_{|Q|} |Y|}
\arrow[d,"\mu_{X,Y}"]
\arrow[rrr,"u_{X,Y}^{|Q|}"]
\arrow[dr, dashed]
&&&
{|Y| \boxtimes_{|Q|} |X|}
\arrow[d,"\mu_{Y,X}"]
\\
{|X\otimes_Q Y|}
\arrow[r,hookrightarrow]
&
{|X\boxtimes_M Y|}
\arrow[r,"|u_{X,Y}^M|"]
&
{|Y\boxtimes_M X|}
\arrow[ur, dashed]
\arrow[r, two heads]
&
{|Y\otimes_Q X|}
\end{tikzcd}
$$
The two triangles on either side commute by \eqref{eq:CanonicalProjectorTriangle}, so it remains to prove the inner square commutes.
Graphically denoting the $\rm II_1$ factor $|Q|$ and the canonical projector $|X\boxtimes_M Y| \to |X|\otimes_{|Q|} |Y|$ as in \eqref{eq:CanonicalProjectors},
as realization is fully faithful,
this is the condition that
\[
\tikzmath{
\begin{scope}
\clip[rounded corners = 5] (-1.35,-.9) rectangle (1.1,.9);
\filldraw[\QrColor] (-.8,-.9) rectangle (1.1,.9);
\filldraw[\rColor] (-1.2,-.9) rectangle (-.8,.9);
\filldraw[\rColor] (0,-.5) -- (.4,-.5) arc (-90:0:.4cm) -- (.8,.1) arc (0:90:.4cm) -- (0,.5) arc (90:180:.4cm) -- (-.4,-.1) arc (180:270:.4cm);
\end{scope}
\draw[dashed] (-1.2,-.9) -- (-1.2,.9);
\draw[\QsColor,thick] (-.8,-.9) -- (-.8,.9);
\draw[\QsColor,thick] (0,-.5) -- (.4,-.5) arc (-90:0:.4cm) -- (.8,.1) arc (0:90:.4cm) -- (0,.5) arc (90:180:.4cm) -- (-.4,-.1) arc (180:270:.4cm);
\draw[\YColor,thick] (.4,-.9) -- (.4,-.5) .. controls ++(90:.45cm) and ++(270:.45cm) .. (0,.5) -- (0,.9);
\filldraw[\rColor] (.2,0) circle (.05cm);
\draw[\XColor,thick] (0,-.9) -- (0,-.5) .. controls ++(90:.45cm) and ++(270:.45cm) .. (.4,.5) -- (.4,.9);
}
\overset{?}{=}
\tikzmath{
\begin{scope}
\clip[rounded corners = 5] (-.95,-.5) rectangle (.7,.5);
\filldraw[\QrColor] (-.4,-.9) rectangle (.7,.9);
\filldraw[\rColor] (-.8,-.9) rectangle (-.4,.9);
\end{scope}
\draw[dashed] (-.8,-.5) -- (-.8,.5);
\draw[\QsColor,thick] (-.4,-.5) -- (-.4,.5);
\draw[\YColor,thick] (.4,-.5) .. controls ++(90:.45cm) and ++(270:.45cm) .. (0,.5);
\filldraw[\QrColor] (.2,0) circle (.05cm);
\draw[\XColor,thick] (0,-.5) .. controls ++(90:.45cm) and ++(270:.45cm) .. (.4,.5);
}\,.
\]
Let $\{c_k\}$ be a $Y_{|Q|}$-basis, and let $\{b_i^{(n)}\}$ be an approximately inner $X_{|Q|}$-basis. 
Let $\{q_l\}$ be a $Q_M$-basis and $\{q_j^{(n)}\}$ be an approximately inner $Q_M$-basis.
According to Proposition \ref{prop:(A)PPBfromQtoM}, 
$\{c_kq_l\}$ is a $Y_M$ basis and $\{b_i^{(n)}q_j^{(n)}\}$ is an approximately inner $X_M$ basis.
Then 
\begin{align*}
\tikzmath{
\begin{scope}
\clip[rounded corners = 5] (-.95,-.9) rectangle (1.1,.9);
\filldraw[\QrColor] (-.4,-.9) rectangle (1.1,.9);
\filldraw[\rColor] (-.8,-.9) -- (-.4,-.9) arc (180:90:.4cm) -- (.4,-.5) arc (-90:0:.4cm) -- (.8,.1) arc (0:90:.4cm) -- (0,.5) arc (270:180:.4cm) -- (-.8,.9);
\end{scope}
\draw[dashed] (-.8,-.9) -- (-.8,.9);
\draw[\QsColor,thick] (-.4,-.9) arc (180:90:.4cm) -- (.4,-.5) arc (-90:0:.4cm) -- (.8,.1) arc (0:90:.4cm) -- (0,.5) arc (270:180:.4cm);
\draw[\YColor,thick] (.4,-.9) -- (.4,-.5) .. controls ++(90:.45cm) and ++(270:.45cm) .. (0,.5) -- (0,.9);
\filldraw[\rColor] (.2,0) circle (.05cm);
\draw[\XColor,thick] (0,-.9) -- (0,-.5) .. controls ++(90:.45cm) and ++(270:.45cm) .. (.4,.5) -- (.4,.9);
}
&\approx
\sum_{i,j,k,l}
\tikzmath{
\begin{scope}
\clip[rounded corners = 5] (-1.55,-5.1) rectangle (1.3,5.1);
\filldraw[\rColor] (1.3,-5.1) -- (-1.4,-5.1) -- (-1.4,-3.7) -- (-1,-3.4) .. controls ++(90:.2cm) and ++(270:.2cm) .. (-.6,-3) -- (-.4,-2.7) -- (-.4,-1.5) -- (-.1,-1.5) -- (-.1,-.5) -- (.25,-.5) -- (.25,.5) -- (-.1,.5) -- (-.1,1.5) -- (-.4,1.5) -- (-.4,2.5) -- (-.6,3) .. controls ++(90:.2cm) and ++(270:.2cm) .. (-1,3.4) -- (-1.4,3.7) -- (-1.4,5.1) -- (1.3,5.1);
\filldraw[\QrColor] (1.3,-5.1) -- (-1,-5.1) arc (180:90:.4cm) -- (.2,-4.7) arc (-90:0:.8cm) -- (1,3.9) arc (0:90:.8cm) -- (-.6,4.7) arc (270:180:.4cm) -- (1.3,5.1);
\end{scope}
\draw[dashed] (-1.4,-5.1) -- (-1.4,-3.7);
\draw[dashed] (-1,-3.4) .. controls ++(90:.2cm) and ++(270:.2cm) .. (-.6,-3);
\draw[dashed] (-.4,-2.7) -- (-.4,-1.5);
\draw[dashed] (-.1,-1.5) -- (-.1,-.5);
\draw[dashed] (.25,-.5) -- (.25,.5);
\draw[dashed] (-.1,.5) -- (-.1,1.5);
\draw[dashed] (-.4,1.5) -- (-.4,2.5);
\draw[dashed] (-.6,3) .. controls ++(90:.2cm) and ++(270:.2cm) .. (-1,3.4);
\draw[dashed] (-1.4,3.7) -- (-1.4,5.1);
\draw[\XColor,thick] (-.6,-5.1) -- (-.6,-3.7);
\draw[\YColor,thick] (.2,-5.1) -- (.2,-1.5);
\draw[\XColor,thick] (.2,1.5) -- (.2,5.1);
\draw[\YColor,thick] (-.6,3.7) -- (-.6,5.1);
\draw[\QsColor,thick] (-1,-5.1) arc (180:90:.4cm) -- (.2,-4.7) arc (-90:0:.8cm) -- (1,3.9) arc (0:90:.8cm) -- (-.6,4.7) arc (270:180:.4cm);
\draw[\QsColor,thick] (-.2,2.7) -- (-.2,4) arc (0:90:.4cm);
\draw[\QsColor,thick] (.6,.5) -- (.6,1.8) arc (0:90:.4cm);
\draw[\QsColor,thick] (-.2,-2.7) -- (-.2,-4) arc (0:-90:.4cm);
\draw[\QsColor,thick] (.6,-.5) -- (.6,-1.8) arc (0:-90:.4cm);
\filldraw[\XColor] (-.6,-4.4) circle (.05cm);
\filldraw[\YColor] (.2,-2.2) circle (.05cm);
\filldraw[\XColor] (.2,2.2) circle (.05cm);
\filldraw[\YColor] (-.6,4.4) circle (.05cm);
\roundNbox{unshaded}{(-1,3.7)}{.3}{.25}{.25}{\scriptsize{$c_k$}};
\roundNbox{unshaded}{(-.4,2.7)}{.3}{.1}{.1}{\scriptsize{$q_l$}};
\roundNbox{unshaded}{(-.1,1.5)}{.3}{.15}{.15}{\scriptsize{$b_i^{(n)}$}};
\roundNbox{unshaded}{(.25,.5)}{.3}{.2}{.2}{\scriptsize{$q_j^{(n)}$}};
\roundNbox{unshaded}{(.25,-.5)}{.3}{.2}{.2}{\scriptsize{$q_l^\dag$}};
\roundNbox{unshaded}{(-.1,-1.5)}{.3}{.15}{.15}{\scriptsize{$c_k^\dag$}};
\roundNbox{unshaded}{(-.4,-2.7)}{.3}{.1}{.1}{\scriptsize{$(q_j^{(\!n\!)})^\dag$}};
\roundNbox{unshaded}{(-1,-3.7)}{.3}{.25}{.25}{\scriptsize{$(b_i^{(n)})^\dag$}};
}
=
\sum_{i,j,k,l}
\tikzmath{
\begin{scope}
\clip[rounded corners = 5] (-1.55,-5.1) rectangle (1.3,5.1);
\filldraw[\rColor] (1.3,-5.1) -- (-1.4,-5.1) -- (-1.4,-3.7) -- (-1,-3.4) .. controls ++(90:.2cm) and ++(270:.2cm) .. (-.6,-2.8) -- (-.4,-2.7) -- (-.4,-1.5) -- (-.1,-1.5) -- (-.1,-.5) -- (.25,-.5) -- (.25,.5) -- (-.1,.5) -- (-.1,1.5) -- (-.4,1.5) -- (-.4,2.5) -- (-.6,2.8) .. controls ++(90:.2cm) and ++(270:.2cm) .. (-1,3.4) -- (-1.4,3.7) -- (-1.4,5.1) -- (1.3,5.1);
\filldraw[\QrColor] (1.3,-5.1) -- (-1,-5.1) arc (180:90:.4cm) -- (.2,-4.7) arc (-90:0:.8cm) -- (1,3.9) arc (0:90:.8cm) -- (-.6,4.7) arc (270:180:.4cm) -- (1.3,5.1);
\end{scope}
\draw[dashed] (-1.4,-5.1) -- (-1.4,-3.7);
\draw[dashed] (-1,-3.4) .. controls ++(90:.2cm) and ++(270:.2cm) .. (-.6,-2.8);
\draw[dashed] (-.4,-2.7) -- (-.4,-1.5);
\draw[dashed] (-.1,-1.5) -- (-.1,-.5);
\draw[dashed] (.25,-.5) -- (.25,.5);
\draw[dashed] (-.1,.5) -- (-.1,1.5);
\draw[dashed] (-.4,1.5) -- (-.4,2.5);
\draw[dashed] (-.6,2.8) .. controls ++(90:.2cm) and ++(270:.2cm) .. (-1,3.4);
\draw[dashed] (-1.4,3.7) -- (-1.4,5.1);
\draw[\XColor,thick] (-.6,-5.1) -- (-.6,-3.7);
\draw[\YColor,thick] (.2,-5.1) -- (.2,-1.5);
\draw[\XColor,thick] (.2,1.5) -- (.2,5.1);
\draw[\YColor,thick] (-.6,3.7) -- (-.6,5.1);
\draw[\QsColor,thick] (-1,-5.1) arc (180:90:.4cm) -- (.2,-4.7) arc (-90:0:.8cm) -- (1,3.9) arc (0:90:.8cm) -- (-.6,4.7) arc (270:180:.4cm);
\draw[\QsColor,thick] (-.2,2.8) arc (180:90:.4cm);
\draw[\QsColor,thick] (.6,.5) -- (.6,2.8) arc (180:90:.4cm);
\draw[\QsColor,thick] (-.2,-2.8) arc (180:270:.4cm);
\draw[\QsColor,thick] (.6,-.5) -- (.6,-2.8) arc (180:270:.4cm);
\filldraw[\YColor] (.2,-3.2) circle (.05cm);
\filldraw[\XColor] (.2,3.2) circle (.05cm);
\filldraw[\QsColor] (1,-3.2) circle (.05cm);
\filldraw[\QsColor] (1,3.2) circle (.05cm);
\roundNbox{unshaded}{(-1,3.7)}{.3}{.25}{.25}{\scriptsize{$c_k$}};
\roundNbox{unshaded}{(-.4,2.5)}{.3}{.1}{.1}{\scriptsize{$q_l$}};
\roundNbox{unshaded}{(-.1,1.5)}{.3}{.15}{.15}{\scriptsize{$b_i^{(n)}$}};
\roundNbox{unshaded}{(.25,.5)}{.3}{.2}{.2}{\scriptsize{$q_j^{(n)}$}};
\roundNbox{unshaded}{(.25,-.5)}{.3}{.2}{.2}{\scriptsize{$q_l^\dag$}};
\roundNbox{unshaded}{(-.1,-1.5)}{.3}{.15}{.15}{\scriptsize{$c_k^\dag$}};
\roundNbox{unshaded}{(-.4,-2.5)}{.3}{.1}{.1}{\scriptsize{$(q_j^{(\!n\!)})^\dag$}};
\roundNbox{unshaded}{(-1,-3.7)}{.3}{.25}{.25}{\scriptsize{$(b_i^{(n)})^\dag$}};
}
\approx
\sum_{i,j,k,l}
\tikzmath{
\begin{scope}
\clip[rounded corners = 5] (-1.55,-5.1) rectangle (1.3,5.1);
\filldraw[\rColor] (1.3,-5.1) -- (-1.4,-5.1) -- (-1.4,-3.7) -- (-1,-3.4) .. controls ++(90:.2cm) and ++(270:.2cm) .. (-.6,-2.8) -- (-.4,-2.7) -- (-.4,-1.5) -- (-.1,-1.5) -- (-.1,-.5) -- (.25,-.5) -- (.25,.5) -- (-.1,.5) -- (-.1,1.5) -- (-.4,1.5) -- (-.4,2.5) -- (-.6,2.8) .. controls ++(90:.2cm) and ++(270:.2cm) .. (-1,3.4) -- (-1.4,3.7) -- (-1.4,5.1) -- (1.3,5.1);
\filldraw[\QrColor] (1.3,-5.1) -- (-1,-5.1) arc (180:90:.4cm) -- (-.2,-4.7) arc (-90:0:1.2cm) -- (1,3.5) arc (0:90:1.2cm) -- (-.6,4.7) arc (270:180:.4cm) -- (1.3,5.1);
\end{scope}
\draw[dashed] (-1.4,-5.1) -- (-1.4,-3.7);
\draw[dashed] (-1,-3.4) .. controls ++(90:.2cm) and ++(270:.2cm) .. (-.6,-2.8);
\draw[dashed] (-.4,-2.7) -- (-.4,-1.5);
\draw[dashed] (-.1,-1.5) -- (-.1,-.5);
\draw[dashed] (.25,-.5) -- (.25,.5);
\draw[dashed] (-.1,.5) -- (-.1,1.5);
\draw[dashed] (-.4,1.5) -- (-.4,2.5);
\draw[dashed] (-.6,2.8) .. controls ++(90:.2cm) and ++(270:.2cm) .. (-1,3.4);
\draw[dashed] (-1.4,3.7) -- (-1.4,5.1);
\draw[\XColor,thick] (-.6,-5.1) -- (-.6,-3.7);
\draw[\YColor,thick] (-.2,-5.1) -- (-.2,-2.5);
\draw[\XColor,thick] (-.2,2.5) -- (-.2,5.1);
\draw[\YColor,thick] (-.6,3.7) -- (-.6,5.1);
\draw[\QsColor,thick] (-1,-5.1) arc (180:90:.4cm) -- (-.2,-4.7) arc (-90:0:1.2cm) -- (1,3.5) arc (0:90:1.2cm) -- (-.6,4.7) arc (270:180:.4cm);
\draw[\QsColor,thick] (.2,1.5) -- (.2,2.8) arc (0:90:.4cm);
\draw[\QsColor,thick] (.6,.5) -- (.6,2.8) arc (180:90:.4cm);
\draw[\QsColor,thick] (.2,-1.5) -- (.2,-2.8) arc (0:-90:.4cm);
\draw[\QsColor,thick] (.6,-.5) -- (.6,-2.8) arc (180:270:.4cm);
\filldraw[\YColor] (-.2,-3.2) circle (.05cm);
\filldraw[\XColor] (-.2,3.2) circle (.05cm);
\filldraw[\QsColor] (1,-3.2) circle (.05cm);
\filldraw[\QsColor] (1,3.2) circle (.05cm);
\roundNbox{unshaded}{(-1,3.7)}{.3}{.25}{.25}{\scriptsize{$c_k$}};
\roundNbox{unshaded}{(-.4,2.5)}{.3}{.1}{.1}{\scriptsize{$b_i^{(n)}$}};
\roundNbox{unshaded}{(-.1,1.5)}{.3}{.15}{.15}{\scriptsize{$q_l$}};
\roundNbox{unshaded}{(.25,.5)}{.3}{.2}{.2}{\scriptsize{$q_j^{(n)}$}};
\roundNbox{unshaded}{(.25,-.5)}{.3}{.2}{.2}{\scriptsize{$q_l^\dag$}};
\roundNbox{unshaded}{(-.1,-1.5)}{.3}{.15}{.15}{\scriptsize{$(q_j^{(\!n\!)})^\dag$}};
\roundNbox{unshaded}{(-.4,-2.5)}{.3}{.1}{.1}{\scriptsize{$c_k^\dag$}};
\roundNbox{unshaded}{(-1,-3.7)}{.3}{.25}{.25}{\scriptsize{$(b_i^{(n)})^\dag$}};
}
\\
&\approx
\sum_{i,k}
\tikzmath{
\begin{scope}
\clip[rounded corners = 5] (-1.55,-3) rectangle (1.3,3);
\filldraw[\rColor] (1.3,-3) -- (-1.4,-3) -- (-1.4,-2) -- (-1,-1.7) .. controls ++(90:.2cm) and ++(270:.2cm) .. (-.6,-1.1) -- (-.4,-1) -- (-.4,1) -- (-.6,1.1) .. controls ++(90:.2cm) and ++(270:.2cm) .. (-1,1.7) -- (-1.4,2) -- (-1.4,3) -- (1.3,3);
\filldraw[\QrColor] (1.3,-3) -- (-1,-3) arc (180:90:.4cm) -- (-.2,-2.6) arc (-90:0:1.2cm) -- (1,1.4) arc (0:90:1.2cm) -- (-.6,2.6) arc (270:180:.4cm) -- (1.3,3);
\end{scope}
\draw[dashed] (-1.4,-3) -- (-1.4,-2);
\draw[dashed] (-1,-1.7) .. controls ++(90:.2cm) and ++(270:.2cm) .. (-.6,-1.1);
\draw[dashed] (-.4,-1) -- (-.4,1);
\draw[dashed] (-.6,1.1) .. controls ++(90:.2cm) and ++(270:.2cm) .. (-1,1.7);
\draw[dashed] (-1.4,2) -- (-1.4,3);
\draw[\QsColor,thick] (.6,-.5) .. controls ++(90:.45cm) and ++(270:.45cm) .. (.2,.5);
\filldraw[\rColor] (.4,0) circle (.05cm);
\draw[\QsColor,thick] (.2,-.5) .. controls ++(90:.45cm) and ++(270:.45cm) .. (.6,.5);
\draw[\QsColor,thick] (-1,-3) arc (180:90:.4cm) -- (-.2,-2.6) arc (-90:0:1.2cm) -- (1,1.4) arc (0:90:1.2cm) -- (-.6,2.6) arc (270:180:.4cm);
\draw[\QsColor,thick] (.2,.5) -- (.2,1.1) arc (0:90:.4cm);
\draw[\QsColor,thick] (.6,.5) -- (.6,1.1) arc (180:90:.4cm);
\draw[\QsColor,thick] (.2,-.5) -- (.2,-1.1) arc (0:-90:.4cm);
\draw[\QsColor,thick] (.6,-.5) -- (.6,-1.1) arc (180:270:.4cm);
\draw[\XColor,thick] (-.6,-3) -- (-.6,-2);
\draw[\YColor,thick] (-.2,-3) -- (-.2,-.8);
\draw[\XColor,thick] (-.2,.8) -- (-.2,3);
\draw[\YColor,thick] (-.6,2) -- (-.6,3);
\filldraw[\YColor] (-.2,-1.5) circle (.05cm);
\filldraw[\XColor] (-.2,1.5) circle (.05cm);
\filldraw[\QsColor] (1,-1.5) circle (.05cm);
\filldraw[\QsColor] (1,1.5) circle (.05cm);
\roundNbox{unshaded}{(-1,2)}{.3}{.25}{.25}{\scriptsize{$c_k$}};
\roundNbox{unshaded}{(-.4,.8)}{.3}{.1}{.1}{\scriptsize{$b_i^{(n)}$}};
\roundNbox{unshaded}{(-.4,-.8)}{.3}{.1}{.1}{\scriptsize{$c_k^\dag$}};
\roundNbox{unshaded}{(-1,-2)}{.3}{.25}{.25}{\scriptsize{$(b_i^{(n)})^\dag$}};
}
=
\sum_{i,k}
\tikzmath{
\begin{scope}
\clip[rounded corners = 5] (-1.55,-2.5) rectangle (.5,2.5);
\filldraw[\rColor] (.5,-2.5) -- (-1.4,-2.5) -- (-1.4,-1.5) -- (-1,-1.2) .. controls ++(90:.2cm) and ++(270:.2cm) .. (-.6,-.8) -- (-.4,-.5) -- (-.4,.5) -- (-.6,.8) .. controls ++(90:.2cm) and ++(270:.2cm) .. (-1,1.2) -- (-1.4,1.5) -- (-1.4,2.5) -- (.5,2.5);
\filldraw[\QrColor] (.5,-2.5) -- (-1,-2.5) arc (180:90:.4cm) -- (-.2,-2.1) arc (-90:0:.4cm) -- (.2,1.7) arc (0:90:.4cm) -- (-.6,2.1) arc (270:180:.4cm) -- (.5,2.5);
\end{scope}
\draw[dashed] (-1.4,-2.5) -- (-1.4,-1.5);
\draw[dashed] (-1,-1.2) .. controls ++(90:.2cm) and ++(270:.2cm) .. (-.6,-.8);
\draw[dashed] (-.4,-.5) -- (-.4,.5);
\draw[dashed] (-.6,.8) .. controls ++(90:.2cm) and ++(270:.2cm) .. (-1,1.2);
\draw[dashed] (-1.4,1.5) -- (-1.4,2.5);
\draw[\QsColor,thick] (-1,-2.5) arc (180:90:.4cm) -- (-.2,-2.1) arc (-90:0:.4cm) -- (.2,1.7) arc (0:90:.4cm) -- (-.6,2.1) arc (270:180:.4cm);
\draw[\XColor,thick] (-.6,-2.5) -- (-.6,-1.5);
\draw[\YColor,thick] (-.2,-2.5) -- (-.2,-.5);
\draw[\XColor,thick] (-.2,.5) -- (-.2,2.5);
\draw[\YColor,thick] (-.6,1.5) -- (-.6,2.5);
\roundNbox{unshaded}{(-1,1.5)}{.3}{.25}{.25}{\scriptsize{$c_k$}};
\roundNbox{unshaded}{(-.4,.5)}{.3}{.1}{.1}{\scriptsize{$b_i^{(n)}$}};
\roundNbox{unshaded}{(-.4,-.5)}{.3}{.1}{.1}{\scriptsize{$c_k^\dag$}};
\roundNbox{unshaded}{(-1,-1.5)}{.3}{.25}{.25}{\scriptsize{$(b_i^{(n)})^\dag$}};
}
=
\sum_{i,k}
\tikzmath{
\begin{scope}
\clip[rounded corners = 5] (-1.55,-3) rectangle (.9,3);
\filldraw[\rColor] (1.3,-3) -- (-1.4,-3) -- (-1.4,-1.9) -- (-1,-1.6) .. controls ++(90:.2cm) and ++(270:.2cm) .. (-.2,-.8) -- (0,-.5) -- (0,.5) -- (-.2,.8) .. controls ++(90:.2cm) and ++(270:.2cm) .. (-1,1.6) -- (-1.4,1.9) -- (-1.4,3) -- (1.3,3);
\filldraw[\QrColor] (1.3,-3) -- (-1,-3) arc (180:90:.4cm) arc (-90:0:.4cm) -- (-.2,-1.6) arc (180:90:.4cm) arc (-90:0:.4cm) -- (.6,.8) arc (0:90:.4cm) arc (270:180:.4cm) -- (-.2,2.2) arc (0:90:.4cm) arc (270:180:.4cm) -- (1.3,3);
\end{scope}
\draw[dashed] (-1.4,-3) -- (-1.4,-1.9);
\draw[dashed] (-1,-1.6) .. controls ++(90:.2cm) and ++(270:.2cm) .. (-.2,-.8);
\draw[dashed] (0,-.5) -- (0,.5);
\draw[dashed] (-.2,.8) .. controls ++(90:.2cm) and ++(270:.2cm) .. (-1,1.6);
\draw[dashed] (-1.4,1.9) -- (-1.4,3);
\draw[\QsColor,thick] (-1,-3) arc (180:90:.4cm) arc (-90:0:.4cm) -- (-.2,-1.6) arc (180:90:.4cm) arc (-90:0:.4cm) -- (.6,.8) arc (0:90:.4cm) arc (270:180:.4cm) -- (-.2,2.2) arc (0:90:.4cm) arc (270:180:.4cm);
\draw[\XColor,thick] (-.6,-3) -- (-.6,-1.9);
\draw[\YColor,thick] (.2,-3) -- (.2,-.5);
\draw[\XColor,thick] (.2,.5) -- (.2,3);
\draw[\YColor,thick] (-.6,1.9) -- (-.6,3);
\roundNbox{unshaded}{(-1,1.9)}{.3}{.25}{.25}{\scriptsize{$c_k$}};
\roundNbox{unshaded}{(0,.5)}{.3}{.1}{.1}{\scriptsize{$b_i^{(n)}$}};
\roundNbox{unshaded}{(0,-.5)}{.3}{.1}{.1}{\scriptsize{$c_k^\dag$}};
\roundNbox{unshaded}{(-1,-1.9)}{.3}{.25}{.25}{\scriptsize{$(b_i^{(n)})^\dag$}};
}
\approx
\tikzmath{
\begin{scope}
\clip[rounded corners = 5] (-.95,-.5) rectangle (.7,.5);
\filldraw[\QrColor] (-.4,-.9) rectangle (.7,.9);
\filldraw[\rColor] (-.8,-.9) rectangle (-.4,.9);
\end{scope}
\draw[dashed] (-.8,-.5) -- (-.8,.5);
\draw[\QsColor,thick] (-.4,-.5) -- (-.4,.5);
\draw[\YColor,thick] (.4,-.5) .. controls ++(90:.45cm) and ++(270:.45cm) .. (0,.5);
\filldraw[\QrColor] (.2,0) circle (.05cm);
\draw[\XColor,thick] (0,-.5) .. controls ++(90:.45cm) and ++(270:.45cm) .. (.4,.5);
}\,.
\end{align*}
The first $\approx$ uses \eqref{eq:OverBraiding} for $X\boxtimes_M Y$, and the second equality uses associativity of the bimodule actions.
The third $\approx$ uses $X$ is approximately inner over $|Q|$, $Y$ is centrally trivial over $|Q|$ and $(q^{(n)}_l)$ is a central sequence in $|Q|$, and the fourth $\approx$ uses \eqref{eq:ExpandBraidingsQ}.
The fifth equality uses an argument similar to \eqref{eq:UseCommutativityForRichtAction}.
The sixth equality is just isotopy,
and the final $\approx$ uses \eqref{eq:OverBraiding} for $X \boxtimes_{|Q|} Y$.

Finally, since 
\[
\tikzmath{
\begin{scope}
\clip[rounded corners = 5] (-.55,-.6) rectangle (.7,.6);
\filldraw[\rColor] (-.4,-.9) rectangle (.7,.9);
\filldraw[\QrColor] (0,-.9) rectangle (.4,.9);
\end{scope}
\draw[dashed] (-.4,-.6) -- (-.4,.6);
\draw[\QsColor,thick] (0,-.6) -- (0,.6);
\draw[\QsColor,thick] (.4,-.6) -- (.4,.6);
}
=
\tikzmath{
\begin{scope}
\clip[rounded corners = 5] (-.55,-.6) rectangle (.7,.6);
\filldraw[\rColor] (-.4,-.9) rectangle (.7,.9);
\filldraw[\QrColor] (.4,-.6) arc (0:90:.4cm) -- (0,-.6);
\filldraw[\QrColor] (.4,.6) arc (0:-90:.4cm) -- (0,.6); 
\end{scope}
\draw[dashed] (-.4,-.6) -- (-.4,.6);
\draw[\QsColor,thick] (0,-.6) -- (0,.6);
\draw[\QsColor,thick] (.4,-.6) arc (0:90:.4cm);
\draw[\QsColor,thick] (.4,.6) arc (0:-90:.4cm);
\filldraw[\QsColor] (0,-.2) circle (.05cm);
\filldraw[\QsColor] (0,.2) circle (.05cm);
}
=
\sum_j
\tikzmath{
\begin{scope}
\clip[rounded corners = 5] (-.55,-1.5) rectangle (.7,1.5);
\filldraw[\rColor] (.7,-1.5) -- (-.4,-1.5) -- (-.4,-.5) -- (-.2,-.5) -- (-.2,.5) -- (-.4,.5) -- (-.4,1.5) -- (.7,1.5);
\filldraw[\QrColor] (.4,-1.5) arc (0:90:.4cm) -- (0,-1.5);
\filldraw[\QrColor] (.4,1.5) arc (0:-90:.4cm) -- (0,1.5); 
\end{scope}
\draw[dashed] (-.4,-1.5) -- (-.4,-.5);
\draw[dashed] (-.2,-.5) -- (-.2,.5);
\draw[dashed] (-.4,.5) -- (-.4,1.5);
\draw[\QsColor,thick] (0,-1.5) -- (0,-.5);
\draw[\QsColor,thick] (0,.5) -- (0,1.5);
\draw[\QsColor,thick] (.4,-1.5) arc (0:90:.4cm);
\draw[\QsColor,thick] (.4,1.5) arc (0:-90:.4cm);
\filldraw[\QsColor] (0,-1.1) circle (.05cm);
\filldraw[\QsColor] (0,1.1) circle (.05cm);
\roundNbox{unshaded}{(-.2,-.5)}{.3}{.1}{.1}{\scriptsize{$q_j^\dag$}};
\roundNbox{unshaded}{(-.2,.5)}{.3}{.1}{.1}{\scriptsize{$q_j$}};
}
\qquad\text{and}\qquad
\tikzmath{
\filldraw[\rColor] (0,-.8) rectangle (.8,.8);
\filldraw[\QrColor] (0,-.4) arc (-90:0:.4cm) arc (180:90:.4cm) -- (.8,.8) -- (0,.8);
\draw[\QsColor,thick] (0,-.8) -- (0,.8);
\draw[\QsColor,thick] (.8,-.8) -- (.8,.8);
\draw[\QsColor,thick] (0,-.4) arc (-90:0:.4cm) arc (180:90:.4cm);
\filldraw[\QsColor] (0,-.4) circle (.05cm);
\filldraw[\QsColor] (.8,.4) circle (.05cm);
}
=
\tikzmath{
\filldraw[\rColor] (0,-.8) rectangle (.8,.8);
\filldraw[\QrColor] (0,.4) arc (90:0:.4cm) arc (180:270:.4cm) -- (.8,.8) -- (0,.8);
\draw[\QsColor,thick] (0,-.8) -- (0,.8);
\draw[\QsColor,thick] (.8,-.8) -- (.8,.8);
\draw[\QsColor,thick] (0,.4) arc (90:0:.4cm) arc (180:270:.4cm);
\filldraw[\QsColor] (0,.4) circle (.05cm);
\filldraw[\QsColor] (.8,-.4) circle (.05cm);
}\,,
\]
we have
\[
\tikzmath{
\begin{scope}
\clip[rounded corners = 5] (-1.35,-.9) rectangle (1.1,.9);
\filldraw[\QrColor] (-.8,-.9) rectangle (1.1,.9);
\filldraw[\rColor] (-1.2,-.9) rectangle (-.8,.9);
\filldraw[\rColor] (0,-.5) -- (.4,-.5) arc (-90:0:.4cm) -- (.8,.1) arc (0:90:.4cm) -- (0,.5) arc (90:180:.4cm) -- (-.4,-.1) arc (180:270:.4cm);
\end{scope}
\draw[dashed] (-1.2,-.9) -- (-1.2,.9);
\draw[\QsColor,thick] (-.8,-.9) -- (-.8,.9);
\draw[\QsColor,thick] (0,-.5) -- (.4,-.5) arc (-90:0:.4cm) -- (.8,.1) arc (0:90:.4cm) -- (0,.5) arc (90:180:.4cm) -- (-.4,-.1) arc (180:270:.4cm);
\draw[\YColor,thick] (.4,-.9) -- (.4,-.5) .. controls ++(90:.45cm) and ++(270:.45cm) .. (0,.5) -- (0,.9);
\filldraw[\rColor] (.2,0) circle (.05cm);
\draw[\XColor,thick] (0,-.9) -- (0,-.5) .. controls ++(90:.45cm) and ++(270:.45cm) .. (.4,.5) -- (.4,.9);
}
=
\sum_j
\tikzmath{
\begin{scope}
\clip[rounded corners = 5] (-1.35,-2.2) rectangle (1.1,2.2);
\filldraw[\rColor] (1.1,-2.2) -- (-1.2,-2.2) -- (-1.2,-1.2) -- (-1,-.9) -- (-1,.9) -- (-1.2,1.2) -- (-1.2,2.2) -- (1.1,2.2);
\filldraw[\QrColor] (1.1,-2.2) -- (-.8,-2.2) -- (-.8,-1.9) arc (-90:0:.4cm) -- (-.4,-.9) arc (180:90:.4cm) -- (.4,-.5) arc (-90:0:.4cm) -- (.8,.1) arc (0:90:.4cm) -- (0,.5) arc (270:180:.4cm) -- (-.4,1.5) arc (0:90:.4cm) -- (-.8,2.2) -- (1.1,2.2);
\end{scope}
\draw[dashed] (-1.2,-2.2) -- (-1.2,-1.2);
\draw[dashed] (-1,-.9) -- (-1,.9);
\draw[dashed] (-1.2,1.2) -- (-1.2,2.2);
\draw[\QsColor,thick] (-.8,1.5) -- (-.8,2.2);
\draw[\QsColor,thick] (-.8,-2.2) -- (-.8,-1.5);
\draw[\QsColor,thick] (-.8,-1.9) arc (-90:0:.4cm) -- (-.4,-.9) arc (180:90:.4cm) -- (.4,-.5) arc (-90:0:.4cm) -- (.8,.1) arc (0:90:.4cm) -- (0,.5) arc (270:180:.4cm) -- (-.4,1.5) arc (0:90:.4cm);
\draw[\YColor,thick] (.4,-2.2) -- (.4,-.5) .. controls ++(90:.45cm) and ++(270:.45cm) .. (0,.5) -- (0,2.2);
\filldraw[\rColor] (.2,0) circle (.05cm);
\draw[\XColor,thick] (0,-2.2) -- (0,-.5) .. controls ++(90:.45cm) and ++(270:.45cm) .. (.4,.5) -- (.4,2.2);
\filldraw[\QsColor] (-.8,-1.9) circle (.05cm);
\filldraw[\QsColor] (-.8,1.9) circle (.05cm);
\roundNbox{unshaded}{(-1,-1.2)}{.3}{.1}{.1}{\scriptsize{$q_j^\dag$}};
\roundNbox{unshaded}{(-1,1.2)}{.3}{.1}{.1}{\scriptsize{$q_j$}};
}
=
\sum_j
\tikzmath{
\begin{scope}
\clip[rounded corners = 5] (-1.45,-1.8) rectangle (.7,1.8);
\filldraw[\rColor] (.7,-1.8) -- (-1.2,-1.8) -- (-1.2,-.8) -- (-1,-.8) -- (-1,.8) -- (-1.2,.8) -- (-1.2,1.8) -- (.7,1.8);
\filldraw[\QrColor] (.7,-1.8) -- (-.8,-1.8) -- (-.8,-1.5) arc (-90:0:.4cm) -- (-.4,1.1) arc (0:90:.4cm) -- (-.8,1.8) -- (.7,1.8);
\end{scope}
\draw[dashed] (-1.2,-1.8) -- (-1.2,-.8);
\draw[dashed] (-1,-.8) -- (-1,.8);
\draw[dashed] (-1.2,.8) -- (-1.2,1.8);
\draw[\QsColor,thick] (-.8,-1.8) -- (-.8,-.8);
\draw[\QsColor,thick] (-.8,.8) -- (-.8,1.8);
\draw[\QsColor,thick] (-.8,-1.5) arc (-90:0:.4cm) -- (-.4,1.1) arc (0:90:.4cm);
\draw[\YColor,thick] (.4,-1.8) -- (.4,-.5) .. controls ++(90:.45cm) and ++(270:.45cm) .. (0,.5) -- (0,1.8);
\filldraw[\QrColor] (.2,0) circle (.05cm);
\draw[\XColor,thick] (0,-1.8) -- (0,-.5) .. controls ++(90:.45cm) and ++(270:.45cm) .. (.4,.5) -- (.4,1.8);
\filldraw[\QsColor] (-.8,-1.5) circle (.05cm);
\filldraw[\QsColor] (-.8,1.5) circle (.05cm);
\roundNbox{unshaded}{(-1,-.8)}{.3}{.1}{.1}{\scriptsize{$q_j^\dag$}};
\roundNbox{unshaded}{(-1,.8)}{.3}{.1}{.1}{\scriptsize{$q_j$}};
}
=
\tikzmath{
\begin{scope}
\clip[rounded corners = 5] (-.95,-.5) rectangle (.7,.5);
\filldraw[\QrColor] (-.4,-.9) rectangle (.7,.9);
\filldraw[\rColor] (-.8,-.9) rectangle (-.4,.9);
\end{scope}
\draw[dashed] (-.8,-.5) -- (-.8,.5);
\draw[\QsColor,thick] (-.4,-.5) -- (-.4,.5);
\draw[\YColor,thick] (.4,-.5) .. controls ++(90:.45cm) and ++(270:.45cm) .. (0,.5);
\filldraw[\QrColor] (.2,0) circle (.05cm);
\draw[\XColor,thick] (0,-.5) .. controls ++(90:.45cm) and ++(270:.45cm) .. (.4,.5);
}\,.
\qedhere
\]
\end{proof}

%%%%%%%%%%%%%%%%%%%%%%%%%%%%%%%%%%%%%%%%%%%%%%%%%%%%%%%%%%%%%%%%%%%%%%%%%%%%%%%%%%%%
\section{Calculation of \texorpdfstring{$\Chi(M_\infty)$}{Minfty} for a non-Gamma finite depth \texorpdfstring{$\rm II_1$}{II1} subfactor}
\label{sec:Calculation}

In this section, we calculate $\Chi(M_\infty)$ for the inductive limit $\rm II_1$ factor obtained from iterating Jones' basic construction for a finite depth finite index non-Gamma $\rm II_1$ subfactor $N\subseteq M$.
These examples are motivated by \cite{MR2661553}.

Suppose $N\subseteq M$ a finite depth, finite index $\rm II_1$ subfactor, and let $\cC={}_N\cC_N$ denote the unitary fusion category of $N-N$ bimodules generated by ${}_NM_N$.
The results of \cite[\S3 and 4]{MR3801484} give a bijective correspondence between equivalence classes of (bifinite) bimodules of $M_{\infty}$ which restrict to $R\otimes N$-bimodules of the form $\cC^{\op}\boxtimes \cC$ and objects of the Drinfeld center $\cZ(\cC)$.
The main goal of this section is to extend this bijection to a fully faithful unitary tensor functor $\Phi:\cZ(\cC)\to \fgpBim(M_\infty)$ such that when $N$ is non-Gamma, $\Phi$ takes values in $\Chi(M_\infty)$ and is a braided unitary equivalence. 
To do so, we rely on the Q-system realization language from \cite{2105.12010} together with the coend realization viewpoint of \cite{MR3948170}.

We begin this section with some basics on unitary fusion categories and subfactor standard invariants.

%%%%%%%%%%%%%%%%%%%%%%%%%%%%%%%%%%%%%%%%%%
\subsection{Basics on unitary fusion categories and subfactor standard invariants}

A \emph{unitary fusion category} is a unitary tensor category with only finitely many isomorphism classes of simple objects.
A unitary fusion category $\cC$ has three commuting involutions $\dag, \vee, \overline{\,\cdot\,}$, and the composite of any two is the third.
Here, $\vee$ is the unique unitary dual functor \cite{MR2091457,MR4133163} giving the canonical unitary spherical structure of $\cC$ \cite{MR1444286}, and we may define $\overline{\,\cdot\,}:=\vee\dag = \dag \vee$.

\begin{defn}
\label{def:Z(C)}
The \emph{Drinfeld center} of a unitary fusion category $\cC$ is $\cZ(\cC)=\End_{\cC-\cC}(\cC)$, the Morita dual of $\cC^{\rm mp}\boxtimes \cC$ acting on $\cC$ by $(a^{\rm mp}\boxtimes b)\rhd c := b\otimes c \otimes a$, where $\cC^{\rm mp}$ is the monoidal opposite of $\cC$ from Remark \ref{rem:OppositeWithBraiding}.
Note that the unitary dual functor $\vee$ gives a unitary tensor equivalence $\cC^{\rm mp}\to \cC^{\op}$, the opposite fusion category with the opposite arrows, but the same tensor product.
It is useful in the subsections below to identify $\cZ(\cC)$ with $\End_{\cC^{\op}\boxtimes \cC}(\cC)$
with the action $(a^{\op}\boxtimes b)\rhd c := b\otimes c \otimes \overline{a}$.

Now $\cZ(\cC)$ has another description in terms of pairs $(z,\sigma_z)$ of an object $z\in \cC$ equipped with a half-braiding $\sigma_z$, where $\cZ(\cC)$ acts on $\cC$ via the forgetful functor $(z,\sigma_z)\mapsto z$ \cite[\S7.13 and 8.5]{MR3242743}.
Our convention for the half-braiding $\sigma_z$ is that the strands for objects in $\cC$ pass over the $z$-strand:
$$
\sigma_{c,z}
:=
\tikzmath{
\draw[\YColor,thick] (.4,-.5) node[below]{$\scriptstyle z$} .. controls ++(90:.45cm) and ++(270:.45cm) .. (0,.5);
\draw[\XColor,thick, knot] (0,-.5) node[below]{$\scriptstyle c$} .. controls ++(90:.45cm) and ++(270:.45cm) .. (.4,.5);
}
$$
Thus the braiding $(z,\sigma_z)\otimes (w,\rho_w) \to (w, \rho_w)\otimes (z,\sigma_z)$ in $\cZ(\cC)$ is given by $\rho_{z,w}$.
\end{defn}

\begin{defn}
There are many equivalent notions of the standard invariant for a finite index $\rm II_1$ subfactor $N\subseteq M$.
For this article, the standard invariant will mean the $2\times 2$ unitary multitensor category $\cC(N\subseteq M)$ of $N-N$, $N-M$, $M-N$, and $M-M$ bimodules generated by $L^2M$ under $\boxtimes$, $\oplus$, and $\subseteq$, with generating object ${}_NL^2M_M$.
$$
\cC(N\subseteq M) =
\begin{pmatrix}
{}_N\cC_N & {}_N\cC_M
\\
{}_M\cC_N & {}_M \cC_M
\end{pmatrix}
$$
Observe that $\cC(N\subseteq M)$ is multifusion if and only if $N\subseteq M$ is finite depth.
In this case, the corners ${}_N\cC_N$ and ${}_M \cC_M$ of $N-N$ and $M-M$ bimodules generated by $L^2M$ respectively are unitary fusion categories which are Morita equivalent,
and thus share the same Drinfeld center $\cZ(\cC)$.
\end{defn}

\begin{rem}
\label{rem:MoritaEquivalenceRealization}
Suppose $\cC$ and $\cD$ are two unitary fusion categories and ${}_\cC \cM_\cD$ is an indecomposable unitary $\cC-\cD$ bimodule category witnessing a Morita equivalence.
Using the internal hom \cite{MR1976459} (see also \cite[Appendix~A]{MR3933035}), we can form a $2\times 2$ unitary multifusion category by
\begin{equation}
\label{eq:2x2UMFC}
\begin{pmatrix}
\cC & \cM 
\\
\cM^{\op} & \cD
\end{pmatrix}.
\end{equation}
For a simple $X\in \cM$, we get two Q-systems $X\otimes \overline{X} =\underline{\End}_\cC(X)\in \cC$ and $\overline{X}\otimes X=\underline{\End}_\cD(X) \in \cD$.
The map $\Ad(X) : d\mapsto X\otimes d\otimes \overline{X}$ gives a unitary tensor equivalence between $\cD$ and $X\otimes \overline{X}-X\otimes \overline{X}$ bimodules in $\cC$.
A similar result holds on the other side.

Suppose now we have a fully faithful unitary tensor functor $F: \cC\to \fgpBim(N)$ for a $\rm II_1$ factor $N$.
Then the realization $M:= |X\otimes \overline{X}|$ is a $\rm II_1$ factor containing $N$, and the standard invariant of $N\subseteq M$ is unitarily equivalent to the $2\times 2$ unitary multifusion category \eqref{eq:2x2UMFC} with generator $|X|$ as an $N-M$ bimodule. 
By Remark \ref{rem:QBimodules}, we get a fully faithful unitary tensor functor $G: \cD \to \fgpBim(M)$ from realization as
$$
\cD 
\xrightarrow{\Ad(X)} \Bim_\cC(X\otimes \overline{X})
\xrightarrow{|\,\cdot\,|} 
\fgpBim(M).
$$
\end{rem}

We now give an important example of Remark \ref{rem:MoritaEquivalenceRealization} which will be used in this section below.

\begin{ex}
\label{ex:MoritaEquivalenceBetweenCopCandZ(C)}
Let $\cC$ be a unitary fusion category, and consider the $\cC\boxtimes \cC^{\rm mp} - \cZ(\cC)$
Morita equivalence bimodule $\cC$.
One calculates that $\underline{\End}_{\cC^{\rm mp}\boxtimes \cC}(1_\cC) = \bigoplus_{c\in \Irr(\cC)} \overline{c}^{\rm mp}\boxtimes c$, which we call the \emph{symmetric enveloping Q-system} after \cite{MR1302385,MR1729488}. The infinite version of this algebra object plays a very important role for the study of analytic properties of infinite unitary tensor categories \cite{MR3406647}.
Identifying $\cC^{\rm mp}\cong \cC^{\op}$ via $\vee$, which will be useful in the sequel, the symmetric enveloping Q-system is given by $S:=\bigoplus_{c\in \Irr(\cC)} c^{\op}\boxtimes c$.
By Remark \ref{rem:MoritaEquivalenceRealization}, $\cZ(\cC)\cong \Bim_{\cC^{\op}\boxtimes \cC}(S)$.
On the other hand, one calculates that $\underline{\End}_{\cZ(\cC)}(1_\cC) = I(1_\cC)$, where $I: \cC\to \cZ(\cC)$ is the adjoint of the forgetful functor.
\end{ex}

%%%%%%%%%%%%%%%%%%%%%%%%%%%%%%%%%%%%%%%%%%%%%%%%%%%%%%%%%%%%%%%%%%%%%%%
\subsection{Q-system realization as a coend}

Suppose $\cC$ is a unitary fusion category and $G: \cC\to \fgpBim(N)$ is a unitary tensor functor, where $N$ is a $\rm II_1$ factor.
Given a Q-system $Q\in \cC$, the realization 
$|G(Q)|$ is a $\rm II_1$ multifactor (finite direct sum of $\rm II_1$ factors) which is a factor if and only if $Q$ is simple as a $Q-Q$ bimodule in $\cC$.

By the Yoneda lemma, we have a canonical isomorphism of vector spaces
$$
|G(Q)|
:=
\Hom(N_N \to N \boxtimes_N G(Q)_N)
\cong
\bigoplus_{c\in \cC} \cC(c\to Q) \otimes_\bbC G(c).
$$
We graphically represent elements of this tensor product by
$$
\sum_{c\in\Irr(\cC)}
\tikzmath{
\begin{scope}
\clip[rounded corners = 5] (-.7,0) -- (-.7,-1.4) -- (.7,-1.4) -- (.7,0);
\filldraw[\rColor] (-.7,0) -- (-.7,-1.4) -- (.7,-1.4) -- (.7,0);
\end{scope}
\draw[dotted] (-.7,0) -- (.7,0);
\draw[\NColor,thick] (0,-1.4) -- (0,-.7);
\draw (0,-.7) -- (0,.7); 
\draw[\QsColor, thick] (0,1) -- (0,1.4);
\roundNbox{unshaded}{(0,-.7)}{.3}{0}{0}{$\xi_c$};
\roundNbox{unshaded}{(0,.7)}{.3}{0}{0}{$f_c$};
\node at (.2,.2) {\scriptsize{$c$}};
\node at (.2,-.2) {\scriptsize{$c$}};
\node at (.2,1.2) {\scriptsize{$Q$}};
}
:=
\sum_{c\in\Irr(\cC)}
\tikzmath{
\draw[\QsColor, thick] (0,.7) -- (0,0);
\draw (0,0) -- (0,-.7);
\roundNbox{unshaded}{(0,0)}{.3}{0}{0}{$f_c$};
\node at (.2,.5) {\scriptsize{$Q$}};
\node at (.2,-.5) {\scriptsize{$c$}};
}
\otimes_{\bbC}
\tikzmath{
\begin{scope}
\clip[rounded corners = 5] (-.7,-.7) rectangle (.7,.7);
\filldraw[\rColor] (-.7,-.7) rectangle (.7,.7);
\end{scope}
\draw (0,0) -- (0,.7);
\draw[\NColor,thick] (0,-.7) -- (0,0);
\roundNbox{unshaded}{(0,0)}{.3}{0}{0}{$\xi_c$};
\node at (0,.9) {\scriptsize{$c$}};
}
=
\sum_{c\in \Irr(\cC)} f_c \otimes_\bbC \xi_c.
$$
Here, the orange line represents the 
functor $G^\circ := \Forget\circ \,G$
viewed as a $\rm W^*$-algebra object in $\Fun(\cC \to \Vect)$ \cite[Prop.~2.18]{MR3948170}, where $\Forget: \fgpBim(N) \to \Vect$ is the forgetful functor.
The shaded half of the diagram is read top to bottom, and the tensorator $G^2_{a,b}$ is denoted by appending an orange trivalent vertex below.

Under this isomorphism of vector spaces, the multiplication and $*$-structure are given by
$$
(f_a\otimes \xi_a)(g_b\otimes \eta_b)
=
\sum_{\substack{ a,b,c\in\Irr(\cC) \\ \alpha\in\ONB(a\otimes b\to c)}}
\tikzmath{
\begin{scope}
\clip[rounded corners = 5] (-1.2,0) -- (-1.2,-2.1) -- (1.2,-2.1) -- (1.2,0);
\filldraw[\rColor] (-1.2,0) -- (-1.2,-2.1) -- (1.2,-2.1) -- (1.2,0);
\end{scope}
\draw[dotted] (-1.2,0) -- (1.2,0);
\draw (0,-.3) -- (0,.3); 
\draw (.4,.7) arc (0:-180:.4cm);
\draw (.4,-.7) arc (0:180:.4cm);
\filldraw (0,-.3) circle (.05cm);
\filldraw (0,.3) circle (.05cm);
\draw[\QsColor, thick] (-.4,1.3) arc (180:0:.4cm);
\draw[\QsColor, thick] (0,1.7) -- (0,2.1);
\filldraw (0,1.7) circle (.05cm);
\draw[\NColor, thick] (-.4,-1.3) arc (-180:0:.4cm);
\draw[\NColor, thick] (0,-1.7) -- (0,-2.1);
\filldraw[\NColor] (0,-1.7) circle (.05cm);
\roundNbox{unshaded}{(.4,-1)}{.3}{0}{0}{$\eta_b$};
\roundNbox{unshaded}{(.4,1)}{.3}{0}{0}{$g_b$};
\roundNbox{unshaded}{(-.4,-1)}{.3}{0}{0}{$\xi_a$};
\roundNbox{unshaded}{(-.4,1)}{.3}{0}{0}{$f_a$};
\node at (.6,.5) {\scriptsize{$b$}};
\node at (.6,-.5) {\scriptsize{$b$}};
\node at (-.6,.5) {\scriptsize{$a$}};
\node at (-.6,-.5) {\scriptsize{$a$}};
\node at (.2,.2) {\scriptsize{$c$}};
\node at (.2,-.2) {\scriptsize{$c$}};
\node at (.6,1.5) {\scriptsize{$Q$}};
\node at (-.6,1.5) {\scriptsize{$Q$}};
\node at (0,.5) {\scriptsize{$\alpha^\dag$}};
\node at (0,-.5) {\scriptsize{$\alpha$}};
}
\qquad\qquad
(f_a\otimes \xi_a)^*
=
\tikzmath{
\begin{scope}
\clip[rounded corners = 5] (-.7,0) -- (-.7,-1.4) -- (.7,-1.4) -- (.7,0);
\filldraw[\rColor] (-.7,0) -- (-.7,-1.4) -- (.7,-1.4) -- (.7,0);
\end{scope}
\draw[dotted] (-.7,0) -- (.7,0);
\draw[\NColor,thick] (0,-1.4) -- (0,-.7);
\draw (0,-.7) -- (0,.7); 
\draw[\QsColor, thick] (0,1) -- (0,1.4);
\roundNbox{unshaded}{(0,-.7)}{.3}{0}{0}{$\overline{\xi_a}$};
\roundNbox{unshaded}{(0,.7)}{.3}{0}{0}{$\overline{f_a}$};
\node at (.2,.2) {\scriptsize{$\overline{a}$}};
\node at (.2,-.2) {\scriptsize{$\overline{a}$}};
\node at (.2,1.2) {\scriptsize{$Q$}};
}\,,
$$
and the unit is given by
$$
1=
\tikzmath{
\begin{scope}
\clip[rounded corners = 5] (-.7,0) -- (-.7,-1.4) -- (.7,-1.4) -- (.7,0);
\filldraw[\rColor] (-.7,0) -- (-.7,-1.4) -- (.7,-1.4) -- (.7,0);
\end{scope}
\draw[dotted] (-.7,0) -- (.7,0);
\filldraw (0,.3) circle (.05cm);
\draw[\QsColor, thick] (0,.3) -- node[right]{$\scriptstyle Q$} (0,.7);
\draw[\NColor,thick] (0,-1.4) -- (0,-.7);
\roundNbox{unshaded}{(0,-.7)}{.3}{0}{0}{\scriptsize{$\Omega_N$}};
}\,;
\qquad\qquad
\Omega_N = 1\in {}_NN_N.
$$

\begin{ex}
\label{ex:RealizedJonesTower}
Suppose now $N\subseteq M$ is a finite depth, finite index $\rm II_1$ subfactor.
The algebra $M$ considered as an $N-N$ bimodule ${}_NM_N \in {}_N\cC_N$ is the canonical Q-system $X\boxtimes_M \overline{X}$ corresponding to the generator $X:={}_NM_M \in {}_N\cC_{M}$ as discussed in Example \ref{ex:RealizeSeparatedQSystem}.
The realization $|{}_NM_N|$ is canonically $*$-isomorphic to $M$:
$$
|{}_NM_N|
=
\Hom(N_N \to N \boxtimes_N M_N)
=
\Hom(N_N \to N \boxtimes_N M \boxtimes_M M_N)
\cong
\Hom(M_M \to M_M)
=
M.
$$ 

Now consider the Jones tower obtained by iterating Jones' basic construction defined inductively by $M_{n+1} := \End((M_n)_{M_{n-1}})$ \cite{MR0696688,MR999799}
$$
M_0=N \subseteq M = M_1 \subseteq M_2 \subseteq M_3\subseteq \cdots.
$$
The $\rm II_1$ factor $M_n$ is $*$-isomorphic to the realization of the Q-system $(X\boxtimes_M \overline{X})^{\boxtimes n}\cong X^{\alt\boxtimes n} \boxtimes \overline{X^{\alt\boxtimes n}}$ which has multiplication and unit given by
$$
\tikzmath{
\draw (-.3,0) node[below]{$\scriptstyle n$} arc (180:0:.3cm) node[below]{$\scriptstyle n$};
\draw (-.6,0) node[below]{$\scriptstyle n$} .. controls ++(90:.3cm) and ++(270:.3cm) .. (-.3,.8);
\draw (.6,0) node[below]{$\scriptstyle n$} .. controls ++(90:.3cm) and ++(270:.3cm) .. (.3,.8);
}
\qquad\qquad
\tikzmath{
\draw (-.3,0) node[above]{$\scriptstyle n$} -- (-.3,-.3) arc (-180:0:.3cm) -- (.3,0) node[above]{$\scriptstyle n$};  
}
\qquad\qquad
\tikzmath{
\draw (0,0) --node[right]{$\scriptstyle n$} (0,1);
}
=
\id_{X^{\alt\boxtimes n}};
\qquad\qquad
X^{\alt\boxtimes n}:=\underbrace{X\boxtimes_M \overline{X} \boxtimes_N \cdots \boxtimes X^?}_{\text{$n$ tensorands}}.
$$
Indeed,
$$
\Hom(N_N \to N \boxtimes_N X^{\alt\boxtimes n} \boxtimes \overline{X^{\alt\boxtimes n}})
\cong
\End(N\boxtimes_N X^{\alt\boxtimes n}_{M \text{ or }N})
\cong 
M_n
$$
by the multistep Jones basic construction \cite{MR965748,MR1424954}.
Another way to see this is to use Remark \ref{rem:MoritaEquivalenceRealization}; 
for example, the map $\Ad(X)$ takes the basic construction $\langle M, N\rangle=\overline{X}\boxtimes_N X$ to 
$(X\boxtimes_M \overline{X})^{\boxtimes 2}$ with the multiplicaiton as claimed.

As a coend realization, we have a canonical Frobenius reciprocity isomorphism
\begin{equation}
\label{eq:CoendForMn}
M_n = 
\bigoplus_{c\in \Irr(\cC)} 
\cC(c\to X^{\alt\boxtimes n} \boxtimes \overline{X^{\alt\boxtimes n}}) 
\otimes_\bbC G(c)
\cong
\bigoplus_{c\in \Irr(\cC)} 
\cC(c\boxtimes X^{\alt\boxtimes n} \to X^{\alt\boxtimes n}) 
\otimes_\bbC G(c).
\end{equation}
Under this isomorphism, in the coend realization diagrammatic calculus, we can draw the $X^{\alt\boxtimes n}$ horizontally, where the horizontal line should be viewed as slightly tilted going from the bottom right to the top left, as indicated by the cyan arrows below.
\begin{equation}
\label{eq:BottomRightTopLeft}
\tikzmath{
\draw (-.7,0) -- (.7,0);
\draw (0,0) -- (0,-.7);
\draw[cyan, dashed] (-.7,-.7) -- (.7,.7);
\draw[cyan, ->] (-.6,-.7) -- (-.8,-.5);
\draw[cyan, ->] (.7,.6) -- (.5,.8);
\roundNbox{unshaded}{(0,0)}{.3}{0}{0}{$f$};
\node at (-.9,0) {\scriptsize{$k$}};
\node at (.9,0) {\scriptsize{$k$}};
\node at (.2,-.5) {\scriptsize{$c$}};
}
\in 
\cC(c\boxtimes X^{\alt\boxtimes n} \to X^{\alt\boxtimes n})
\end{equation}
The multiplication, $*$, and unit in the realization $|X^{\alt\boxtimes n} \boxtimes \overline{X^{\alt\boxtimes n}}|\cong M_n$
are now represented respectively by
$$
\sum_{\substack{ a,b,c\in\Irr(\cC) \\ \alpha\in\ONB(a\otimes b\to c)}}
\tikzmath{
\begin{scope}
\clip[rounded corners = 5] (-1.2,0) -- (-1.2,-2.2) -- (1.2,-2.2) -- (1.2,0);
\filldraw[\rColor] (-1.2,0) -- (-1.2,-2.2) -- (1.2,-2.2) -- (1.2,0);
\end{scope}
\draw[dotted] (-1.2,0) -- (1.2,0);
\draw[\NColor,thick] (-.5,-1.4) arc (-180:0:.5cm);
\draw[\NColor, thick] (0,-1.9) -- (0,-2.2);
\filldraw[\NColor] (0,-1.9) circle (.05cm);
\draw (0,-.3) -- (0,.3); 
\draw (.5,.8) arc (0:-180:.5cm);
\draw (.5,-.8) arc (0:180:.5cm);
\draw (-1.2,1.1) -- (1.2,1.1);
\filldraw (0,-.3) circle (.05cm);
\filldraw (0,.3) circle (.05cm);
\roundNbox{unshaded}{(.5,-1.1)}{.3}{0}{0}{$\eta_b$};
\roundNbox{unshaded}{(.5,1.1)}{.3}{0}{0}{$g_b$};
\roundNbox{unshaded}{(-.5,-1.1)}{.3}{0}{0}{$\xi_a$};
\roundNbox{unshaded}{(-.5,1.1)}{.3}{0}{0}{$f_a$};
\node at (.7,.6) {\scriptsize{$b$}};
\node at (.7,-.6) {\scriptsize{$b$}};
\node at (-.7,.6) {\scriptsize{$a$}};
\node at (-.7,-.6) {\scriptsize{$a$}};
\node at (.2,.2) {\scriptsize{$c$}};
\node at (.2,-.2) {\scriptsize{$c$}};
\node at (1,1.3) {\scriptsize{$n$}};
\node at (-1,1.3) {\scriptsize{$n$}};
\node at (0,1.3) {\scriptsize{$n$}};
\node at (0,.5) {\scriptsize{$\alpha^\dag$}};
\node at (0,-.5) {\scriptsize{$\alpha$}};
}\,,
\qquad\qquad
\tikzmath{
\begin{scope}
\clip[rounded corners = 5] (-.7,0) -- (-.7,-1.4) -- (.7,-1.4) -- (.7,0);
\filldraw[\rColor] (-.7,0) -- (-.7,-1.4) -- (.7,-1.4) -- (.7,0);
\end{scope}
\draw[dotted] (-.7,0) -- (.7,0);
\draw[\NColor,thick] (0,-1.4) -- (0,-.7);
\draw (0,-.7) -- (0,.7); 
\draw (-.7,.7) -- (.7,.7);
\roundNbox{unshaded}{(0,-.7)}{.3}{0}{0}{$\overline{\xi_a}$};
\roundNbox{unshaded}{(0,.7)}{.3}{0}{0}{$\overline{f_a}$};
\node at (.2,.2) {\scriptsize{$\overline{a}$}};
\node at (.2,-.2) {\scriptsize{$\overline{a}$}};
\node at (.5,.9) {\scriptsize{$n$}};
\node at (-.5,.9) {\scriptsize{$n$}};
}\,,
\qquad\qquad
\text{and}
\qquad\qquad
\tikzmath{
\begin{scope}
\clip[rounded corners = 5] (-.7,0) -- (-.7,-1.4) -- (.7,-1.4) -- (.7,0);
\filldraw[\rColor] (-.7,0) -- (-.7,-1.4) -- (.7,-1.4) -- (.7,0);
\end{scope}
\draw[dotted] (-.7,0) -- (.7,0);
\draw (-.7,.7) -- (.7,.7);
\node at (0,.9) {\scriptsize{$n$}};
}
=
\tikzmath{
\begin{scope}
\clip[rounded corners = 5] (-.7,0) -- (-.7,-1.4) -- (.7,-1.4) -- (.7,0);
\filldraw[\rColor] (-.7,0) -- (-.7,-1.4) -- (.7,-1.4) -- (.7,0);
\end{scope}
\draw[dotted] (-.7,0) -- (.7,0);
\draw (-.7,.7) -- (.7,.7);
\draw[\NColor,thick] (0,-1.4) -- (0,-.7);
\roundNbox{unshaded}{(0,-.7)}{.3}{0}{0}{\scriptsize{$\Omega_N$}};
\node at (0,.9) {\scriptsize{$n$}};
}\,.
$$
Here, the $X^{\alt\boxtimes n}$ horizontal strand is read bottom to top, i.e.,
$$
\tikzmath{
\draw (-.7,0) -- node[above]{$\scriptstyle n$} (.7,0);
}
\qquad
=
\tikzmath{
\draw (-.7,0) -- node[above]{$\scriptstyle X^{\alt\boxtimes n}$} (.7,0);
}
\qquad
=
\qquad
\tikzmath{
\draw (-.7,-1) -- node[above]{$\scriptstyle X$} (.7,-1);
\draw (-.7,-.5) -- node[above]{$\scriptstyle \overline{X}$} (.7,-.5);
\draw (-.7,0) -- node[above]{$\scriptstyle X$} (.7,0);
\node at (0,.7) {$\scriptstyle \vdots$};
\draw (-.7,1) -- node[above]{$\scriptstyle X^?$} (.7,1);
}\,.
$$
\end{ex}

%%%%%%%%%%%%%%%%%%%%%%%%%%%%%%%%%%%%%%%%%%%%%%%%%%%%%%%%%%%%%%%%%%%%%%%
\subsection{The inductive limit factor as a realization}
\label{sec:InductiveLimit}

We now give a graphical representation of the inductive limit $\rm II_1$ factor $M_\infty$ from the realized Jones tower from Example \ref{ex:RealizedJonesTower}.
We begin with a short remark about inductive limits in the category of tracial von Neumann algebras and trace-preserving unital $*$-homomorphisms, followed by a brief review of construction of the hyperfinite $\rm II_1$ subfactor $R\subseteq P$ with the opposite standard invariant as our finite depth, finite index $\rm II_1$ subfactor $N\subseteq M$.

\begin{rem}
Consider the category whose objects are pairs $(M,\tr)$  where $M$ is a separable von Neumann algebra and $\tr$ is a faithful normal tracial state, and whose morphisms are trace-preserving unital $*$-homomorphisms (which are automatically normal by \cite[III.2.2.2]{MR2188261} and \cite[Prop.~9.1.1]{JonesVNA}).
It is well known that this category admits inductive limits.
%for example, see \cite[Prop.~5.5 and 5.7]{MR3729731} for arbitrary von Neumann algebras with unital $*$-homomorphisms.
We briefly recall the construction for completeness and convenience of the reader. 

For an increasing sequence of tracial von Neumann algebras $(M_n,\tr_n)$, we get a tracial von Neumann algebra $\varinjlim M_n$ by taking the bicommutant of $\bigcup_n M_n$ on its GNS Hilbert space with respect to the limit trace, which is equipped with the faithful tracial state $\varinjlim \tr := \langle \,\cdot\, \Omega, \Omega\rangle_{L^2\bigcup_n M_n}$.
For any tracial von Neumann algebra $(N,\tr)$ and trace-preserving maps $\varphi_n :M_n \to N$ compatible with the inclusions, we get a unique trace-preserving map $\varphi:\bigcup_n M_n \to N$,
which extends uniquely to a trace-preserving map $\varinjlim M_n\to N$.

Indeed, for a fixed $x\in \varinjlim M_n$, let $\xi \in L^2N$ be the image of $x\Omega$ under the induced map of GNS spaces $L^2(\bigcup_n M_n) \to L^2N$ given by the extension of $m\Omega \mapsto \varphi(m)\Omega$.
Since there is a bounded sequence $(x_k)\subseteq \bigcup_n M_n$ with $x_k \to x$ in $\|\cdot\|_2$, we see
\begin{align*}
\|L_\xi n\Omega\|_2 
&= 
\|\xi n \|_2 
= 
\| Jn^*J \xi\|_2 
= 
\lim \|Jn^*J \varphi(x_k)\Omega\|_2
=
\lim \|\varphi(x_k) n\Omega\|_2
\\&\leq
\limsup \|\varphi(x_k)\|\cdot \|n\Omega\|_2.
=
\limsup \|x_k\|\cdot \|n\Omega\|_2.
\end{align*}
Thus $\xi$ is $N$-bounded, and necessarily of the form $\varphi(x)\Omega$ for some $\varphi(x)\in N$.
Since multiplication is jointly SOT-continuous on bounded subsets 
and $*$ is SOT-continuous on bounded subsets
of $\varinjlim M_n$,
it is easily verified that the above definition of $\varphi(x)$ extends $\varphi$ to a unital $*$-homomorphism $\varinjlim M_n \to N$.
\end{rem}

For the rest of this section, we fix a finite depth finite index $\rm II_1$ subfactor $N\subseteq M$.
Recall that its standard invariant can also be described as the subfactor planar algebra $\cP_\bullet$ whose box spaces are given by
$$
\cP_{k,+}:= \End(X^{\alt\boxtimes k})
\qquad\text{and}\qquad
\cP_{k,-}:= \End(\overline{X}^{\alt\boxtimes k}).
$$
These finite dimensional von Neumann algebras are equipped with the canonical traces which agree with the categorical traces on $\cC(N\subseteq M)$.

We now rapidly recall how to construct a hyperfinite $\rm II_1$ subfactor with the opposite standard invariant \cite{MR996454,MR1278111}.
For a detailed diagrammatic exposition (in the multifactor setting), see \cite[\S5.1]{2010.01067}.
Since $N\subseteq M$ is finite depth, the inductive limit tracial von Neumann algebras
$\varinjlim \cP_{k,\pm}$ are hyperfinite $\rm II_1$ factor.
We have a trace-preserving injection $\cP_{n,+}\hookrightarrow \cP_{n+1,-}$ by adding an $\overline{X}$ strand to the left, giving a $\rm II_1$ subfactor
$$
R := \varinjlim \cP_{n,+} \subseteq \varinjlim \cP_{n,-} =:P.
$$
It is well known that by the Ocneanu Compactness Theorem \cite[Thm.~5.7.6]{MR1473221}, the inclusion $R^{\op}\subseteq P^{\op}$ has the same standard invariant as $N\subseteq M$ \cite{MR996454,MR1278111}, and thus $R\subseteq P$ has the opposite standard invariant.

Letting $\cC:={}_N\cC_N$ be the unitary fusion category generated by ${}_NM_N$, the above construction gives a fully faithful unitary tensor functor $F:\cC^{\op}\to \fgpBim(R)$.
By \cite[Prop.~3.2]{MR1424954}, simple objects $c\in \Irr(\cC)$ correspond to minimal projections $p\in M_0'\cap M_{2n}$ under the correspondence
$$
M_0'\cap M_{2n} \ni p \longmapsto pM_n \in \fgpBim(N).
$$
Now consider the Jones tower
$$
R_0 = R \subseteq P = R_1 \subseteq R_2 \subseteq  R_3\subseteq \cdots 
$$
of our hyperfinite $\rm II_1$ subfactor $R\subseteq P$.
Simple objects $c^{\op}\in \Irr(\cC^{\op})$ correspond to the opposite projections $p^{\op}\in R_0'\cap R_{2n} \cong (M_0'\cap M_{2n})^{\op}$, which corresponds to the bimodule $p^{\op}R_n$.
As $R_n$ is also isomorphic to $\varinjlim \cP_{k,\pm}$ (here, $\pm$ depends on the parity of $n$), we can realize the bimodule $p^{\op}R_n$ graphically as an inductive limit:
\begin{equation}
\label{eq:HyperfiniteAction}
F(c^{\op})
=
p^{\op}R_n
\cong
\varinjlim
\cC(
c\boxtimes X^{\alt\boxtimes k}\to X^{\alt\boxtimes k}
)
\end{equation}
under the isometric right $R_n$-inner product preserving inclusions
$$
\tikzmath{
\draw (-.7,0) -- (.7,0);
\draw (0,0) -- (0,-.7);
\roundNbox{unshaded}{(0,0)}{.3}{0}{0}{$f$};
\node at (-.9,0) {\scriptsize{$k$}};
\node at (.9,0) {\scriptsize{$k$}};
\node at (.2,-.5) {\scriptsize{$c$}};
}
\longmapsto
\tikzmath{
\draw (-.7,0) -- (.7,0);
\draw (-.7,.5) -- (.7,.5);
\draw (0,0) -- (0,-.7);
\roundNbox{unshaded}{(0,0)}{.3}{0}{0}{$f$};
\node at (-.9,0) {\scriptsize{$k$}};
\node at (.9,0) {\scriptsize{$k$}};
\node at (.2,-.5) {\scriptsize{$c$}};
}\,.
$$
Indeed, $p^{\op}R_n\cong p^{\op}R_{n}f_n\subseteq p^{\op}R_{2n}f_n$, where $f_n$ is the multistep Jones projection \cite{MR965748,MR1424954}.
In diagrams, for $\xi \in \cC(c\boxtimes X^{\alt\boxtimes k+n}\to X^{\alt\boxtimes k+n}) \subseteq R_n \hookrightarrow R_{2n}$, we have 
$$
f_n = 
\frac{1}{[M:N]^{n/2}}
\tikzmath{
\draw (-.5,.3) node[above]{$\scriptstyle n$} arc (90:-90:.3cm);
\draw (.5,.3) node[above]{$\scriptstyle n$} arc (90:270:.3cm);
}
\qquad\qquad
\Longrightarrow
\qquad\qquad
p^{\op} \xi f_n
=
\frac{1}{[M:N]^{n/2}}
\tikzmath{
\draw[rounded corners=5, dashed] (-1.5,-1) rectangle (.7,.6);
\draw (-1.7,.2) node[left]{$\scriptstyle k$} -- (1.1,.2) node[right]{$\scriptstyle k$};
\draw (-.7,-.2) -- node[above]{$\scriptstyle n$} (-.3,-.2);
\draw (.3,-.2) arc (90:-90:.2cm) -- node[below]{$\scriptstyle n$} (-.7,-.6);
\draw (1.1,-.2) node[right]{$\scriptstyle n$} arc (90:270:.2cm) node[right]{$\scriptstyle n$};
\draw (-1.3,-.4) -- (-1.7,-.4) node[left]{$\scriptstyle 2n$};
\roundNbox{unshaded}{(0,0)}{.4}{-.1}{-.1}{$\xi$};
\roundNbox{unshaded}{(-1,-.4)}{.4}{-.1}{-.1}{$p$};
}\,.
$$
It is now visibly evident how to implement the isomorphism \eqref{eq:HyperfiniteAction}.

Now, since we have two subfactors $N\subseteq M$ and $R\subseteq P$ with opposite standard invariants, we get two fully faithful unitary tensor functors $F: \cC^{\op}\to \fgpBim(R)$ and $G: \cC\to \fgpBim(N)$.
Consider the \emph{symmetric enveloping Q-system} 
$S:= \bigoplus_{c\in \Irr(\cC)} c^{\op}\boxtimes c \in \cC^{\op}\boxtimes \cC$, which is simple as an $S-S$ bimodule, giving the realized 
$\rm II_1$ factor $|(F\boxtimes G)(S)|$.

\begin{prop}
\label{prop:SymmetricEvelopeInductiveLimit}
The realization $|(F\boxtimes G)(S)|$ is $*$-isomorphic to the inductive limit $\rm II_1$ factor $M_\infty = \varinjlim M_n$.
\end{prop}
\begin{proof}
Observe that by the Yoneda Lemma, we have canonical isomorphisms
\begin{align*}
|(F\boxtimes G)(S)|
&\cong 
\bigoplus\limits_{a,b\in\Irr(\cC)} (\cC^\op\boxtimes\cC)(a^\op\boxtimes b\to S)\otimes_{\bbC}(F\boxtimes G)(a^\op\boxtimes b) 
\\&=
\bigoplus\limits_{c\in\Irr(\cC)} (\cC^\op\boxtimes\cC)(c^\op\boxtimes c\to S)\otimes_{\bbC}(F\boxtimes G)(c^\op\boxtimes c) 
\\&\cong
\bigoplus\limits_{c\in\Irr(\cC)} (F\boxtimes G)(c^\op\boxtimes c)  
\\&= 
\bigoplus\limits_{c\in\Irr(\cC)} F(c^{\op}) \otimes_{\bbC} G(c)
\end{align*}
Thus $|(F\boxtimes G)(S)|$ is $*$-isomorphic to the more general coend realization of $F^\circ$ and $G^\circ$ from \cite[\S4]{MR3948170} (see also \cite[Ex.~5.33]{MR3687214}), where as above, $\circ$ denotes taking the underlying vector space.
Using this identification,
\[
\bigoplus\limits_{c\in\Irr(\cC)} F(c^{\op}) \otimes_{\bbC} G(c)
\cong
\varinjlim
\bigoplus\limits_{c\in\Irr(\cC)} \cC(c\boxtimes X^{\alt\boxtimes k} \to X^{\alt\boxtimes k})
\otimes_{\bbC} G(c)
\underset{\text{\eqref{eq:CoendForMn}}}{\cong}
\varinjlim M_n.
\qedhere
\]
\end{proof}

%%%%%%%%%%%%%%%%%%%%%%%%%%%%%%%%%%%%%%%%%%%%%%%%%%%%%%%%%%%%%%%%%%%%%%%
\subsection{Embedding \texorpdfstring{$\cZ(\cC)$}{Z(C)} into \texorpdfstring{$\Chi(M_\infty)$}{chi(Minfty)} when \texorpdfstring{$N\subseteq M$}{N in M} is finite depth}

As in \S\ref{sec:InductiveLimit} above, for this section, we fix a finite depth, finite index $\rm II_1$ subfactor $N\subseteq M$, and let $\cC={}_N\cC_N$ denote the unitary fusion category of $N-N$ bimodules generated by ${}_NM_N$.
In this section, we extend this bijection to a unitary tensor functor $\Phi:\cZ(\cC)\to \fgpBim(M_\infty)$ such that when $N$ is non-Gamma, $\Phi$ takes values in $\Chi(M_\infty)$ as is braided.
To do so, we use the description of the inductive limit $\rm II_1$ factor $M_\infty=|(F\boxtimes G)(S)|$ obtained from iterating the Jones basic construction afforded by Proposition \ref{prop:SymmetricEvelopeInductiveLimit}
for the fully faithful unitary tensor functors
$F: \cC^{\op}\to \fgpBim(R)$ from \eqref{eq:HyperfiniteAction} and $G: \cC\to \fgpBim(N)$ associated to the subfactor $N\subset M$.

First, applying \cite[Cor.~C]{2105.12010} in the $\rm W^*$ setting to the fully faithful unitary tensor functor $F\boxtimes G : \cC^{\op}\boxtimes \cC\to \fgpBim(R\otimes N)$, bimodule realization gives a fully faithful tensor functor from $\Bim_{\cC^{\op}\boxtimes \cC}(S) \to \fgpBim(M_\infty)$.
(See also Remark \ref{rem:QBimodules} above.)
Explicitly, on an $S-S$ bimodule
$X
=
\bigoplus_{a,b\in \Irr(\cC)} X_{ab} \otimes (a^{\op} \boxtimes b)$,
we have
$$
|X|
=
\Hom\left(
(R\otimes N)_{R\otimes N} 
\to 
\bigoplus_{a,b\in \Irr(\cC)} X_{ab} \otimes \big(F(a^{\op}) \otimes G(b)\big)_{R\otimes N}
\right)
\cong
\bigoplus_{a,b\in \Irr(\cC)}
X_{ab}\otimes F(a^{\op}) \otimes G(b).
$$
Using the well-known equivalence 
$\Bim_{\cC^{\op}\boxtimes \cC}(S)\cong \cZ(\cC)$
from Example \ref{ex:MoritaEquivalenceBetweenCopCandZ(C)},
we get the following proposition.

\begin{prop}
Bimodule realization gives a fully faithful unitary tensor functor $\Phi:\cZ(\cC) \to \fgpBim(M_\infty)$.
\end{prop}

We now want an explicit model for $\Phi(z,\sigma_z)$ for each object $(z,\sigma_z)\in \cZ(\cC)$.
To do so, we give an explicit description of the $S-S$ bimodule $X_z \in \Bim_{\cC^{\op}\boxtimes \cC}(S)$ under the unitary tensor equivalence.

\begin{defn}
Given $(z,\sigma_z)\in \cZ(\cC)$, we define $X_z:=\bigoplus_{a,b\in\Irr(\cC)} \cC(b\to z\otimes a) \otimes (a^\op\boxtimes b)$.
The left $S$-module structure of $X_z$ is given as follows.
First, we observe
\begin{align*}
S\otimes_{\cC^{\op}\boxtimes \cC} X_z
&= 
\bigoplus_{a,b,c\in\Irr(\cC)}
\cC(b\to z\otimes a) \otimes 
( (c^\op\otimes a^\op)\boxtimes (c\otimes b) )
\\
&= 
\bigoplus_{a,b,c,d,e\in\Irr(\cC)} 
\cC(b\to z\otimes a) \otimes \cC^\op(d^\op\to c^\op\otimes a^\op) \otimes \cC(e\to c\otimes b) \otimes
(d^\op\boxtimes e)
\\
&= 
\bigoplus_{a,b,c,d,e\in\Irr(\cC)} 
\cC(b\to z\otimes a) \otimes \cC(c\otimes a\to d) \otimes \cC(e\to c\otimes b) \otimes
(d^\op\boxtimes e)
\end{align*}
The left action map is given component-wise by
$$
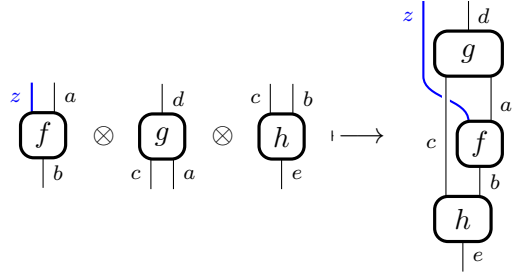
\tikzmath{
\draw[thick, \zColor] (-.15,.3) -- node[left]{$\scriptstyle z$} (-.15,.7);
\draw (.15,.3) -- node[right]{$\scriptstyle a$} (.15,.7);
\draw (0,-.7) -- node[right]{$\scriptstyle b$} (0,-.3);
\roundNbox{fill=white}{(0,0)}{.3}{0}{0}{$f$}
}
\otimes
\tikzmath{
\draw (-.15,-.3) -- node[left]{$\scriptstyle c$} (-.15,-.7);
\draw (.15,-.3) -- node[right]{$\scriptstyle a$} (.15,-.7);
\draw (0,.7) -- node[right]{$\scriptstyle d$} (0,.3);
\roundNbox{fill=white}{(0,0)}{.3}{0}{0}{$g$}
}
\otimes
\tikzmath{
\draw (-.15,.3) -- node[left]{$\scriptstyle c$} (-.15,.7);
\draw (.15,.3) -- node[right]{$\scriptstyle b$} (.15,.7);
\draw (0,-.7) -- node[right]{$\scriptstyle e$} (0,-.3);
\roundNbox{fill=white}{(0,0)}{.3}{0}{0}{$h$}
}
\longmapsto
\tikzmath{
\draw (0,1) -- (0,1.9);
\draw (-.075,-1.7) -- (-.075,-1);
\draw (.3,0) -- (.3,.9);
\draw (.15,-.7) -- (.15,0);
\draw[\zColor,thick] (0,.3) .. controls ++(90:.3cm) and ++(270:.3cm) .. (-.6,.9) -- (-.6,1.9);
\draw[knot] (-.3,-.7) -- (-.3,.9);
\roundNbox{fill=white}{(0,1.2)}{.3}{.15}{.15}{$g$};
\roundNbox{fill=white}{(-.075,-1)}{.3}{.075}{.075}{$h$};
\roundNbox{fill=white}{(.15,0)}{.3}{0}{0}{$f$};
\node at (.125,-1.5) {\scriptsize{$e$}};
\node at (-.5,0) {\scriptsize{$c$}};
\node at (.35,-.5) {\scriptsize{$b$}};
\node at (.5,.5) {\scriptsize{$a$}};
\node[\zColor] at (-.8,1.7) {\scriptsize{$z$}};
\node at (.2,1.7) {\scriptsize{$d$}};
}
$$

The right $S$-module structure is defined similarly.
Observe
\begin{align*}
X_z\otimes_{\cC^{\op}\boxtimes \cC} S 
&= 
\bigoplus_{a,b,c\in\Irr(\cC)} 
\cC(b\to z\otimes a) 
\otimes 
(a^\op\otimes c^\op)\boxtimes (b\otimes c) 
\\
&= 
\bigoplus_{a,b,c,d,e\in\Irr(\cC)} 
\cC(b\to z\otimes a)\otimes \cC(a\otimes c\to d)\otimes \cC(e\to b\otimes c) \otimes
(d^\op\boxtimes e),
\end{align*}
and the right action map is given component-wise by
$$
\tikzmath{
\draw[thick, \zColor] (-.15,.3) -- node[left]{$\scriptstyle z$} (-.15,.7);
\draw (.15,.3) -- node[right]{$\scriptstyle a$} (.15,.7);
\draw (0,-.7) -- node[right]{$\scriptstyle b$} (0,-.3);
\roundNbox{fill=white}{(0,0)}{.3}{0}{0}{$f$};
}
\otimes
\tikzmath{
\draw (-.15,-.3) -- node[left]{$\scriptstyle a$} (-.15,-.7);
\draw (.15,-.3) -- node[right]{$\scriptstyle c$} (.15,-.7);
\draw (0,.7) -- node[right]{$\scriptstyle d$} (0,.3);
\roundNbox{fill=white}{(0,0)}{.3}{0}{0}{$g$};
}
\otimes
\tikzmath{
\draw (-.15,.3) -- node[left]{$\scriptstyle b$} (-.15,.7);
\draw (.15,.3) -- node[right]{$\scriptstyle c$} (.15,.7);
\draw (0,-.7) -- node[right]{$\scriptstyle e$} (0,-.3);
\roundNbox{fill=white}{(0,0)}{.3}{0}{0}{$h$};
}
\longmapsto
\tikzmath{
\draw (0,0) -- (0,.7);
\draw (-.2,-.7) -- (-.2,0);
\draw (.4,-.7) -- (.4,.7);
\draw (.2,1) -- (.2,1.7);
\draw (.1,-1.7) -- (.1,-1);
\draw[thick, \zColor] (-.4,0) -- (-.4,1.7);
\roundNbox{fill=white}{(.2,1)}{.3}{.1}{.1}{$g$};
\roundNbox{fill=white}{(-.2,0)}{.3}{.1}{.1}{$f$};
\roundNbox{fill=white}{(.1,-1)}{.3}{.2}{.2}{$h$};
\node at (.3,-1.5) {\scriptsize{$e$}};
\node at (-.4,-.5) {\scriptsize{$b$}};
\node at (.6,0) {\scriptsize{$c$}};
\node at (-.2,.5) {\scriptsize{$a$}};
\node[\zColor] at (-.6,1.5) {\scriptsize{$z$}};
\node at (.4,1.5) {\scriptsize{$d$}};
}
$$
\end{defn}

We now describe the realization 
\begin{equation}
\label{eq:RealizationXz}
|X_z|
=
\varinjlim
\bigoplus_{a,b\in \Irr(\cC)}
\cC(b\to z\otimes a)
\otimes_\bbC
\cC(a\boxtimes X^{\alt \boxtimes n}\to X^{\alt \boxtimes n})
\otimes_\bbC
G(b).
\end{equation}
First, note that by semisimplicity and Frobenius reciprocity, we can alternatively describe the first two tensorands in the direct sum for $|X_z|$ as 
\begin{equation}
\label{eq:SemisimpleIso}
\left.
\begin{aligned}
\bigoplus_{a\in \Irr(\cC)}
\cC(b\to z\otimes a)
&\otimes_\bbC
\cC(a\boxtimes X^{\alt \boxtimes n} \to X^{\alt \boxtimes n})
\\&\cong
\bigoplus_{a\in \Irr(\cC)}
\cC( \overline{z}\boxtimes b\to a)
\otimes_\bbC
\cC(a \to X^{\alt \boxtimes n}\boxtimes \overline{X^{\alt \boxtimes n}})
\\&\cong
\cC(\overline{z}\boxtimes b 
\to
X^{\alt \boxtimes n}\boxtimes \overline{X^{\alt \boxtimes n}})
\\&\cong
\cC(
b\boxtimes X^{\alt \boxtimes n}
\to
z\boxtimes 
X^{\alt \boxtimes n})
\end{aligned}
\right\}
\end{equation}
We may thus graphically represent elements in a dense subspace of $|X_z|$ as
$$
\bigoplus_{b\in\Irr(\cC)}
\tikzmath{
\begin{scope}
\clip[rounded corners = 5] (-.7,0) -- (-.7,-1.4) -- (.7,-1.4) -- (.7,0);
\filldraw[\rColor] (-.7,0) -- (-.7,-1.4) -- (.7,-1.4) -- (.7,0);
\end{scope}
\draw[dotted] (-.7,0) -- (.7,0);
\draw[\NColor,thick] (0,-1.4) -- (0,-.7);
\draw (0,-.7) -- (0,.7); 
\draw (-.7,.95) node[left]{$\scriptstyle n$} -- (.7,.95) node[right]{$\scriptstyle n$};
\draw[thick, \zColor] (-.7,.65) node[left]{$\scriptstyle z$} -- (0,.65);
\draw[cyan, dashed] (-.7,.1) -- (.7,1.5);
\draw[cyan, ->] (-.6,.1) -- (-.8,.3);
\draw[cyan, ->] (.7,1.4) -- (.5,1.6);
\roundNbox{unshaded}{(0,-.7)}{.3}{0}{0}{$\xi_b$};
\roundNbox{unshaded}{(0,.8)}{.35}{-.05}{-.05}{$f_b$};
\node at (.2,.2) {\scriptsize{$b$}};
\node at (.2,-.2) {\scriptsize{$b$}};
}
\in 
\varinjlim
\bigoplus_{b\in \Irr(\cC)}
\cC(b\boxtimes X^{\alt \boxtimes n} \to z\boxtimes X^{\alt \boxtimes n})
\otimes_\bbC 
G(b)
$$
where as in \eqref{eq:BottomRightTopLeft}, the horizontal line in the top half of the diagram should be viewed as slightly tilted going from the bottom right to the top left, as indicated by the cyan arrows above.
The right $M_\infty$-action on $|X_z|$ is the obvious diagrammatic one, and the left one is similar, but uses the half-braiding for $z$: 
\[
(f_a\otimes\xi_a)\rhd (g_b\otimes \eta_b) 
:=
\sum_{\substack{
c\in \Irr(\cC)
\\
\alpha\in\ONB(a\otimes b\to c)
}}
\tikzmath{
\begin{scope}
\clip[rounded corners = 5] (-1.2,0) -- (-1.2,-2.2) -- (1.2,-2.2) -- (1.2,0);
\filldraw[\rColor] (-1.2,0) -- (-1.2,-2.2) -- (1.2,-2.2) -- (1.2,0);
\end{scope}
\draw[dotted] (-1.2,0) -- (1.2,0);
\draw[\NColor,thick] (-.5,-1.4) arc (-180:0:.5cm);
\draw[\NColor, thick] (0,-1.9) -- (0,-2.2);
\filldraw[\NColor] (0,-1.9) circle (.05cm);
\draw[thick, \zColor] (.2,1.1) .. controls ++(180:.3cm) and ++(0:.3cm) .. (-.4,.7) -- (-1.2,.7) node[above, xshift=.2cm]{$\scriptstyle z$};
\draw (-1.2,1.3) node[above, xshift=.2cm]{$\scriptstyle n$} -- node[above]{$\scriptstyle n$} (1.2,1.3) node[above, xshift=-.2cm]{$\scriptstyle n$};
\draw[knot] (.5,.8) arc (0:-180:.5cm) -- (-.5,1);
\draw (.5,-.8) arc (0:180:.5cm);
\draw (0,-.3) -- (0,.3); 
\filldraw (0,-.3) circle (.05cm);
\filldraw (0,.3) circle (.05cm);
\roundNbox{unshaded}{(.5,-1.1)}{.3}{0}{0}{$\eta_b$};
\roundNbox{unshaded}{(.5,1.2)}{.4}{-.1}{-.1}{$g_b$};
\roundNbox{unshaded}{(-.5,-1.1)}{.3}{0}{0}{$\xi_a$};
\roundNbox{unshaded}{(-.5,1.3)}{.3}{0}{0}{$f_a$};
\node at (.7,.6) {\scriptsize{$b$}};
\node at (.7,-.6) {\scriptsize{$b$}};
\node at (-.7,.5) {\scriptsize{$a$}};
\node at (-.7,-.6) {\scriptsize{$a$}};
\node at (.2,.2) {\scriptsize{$c$}};
\node at (.2,-.2) {\scriptsize{$c$}};
\node at (0,.5) {\scriptsize{$\alpha^\dag$}};
\node at (0,-.5) {\scriptsize{$\alpha$}};
}
\,.
\]
The $M_\infty$-valued inner product of $|X_z|$ is given by
\[
\langle f_a\otimes \xi_a|g_b\otimes \eta_b\rangle_{M_\infty}^{|X_z|}
:=
\sum_{\substack{
c\in \Irr(\cC)
\\
\alpha\in\ONB(a\otimes b\to c)
}}
\tikzmath{
\begin{scope}
\clip[rounded corners = 5] (-1.2,0) -- (-1.2,-2.2) -- (1.2,-2.2) -- (1.2,0);
\filldraw[\rColor] (-1.2,0) -- (-1.2,-2.2) -- (1.2,-2.2) -- (1.2,0);
\end{scope}
\draw[dotted] (-1.2,0) -- (1.2,0);
\draw[\NColor,thick] (-.5,-1.4) arc (-180:0:.5cm);
\draw[\NColor, thick] (0,-1.9) -- (0,-2.2);
\filldraw[\NColor] (0,-1.9) circle (.05cm);
\draw[thick, \zColor] (.2,1.1) -- node[below]{$\scriptstyle z$} (-.2,1.1);
\draw (-1.2,1.3) node[above, xshift=.2cm]{$\scriptstyle n$} -- node[above]{$\scriptstyle n$} (1.2,1.3) node[above, xshift=-.2cm]{$\scriptstyle n$};
\draw (.5,.8) arc (0:-180:.5cm);
\draw (.5,-.8) arc (0:180:.5cm);
\draw (0,-.3) -- (0,.3); 
\filldraw (0,-.3) circle (.05cm);
\filldraw (0,.3) circle (.05cm);
\roundNbox{unshaded}{(.5,-1.1)}{.3}{0}{0}{$\eta_b$};
\roundNbox{unshaded}{(.5,1.2)}{.4}{-.1}{-.1}{$g_b$};
\roundNbox{unshaded}{(-.5,-1.1)}{.3}{0}{0}{$\overline{\xi_a}$};
\roundNbox{unshaded}{(-.5,1.2)}{.4}{-.1}{-.1}{$\overline{f_a}$};
\node at (.7,.6) {\scriptsize{$b$}};
\node at (.7,-.6) {\scriptsize{$b$}};
\node at (-.7,.6) {\scriptsize{$a$}};
\node at (-.7,-.6) {\scriptsize{$a$}};
\node at (.2,.2) {\scriptsize{$c$}};
\node at (.2,-.2) {\scriptsize{$c$}};
\node at (0,.5) {\scriptsize{$\alpha^\dag$}};
\node at (0,-.5) {\scriptsize{$\alpha$}};
}
\,.
\]
By the definition of the realization \eqref{eq:RealizationXz} for $|X_z|$ and $|X_w|$
and the right and left $S$-action on $X_z$ and $X_w$ respectively, the tensorator $\Phi^2_{z,w}$ is given on $|X_z|\boxtimes_{M_\infty} |X_w|$ by
$$
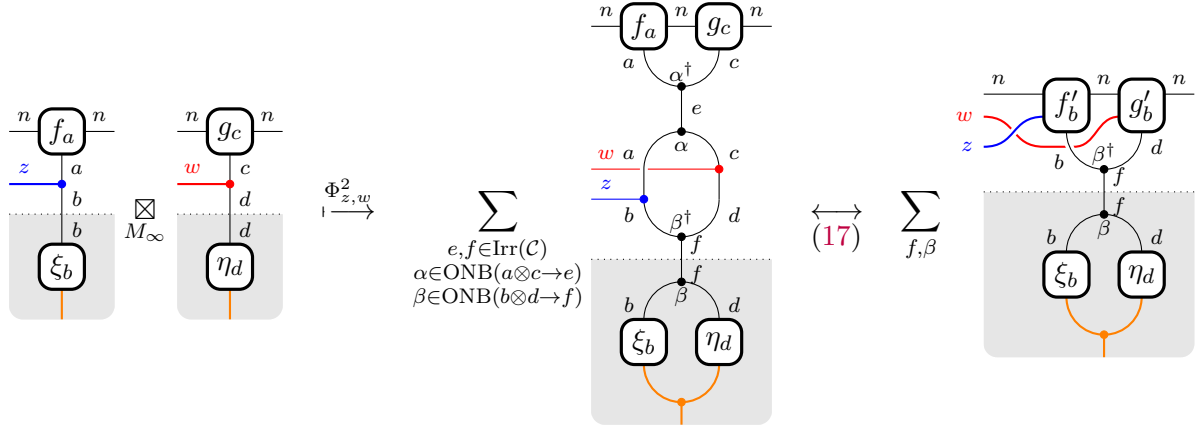
\tikzmath{
\begin{scope}
\clip[rounded corners = 5] (-.7,0) -- (-.7,-1.4) -- (.7,-1.4) -- (.7,0);
\filldraw[\rColor] (-.7,0) -- (-.7,-1.4) -- (.7,-1.4) -- (.7,0);
\end{scope}
\draw[dotted] (-.7,0) -- (.7,0);
\draw[\NColor,thick] (0,-1.4) -- (0,-.7);
\draw (0,-.7) -- (0,.8); 
\draw (-.7,1.1) node[above, xshift=.2cm]{$\scriptstyle n$} -- (.7,1.1) node[above, xshift=-.2cm]{$\scriptstyle n$};
\draw[\zColor,thick] (0,.4) -- (-.7,.4) node[above, xshift=.2cm]{$\scriptstyle z$};
\filldraw[\zColor] (0,.4) circle (.05cm);
\roundNbox{unshaded}{(0,1.1)}{.3}{0}{0}{$f_a$};
\roundNbox{unshaded}{(0,-.7)}{.3}{0}{0}{$\xi_b$};
\node at (.2,.6) {$\scriptstyle a$};
\node at (.2,.2) {$\scriptstyle b$};
\node at (.2,-.2) {$\scriptstyle b$};
}
\underset{M_\infty}{\boxtimes}
\tikzmath{
\begin{scope}
\clip[rounded corners = 5] (-.7,0) -- (-.7,-1.4) -- (.7,-1.4) -- (.7,0);
\filldraw[\rColor] (-.7,0) -- (-.7,-1.4) -- (.7,-1.4) -- (.7,0);
\end{scope}
\draw[dotted] (-.7,0) -- (.7,0);
\draw[\NColor,thick] (0,-1.4) -- (0,-.7);
\draw (0,-.7) -- (0,.8); 
\draw (-.7,1.1) node[above, xshift=.2cm]{$\scriptstyle n$} -- (.7,1.1) node[above, xshift=-.2cm]{$\scriptstyle n$};
\draw[\wColor,thick] (0,.4) -- (-.7,.4) node[above, xshift=.2cm]{$\scriptstyle w$};
\filldraw[\wColor] (0,.4) circle (.05cm);
\roundNbox{unshaded}{(0,1.1)}{.3}{0}{0}{$g_c$};
\roundNbox{unshaded}{(0,-.7)}{.3}{0}{0}{$\eta_d$};
\node at (.2,.6) {$\scriptstyle c$};
\node at (.2,.2) {$\scriptstyle d$};
\node at (.2,-.2) {$\scriptstyle d$};
}
\quad
\overset{\Phi^2_{z,w}}{\longmapsto}
\quad
\sum_{\substack{
e,f\in \Irr(\cC)
\\
\alpha\in\ONB(a\otimes c\to e)
\\
\beta\in\ONB(b\otimes d\to f)
}}
\tikzmath{
\begin{scope}
\clip[rounded corners = 5] (-1.2,0) -- (-1.2,-2.2) -- (1.2,-2.2) -- (1.2,0);
\filldraw[\rColor] (-1.2,0) -- (-1.2,-2.2) -- (1.2,-2.2) -- (1.2,0);
\end{scope}
\draw[dotted] (-1.2,0) -- (1.2,0);
\draw[\NColor,thick] (-.5,-1.4) arc (-180:0:.5cm);
\draw[\NColor, thick] (0,-1.9) -- (0,-2.2);
\filldraw[\NColor] (0,-1.9) circle (.05cm);
\draw (0,-.3) -- (0,.3); 
\draw (.5,.8) arc (0:-180:.5cm);
\draw (.5,-.8) arc (0:180:.5cm);
\filldraw (0,-.3) circle (.05cm);
\filldraw (0,.3) circle (.05cm);
\draw[\zColor] (-.5,.8) -- (-1.2,.8) node[above, xshift=.2cm]{$\scriptstyle z$};
\draw[\wColor] (.5,1.2) -- (-1.2,1.2) node[above, xshift=.2cm]{$\scriptstyle w$};
\draw[knot] (-.5,.8) -- (-.5,1.2) arc (180:0:.5cm) -- (.5,.8);
\filldraw[\zColor] (-.5,.8) circle (.05cm);
\filldraw[\wColor] (.5,1.2) circle (.05cm);
\draw (0,1.7) -- (0,2.3);
\draw (-.5,2.8) arc (-180:0:.5cm);
\filldraw (0,1.7) circle (.05cm);
\filldraw (0,2.3) circle (.05cm);
\draw (-1.2,3.1) node[above, xshift=.2cm]{$\scriptstyle n$} -- node[above]{$\scriptstyle n$} (1.2,3.1) node[above, xshift=-.2cm]{$\scriptstyle n$};
\roundNbox{unshaded}{(.5,-1.1)}{.3}{0}{0}{$\eta_d$};
\roundNbox{unshaded}{(.5,3.1)}{.3}{0}{0}{$g_c$};
\roundNbox{unshaded}{(-.5,-1.1)}{.3}{0}{0}{$\xi_b$};
\roundNbox{unshaded}{(-.5,3.1)}{.3}{0}{0}{$f_a$};
\node at (.7,2.6) {\scriptsize{$c$}};
\node at (.7,1.4) {\scriptsize{$c$}};
\node at (-.7,2.6) {\scriptsize{$a$}};
\node at (-.7,1.4) {\scriptsize{$a$}};
\node at (.2,2) {\scriptsize{$e$}};
\node at (.7,.6) {\scriptsize{$d$}};
\node at (.7,-.6) {\scriptsize{$d$}};
\node at (-.7,.6) {\scriptsize{$b$}};
\node at (-.7,-.6) {\scriptsize{$b$}};
\node at (.2,.2) {\scriptsize{$f$}};
\node at (.2,-.2) {\scriptsize{$f$}};
\node at (0,.5) {\scriptsize{$\beta^\dag$}};
\node at (0,-.5) {\scriptsize{$\beta$}};
\node at (0,2.5) {\scriptsize{$\alpha^\dag$}};
\node at (0,1.5) {\scriptsize{$\alpha$}};
}
\quad
\underset{\text{\eqref{eq:SemisimpleIso}}}{\longleftrightarrow}
\quad
\sum_{f,\beta}
\tikzmath{
\begin{scope}
\clip[rounded corners = 5] (-1.6,0) -- (-1.6,-2.2) -- (1.2,-2.2) -- (1.2,0);
\filldraw[\rColor] (-1.6,0) -- (-1.6,-2.2) -- (1.2,-2.2) -- (1.2,0);
\end{scope}
\draw[dotted] (-1.6,0) -- (1.2,0);
\draw[\NColor,thick] (-.5,-1.4) arc (-180:0:.5cm);
\draw[\NColor, thick] (0,-1.9) -- (0,-2.2);
\filldraw[\NColor] (0,-1.9) circle (.05cm);
\draw[thick, \wColor] (.2,1) .. controls ++(180:.3cm) and ++(0:.3cm) .. (-.4,.6) -- (-.8,.6) .. controls ++(180:.45cm) and ++(0:.45cm) .. (-1.6,1) node[left]{$\scriptstyle w$};
\draw[thick, \zColor, knot] (-.8,1) .. controls ++(180:.45cm) and ++(0:.45cm) .. (-1.6,.6) node[left]{$\scriptstyle z$};
\draw (-1.6,1.3) node[above, xshift=.2cm]{$\scriptstyle n$} -- (1.2,1.3) node[above, xshift=-.2cm]{$\scriptstyle n$};
\draw[knot] (.5,.8) arc (0:-180:.5cm) -- (-.5,1);
\draw (.5,-.8) arc (0:180:.5cm);
\draw (0,-.3) -- (0,.3); 
\filldraw (0,-.3) circle (.05cm);
\filldraw (0,.3) circle (.05cm);
\roundNbox{unshaded}{(.5,-1.1)}{.3}{0}{0}{$\eta_d$};
\roundNbox{unshaded}{(.5,1.15)}{.35}{-.05}{-.05}{$g_b'$};
\roundNbox{unshaded}{(-.5,-1.1)}{.3}{0}{0}{$\xi_b$};
\roundNbox{unshaded}{(-.5,1.15)}{.35}{-.05}{-.05}{$f_b'$};
\node at (0,1.5) {\scriptsize{$n$}};
\node at (.7,.6) {\scriptsize{$d$}};
\node at (.7,-.6) {\scriptsize{$d$}};
\node at (-.6,.4) {\scriptsize{$b$}};
\node at (-.7,-.6) {\scriptsize{$b$}};
\node at (.2,.2) {\scriptsize{$f$}};
\node at (.2,-.2) {\scriptsize{$f$}};
\node at (0,.5) {\scriptsize{$\beta^\dag$}};
\node at (0,-.5) {\scriptsize{$\beta$}};
}%\todo{}
$$
under the semisimplicity isomorphism \eqref{eq:SemisimpleIso}.

We now show that the image of the unitary tensor functor $\Phi: \cZ(\cC) \to \fgpBim(M_\infty)$ lies in $\Chi(M_\infty)$ when $N$ is non-Gamma.

\begin{lem}
Since $N\subseteq M$ is finite depth, there is a $k>0$ such that every $c\in \Irr(\cC)$ is isomorphic to a summand of $(X\boxtimes \overline{X}))^{\boxtimes k}=X^{\alt\boxtimes 2k}$.
There is a finite subset $\{e_i\}_{i=1}^m\subseteq \cC^\op(X^{\alt \boxtimes 2k} \to z\boxtimes X^{\alt \boxtimes 2k})$ 
such that
\begin{equation}
\label{eq:zX2k-PPbasis}
\id_{z\boxtimes X^{\alt\boxtimes 2k}}
=
\tikzmath{
\draw (-.7,.15) -- node[above]{$\scriptstyle 2k$} (.7,.15);
\draw[thick, \zColor] (-.7,-.15) -- node[below]{$\scriptstyle z$} (.7,-.15);
}
=
\sum_{i=1}^m 
\tikzmath{
\draw (-1.2,.15) node[left]{$\scriptstyle 2k$} -- node[above]{$\scriptstyle 2k$} (1.2,.15) node[right]{$\scriptstyle 2k$};
\draw[thick, \zColor] (-1.2,-.15) node[left]{$\scriptstyle z$} -- (-.5,-.15);
\draw[thick, \zColor] (1.2,-.15) node[right]{$\scriptstyle z$} -- (.5,-.15);
\roundNbox{unshaded}{(-.5,0)}{.35}{-.05}{-.05}{$e_i$};
\roundNbox{unshaded}{(.5,0)}{.35}{-.05}{-.05}{$e_i^\dag$};
}\,.
\end{equation}
\end{lem}
\begin{proof}
For each $c\in \Irr(\cC)$, we pick
\begin{itemize}
\item
a finite family of isometries $\{\iota_c^i:c\to X^{\alt\boxtimes 2k}\}_{i=1}^{m_c}$ such that $(\iota_c^i)^\dag\cdot \iota_c^i=1_c$ and $\sum_{c}\sum_{i=1}^{m_c} \iota_c^i\cdot (\iota_c^i)^\dag =1_{X^{\alt\boxtimes 2k}}$, and
\item
an orthonormal basis $\{\alpha_c\} \subseteq \cC(c\to z\boxtimes X^{\alt\boxtimes 2k})$ under the isometry inner product, i.e., $\alpha_c^\dag \cdot \alpha_c' = \delta_{\alpha_c=\alpha'_c} \id_c$.
\end{itemize}
Then we have (reading diagrams right to left)
\[
\tikzmath{
\draw (-.7,.15) -- node[above]{$\scriptstyle 2k$} (.7,.15);
\draw[thick, \zColor] (-.7,-.15) -- node[below]{$\scriptstyle z$} (.7,-.15);
}
=
\sum_{c\in\Irr(\cC)}
\sum_{\alpha_c}
%\\ \alpha\in\ONB(c\to z\boxtimes X^{\alt\boxtimes 2k})}}
\tikzmath{
\draw (-.3,0) -- (.3,0);
\draw (-.3,0) arc (0:90:.3cm) -- (-.8,.3) node[left]{$\scriptstyle 2k$};
\draw[thick, \zColor] (-.3,0) arc (0:-90:.3cm) -- (-.8,-.3) node[left]{$\scriptstyle z$};
\draw (.3,0) arc (180:90:.3cm) -- (.8,.3) node[right]{$\scriptstyle 2k$};
\draw[thick, \zColor] (.3,0) arc (180:270:.3cm) -- (.8,-.3) node[right]{$\scriptstyle z$};
\filldraw (-.3,0) node[left]{$\scriptstyle \alpha_c$} circle (.05cm);
\filldraw (.3,0) node[right]{$\scriptstyle \alpha_c^\dag$} circle (.05cm);
\node at (0,-.2) {\scriptsize{$c$}};
}
=
\sum_{c\in\Irr(\cC)}
\sum_{\alpha_c}
%\\ \alpha\in\ONB(c\to z\boxtimes X^{\alt\boxtimes 2k})}}
\sum_{i=1}^{m_c}
\tikzmath{
\draw[dashed,rounded corners = 5] (-1.75,.5) rectangle (.2,-.5);
\draw (-1.2,0) -- (1.2,0);
\draw (-1.2,0) arc (0:90:.3cm) -- (-1.9,.3) node[left]{$\scriptstyle 2k$};
\draw[\zColor, thick] (-1.2,0) arc (0:-90:.3cm) -- (-1.9,-.3) node[left]{$\scriptstyle z$};
\draw (1.2,0) arc (180:90:.3cm) -- (1.7,.3) node[right]{$\scriptstyle 2k$};
\draw[\zColor] (1.2,0) arc (180:270:.3cm) -- (1.7,-.3) node[right]{$\scriptstyle z$};
\filldraw (1.2,0) node[right]{$\scriptstyle \alpha_c^\dag$} circle (.05cm);
\filldraw (-1.2,0) node[left]{$\scriptstyle \alpha_c$} circle (.05cm);
\roundNbox{unshaded}{(-.5,0)}{.3}{.05}{.05}{$\scriptstyle (\iota_c^i)^\dag$};
\roundNbox{unshaded}{(.65,0)}{.3}{0}{0}{$\scriptstyle \iota_c^i$};
\node at (-1.05,-.2) {\scriptsize{$c$}};
\node at (.05,-.2) {\scriptsize{$2k$}};
}\,.
\]
So we define our set $\{e_i\}$ to be $\bigcup_{c\in \Irr(\cC)}\{\alpha_c\cdot (\iota^i_c)^\dag\}_{i=1}^{m_c}$.
\end{proof}

\begin{prop}
\label{prop:-otimesXAI}
Let $\{e_i\}$ be as in \eqref{eq:zX2k-PPbasis} above.
For $n\geq 0$, define subsets $\{b_i\}$ and $\{b_i^{(n)}\}$ of $|X_z|$ by 
\begin{equation}
\label{eq:b_i^nOverS}
b_i
:=
\tikzmath{
\begin{scope}
\clip[rounded corners = 5] (-.7,0) -- (-.7,-1.4) -- (.7,-1.4) -- (.7,0);
\filldraw[\rColor] (-.7,0) -- (-.7,-1.4) -- (.7,-1.4) -- (.7,0);
\end{scope}
\draw[dotted] (-.7,0) -- (.7,0);
\draw[dashed] (0,-.7) -- (0,.7);
\draw (.7,.95) node[right]{$\scriptstyle 2k$} -- (-.7,.95) node[left]{$\scriptstyle 2k$};
\draw[\NColor,thick] (0,-1.4) -- (0,-.7);
\draw[\zColor, thick] (-.3,.65) -- (-.7,.65) node[left]{$\scriptstyle z$};
\roundNbox{unshaded}{(0,-.7)}{.3}{0}{0}{\scriptsize{$\Omega_N$}};
\roundNbox{unshaded}{(0,.8)}{.35}{-.05}{-.05}{$e_i$};
\node at (.2,.2) {\scriptsize{$1_\cC$}};
\node at (.2,-.2) {\scriptsize{$1_\cC$}};
}
\qquad\qquad
b_i^{(n)}
:=
\tikzmath{
\begin{scope}
\clip[rounded corners = 5] (-1.3,0) -- (-1.3,-1.4) -- (.7,-1.4) -- (.7,0);
\filldraw[\rColor] (-1.7,0) -- (-1.7,-1.4) -- (.7,-1.4) -- (.7,0);
\end{scope}
\draw[dotted] (-1.3,0) -- (.7,0);
\draw (.7,1.5) node[right]{$\scriptstyle 2k$} -- (-1.3,1.5) node[left]{$\scriptstyle 2k$};
\draw[\NColor,thick] (0,-1.4) -- (0,-.7);
\draw[\zColor, thick] (-.3,1.2) .. controls ++(180:.45cm) and ++(0:.45cm) .. (-1.1,.4) -- (-1.3,.4) node[left]{$\scriptstyle z$};
\draw[knot] (.7,.8) node[right]{$\scriptstyle 2n$} -- (-1.3,.8) node[left]{$\scriptstyle 2n$};
\draw[dashed] (0,-.7) -- (0,1.35);
\roundNbox{unshaded}{(0,-.7)}{.3}{0}{0}{\scriptsize{$\Omega_N$}};
\roundNbox{unshaded}{(0,1.35)}{.35}{-.05}{-.05}{$e_i$};
\node at (.2,.2) {\scriptsize{$1_\cC$}};
\node at (.2,-.2) {\scriptsize{$1_\cC$}};
}.
\end{equation}
Then $\{b_i\}$ is an $|X_z|$-basis and $\{b_i^{(n)}\}$ is an approximately inner $|X_z|_{M_\infty}$-basis.
This implies $|X_z|$ is approximately inner.
\end{prop}
\begin{proof}
The first claim is immediate from \eqref{eq:zX2k-PPbasis}.
Similarly, $\{b_i^{(n)}\}$ is an $|X_z|_{M_\infty}$-basis for every fixed $n$, and moreover, $[b_i^{(n)}, a]=0$ for all $a\in M_{2n}\subseteq M_\infty$.
Since $M_\infty = \varinjlim M_n$, for $a\in M_\infty$, $\|a-E_{M_n}(a)\|_2\to 0$
(e.g., see \cite[Lem.~B.7]{2010.01067}).
Then for all $a\in M_\infty$,
\begin{align*}
\|ab_i^{(n)}-b_i^{(n)}a\|_2 
&\le \|(a-E_{M_n}(a))b_i^{(n)}\|_2+\|E_{M_n}(a)b_i^{(n)}-b_i^{(n)}E_{M_n}(a)\|_2+\|b_i^{(n)}(E_{M_n}(a)-a)\|_2 \\
& \le 2\|b_i^{(n)}\|_2\|E_{M_n}(a)-a\|_2 \longrightarrow 0. 
\qedhere
\end{align*}
\end{proof}

\begin{assumption}
\label{assume:nonGamma}
For the remainder of this section, we now assume the $\rm II_1$ factor $N$ in our finite index finite depth subfactor $N\subseteq M$ is non-Gamma.
This implies $M$ is also non-Gamma by \cite[Prop.~1.11]{MR860811}.
\end{assumption}

\begin{lem}
\label{lem:CentralSequenceCommPPbasis}
Let $\{b_i\}$ be as in \eqref{eq:b_i^nOverS}.
For each central sequence $(a_n)\subseteq N_\infty$, $\|a_nb_i-b_ia_n\|_2\to 0$.
\end{lem}
\begin{proof}
Since $N\subseteq M$ is finite depth
and $N$ is non-Gamma,
by \cite[Prop.~3.2(3)]{MR2661553}, every Jones basic construction $M_n$ has spectral gap in $M_\infty$ for every $n$.
This implies that $\|a_n-E_{M_{k+1}'\cap M_\infty}(a_n)\|_2\to 0$.
Since $[E_{M_{k+1}'\cap M_\infty}(a_n),b_i]=0$, we have
\begin{align*}
\|a_nb_i-b_ia_n\|_2
&\le \|(a_n-E_{M_{k+1}'\cap M_\infty}(a_n))b_i\|_2
+
\|E_{M_{k+1}'\cap M_\infty}(a_n)b_i-b_iE_{M_{k+1}'\cap M_\infty}(a_n)\|_2
\\
&\quad
+\|b_i(E_{M_{k+1}'\cap M_\infty}(a_n)-a_n)\|_2 
\\
& \le \|b_i\|_2\cdot \|a_n-E_{M_{k+1}'\cap M_\infty}(a_n)\|_2
+
\|a_n-E_{N_{k+1}'\cap M_\infty}(a_n)\|_2\cdot \|b_i\|_2
\longrightarrow 
0. \qedhere
\end{align*}
\end{proof}

\begin{prop}
\label{prop:-otimesCT}
$|X_z|$ is centrally trivial.
\end{prop}
\begin{proof}
By Proposition \ref{prop:BimodualCT&CS},
we must show that for each central sequence $(a_n)\subseteq M_\infty$, 
$\|a_nx-xa_n\|_2\to 0$, for every $x\in |X_z|$.
Suppose $(a_n)$ is such a central sequence.
Let $\{b_i\}$ be the $|X_z|_{M_\infty}$-basis as in \eqref{eq:b_i^nOverS}, and let $K>0$ such that $\|b_i\|_2\le K$ for all $i$.
By Lemma \ref{lem:CentralSequenceCommPPbasis},
\begin{align*}
\|a_nx-xa_n\|_2 
&= 
\left\|a_n \sum_i b_i\langle b_i|x\rangle^{|X_z|}_{M_\infty} \right. - \left.\sum_i b_i\langle b_i|x\rangle^{|X_z|}_{M_\infty} a_n\right\|_2 \\& \le 
\left\|\sum_i (a_n b_i- b_ia_n)\langle b_i|x\rangle^{|X_z|}_{M_\infty}\right\|_2 
+ 
\left\|\sum b_i\left(a_n\langle b_i|x\rangle^{|X_z|}_{M_\infty}-\langle b_i|x\rangle^{|X_z|}_{M_\infty} a_n\right)\right\|_2 
\\& \le 
K\|x\|_2 \cdot \sum_i\|a_n b_i- b_i a_n\|_2 
+ 
K\sum_i\left\|a_n\langle b_i|x\rangle^{|X_z|}_{M_\infty} -\langle  b_i|x\rangle^{|X_z|}_{M_\infty} a_n\right\|_2 \\
& \longrightarrow 0. \qedhere
\end{align*}
\end{proof}

Combining the statements of Propositions \ref{prop:-otimesXAI} and \ref{prop:-otimesCT},
our unitary tensor functor $\Phi: \cZ(\cC) \to \fgpBim(M_\infty)$ lands in $\Chi(M_\infty)$.

\begin{prop}
The unitary tensor functor $\Phi:\cZ(\cC)\to \Chi(M_\infty)$ is braided, i.e., for $z,w\in\cZ(\cC)$, the following diagram commutes
\begin{equation}
\label{eq:BraidedFunctor}
\begin{tikzcd}[column sep=4em]
{|X_z|\boxtimes_{M_\infty} |X_w|}
\arrow["u_{|X_z|,|X_w|}", rightarrow]{r} 
\arrow["\Phi^2_{z,w}"', rightarrow]{d} 
&
{|X_w|\boxtimes_{M_\infty} |X_z|} 
\arrow["\Phi^2_{w,z}", rightarrow]{d} 
\\
{|X_{z\otimes w}|} 
\arrow["\Phi(\sigma_{z,w})"', rightarrow]{r} 
&
{|X_{w\otimes z}|}
\end{tikzcd}
\end{equation}
% $\Phi^2_{w,z}\cdot u_{|X_z|,|X_w|} = \Phi(\sigma_{z,w})\cdot\Phi^2_{z,w}$.
\end{prop}
\begin{proof}
Let $\{b_i^{(n)}\}\subseteq |X_z|$ and $\{c_j\}\subseteq |X_w|$ be approximately inner $|X_z|_{M_\infty}$-basis and $|X_w|_{M_\infty}$-basis respectively as in \eqref{eq:b_i^nOverS}.
By \eqref{eq:uXY}, 
$u_{|X_z|,|X_w|} = \lim_n \sum_{i,j} |c_j\boxtimes b_i^{(n)}\rangle \langle b_i^{(n)}\boxtimes c_j|$. 
For $x=(f_a\otimes \xi_a)\in |X_z|$ and $y=(g_b\otimes \eta_b)\in |X_w|$ with $f_a\in \cC(X^{\alt\boxtimes 2t}\to a\boxtimes X^{\alt\boxtimes 2t})$
and
$g_b\in \cC(X^{\alt\boxtimes 2t}\to b\boxtimes X^{\alt\boxtimes 2t})$ for $t$ sufficiently large,
we have
\begin{align*}
\sum_{i,j}
\Phi^2_{w,z}(|c_j\boxtimes b_i^{(n)}\rangle 
&
\langle b_i^{(n)}\boxtimes c_j|x\boxtimes y\rangle_{M_\infty}^{|X_z|\boxtimes_{M_\infty}|X_w|} 
)
\\&=
\sum_{i,j}
\sum_{\substack{
c\in \Irr(\cC)
\\
\alpha\in\ONB(a\otimes b\to c)}}
\tikzmath{
\begin{scope}
\clip[rounded corners = 5] (-2.3,-2.7) rectangle (6,.1);
\filldraw[\rColor] (-2.3,-2.7) rectangle (6,-.4);
\end{scope}
\draw[dotted] (-2.3,-.4) -- (6,-.4);
\draw[\NColor,thick] (4,-1.8) arc (-180:0:.65cm and .5cm);
\draw[\NColor, thick] (4.65,-2.3) -- (4.65,-2.7);
\filldraw[\NColor] (4.65,-2.3) circle (.05cm);
\draw[\zColor, thick] (.2,1.7) .. controls ++(180:.45cm) and ++(0:.45cm) .. (-.8,.2) -- (-1.3,.2) .. controls ++(180:.45cm) and ++(0:.45cm) .. (-2.3,.6) node[left]{$\scriptstyle z$};
\draw[\wColor, thick, knot] (-1.3,.6) .. controls ++(180:.45cm) and ++(0:.45cm) .. (-2.3,.2) node[left]{$\scriptstyle w$};
\draw[\zColor, thick, knot] (2.8,1.7) .. controls ++(0:.45cm) and ++(180:.45cm) .. (3.7,.6);
\draw[\wColor, thick, knot] (1.8,.6) .. controls ++(0:.45cm) and ++(180:.45cm) .. (2.8,.2) -- (4,.2) .. controls ++(0:.45cm) and ++(180:.45cm) .. (5,.6);
\draw[knot] (-2.3,2.4) -- (6,2.4);
\draw[knot] (-2.3,2) node[left]{$\scriptstyle 2k$} -- (6,2);
\draw[knot] (-2.3,1.3) node[left]{$\scriptstyle 2n$} -- (6,1.3);
\draw[knot] (-2.3,.9) node[left]{$\scriptstyle 2k$} -- (6,.9);
\draw[knot] (4,.4) arc (-180:0:.65cm and .5cm);
\draw (4,-1.2) arc (180:0:.65cm and .5cm);
\draw (4.65,-.7) -- (4.65,-.1);
\filldraw[black] (4.65,-.1) circle (.05cm);
\filldraw[black] (4.65,-.7) circle (.05cm);
\roundNbox{unshaded}{(-1,.75)}{.35}{-.05}{-.05}{$c_j$};
\roundNbox{unshaded}{(.5,1.85)}{.35}{-.05}{-.05}{$b_i$};
\roundNbox{unshaded}{(1.5,.75)}{.35}{-.05}{-.05}{$c_j^\dag$};
\roundNbox{unshaded}{(2.5,1.85)}{.35}{-.05}{-.05}{$b_i^\dag$};
\roundNbox{unshaded}{(4,1.5)}{1.1}{-.8}{-.8}{$f_a$};
\roundNbox{unshaded}{(5.3,1.5)}{1.1}{-.8}{-.8}{$g_b$};
\roundNbox{unshaded}{(4,-1.5)}{.3}{0}{0}{$\xi_a$};
\roundNbox{unshaded}{(5.3,-1.5)}{.3}{0}{0}{$\eta_b$};
\node at (3.9,.05) {\scriptsize{$a$}};
\node at (3.9,-.85) {\scriptsize{$a$}};
\node at (5.4,.1) {\scriptsize{$b$}};
\node at (5.4,-.9) {\scriptsize{$b$}};
\node at (4.8,-.25) {\scriptsize{$c$}};
\node at (4.8,-.55) {\scriptsize{$c$}};
\node at (4.65,.1) {\scriptsize{$\alpha^\dag$}};
\node at (4.65,-.9) {\scriptsize{$\alpha$}};
}
\displaybreak[1]\\&=
\sum_{i,j}
\sum_{\substack{
c\in \Irr(\cC)
\\
\alpha\in\ONB(a\otimes b\to c)}}
\tikzmath{
\begin{scope}
\clip[rounded corners = 5] (-1.9,-2.7) rectangle (6,.1);
\filldraw[\rColor] (-1.9,-2.7) rectangle (6,-.4);
\end{scope}
\draw[dotted] (-1.9,-.4) -- (6,-.4);
\draw[\NColor,thick] (4,-1.8) arc (-180:0:.65cm and .5cm);
\draw[\NColor, thick] (4.65,-2.3) -- (4.65,-2.7);
\filldraw[\NColor] (4.65,-2.3) circle (.05cm);
\draw[\zColor, thick, knot] (-.8,1.7) .. controls ++(180:.45cm) and ++(0:.45cm) ..  (-1.8,.6) -- (-1.9,.6) node[left]{$\scriptstyle z$};
\draw[\zColor, thick, knot] (2.8,1.7) .. controls ++(0:.45cm) and ++(180:.45cm) .. (3.7,.6);
\draw[\wColor, thick, knot] (.2,.6) .. controls ++(180:.45cm) and ++(0:.45cm) .. (-.8,.2) -- (-1.9,.2) node[left]{$\scriptstyle w$};
\draw[\wColor, thick, knot] (1.8,.6) .. controls ++(0:.45cm) and ++(180:.45cm) .. (2.8,.2) -- (4,.2) .. controls ++(0:.45cm) and ++(180:.45cm) .. (5,.6);
\draw[knot] (-1.9,2.4) -- (6,2.4);
\draw[knot] (-1.9,2) node[left]{$\scriptstyle 2k$} -- (6,2);
\draw[knot] (-1.9,1.3) node[left]{$\scriptstyle 2n$} -- (6,1.3);
\draw[knot] (-1.9,.9) node[left]{$\scriptstyle 2k$} -- (6,.9);
\draw[knot] (4,.4) arc (-180:0:.65cm and .5cm);
\draw (4,-1.2) arc (180:0:.65cm and .5cm);
\draw (4.65,-.7) -- (4.65,-.1);
\filldraw[black] (4.65,-.1) circle (.05cm);
\filldraw[black] (4.65,-.7) circle (.05cm);
\roundNbox{unshaded}{(-.5,1.85)}{.35}{-.05}{-.05}{$b_i$};
\roundNbox{unshaded}{(.5,.75)}{.35}{-.05}{-.05}{$c_j$};
\roundNbox{unshaded}{(1.5,.75)}{.35}{-.05}{-.05}{$c_j^\dag$};
\roundNbox{unshaded}{(2.5,1.85)}{.35}{-.05}{-.05}{$b_i^\dag$};
\roundNbox{unshaded}{(4,1.5)}{1.1}{-.8}{-.8}{$f_a$};
\roundNbox{unshaded}{(5.3,1.5)}{1.1}{-.8}{-.8}{$g_b$};
\roundNbox{unshaded}{(4,-1.5)}{.3}{0}{0}{$\xi_a$};
\roundNbox{unshaded}{(5.3,-1.5)}{.3}{0}{0}{$\eta_b$};
\node at (3.9,.05) {\scriptsize{$a$}};
\node at (3.9,-.85) {\scriptsize{$a$}};
\node at (5.4,.1) {\scriptsize{$b$}};
\node at (5.4,-.9) {\scriptsize{$b$}};
\node at (4.8,-.25) {\scriptsize{$c$}};
\node at (4.8,-.55) {\scriptsize{$c$}};
\node at (4.65,.1) {\scriptsize{$\alpha^\dag$}};
\node at (4.65,-.9) {\scriptsize{$\alpha$}};
}
\displaybreak[1]\\&=
\sum_{\substack{
c\in \Irr(\cC)
\\
\alpha\in\ONB(a\otimes b\to c)}}
\tikzmath{
\begin{scope}
\clip[rounded corners = 5] (3.3,-2.7) rectangle (6,.1);
\filldraw[\rColor] (2.7,-2.7) rectangle (6,-.4);
\end{scope}
\draw[dotted] (3.3,-.4) -- (6,-.4);
\draw[\NColor,thick] (4,-1.8) arc (-180:0:.65cm and .5cm);
\draw[\NColor, thick] (4.65,-2.3) -- (4.65,-2.7);
\filldraw[\NColor] (4.65,-2.3) circle (.05cm);
\draw (3.3,.9) node[left]{$\scriptstyle 2t$} -- (6,.9);
\draw[\zColor, thick] (3.7,.6) -- (3.3,.6) node[left]{$\scriptstyle z$}; 
\draw[\wColor, thick, knot] (5,.6) .. controls ++(180:.45cm) and ++(0:.45cm) .. (4,.2) -- (3.7,.2) -- (3.3,.2) node[left]{$\scriptstyle w$};
\draw[knot] (4,.4) arc (-180:0:.65cm and .5cm);
\draw (4,-1.2) arc (180:0:.65cm and .5cm);
\draw (4.65,-.7) -- (4.65,-.1);
\filldraw[black] (4.65,-.1) circle (.05cm);
\filldraw[black] (4.65,-.7) circle (.05cm);
\roundNbox{unshaded}{(4,.75)}{.35}{-.05}{-.05}{$f_a$};
\roundNbox{unshaded}{(5.3,.75)}{.35}{-.05}{-.05}{$g_b$};
\roundNbox{unshaded}{(4,-1.5)}{.3}{0}{0}{$\xi_a$};
\roundNbox{unshaded}{(5.3,-1.5)}{.3}{0}{0}{$\eta_b$};
\node at (3.9,.05) {\scriptsize{$a$}};
\node at (3.9,-.85) {\scriptsize{$a$}};
\node at (5.4,.1) {\scriptsize{$b$}};
\node at (5.4,-.9) {\scriptsize{$b$}};
\node at (4.8,-.25) {\scriptsize{$c$}};
\node at (4.8,-.55) {\scriptsize{$c$}};
\node at (4.65,.1) {\scriptsize{$\alpha^\dag$}};
\node at (4.65,-.9) {\scriptsize{$\alpha$}};
}
\\&=
\sum_{\substack{
c\in \Irr(\cC)
\\
\alpha\in\ONB(a\otimes b\to c)}}
\tikzmath{
\begin{scope}
\clip[rounded corners = 5] (1.7,-2.7) rectangle (6,.1);
\filldraw[\rColor] (1.7,-2.7) rectangle (6,-.4);
\end{scope}
\draw[dotted] (1.7,-.4) -- (6,-.4);
\draw[\NColor,thick] (4,-1.8) arc (-180:0:.65cm and .5cm);
\draw[\NColor, thick] (4.65,-2.3) -- (4.65,-2.7);
\filldraw[\NColor] (4.65,-2.3) circle (.05cm);
\draw (1.7,.9) node[left]{$\scriptstyle 2t$} -- (6,.9);
\draw[\wColor, thick] (5,.6) .. controls ++(180:.45cm) and ++(0:.45cm) .. (4,.2) -- (3.7,.2) .. controls ++(180:.45cm) and ++(0:.45cm) .. (2.7,.6) .. controls ++(180:.45cm) and ++(0:.45cm) .. (1.7,.2) node[left]{$\scriptstyle w$};
\draw[\zColor, thick, knot] (3.7,.6) .. controls ++(180:.45cm) and ++(0:.45cm) .. (2.7,.2) .. controls ++(180:.45cm) and ++(0:.45cm) .. (1.7,.6) node[left]{$\scriptstyle z$}; 
\draw[knot] (4,.4) arc (-180:0:.65cm and .5cm);
\draw (4,-1.2) arc (180:0:.65cm and .5cm);
\draw (4.65,-.7) -- (4.65,-.1);
\filldraw[black] (4.65,-.1) circle (.05cm);
\filldraw[black] (4.65,-.7) circle (.05cm);
\roundNbox{unshaded}{(4,.75)}{.35}{-.05}{-.05}{$f_a$};
\roundNbox{unshaded}{(5.3,.75)}{.35}{-.05}{-.05}{$g_b$};
\roundNbox{unshaded}{(4,-1.5)}{.3}{0}{0}{$\xi_a$};
\roundNbox{unshaded}{(5.3,-1.5)}{.3}{0}{0}{$\eta_b$};
\node at (3.9,.05) {\scriptsize{$a$}};
\node at (3.9,-.85) {\scriptsize{$a$}};
\node at (5.4,.1) {\scriptsize{$b$}};
\node at (5.4,-.9) {\scriptsize{$b$}};
\node at (4.8,-.25) {\scriptsize{$c$}};
\node at (4.8,-.55) {\scriptsize{$c$}};
\node at (4.65,.1) {\scriptsize{$\alpha^\dag$}};
\node at (4.65,-.9) {\scriptsize{$\alpha$}};
\draw[dashed,rounded corners = 5] (1.7,.05) rectangle (2.7,.75);
}
\\&=
\Phi(\sigma_{z,w})\cdot \Phi^2_{z,w} (x\boxtimes y).
\end{align*}
Thus \eqref{eq:BraidedFunctor} commutes on a dense subspace of $|X_z|\boxtimes_{M_\infty}|X_w|$, and the result follows.
\end{proof}

\begin{cor}
Let $\cC$ be a braided fusion category,
then there exists a $\mathrm{II}_1$ factor $M$ with a braided fully faithful monoidal functor $\cC\hookrightarrow \cZ(\cC)\to \Chi(M)$.
\end{cor}

%%%%%%%%%%%%%%%%%%%%%%%%%%%%%%%%%%%%%%%%%%%%%%%%%%%%%%%%%%%%%%%%%%%%%%%
\subsection{Calculation of \texorpdfstring{$\Chi(M_\infty)$}{chi(Minfty)} when \texorpdfstring{$N\subseteq M$}{N in M} is non-Gamma and finite depth}

As in \S\ref{sec:InductiveLimit} above we suppose $N\subseteq M$ is a fixed finite depth finite index $\rm II_1$ subfactor.
We also continue Assumption \ref{assume:nonGamma} that $N$ is non-Gamma.
We now prove our main result, which uses a technical result on centralizers in braided tensor categories in \S\ref{sec:TechnicalBTC} below.

\begin{thm}
\label{thm:nonGamma}
Let $N\subseteq M$ be a finite depth finite index inclusion of non-Gamma $\rm II_1$ factors.
Let $M_\infty$ denote the inductive limit $\rm II_1$ factor from the Jones tower, and let $\cC={}_N\cC_N$ be the even part of the standard invariant $\cC(N\subseteq M)$. 
Then $\Chi(M_\infty)\cong \cZ(\cC)$.
\end{thm}
\begin{proof}
Consider our construction of $\Phi:\cZ(\cC)\to\Chi(M_\infty)$.
Let $L:=I(1_\cC)\in \cZ(\cC)$ be the canonical Lagrangian algebra, where $I: \cC\to \cZ(\cC)$ is adjoint to the forgetful functor. 
By Example \ref{ex:MoritaEquivalenceBetweenCopCandZ(C)}, the Q-systems $L\in \cZ(\cC)$ and $S\in \cC^{\op}\boxtimes \cC$ are related as in Remark \ref{rem:MoritaEquivalenceRealization}, i.e., $S = X\otimes \overline{X}$ and $L=\overline{X}\otimes X$ for $X=1_\cC$ in the Morita equivalence $\cC^{\op}\boxtimes \cC - \cZ(\cC)$ bimodule category $\cC$.
Thus we can identify $|\Phi(L)|$ with the basic construction of $R\otimes N\subseteq M_{\infty}$ by the discussion in Example \ref{ex:RealizedJonesTower}.
But this implies that $|\Phi(L)|$ is Morita equivalent to $R\otimes N$, and in particular, by Example \ref{ex:Chi_fus(RboxtimesN)},
$\Chi(|\Phi(L)|)=\Chi(R\otimes N)$ is trivial.
By Proposition \ref{thm:LocalExtension}, $\Chi(M_\infty)_{\Phi(L)}^{\loc}\cong \Chi(|\Phi(L)|)$.
This implies no non-trivial simple object in $\Chi(M_\infty)$ centralizes $\Phi(\cZ(\cC))$,
since the free module functor $x\mapsto x\otimes L$ for $x\in\Phi(\cZ(\cC))'\subseteq \Chi(M_\infty)$ is fully faithful.
Thus $\Phi(\cZ(\cC))'\subseteq \Chi(M_\infty)$ is trivial.
Since $\cZ(\cC)$ is non-degenerately braided by \cite{MR1966525}, 
by Proposition \ref{EXISTCENTRAL} in the next subsection,
$\Phi(\cZ(\cC))=\Chi(M_\infty)$.
\end{proof}

\begin{rem}
Kawahigashi studied a relative version of Connes' $\chi(M)$ and Jones $\kappa$ invariant for finite index $\rm II_1$ subfactors $N\subseteq M$ \cite{MR1230287, MR1321700}. 
In particular, Kawahigashi provides bounds and computations for relative $\chi$ for finite depth finite index subfactors of the hyperfinite $\rm{II}_{1}$ factor. 
For a given finite depth \textit{hyperfinite} subfactor $N\subseteq M$, there exists a non-Gamma inclusion $A\subseteq B$ with the same standard invariant \cite{MR2051399}. 
By \cite[Thm.~4.2]{MR2661553}, $\chi(A_{\infty})\cong \chi(N\subseteq M)$. 
By Theorem \ref{thm:nonGamma} and Example \ref{ex:ConnesChi}, 
$\chi(N\subseteq M)\cong \Inv(\cZ(\cC(N\subseteq M)))$, the group of isomorphism classes of invertible objects of the Drinfeld center of $\cC(N\subseteq M)$.
\end{rem}

\begin{rem}
Suppose $N\subseteq M$ is a finite index inclusion of non-Gamma $\rm II_1$ factors with $A_{2n}$ Jones-Temperley-Lieb standard invariant. 
Then $\cZ(\cC(N\subseteq M))$ is a unitary modular tensor category with no invertible objects. 
This distinguishes the corresponding $M_{\infty}$ factors pairwise, but they all have the same trivial Connes' $\chi$ invariant. 
Popa considers these examples of $\rm{II}_{1}$ factors with trivial $\chi$ which are not s-McDuff in \cite{MR2661553}, answering a question of Connes in the negative. 
(Recall a $\rm{II}_{1}$ factor is \emph{s-McDuff} if it is of the form $R\otimes N$ for $N$ non-Gamma). 
This leads us to ask the natural extension of Connes question.
\end{rem}

\begin{quest}
If $M$ is McDuff and $\Chi(M)$ is trivial, is $M$ s-McDuff?
\end{quest}

%%%%%%%%%%%%%%%%%%%%%%%%%%%%%%%%%%%%%%%%%%
\subsection{A technical result on centralizers in braided tensor categories}
\label{sec:TechnicalBTC}

The goal of this section is to prove the following technical result for braided unitary tensor categories.
We expect this result holds in the greater generality of semisimple ribbon tensor categories; see the paragraph before \cite[Prop.~2.5]{MR1990929} for more details in this direction.

\begin{prop}\label{EXISTCENTRAL}
Suppose $\cC$ is an arbitrary braided unitary tensor category and $\cD\subsetneq \cC$ is a non-degenerately braided proper fusion subcategory. 
There exists $a\in \Irr(\cC)\backslash \Irr(\cD)$ which centralizes $\cD$. 
\end{prop}

Let $\cC$ be a braided unitary tensor category. 
Let $\tr_\cC$ denote the (unnormalized) categorical trace corresponding to the unique unitary spherical structure from minimal solutions to the conjugate equations \cite{MR1444286}.
(The normalization is
$\tr_\cC(\id_c) = \dim(c)$ for each $c\in \cC$.)
For each $a\in \Irr(\cC)$ define a function $\gamma_{a}: \Irr(\cC)\rightarrow \bbC$ by 
$$
\gamma_{a}(b)
:=\frac{1}{d_{b}}\tr_\cC(\sigma_{b,a}\sigma_{a,b}).
$$
Then $\gamma_{a}(b)$ extends linearly to a \textit{character} on the fusion ring, i.e.
$$
\gamma_{a}(b)\gamma_{a}(c)=\sum_{d} N^{d}_{bc}\gamma_{a}(d).
$$
Furthermore, characters are related by the following equation:
\begin{equation}
\label{eq:CharacterProductExpansion}
\frac{\gamma_{a}(c)\gamma_{b}(c)}{d_{c}}=\sum_{e} \frac{d_{e}}{d_{a}d_{b}}N^{c}_{ab} \gamma_{e}(c).
\end{equation}

Suppose $\cD$ is a non-degenerately braided full fusion subcategory of $\cC$ (which is thus modular by unitarity).
Then by non-degeneracy, $\{\gamma_{a}\}_{a\in \Irr(\cD)}$ forms a complete set of characters of $K_0(\cD)$. 
But any $b\in \Irr(\cC)$ also defines a character $\gamma_{b}$, and thus $\gamma_{b}|_{\Irr(\cD)}=\gamma_{f(b)}$ for some uniquely defined $f(b)\in \Irr(\cD)$. 
Thus we have a function $f:\Irr(\cC)\rightarrow \Irr(\cD)$.

Now we define the \emph{fusion hypergroup} of a semisimple unitary tensor category $\cC$ as the fusion algebra $K_0(\cC)$ with the distinguished basis $\lambda_{a}=\frac{[a]}{d_{a}}$. 
We then have 
$$
\lambda_{a}\lambda_{b}=\sum_{c} M^{c}_{ab} \lambda_{c}
\qquad\qquad
\text{where}
\qquad\qquad
M^{c}_{ab}:=\frac{d_{c}}{d_{a}d_{b}} N^{c}_{ab}
.
$$

\begin{lem}
The assignment $f(\lambda_{a}):= \lambda_{f(a)}$ extends to a homomorphism of fusion algebras $K_0(\cC)\rightarrow K_0(\cD)$.
\end{lem}

\begin{proof}
For $x\in \Irr(\cD)$ and $a,b\in \Irr(\cC)$, 
we compute $d_x^{-1}\gamma_{a}(x)\gamma_{b}(x)$ in two ways.
First, we can apply \eqref{eq:CharacterProductExpansion} and then swap $\gamma_{c}$ with $\gamma_{f(c)}$, or 
we can swap
$\gamma_{a}, \gamma_{b}$ with $\gamma_{f(a)},\gamma_{f(b)}$ respectively and then apply \eqref{eq:CharacterProductExpansion}.
Equating these two computations gives the equality
$$
\sum_{c\in \Irr(\cC)} \frac{d_{c}}{d_{a}d_{b}}N^{c}_{ab} \gamma_{f(c)}(x)=\sum_{y\in \Irr(\cD)} \frac{d_{y}}{d_{f(a)}d_{f(b)}} N^{y}_{f(a)f(b)}\gamma_{y}(x),
$$
which implies
\begin{equation}
\label{eq:HypergroupIdentity}
\sum_{y\in \Irr(\cD)}  \left( \sum_{c\in f^{-1}(y)} M^{c}_{ab}  -M^{y}_{f(a)f(b)}\right) \gamma_{y}(x) =0.
\end{equation}
Since \eqref{eq:HypergroupIdentity} holds for all $x\in \Irr(\cD)$, we have 
$$
\sum_{y\in \Irr(\cD)}  \left( \sum_{c\in f^{-1}(y)} M^{c}_{ab}  -M^{y}_{f(a)f(b)}\right) \gamma_{y}=0,
$$
which is an equation in the space of functions on $\Irr(\cD)$. 
But since $\{\gamma_{y}\}_{y\in \Irr(\cD)}$ is a complete set of characters for the fusion algebra $K_0(\cD)$, it forms a basis for the space $\Fun(\Irr(\cD)\to \cC)$ (where we idenitfy $\Fun(\Irr(\cD)\to \cC)$ with the dual space $K_0(\cD)^{\vee}$), and is thus linearly independent. 
This immediately implies
$$
\sum_{c\in f^{-1}(y)} M^{c}_{ab}  =M^{y}_{f(a)f(b)}
\qquad\qquad
\forall\,y\in \Irr(\cD).
$$
We then see that 
$$
f(\lambda_{a})f(\lambda_{b})=\lambda_{f(a)} \lambda_{f(b)}=\sum_{y} M^{y}_{f(a)f(b)}\lambda_{y}
=
\sum_{y} \left(\sum_{c\in f^{-1}(y)} M^{c}_{ab} \right) \lambda_{y}
=
\sum_{c} M^{c}_{ab} \lambda_{f(c)}
=
f(\lambda_{a}\lambda_{b})
$$
as claimed.
\end{proof}

\begin{proof}[Proof of Proposition~\ref{EXISTCENTRAL}]
We prove the contrapositive.
That is, we will show that if $\cC$ is a braided unitary tensor category and
$\cD\subseteq \cC$ is a non-degenerately braided fusion subcategory, then the absence of a nontrivial centralizing simple object for $\cD$ in $\Irr(\cC)$ implies $\cC=\cD$.

First, by \cite[Prop.~2.5]{MR1990929},  $c\in \Irr(\cC)$ centralizes $\cD$ if and only if $\gamma_{c}|_{\cD}=\gamma_{f(c)}=\gamma_{1_\cD}$.
Suppose that the only $c\in \Irr(\cC)$ for which $f(c)=1_\cD$ is $c=1_\cC$. 
Let $\tau_\cC$ be the functional on $K_0(\cC)$ which picks off the coefficient of the identity object, and similarly define $\tau_\cD$ on $K_0(\cD)$. 
Note that $\tau_\cC,\tau_\cD$ are positive definite on the $*$-algebras $K_0(\cC),K_0(\cD)$ respectively.
For any $\eta\in K_0(\cC)$, our hypothesis implies $\tau(\eta)=\tau(f(\eta))$, where we have 
extended $f$ linearly. 
But then $f:K_0(\cC)\rightarrow K_0(\cD)$ is injective, since if
$f(\eta)=0$, then $0=f(\eta^{*})f(\eta)=f(\eta^{*}\eta)$, and thus $0=\tau_\cD(f(\eta^{*}\eta))=\tau_\cC(\eta^{*}\eta)$, which implies $\eta=0$. 
In particular, 
$$
\operatorname{rank}(\cC)=\dim(K_0(\cC))\le \dim(K_0(\cD))=\operatorname{rank}(\cD).
$$ 
But as $\cD\subseteq \cC$, we must have $\cD=\cC$.
\end{proof}

%%%%%%%%%%%%%%%%%%%%%%%%%%%%%%%%%%%%%%%%%%%%%%%%%%%%%%%%%%%%%%%%%%%%%%%%%%%%%%%%%%%%%%%%%%
\bibliographystyle{alpha}
{\footnotesize{
\bibliography{bibliography}
%\bibliography{../../../Documents/research/penneys/bibliography}
}}
\end{document}